\documentclass[a4paper,10pt]{article}
\newcommand{\smallsum}{\textstyle\sum}

\newcommand{\tzero}{0}

\usepackage{amsthm,amsmath,amssymb}
\usepackage{bbm,stmaryrd}
\usepackage{epsfig}
\usepackage{enumerate}
\usepackage{geometry}

\newtheorem{remark}{Remark}[section]
\newtheorem{lemma}{Lemma}[section]
\newtheorem{theorem}[lemma]{Theorem}

\newtheorem{prop}[lemma]{Proposition}
\newtheorem{corollary}[lemma]{Corollary}

\providecommand{\eps}{{\ensuremath{\varepsilon}}}

\providecommand{\N}{{\ensuremath{\mathbbm{N}}}}
\providecommand{\R}{{\ensuremath{\mathbbm{R}}}}
\providecommand{\E}{{\ensuremath{\mathbb{E}}}}
\renewcommand{\P}{{\ensuremath{\mathbb{P}}}}

\providecommand{\N}{{\ensuremath{\mathbbm{N}}}}
\providecommand{\R}{{\ensuremath{\mathbbm{R}}}}

\providecommand{\E}{{\ensuremath{\mathbb{E}}}}
\renewcommand{\P}{{\ensuremath{\mathbb{P}}}}

\providecommand{\HS}{{\ensuremath{\textup{HS}}}}

\title{On a perturbation theory and on strong convergence rates 
for stochastic ordinary and partial differential equations
with non-globally monotone coefficients}

\begin{document}
\author{Martin Hutzenthaler and Arnulf Jentzen}
\maketitle
\begin{abstract}
\end{abstract}
  We develope a perturbation theory for stochastic differential equations (SDEs)
%   of evolutionary type 
  by which we mean both
%   stochastic evolution equations (SEEs) including, for example,
  stochastic ordinary differential equations (SODEs)
  and stochastic partial differential equations (SPDEs).
  In particular, we estimate the $ L^p $-distance
  between the solution process of an SDE and an
  arbitrary It\^o process, which we view as a perturbation of the solution process of the SDE, 
  by the $ L^q $-distances of the differences of the local characteristics
  for suitable $ p, q > 0 $.
  As application of our perturbation theory,
  we establish 
  strong convergence rates 
%   the strong convergence rate $ \tfrac{ 1 }{ 2 } $ 
  for numerical approximations of a class of
  SODEs with non-globally monotone coefficients.
%  including
%  the stochastic Lorenz equation with bounded noise,
%  the stochastic van der Pol oscillator
%  and the stochastic Duffing
  As another application of our perturbation theory, we prove strong convergence rates
  for spectral Galerkin approximations of solutions
  of semilinear SPDEs with non-globally monotone nonlinearities
  including Cahn-Hilliard-Cook type equations
  and stochastic Burgers equations.
  Further applications of the perturbation theory include 
  the regularity of solutions of SDEs with respect to the initial values 
  and small-noise analysis for ordinary and partial differential equations.
\tableofcontents
%\newpage

\section{Introduction}
In this article we develop 
a \emph{perturbation theory}
for stochastic differential equations (SDEs)
by which we mean both
stochastic ordinary differential equations (SODEs)
and
stochastic partial differential equations (SPDEs).
To illustrate this perturbation theory,
we use the following setting
in this introductory section.
Let 
$
  ( H, \left< \cdot , \cdot \right>_H, \left\| \cdot \right\|_H )
$
and
$
  ( U, \left< \cdot , \cdot \right>_U, \left\| \cdot \right\|_U )
$
be separable $ \R $-Hilbert spaces,
let 
$ D \subseteq H $ be a Borel measurable set,
let
$ 
  \mu \colon D \to H
$
and
$
  \sigma \colon D \to HS( U, H ) 
$
be Borel measurable functions,
let 
$ T \in (0,\infty) $,
let
$ 
  ( 
    \Omega, \mathcal{F}, \P,  
    ( \mathcal{F}_t )_{ t \in [ 0 , T ] }
  ) 
$
be a stochastic basis,
let
$
  ( W_t )_{ t \in [ 0, T ] }
$
be a cylindrical $ I_U $-Wiener process
with respect to $ ( \mathcal{F}_t )_{ t \in [ 0, T ] } $,
%let $ V = ( V(x, y) )_{ (x, y) \in H^2 } \in C^2( H^2, \R ) $,
let 
$ X, Y \colon [ 0 , T ] \times \Omega \to D $
be adapted stochastic processes
with continuous sample paths
and let
$ a \colon [ 0, T ] \times \Omega \to H $
and
$ b \colon [ 0, T ] \times \Omega \to HS( U, H ) $
be predictable stochastic processes
with 
$
  \int_{ 0 }^T
  \| a_s \|_H
  +
  \| b_s \|_{ HS( U, H ) }^2
  +
  \| \mu( X_s ) \|_H
  +
  \| \sigma( X_s ) \|^2_{ HS( U, H ) }
  +
  \| \mu( Y_s ) \|_H
  +
  \| \sigma( Y_s ) \|^2_{ HS( U, H ) }
  \,
  ds
  < \infty
$
$ \P $-a.s.\ and
%$
\begin{align}
\label{eq:SDE.intro}
  X_t 
  &= 
  X_0 
  +
  \int_{ 0 }^t \mu( X_s ) \, ds
  +
  \int_{ 0 }^t \sigma( X_s ) \, dW_s
%$
%$ \P $-a.s. and
%$
  \\
  Y_t 
  &= 
  Y_0 
  +
  \int_{ 0 }^t a_s \, ds
  +
  \int_{ 0 }^t b_s \, dW_s
\label{eq:perturbation.intro}
\end{align}
%$
$ \P $-a.s.\ for all
$ t \in [0, T] $.
The process $ X $ is thus a solution process
of the SDE~\eqref{eq:SDE.intro}
and the process $ Y $ is a general It\^{o} process
with the drift process $ a $, the 
diffusion process $ b $
and the Wiener process $ W $.
We view the stochastic process $ Y $ 
as a \emph{perturbation} of the solution
process of the 
SDE~\eqref{eq:SDE.intro}
and we are interested 
in estimates for the strong
\emph{perturbation error}
$ 
  \| X_t - Y_t \|_{ L^p( \Omega; H ) } 
$
at some fixed (or random) time
$ t \in [0,T] $
for $ p \in (0, \infty) $.

Informally speaking, we 
\emph{estimate the global perturbation error by 
the local perturbation error}.
More formally,
for every $ p \in (0,\infty) $
we estimate the global perturbation error
$ \| X_T - Y_T \|_{ L^p( \Omega; H ) } $
%, $t\in[0,T]$, 
by the $ L^q $-norms
of the difference $ X_0 - Y_0 $ at time $ 0 $
and
of the differences 
$ 
  a - \mu( Y ) 
  =
  ( a_t - \mu( Y_t ) )_{ t \in [0,T] } 
$ 
and 
$ 
  b - \sigma( Y ) 
  =
  ( b_t - \sigma( Y_t ) )_{ t \in [0,T] } 
$ 
of the local characteristics
where $ q \in (0,\infty) $ is appropriate;
see Theorem~\ref{thm:perturbation.estimate} 
below for details.
This perturbation result can then be applied 
to any stochastic process that is an
It\^{o} process with respect to 
the Wiener process $W$.
Possible applications include
\begin{enumerate}[(i)]
  \item  
     \emph{local Lipschitz continuity
    % in the strong sense 
    of solutions of SDEs with respect to the initital value}
     (choose $ a_t = \mu( Y_t ) $ and $ b_t = \sigma( Y_t ) $ for $ t \in [0,T] $; see 
    Corollary~\ref{cor:perturbation_initial}
    below and 
    Cox et 
    al.~\cite{CoxHutzenthalerJentzen2013} for details),
     \label{i:problem1}
  \item
     strong convergence rates for \emph{time-discrete numerical approximations
     of SODEs}
     (e.g., the Euler-Maruyama approximation with $N\in\N$ time discretization steps
     is given by
     $ a_t = \mu( Y_{ \frac{ k T }{ N } } ) $ 
      and
     $ b_t = \sigma( Y_{ \frac{ k T }{ N } } ) $ 
    for $ t \in [ \frac{ kT }{ N }, \frac{ (k+1) T }{ N } )
    $ and $ k \in \N $;
    see Subsection~\ref{sec:SODEs} below),
     \label{i:problem2}
  \item
    strong convergence rates for
     \emph{Galerkin approximations
     for SPDEs}
     (choose $ a_t = P( \mu( Y_t ) ) $ 
    and $ b_t u = P( \sigma( Y_t) u ) $ 
    for $ u \in U $, $ t \in [0,T] $
    and some suitable projection operator $ P \in L( H ) $;
    see 
    Subsection~\ref{ssec:Galerkin}
    below)\label{i:problem3} and
   \item
    strong convergence rates for
     \emph{small noise perturbations}
     of solutions of deterministic 
    differential equations
     (choose $ \sigma = 0 $, $ a_t = \mu( Y_t ) $ 
    and $ b_t = \eps \, \tilde{\sigma}( Y_t ) $
    for $ t \in [0,T] $
     where 
      $
       \tilde{\sigma} \colon D \to HS( U, H ) 
     $
     is a suitable Borel measurable function and
    where $ \eps > 0 $ is a sufficiently small 
    parameter; see Subsection~\ref{ssec:small.noise}
    below).
     \label{i:problem4}
\end{enumerate}
In the literature, a frequently used 
method to estimate 
strong perturbation errors is to employ
Gronwall's lemma together with
the popular \emph{global monotonicity}
assumption
(see, e.g., Minty~\cite{Minty1962,Minty1963} for deterministic equations
and condition (4.19) in Pardoux~\cite{Pardoux1975} for SODEs)
% on the coefficient functions,
% that is,
% to assume
that there exists a real number $ c \in \R $ 
such that for all $ x, y \in D $ 
it holds that
\begin{equation} 
\label{eq:global.monotonicity}
  \left\langle x-y,\mu(x)-\mu(y)\right\rangle_{H}+\tfrac{1}{2}\left\|\sigma(x)-\sigma(y)\right\|^2_{\HS(U,H)}
  \leq c\left\|x-y\right\|_H^2.
\end{equation}
%To illustrate this, we derive strong local Lipschitz continuity in the intial value with this approach.
%In the case $a_s=\mu(Y_s)$ and $b_s=\sigma(Y_s)$ for all $s\in[0,T]$,
%It\^o's formula and the global monotonicity assumption~\ref{eq:global.monotonicity} imply
%that
%\begin{equation}  \begin{split}
%  \left\|X_t-Y_t\right\|^2
%  &
%  =
%  \left\|X_0-Y_0\right\|^2
%  +2\int_0^t
%  \left\langle X_s-Y_s,\mu(X_s)-\mu(Y_s)\right\rangle_{H}
%   +\tfrac{1}{2}\left\|\sigma(X_s)-\sigma(Y_s)\right\|^2_{\HS(U,H)}\,ds
%  +2\int_0^t \left\langle X_s-Y_s,(\sigma(X_s)-\sigma(Y_s))dW_s\right\rangle_H
%  \\&
%  \leq
%  \left\|X_0-Y_0\right\|^2
%  +2c\int_0^t
%  \left\| X_s-Y_s\right\|_{H}^2\,ds
%  +2\int_0^t \left\langle X_s-Y_s,(\sigma(X_s)-\sigma(Y_s))dW_s\right\rangle_H
%\end{split}     \end{equation}
%$\P$-a.s.~for all $t\in[0,T]$.
%After taking expectations, Gronwall's lemma can then be applied
%to infer that
%$\|X_t-Y_t\|_{L^2(\Omega;H)}\leq
% \|X_0-Y_0\|_{L^2(\Omega;H)}e^{ct}$
%for all $t\in[0,T]$.
%Note that this global monotonicity assumption is more general than globalLipchitz
Under the global monotonicity assumption~\eqref{eq:global.monotonicity},
there are a multitute of results in the literature 
and, at least partially, the above problems~\eqref{i:problem1}--\eqref{i:problem4}
have been solved under this assumption
(cf., e.g., Proposition 4.2.10 in Pr\'ev\^ot \& R\"ockner~\cite{PrevotRoeckner2007}
for problem~\eqref{i:problem1}, 
% Theorem~2.4 in 
Hu~\cite{Hu1996},
% Theorem~5.3 in 
% Higham, Mao \& Stuart~\cite{hms02},
% Theorems~2 and 3 in 
Sabanis~\cite{Sabanis2013Arxiv}
for problem~\eqref{i:problem2},
% Theorem~3.5 in 
Liu~\cite{Liu2003},
Sauer \& Stannat~\cite{SauerStannat2013}
for problem~\eqref{i:problem3}).
Unfortunately,
the global monotonicity assumption~\eqref{eq:global.monotonicity} is 
too restrictive in the sense that
the nonlinearities
in the coefficient functions of the majority
of nonlinear (stochastic) differential equations 
from applications
do not satisfy the global monotonicity assumption~\eqref{eq:global.monotonicity}
(see, e.g., Sections~\ref{sec:SODEs} 
and~\ref{ssec:Galerkin} below
for a few examples).

Beyond the global monotonicity assumption~\eqref{eq:global.monotonicity},
we are not aware of a general technique for
estimating global perturbation errors by local perturbation errors.
In the literature, there exist the following results
%For specific questions, however, there exist  approaches
%for proving convergence in the strong sense to the solution of an SODE or SPDE
for SDEs with non-globally monotone nonlinearities for the problems~\eqref{i:problem1}--\eqref{i:problem4}.
Problem~\eqref{i:problem1} -- which is in a certain sense the simpliest of
the problems~\eqref{i:problem1}--\eqref{i:problem4}
as there is only a perturbation of the initial value but
no perturbation of the 
dynamics of \eqref{eq:SDE.intro} --
is already solved
%in the sense that there exists a theory which
%implies local Lipschitz continuity in the strong sense with respect to the initial value
for a large class of SDEs with non-globally monotone nonlinearities
(see, e.g.,
% Theorem 5.1 and Lemma 6.1 in 
Li~\cite{Li1994},
% Lemma 4.10 in 
Hairer \& Mattingly~\cite{HairerMattingly2006},
% Lemma 2.3 in 
Zhang~\cite{Zhang2010} 
% and, in particular,
% Theorem 1.1 and the examples in Section 4 in 
and
Cox et.\ al~\cite{CoxHutzenthalerJentzen2013}).
Problem~\eqref{i:problem2}
has been solved for a large class of one-dimensional
square-root diffusion processes 
with inaccessible boundaries
(see
Gy\"{o}ngy \& Rasonyi~\cite{GyoengyRasonyi2011},
Dereich, Neuenkirch \& Szpruch~\cite{DereichNeuenkirchSzpruch2012},
Alfonsi~\cite{Alfonsi2012},
Neuenkirch \& Szpruch~\cite{NeuenkirchSzpruch2012}).
% \cite{
% GyoengyRasonyi2011,
% DereichNeuenkirchSzpruch2012,% Theorem 1.1 (CIR)
% Alfonsi2012,%Theorem 2 (CIR) und Proposition 3 (general)
% NeuenkirchSzpruch2012%, Prop 3.1 (CIR), Heston 3/2, CEV, WF-Diffusion, Ait-Sahalia, Thm 2.7 (general)
% }).
%
%
We are not aware of any result in the literature that solves problem~\eqref{i:problem2}
in the case of a multi-dimensional SODE
which fails to satisfy \eqref{eq:global.monotonicity}.
Regarding
problem~\eqref{i:problem3},
we are aware of exactly 
one result in the literature
on SPDEs with non-globally monotone nonlinearities, that is,
the work of D\"orsek~\cite{Doersek2012}.
More precisely,
Corollary~3.2 in \cite{Doersek2012}
establishes the strong convergence rate $ 1 $
% in the $ L^2( \Omega; H ) $-norm
for spectral Galerkin approximations 
of the vorticity formulation of the two dimensional stochastic
Navier-Stokes equations with degenerate additive noise.
For problem~\eqref{i:problem4}, we have not found results in the literature
on SDEs with non-globally monotone nonlinearities.
%Again
%finiteness of certain exponential moments from
%Lemma 4.10 in Hairer \& Mattingly~\cite{HairerMattingly2006}
%play an important role in the proof of
%Corollary 3.2 in D\"orsek~\cite{Doersek2012}.

%Our new approach for estimating global approximation errors relies on
%a simple application of It\^o's formula.
An important observation of this article is that there exist
exponential integrating factors 
$
  \exp( \int_0^t \chi_s \, ds ) 
$, 
$ t \in [0,T] $,
%such that
%if we apply It\^o's formula
such that, informally speaking,
the rescaled squared distances
$
  \| X_t - Y_t \|^2_H
  \exp( - \smallint_0^t \chi_s \, ds )
$,
$ 
  t \in [0,T] 
$,
are sums and integrals over local perturbation errors
where $ ( \chi_t )_{ t \in [0,T] } $ is a suitable stochastic process.
The following proposition, 
Proposition~\ref{prop:main_perturbation_intro} below,
formalizes this idea and establishes a \emph{pathwise perturbation formula}.
In Proposition~\ref{prop:main_perturbation_intro}
the squared Hilbert-space distance 
$ \left\| v - w \right\|_H^2 $, $ v, w \in H $,
is replaced by a more general function $ V( v, w ) $, $ v, w \in H $,
to measure distances.
It proved very beneficial in the case of some SDEs such as
Cox-Ingersoll-Ross processes or the Cahn-Hilliard-Cook equation with space-time white noise
to measure distance between the solution $ X $ and its perturbation $ Y $
with a general function $ V \in C^2( H^2 , \R ) $ rather than with 
the squared Hilbert space distance
(see Sections~4.9 and 4.12.2 in Cox et.~al~\cite{CoxHutzenthalerJentzen2013}
for details).
Next we note that
in the perturbation formula~\eqref{eq:perturbation.formula.intro} below,
there appears an operator
$ \overline{\mathcal{G}}_{ \mu, \sigma } \colon C^2( H^2 , \R ) \to C( H^2 , \R ) $
defined in~\eqref{eq:extended_generator} below
which is the formal generator of the bivariate process consisting of two
solution processes of the SDE~\eqref{eq:SDE.intro};
see also Ichikawa~\cite{Ichikawa1984}, Maslowski~\cite{Maslowski1986} 
and, e.g., Leha \& Ritter~\cite{LehaRitter1994,LehaRitter2003}
for references where this operator has been introduced and used in the literature.
% 
% This operator was already used in Theorem 1.2 in Maslowski~\cite{Maslowski1986}
% for studying long-time stability properties of SDEs under the
% assumption $(\overline{\mathcal{G}}_{\mu,\sigma}V)(x,y)\leq 0$ for all $x,y\in H$
% (cf.~also Ichikawa~\cite{Ichikawa1984} and, e.g., Leha \& Ritter~\cite{LehaRitter1994,LehaRitter2003}).
%
% %%%%%%%%%%%%%%%%%%%%%%%%%%%%%%%%%%%%%%%%%%%%%%%%%%%%%%
%
%
%
\begin{prop}[Perturbation formula]
\label{prop:main_perturbation_intro}
Assume the above setting,
let
$ V = ( V(x, y) )_{ ( x, y ) \in H^2 } \in C^2( H^2, \R ) $,
let 
$ \chi \colon [0,T] \times \Omega \to \R $ be a predictable stochastic process 
with $ \int_0^T | \chi_s | \, ds < \infty $ $ \P $-a.s.\ 
and let 
$ e_k \in U $,
$ k \in \N$,
be an orthonormal basis of $ U $.
Then
\begin{equation}
\begin{split}
\label{eq:perturbation.formula.intro}
  \frac{
    V( X_t, Y_t )
  }{
    \exp\!\big(
      \int_{0}^t \chi_u \, du
    \big)
  }
 =&
 \, V( X_{ 0 }, Y_{ 0 } )
  +
  \smallint\limits_0^t
  \tfrac{
    (
        \partial_x
      V
    )( X_s, Y_s ) 
    \, \sigma( X_s )
    +
    (
        \partial_{ y }
      V
    )( X_s, Y_s ) 
    \, b_s
  }{
    \exp(
      \int_{0}^s \chi_u \, du
    )
  }
  \, dW_s
\\ & 
  +
  \smallint\limits_{ 0 }^t
  \tfrac{
    ( \overline{\mathcal{G}}_{ \mu, \sigma } V)( X_s, Y_s )
    - 
    \chi_s V( X_s, Y_s )
    +
    \sum_{ k = 1 }^{ \infty }
    \left(
      \partial_x \partial_y 
      V
    \right)( X_s, Y_s ) 
    \left(
      \sigma( X_s ) e_k
      ,
      [ b_s - \sigma( Y_s ) ] e_k
    \right)
  }{
    \exp(
      \int_{0}^s \chi_u \, du
    )
  }
  \, ds
\\ & 
  +
  \smallint\limits_{ 0 }^t
%  \int_{  \tzero  }^t
  \tfrac{
    ( \partial_y V )( X_s, Y_s ) 
    \left[
      a_s 
      -
      \mu( Y_s )
    \right]
    +
    \frac{ 1 }{ 2 }
    \operatorname{trace}\left(
      \left[ 
        b_s + \sigma( Y_s ) 
      \right]^*
      ( \operatorname{Hess}_{ y } V)( X_s, Y_s )
      \left[ 
        b_s
        -
        \sigma( Y_s )
      \right]
    \right)
  }{
    \exp(
      \int_{0}^s \chi_u \, du
    )
  }
  \, ds
\end{split}
\end{equation}
$ \P $-a.s.\ for all 
$ t \in [ 0, T ] $.
\end{prop}
Proposition~\ref{prop:main_perturbation_intro}
follows immediately from It\^{o}'s formula together with the addition and the subtraction of a suitable term;
see Proposition~\ref{prop:main_perturbation}
below for details.
Proposition~\ref{prop:main_perturbation_intro} turned out to be rather
useful to develop a perturbation theory for the SDE~\eqref{eq:SDE.intro}
and, thereby, to partially solve the problems~\eqref{i:problem1}--\eqref{i:problem4}
without assuming global monotonicity.
In the formulation of Proposition~\ref{prop:main_perturbation_intro},
the exponential integrating factors
$\exp(\smallint_0^t\chi_s\,ds)$, $t\in[0,T]$, can be quite arbitrary.
However,
it is essential to observe that if $\chi$ can be chosen such that
$
    ( \overline{\mathcal{G}}_{ \mu, \sigma } V)( X_s, Y_s )
    - 
    \chi_s V( X_s, Y_s )
    \leq 0
$
$ \P $-a.s.\ for all
$ s \in [ 0, T ] $,
then the expectation of the right-hand side of~\eqref{eq:perturbation.formula.intro}
is
-- informally speaking --
dominated by sums and integrals over 
the local perturbation errors $a-\mu(Y)$ and $b-\sigma(Y)$
times random factors.
The exponential integrating factors
$\exp(\smallint_0^t\chi_s\,ds)$, $t\in[0,T]$,
on the left-hand side of~\eqref{eq:perturbation.formula.intro}
and the random factors on the right-hand side of \eqref{eq:perturbation.formula.intro}
can then
-- roughly speaking --
be estimated by using H\"older's inequality and Young's inequality.
In the case of $V(x,y)=\|x-y\|^p$ for $ x, y \in H $ and some $p\in[2,\infty)$,
this leads to the following \emph{perturbation estimate}~\eqref{eq:perturbation.estimate.intro}.
For more general estimates including a general `distance-type' function $ V $
see Section~\ref{ssec:perturbation.estimates}.
%
% %%%%%%%%%%%%%%%%%%%%%%%%%%%%%%%%%%%%%%%%%%%%%%%%%%%%%%
%
\begin{theorem}%[Perturbation estimate]
\label{thm:perturbation.estimate}
Assume the above setting,
let
$ \varepsilon\in [0,\infty] $,
$ p \in [2,\infty) $, 
%let $ \mu \in \mathcal{L}^0( D, H ) $, $ \sigma \in \mathcal{L}^0( D, HS( U, H ) ) $,
let $ \tau \colon \Omega \to [  \tzero  , T ] $ 
be a stopping time and assume
%\begin{equation}
%\label{eq:perturbation_estimate_unique_assumption_intro}
{\small $
    \int_{  \tzero  }^{ \tau } 
      \big[
          \langle 
            X_s - Y_s , \mu( X_s ) - \mu( Y_s ) 
          \rangle_H
          +
          \frac{ ( p - 1 ) \, ( 1 + \varepsilon ) }{ 2 }
          \|
            \sigma( X_s ) - \sigma( Y_s )
          \|^2_{ HS( U, H ) }
    \big]^+ \!
    / \,
          \| X_s - Y_s \|^{ 2 }_H
    \, ds
  < \infty
$}
$ \P $-a.s.
Then it holds for all 
$ \alpha, \beta \in (0,\infty) $,
$ r, q \in (0,\infty] $
with 
$ \frac{ 1 }{ p } + \frac{ 1 }{ q } = \frac{ 1 }{ r } $
that
{\footnotesize
\begin{align}
&
  \big\| X_{ \tau } - Y_{ \tau }
  \big\|_{
    L^r( \Omega; H )
  }
  \leq
  \left\|
  \exp\!\left(
    \smallint_{  \tzero  }^{ \tau }
  \bigg[
  \tfrac{
    \left< 
      X_s - Y_s , \mu( X_s ) - \mu( Y_s ) 
    \right>_H
    +
    \frac{ ( p - 1 ) \, ( 1 + \varepsilon ) }{ 2 }
    \left\|
      \sigma( X_s ) - \sigma( Y_s )
    \right\|^2_{ HS( U, H ) }
  }{
    \left\| X_s - Y_s \right\|^2_H
  }
  +
  \tfrac{
      ( 1 - \frac{ 1 }{ p } ) 
  }{
      \alpha
  }
      +
  \tfrac{
      ( \frac{ 1 }{ 2 } - \frac{ 1 }{ p } )
  }{
      \beta
  }
  \bigg]^+
  \!
    ds
  \right)
  \right\|_{
    L^q( \Omega; \R )
  }
\nonumber
\\ & \cdot
\label{eq:perturbation.estimate.intro}
  \left[
    \big\| X_{  \tzero  } - Y_{  \tzero  } 
    \big\|_{ L^p( \Omega; H ) }
  +
    \alpha^{
      ( 1 - \frac{ 1 }{ p } )
    }
      \left\|
        a -
        \mu( Y )
      \right\|_{ L^p( \llbracket  \tzero , \tau \rrbracket ; H ) }
    +
      \beta^{ ( \frac{ 1 }{ 2 } - \frac{ 1 }{ p } ) }
      \sqrt{ ( p - 1 ) ( 1 + 1 / \varepsilon ) }
        \left\|
          b -
          \sigma( Y )
        \right\|_{ 
          L^p( \llbracket  \tzero  , \tau \rrbracket ; HS( U, H ) )
        }
  \right]
  .
\end{align}}
\end{theorem}

In Theorem~\ref{thm:perturbation.estimate}
the expression
$ 
  \llbracket 0 , \tau \rrbracket 
  :=
  \left\{
    ( t, \omega ) \in [0,T] \times \Omega 
    \colon
    t \leq \tau( \omega )
  \right\}
$
denotes the stochastic interval
from $ 0 $ to $ \tau $
(see, e.g., K\"{u}hn~\cite{Kuehn2004})
and in the formulation of 
Theorem~\ref{thm:perturbation.estimate}
the convention 
$ \frac{ 0 }{ 0 } := 0 $ 
is used.
Theorem~\ref{thm:perturbation.estimate} follows immediately from Corollary~\ref{cor:norm2}
below
which, in turn, follows from 
Theorem~\ref{thm:norm2_expectation_max0} 
below.
Theorem~\ref{thm:perturbation.estimate}
can be applied to prove local Lipschitz continuity in the strong sense
with respect to the initial value by choosing
$ \tau = T $, 
$ \eps = 0 $, 
% $ \delta = \rho = \infty $
$ a = \mu( Y ) $,
$ b = \sigma( Y ) $.
% for $ t \in [0,T] $.
Thereby one obtains a quite similar inequality than Corollary 2.19
in Cox et.~al~\cite{CoxHutzenthalerJentzen2013}
(see also Corollary~\ref{cor:perturbation_initial} below).
Local Lipschitz continuity with respect to the initial value
follows then from finiteness of the 
\emph{exponential moment}
on the right-hand side of~\eqref{eq:perturbation.estimate.intro}
which, in turn, is implied by conditions similar
to~\eqref{eq:local.monotonicity} and \eqref{eq:generator.condition.exp.moments} below
in the case $a=\mu(Y)$ and $b=\sigma(Y)$
(see, e.g., Lemma 2.22
in \cite{CoxHutzenthalerJentzen2013}
for details
and, e.g., 
also \cite{HairerMattingly2006,
FangImkellerZhang2007,
EsSarhirStannat2010,
BouRabeeHairer2013,
HieberStannat2013}
for some instructive results on exponential moments).
Note that 
the counterexamples in 
Hairer et.\ al~\cite{hhj12}
show that some condition on $ \mu $ 
and $ \sigma $ beyond smoothness
and global boundedness
is necessary to ensure that the exponential
moment on the right-hand side of
\eqref{eq:perturbation.estimate.intro}
is finite and, thereby, that solutions of
\eqref{eq:SDE.intro} are locally Lipschitz
continuous with respect to
the initial values.

In order to demonstrate the flexibility of Theorem~\ref{thm:perturbation.estimate}
(and Theorem~\ref{thm:norm2_expectation_max0} below),
we partially solve two well-known approximation problems.
In our first application of Theorem~\ref{thm:perturbation.estimate},
we
establish
in Theorem~\ref{thm:ST2} below
the strong convergence rate $ 1 / 2 $ 
for suitable numerical approximations 
for a large class of 
finite-dimensional 
SODEs with non-globally monotone coefficients.
We point out that strong convergence rates are particularly
important for efficient multilevel Monte Carlo methods
(see Giles~\cite{g08a}, Heinrich~\cite{h98}, Kebaier~\cite{k05}).
In the literature, strong convergence rates for time-discrete approximation processes 
for multi-dimensional SODEs are
% in general 
only known under the global monotonicity assumption~\eqref{eq:global.monotonicity}
(cf., e.g.,
\cite{Hu1996,
hms02,
HutzenthalerJentzenKloeden2012,
TretyakovZhang2012,
MaoSzpruch2013Rate,
KloedenNeuenkirch2013,
Sabanis2013ECP,
Sabanis2013Arxiv}
and the references mentioned therein).
In addition,
strong convergence without rates
has been established for time-discrete approximation processes
for multi-dimensional SDEs with
non-globally monotone coefficients
%, e.g.,
in~\cite{BrzezniakCarelliProhl2013,
HutzenthalerJentzen2014Memoires,
Szpruch2013Vstable,
Sabanis2013Arxiv,
KovacsLarssonLindgren2013}.
To the best of our knowledge,
Theorem~\ref{thm:ST2} is the first result
which proves a strong convergence rate
of time-discrete approximation processes
for a multi-dimensional SODE with
non-globally monotone coefficients.
In particular,
to the best of our knowledge,
Theorem~\ref{thm:ST2} is the first result
which implies a strong convergence rate
for the stochastic Lorenz equation with bounded noise
(see Section~\ref{ssec:Lorenz}),
for the stochastic van der Pol oscillator
(see Section~\ref{ssec:van.der.Pol}),
for the stochastic Duffing-van der Pol oscillator
(see Section~\ref{ssec:stochastic.Duffing.van.der.Pol.oscillator}),
for a model from experimental psychology
(see Section~\ref{ssec:experimental.psychology}),
%,
%for the Brownian dynamics
%(see Section~\ref{ssec:overdamped.Langevin.dynamics}),
%for the Langevin dynamics
%(see Section~\ref{ssec:Langevin.dynamics})
for the overdamped Langevin dynamics 
under suitable assumptions
(see
Section~\ref{ssec:overdamped.Langevin.dynamics})
or for the stochastic Duffing oscillator
with additive noise
(see Section~\ref{ssec:Langevin.dynamics}).
In inequality~\eqref{eq:generator.condition.exp.moments} below,
there appears an operator
$ 
  \mathcal{G}_{ \mu, \sigma } 
  \colon C^2( H , \R ) \to C( H ,\R)
$
defined in~\eqref{eq:generator} below
which is the 
% formal 
generator
associated with
the SDE~\eqref{eq:SDE.intro}.
Theorem~\ref{thm:ST2} follows  immediately from Proposition~\ref{prop:ST2}
below.
%
% %%%%%%%%%%%%%%%%%%%%%%%%%%%%%%%%%%%%%%%%%%%%%%%%%%%%%%
%
\begin{theorem}[Strong convergence rates for numerical approximations]
\label{thm:ST2}
Assume the above setting,
let $ d, m \in \N $,
$ c, r \in (0,\infty) $, 
$ q_0, q_1 \in (0,\infty] $,
$ \alpha \in [0,\infty) $,
$ p, q \in [2,\infty) $
with 
$
  \frac{ 1 }{ p } + \frac{ 1 }{ q_0 } + \frac{ 1 }{ q_1 } = \frac{ 1 }{ r } 
$,
assume 
$ H = D = \R^d $, $ U = \R^m $,
let
$ U_1 \in C^1( \R^d, [0,\infty) ) $,
$ \mu \in C^1( \R^d, \R^d ) $,
$ \sigma \in C^1( \R^d , \R^{ d \times m } ) $
have at most polynomially growing derivatives,
let
$
  U_0 \in C^3( \R^d, [1,\infty) )
$
satisfy
$
  \sum_{ i = 1 }^3
  \|
    ( U_0^{ (i) } )( x )
  \|_{
    L^{ (i) }( \R^d , \R ) 
  }
  \leq 
  c
  \left|
    U_0(x)
  \right|^{ ( 1 - 1 / q ) }
$,
$ \| x \|^{ 1 / c } \leq c \left( 1 + U_0( x ) \right) $,
$
  \E\big[
    e^{ U_0( X_0 ) }
  \big]
  < \infty
$
and
{\small
\begin{align}
\label{eq:local.monotonicity}
  \tfrac{
    \left< 
      x - y , \mu( x ) - \mu( y ) 
    \right>_{ \R^d }
    +
    \frac{ ( p - 1 ) \, ( 1 + 1/c ) }{ 2 }
    \left\|
      \sigma( x ) - \sigma( y )
    \right\|^2_{ HS( \R^m , \R^d ) }
  }{
    \left\| x - y \right\|^2_{ \R^d }
  }
& \leq
    c
    +
    \tfrac{
      U_0( x ) + U_0( y )
    }{
      2 q_0 T e^{ \alpha T }
    }
    +
    \tfrac{
      U_1( x ) + U_1( y )
    }{
      2 q_1 e^{ \alpha T }
    }
\\
\label{eq:generator.condition.exp.moments}
  ( \mathcal{G}_{ \mu, \sigma } U_0 )( x )
  +
  \tfrac{ 1 }{ 2 } 
  \, \| \sigma( x )^* ( \nabla U_0 )( x ) \|^2_{ \R^d } 
  +
  U_1( x )
& \leq
  \alpha \, U_0( x )
  +
  c
\end{align}}for 
all $ x, y \in \R^d $
with $ x \neq y $
and let
%$ 
%  X \colon [ 0, T ] \times \Omega \to \R^d 
%$
%and 
$ 
  Z^{ N } \colon \{0,1,\ldots,N\} \times \Omega \to \R^d 
$,
$ N \in \N $,
satisfy
$ Z_0^{ N } = X_0 $
and 
\begin{equation}
\label{eq:numerical.scheme.Z}
  Z_{n+1}^{ N }
  = 
  Z_{ n }^{ N }
  +
  \mathbbm{1}_{
    \left\{
      \| Z_{ n }^{ N } \|_{ \R^d }
      < 
      \exp(
        | \ln( T / N ) |^{ 1 / 2 }
      )
    \right\}
  }
  \left[
  \tfrac{
    \mu( Z_{ n }^{ N } ) \frac{T}{N}
    +
    \sigma( Z_{ n }^{ N } )
    ( W_{ (n+1)T/N } - W_{ nT/N } )
  }{
    1 +
    \|
      \mu( Z_{ n }^{ N } ) \frac{T}{N}
      +
      \sigma( Z_{ n }^{ N } )
      ( W_{ (n+1)T/N } - W_{ nT/N } )
    \|^2_{ \R^d }
  }
  \right]
\end{equation}
for all $ n \in \{ 0, 1, \ldots , N - 1 \} $, $ N \in \N $.
Then there exists a real number $ C \in [0,\infty) $
such that 
for all $ N \in \N $
it holds that
\begin{equation}
  \sup\nolimits_{ n\in\{0,1,\ldots,N\}}
  \| 
    X_{\frac{nT}{N}} - Z^{ N }_n
  \|_{
    L^r( \Omega; \R^d )
  }
\leq
  C
  \,
  N^{
    -
    1 / 2
  }
  .
\end{equation}
\end{theorem}
The numerical
scheme~\eqref{eq:numerical.scheme.Z}
has been proposed in~\cite{HutzenthalerJentzenWang2013}.
Note that we cannot replace 
scheme~\eqref{eq:numerical.scheme.Z}
by the well-known Euler-Maruyama scheme
since Euler-Maruyama approximations
diverge in the strong sense in the case of superlinearly
growing coefficient functions
(see Theorem 2.1 in~\cite{hjk11} and 
\cite{HutzenthalerJentzenKloeden2013}).
As sketched above,
\emph{exponential integrability properties} 
play an
important role 
in the perturbation theory developed
in this article.
The advantage of the numerical approximations~\eqref{eq:numerical.scheme.Z}
is to preserve 
\emph{exponential integrability} properties of the exact
solution under minor additional assumptions
(see~\cite{HutzenthalerJentzenWang2013} for more details).
Condition 
\eqref{eq:generator.condition.exp.moments} 
ensures that both the exact solution and the
numerical approximations admit suitable 
exponential integrability properties
and assumption \eqref{eq:local.monotonicity}
ensures that the exponential term on the right
hand side of \eqref{eq:perturbation.estimate.intro}
can be estimated in an appropriate way.
Observe that if we choose $ q_0 = q_1 = \infty $
in Theorem~\ref{thm:ST2},
then condition~\eqref{eq:local.monotonicity}
essentially reduces to the global
monotonicity 
assumption~\eqref{eq:global.monotonicity}.
%for a.s.\ divergence results of the associated multilevel Monte Carlo method).
%
%
%
%
%
%
%

Our second application of Theorem~\ref{thm:perturbation.estimate}
and of the more general
Theorem~\ref{thm:norm2_expectation_max0} 
below
concerns the approximation and the analysis 
of SPDEs.
In the literature, there are a number of results
which prove pathwise convergence
rates or convergence rates 
for convergence in probability
for spatially discrete approximation 
processes of SPDEs with non-globally 
monotone nonlinearities
(see, e.g.,
\cite{
p01,%Thm4.3
ag06,%Thm22 on Burgers
CarelliProhl2011,%Thm5.1 on SNSE 2d
BloemkerJentzen2013,
BloemkerKamraniHosseini2013,
KovacsLarssonLindgren2013%Thm5.3
})
or which prove strong convergence 
without convergence rates
for spatially discrete
%(and temporally discrete)
approximation processes of SPDEs
with non-globally monotone nonlinearities
(see, e.g., \cite{
GyoengyMillet2005,%Thm2.10
KovacsLarssonMesforush2011,%
BrzezniakCarelliProhl2013, % Theorem am Ende des Artikels!
KovacsLarssonLindgren2013%Thm5.3
}).
We are aware of only one result
which establishes a strong convergence rate
for spatially discrete approximation processes of SPDEs
with non-globally monotone nonlinearities
namely the above mentioned
Corollary 3.2 in D\"orsek~\cite{Doersek2012}.
%This result is
%which proves  strong convergence rate $1$
%of spectral Galerkin approximations 
%for the vorticity formulation of the stochastic
%Navier-Stokes equations on the two-dimensional torus.
Now our perturbation estimate~\eqref{eq:perturbation.estimate.intro}
in Theorem~\ref{thm:perturbation.estimate}
and its more general version~\eqref{eq:norm2_expectation}
in Theorem~\ref{thm:norm2_expectation_max0} below
result in
Theorem~\ref{thm:Galerkin2} below
which can be applied to semilinear SPDEs with non-globally monotone nonlinearities
to establish 
strong convergence rates for
Galerkin approximations.
In particular, we apply
Theorem~\ref{thm:Galerkin2} below
to obtain for the first time
a strong convergence rate
for spectral Galerkin approximations
for Cahn-Hilliard-Cook type SPDEs
(see inequality~\eqref{eq:CHC.Galerkin.estimate}
in Section~\ref{ssec:Cahn_Hilliard} below for details)
and for stochastic Burgers equations with
bounded diffusion coefficients
(see inequality~\eqref{eq:Burgers.Galerkin.estimate}
in Section~\ref{ssec:stochastic.Burgers.equation} below for details).
Theorem~\ref{thm:Galerkin2} follows  immediately from Proposition~\ref{prop:Galerkin2}
below.
%
% %%%%%%%%%%%%%%%%%%%%%%%%%%%%%%%%%%%%%%%%%%%%%%%%%%%%%%
%
\begin{theorem}[Strong convergence rates for Galerkin approximations]
\label{thm:Galerkin2}
Assume the above setting,
let
$
  \varphi \colon D^2 \to \R
$
be a Borel measurable mapping,
let 
$ \varepsilon \in [0,\infty] $,
$ r \in (0,\infty) $, 
$ q_0, q_1, \hat{q}_0, \hat{q}_1 \in (0,\infty] $,
$ c, \alpha, \beta, \hat{\alpha}, \hat{\beta} \in [0,\infty) $,
$ p \in [2,\infty) $,
$
  U_0, \hat{U}_0 \in C^2( H, [0,\infty) )
$,
$ 
  U_1, \hat{U}_1 \in C( D, [0,\infty) ) 
$,
%$ \mu \in \mathcal{L}^0( D ; H ) $,
%$ \sigma \in \mathcal{L}^0( D ; HS( U, H ) ) $,
$ P \in L( H ) $
satisfy 
% $ P^2 = P = P^* $,
% $ \| P \|_{ L(H) } \leq 1 $,
$ P( H ) \subseteq D $,
$
  \frac{ 1 }{ p } + \frac{ 1 }{ q_0 } + \frac{ 1 }{ q_1 } 
  + \frac{ 1 }{ \hat{q}_0 } + \frac{ 1 }{ \hat{q}_1 } 
  = \frac{ 1 }{ r } 
$,
$
  \E\big[
    e^{
      U_0( X_0 ) 
    }
    +
    e^{
      \hat{U}_0( Y_0 )
    }
  \big] 
  < \infty
$
and
{\small
\begin{eqnarray}
&
  ( \mathcal{G}_{ \mu, \sigma } U_0 )( x )
  +
  \tfrac{ 1 }{ 2 } 
  \, \| \sigma( x )^* ( \nabla U_0 )( x ) \|^2_H 
  +
  U_1( x )
\leq
  \alpha \, U_0( x )
  +
  \beta
  ,
\nonumber
\\[0.5ex]
&
\label{eq:Galerkin_assumption.intro}
  ( \mathcal{G}_{ P \mu, P \sigma } \hat{U}_0 )( y )
  +
  \tfrac{ 1 }{ 2 } 
  \, \| \sigma( y )^* P^* ( \nabla \hat{U}_0 )( y ) \|^2_H 
  +
  \hat{U}_1( y )
\leq
  \hat{\alpha} \, \hat{U}_0( y )
  +
  \hat{\beta}
  ,
\\[0.5ex]
&
\nonumber
    \left< 
      P x - y , P \mu( P x ) - P \mu( y ) 
    \right>_H
    +
    \frac{ ( p - 1 ) \, ( 1 + \varepsilon ) }{ 2 }
    \left\|
      P \sigma( P x ) - P \sigma( y ) 
    \right\|^2_{ HS( U , H ) }
  +
    \left< 
      y - P x , P \mu( P x ) - P \mu( x ) 
    \right>_H
\\ & 
\nonumber
    +
    \frac{ ( p - 1 ) \, ( 1 + 1 / \varepsilon ) }{ 2 }
    \left\|
      P \sigma( P x ) - P \sigma( x )
    \right\|^2_{ HS( U , H ) }
\leq
  \tfrac{
    \left|
      \varphi( x )
    \right|^2
  }{ 2 }
  +
  \left[
    c
    +
    \tfrac{
      U_0( x ) 
    }{
      q_0 T e^{ \alpha T }
    }
    +
    \tfrac{
      \hat{U}_0( y )
    }{
      \hat{q}_0 T e^{ \hat{\alpha} T }
    }
    +
    \tfrac{
      U_1( x ) 
    }{
      q_1 e^{ \alpha T }
    }
    +
    \tfrac{
      \hat{U}_1( y )
    }{
      \hat{q}_1 e^{ \hat{\alpha} T }
    }
  \right] 
  \left\| P x - y \right\|^2_H
\end{eqnarray}}for 
all $ x \in D $,
$ y \in P( H ) $
%let
%$ X, Y \colon [  \tzero , T ] \times \Omega \to D $
%be predictable stochastic processes
%with 
%$
%  \int_{  \tzero  }^T
%  \| \mu( X_s ) \|_H
%  +
%  \| \sigma( X_s ) \|^2_{ HS( U, H ) }
%  +
%  \| \mu( Y_s ) \|_H
%  +
%  \| \sigma( Y_s ) \|^2_{ HS( U, H ) }
%  \,
%  ds
%  < \infty
%$
%$ \P $-a.s.
%$
%  X_t 
%  = 
%  X_{  \tzero  } 
%  +
%  \int_{  \tzero  }^t \mu( X_s ) \, ds
%  +
%  \int_{  \tzero  }^t \sigma( X_s ) \, dW_s
%$
%$ \P $-a.s. and
and assume that
$
  \int_0^T
  \| \mu( P X_s ) \|_H 
  +
  \| \sigma( P X_s ) \|^2_{ HS( U, H ) }
  \, ds
$
and
$
  Y_t 
  = 
  P X_{  \tzero  } 
  +
  \int_{  \tzero  }^t P \mu( Y_s ) \, ds
  +
  \int_{  \tzero  }^t P \sigma( Y_s ) \, dW_s
$
$ \P $-a.s.\ for all
$ t \in [ \tzero ,T] $.
Then 
\begin{align}
&
  \sup_{ t \in [  \tzero , T] }
  \left\|
    X_t - Y_t
  \right\|_{
    L^r( \Omega; H )
  }
\leq
  T^{
    ( \frac{ 1 }{ 2 } - \frac{ 1 }{ p } )
  }
  \exp\!\left(
    \tfrac{ 1 }{ 2 } - \tfrac{ 1 }{ p }
    +
    \smallint_0^T
    c
    +
    \smallsum\limits_{ i = 0 }^1
    \left[
    \tfrac{
      \beta
    }{
      q_i e^{ \alpha s }
    }
    +
    \tfrac{
      \hat{\beta}
    }{
      \hat{q}_i e^{ \hat{ \alpha } s }
    }
    \right]
    ds
  \right)
\\ & \quad 
\nonumber
  \cdot
  \left\|
    \varphi( X )
  \right\|_{
    L^p( [  \tzero , T ] \times \Omega ; \R )
  }
  \left|
  \E\Big[
    e^{
      U_0( X_0 ) 
    }
  \Big]
  \right|^{
    \left[ 
      \frac{ 1 }{ q_0 }
      +
      \frac{ 1 }{ q_1 }
    \right]
  }
  \left|
  \E\!\left[
    e^{
      \hat{U}_0( Y_0 )
    }
  \right] 
  \right|^{
    \left[
      \frac{ 1 }{ \hat{q}_0 }
      +
      \frac{ 1 }{ \hat{q}_1 }
    \right]
  }
  +
  \sup_{ t \in [  \tzero , T] }
  \left\|
    ( I - P ) X_t
  \right\|_{
    L^r( \Omega; H ) 
  }
  .
\end{align}
\end{theorem}

As a third
application of 
Theorem~\ref{thm:perturbation.estimate},
we study SDEs with small noise
(cf., e.g., Theorem 1.2 in Freidlin \& 
Wentzell~\cite{FreidlinWentzell2012}
for the case of globally Lipschitz continuous coefficients).
In particular,
Corollary~\ref{cor:small_noise2} below can be applied
to a number of nonlinear
%ODEs and PDEs
ordinary and partial differential equations
perturbed by a small 
noise term such as the examples in Subsections~\ref{ssec:Lorenz}--\ref{ssec:Langevin.dynamics}
as well as the examples in Subsections~\ref{ssec:Cahn_Hilliard}--\ref{ssec:stochastic.Burgers.equation}.
We refer the reader to 
Section~\ref{ssec:small.noise} for more details.

\subsection{Notation}

Throughout this article the following notation is used.
For two sets $ A $ and $ B $ we denote by
$ \mathcal{M}( A, B ) $ the set of all mappings
from $ A $ to $ B $.
In addition,
for two measurable spaces
$
  ( A, \mathcal{A} )
$
and
$
  ( B, \mathcal{B} )
$
we denote by 
$
  \mathcal{L}^0( A; B )
$
the set of all $ \mathcal{A} $/$ \mathcal{B} $-measurable
functions.
Furthermore, for a natural number $ d \in \N $ we define
\begin{multline}
  \mathcal{C}^3_{ \mathcal{D} }( \R^d, \R )
  :=
%\\
  \bigcup_{ p, c \in [3,\infty) }
  \left\{ 
    f \in C^2( \R^d, \R )
    \colon
    \begin{array}{c}
      f'' \text{ is locally Lipschitz continuous and for}
    \\
      \text{all $ i \in \{ 1, 2, 3 \} $ and $ \lambda_{ \R^d } $-almost all $ x \in \R^d $ it}
    \\
      \text{holds that }
      \| f^{ (i) }(x) \|_{
        L^{ (i) }( \R^d, \R )
      }
      \leq
      c
      \,
      | f( x ) |^{ [ 1 - i / p ] }
    \end{array}
  \right\}
  .
\end{multline}
Moreover, for a natural number $ d \in \N $
and a metric space $ ( E, d_E ) $
we denote by
\begin{multline}
  \mathcal{C}^1_{ \mathcal{P} }( \R^d, E )
  :=
  \Big\{ 
    f \in C( \R^d, E ) 
  \colon
\\
  \Big(
    \exists \, c \in [0,\infty)
    \colon
    \forall \, x, y \in \R^d \colon
      d_E( f(x), f(y) )
    \leq
      c \left( 1 + \| x \|^c + \| y \|^c \right)
      \left\| x - y \right\|
  \Big)
  \Big\}
\end{multline}
the space of locally Lipschitz continuous functions from $ \R^d $ to $ E $
with locally Lipschitz constants that grow at most polynomially.
In addition, 
for 
separable $ \R $-Hilbert spaces
$
  ( H, \left< \cdot , \cdot \right>_H, \left\| \cdot \right\|_H )
$
and
$
  ( U, \left< \cdot , \cdot \right>_U, \left\| \cdot \right\|_U )
$,
an open set $ O \subseteq H $, 
a non-empty set $ \mathcal{O} \subseteq O $
and functions
$ 
  \mu \in \mathcal{M}( \mathcal{O}, H ) 
$,
$
  \sigma \in \mathcal{M}( 
    \mathcal{O}, HS( U, H ) 
  )
$
we define linear operators
$  	
  \mathcal{G}_{ \mu, \sigma }
  \colon
  C^2( O, \R )
  \to
  \mathcal{M}( \mathcal{O}, \R )
$,
$  
  G_{ \sigma }
  \colon
  C^1( O, \R )
  \to
  \mathcal{M}( \mathcal{O}, \R^{ 1 \times m } )
$,
$  
  \overline{ \mathcal{G} }_{ \mu, \sigma }
  \colon
  C^2( O^2, \R )
  \to
  \mathcal{M}( \mathcal{O}^2, \R )
$
and
$  
  \overline{ G }_{ \sigma }
  \colon
  C^1( O^2, \R )
  \to
  \mathcal{M}( \mathcal{O}^2, U^* )
$
by
\begin{align}
\label{eq:generator}
  ( \mathcal{G}_{ \mu, \sigma } \phi)( x )
& :=
  \phi'(x) \,
  \mu( x )
  +
  \tfrac{ 1 }{ 2 }
  \text{tr}\big(
    \sigma(x) \sigma(x)^* 
    ( \text{Hess}_x \phi )( x )
  \big)
  ,
\\[1ex]
\label{eq:noise_operator}
  ( G_{ \sigma } \psi )( x )
& :=
  \psi'(x) \, \sigma(x)
  ,
\\
\label{eq:extended_generator}
  ( \overline{ \mathcal{G} }_{ \mu, \sigma } \bar{ \phi } )( x, y )
& :=
  \big(
    \tfrac{ \partial }{ \partial_x }
    \bar{ \phi }
  \big)(x,y) \, 
  \mu( x )
  +
  \big(
    \tfrac{ \partial }{ \partial_y }
    \bar{ \phi }
  \big)(x,y) \, 
  \mu( y )
  +
  \tfrac{ 1 }{ 2 }
  \smallsum\limits_{ i = 1 }^m
  \big(
    \tfrac{ \partial^2 }{ \partial x^2 }
    \bar{ \phi }
  \big)(x,y) 
  \big( 
    \sigma_i( x ) , 
    \sigma_i( x )
  \big)
\nonumber
\\ & \quad +
  \smallsum\limits_{ i = 1 }^m
  \big(
    \tfrac{ \partial }{ \partial_y }
    \tfrac{ \partial }{ \partial_x }
    \bar{ \phi }
  \big)(x,y)\big( 
    \sigma_i( x ) , 
    \sigma_i( y )
  \big)
  +
  \tfrac{ 1 }{ 2 }
  \smallsum\limits_{ i = 1 }^m
  \big(
    \tfrac{ \partial^2 }{ \partial y^2 }
    \bar{ \phi }
  \big)(x,y) 
  \big(
    \sigma_i( y ) , 
    \sigma_i( y )
  \big)
  ,
\\
\label{eq:extended_noise_operator}
  ( \overline{ G }_{ \sigma } \bar{ \psi } )( x, y )
& :=
  \big(
    \tfrac{ \partial }{ \partial_x }
    \bar{ \psi }
  \big)(x,y) \, 
  \sigma( x )
  +
  \big(
    \tfrac{ \partial }{ \partial_y }
    \bar{ \psi }
  \big)(x,y) \, 
  \sigma( y )
\end{align}
for all $ x, y \in \mathcal{O} $,
$ \phi \in C^2( O, \R ) $,
$ \bar{ \phi } \in C^2( O^2, \R ) $,
$ \psi \in C^1( O , \R ) $,
$ \bar{ \psi } \in C^1( O^2 , \R ) $.
We call the 
linear operator $ \mathcal{G}_{ \mu, \sigma } $
defined in \eqref{eq:generator} \emph{generator},
we call the linear operator $ G_{ \sigma } $
defined in \eqref{eq:noise_operator} \emph{noise operator},
we call the linear operator $ \overline{\mathcal{G}}_{ \mu, \sigma } $
defined in \eqref{eq:extended_generator} \emph{extended generator}
(cf.\ Ichikawa~\cite{Ichikawa1984} and Maslowski~\cite{Maslowski1986})
and we call the linear operator $ \overline{G}_{ \sigma } $
defined in \eqref{eq:extended_noise_operator} \emph{extended noise operator}
(cf., e.g., Cox et al.~\cite{CoxHutzenthalerJentzen2013}).
In addition, for a real number $ T \in (0,\infty) $
we denote by
$
  \mathcal{P}_T 
  =
  \cup_{ n \in \N }
  \{ 
    ( t_0 , t_1, \dots, t_n ) \in \R^{ n + 1 }
    \colon
    0 =  t_0 < t_1 < \dots < t_n = T
  \}
$
the set of all partitions of the interval $ [ 0, T ] $.
Moreover, if
$ T \in (0,\infty) $
is a real number
and if $ ( \Omega, \mathcal{F}, \P ) $
is a probability space
with a normal filtration $ ( \mathcal{F}_t )_{ t \in [0,T] } $
(see, e.g., Definition~2.1.11 in Pr\'{e}v\^{o}t \& R\"{o}ckner~\cite{PrevotRoeckner2007}),
then we say that the quadrupel
$
  ( \Omega, \mathcal{F}, \P, ( \mathcal{F}_t )_{ t \in [0,T] } )
$
is a \emph{stochastic basis}.
Furthermore, for a real number
$  T \in [0,\infty) $,
a stochastic basis
$
  ( 
    \Omega, \mathcal{F}, \P, ( \mathcal{F}_t )_{ t \in [ \tzero ,T] } 
  )
$
and
adapted and product measurable stochastic processes
$ \chi \colon [  \tzero , T ] \times \Omega \to \R $
and
$
  \zeta \colon [  \tzero , T ] \times \Omega \to U^* = L( U, \R ) 
$
with
$
  \int_{  \tzero  }^T
  | \chi_s |
  +
  \| \zeta_s \|^2_{ U^* }
  \, ds
  < \infty
$
$ \P $-a.s.\ we 
denote by
% define
% the \emph{stochastic exponential}
$
  \Psi[ \chi, \zeta ] 
$
the equivalence class (with respect to 
indistinguishability) of adapted $ \R $-valued stochastic processes 
on $ [ \tzero ,T] $
with continuous sample paths
satisfying
\begin{equation}
  \Psi[ \chi, \zeta ]_t
=
  \exp\!\left(
    \smallint_{  \tzero  }^t
    \chi_s
    -
    \tfrac{ 1 }{ 2 }
    \left\| \zeta_s 
    \right\|^2_{ U^* }
    ds
    +
    \smallint_{  \tzero  }^t
    \zeta_s \, dW_s 
  \right)
\end{equation}
$ \P $-a.s.\ for all
$ t \in [ \tzero , T] $.
Throughout this article we also often calculate and 
formulate expressions in the extended real numbers
$ [ - \infty, \infty ] = \R \cup \{ - \infty, \infty \} $.
In particular, we frequently use the convention
$
  \frac{ 0 }{ 0 } = 0 \cdot \infty = 0
$.
Furthermore, for a real number $ a \in \R $ we denote by
$
  a^+ := \max( a, 0 ) 
$
the nonnegative part of $ a $.
In addition, 
for a real number $ T \in [ 0, \infty ) $,
a stochastic basis
$ 
  ( \Omega, \mathcal{F}, \P, ( \mathcal{F}_t )_{ t \in [0,T] } )
$
and a stopping time $ \tau \colon \Omega \to [0,T] $
we denote by
$ 
  \llbracket 0 , \tau \rrbracket 
  :=
  \left\{
    ( t, \omega ) \in [0,T] \times \Omega 
    \colon
    t \leq \tau( \omega )
  \right\}
$
the probabilistic interval from $ 0 $ to $ \tau $
(see, e.g., Definition~3.1 in K\"{u}hn~\cite{Kuehn2004}).

\subsection{Setting}
\label{sec:setting}

Throughout this article the following setting is frequently used.
Let 
$
  ( H, \left< \cdot , \cdot \right>_H, \left\| \cdot \right\|_H )
$
and
$
  ( U, \left< \cdot , \cdot \right>_U, 
$
$
  \left\| \cdot \right\|_U )
$
be separable $ \R $-Hilbert spaces,
let $ O \subseteq H $ be an open set,
let $ \mathcal{O} \in \mathcal{B}( O ) $,
$ T \in ( 0, \infty ) $,
let 
$ 
  ( \Omega, \mathcal{F}, \P, 
    ( \mathcal{F}_t )_{ t \in [  \tzero  , T ] } ) 
$
be a stochastic basis,
let
$
  ( W_t )_{ t \in [  \tzero , T ] }
$
be a cylindrical $ I_U $-Wiener process
with respect to $ ( \mathcal{F}_t )_{ t \in [  \tzero , T ] } $
and let 
$ e_k \in U $,
$ k \in \N $,
be an orthonormal basis of $ U $.

\section{A perturbation theory for stochastic differential equations (SDEs)}
\label{sec:perturbation_theory}

\subsection{It\^{o}'s formula and an exponential integrating factor}

\begin{lemma}
\label{lem:perturbation_formula}
Assume the setting in Subsection~\ref{sec:setting},
let $ V \in C^2( O, \R ) $
and 
let 
$ X \colon [  \tzero  , T ] \times \Omega \to \mathcal{O} $,
$ a \colon [  \tzero , T ] \times \Omega \to H $,
$ b \colon [  \tzero , T ] \times \Omega \to HS( U, H ) $,
$ \chi \colon [  \tzero , T ] \times \Omega \to \R $,
%and
$ \zeta \colon [  \tzero , T ] \times \Omega \to U^* $
be predictable stochastic processes
with 
$
  \int_{  \tzero  }^T
  \| a_s \|_H
  +
  \| b_s \|_{ HS( U, H ) }^2
  +
  | \chi_s |
  +
  \| \zeta_s \|^2_{ U^* }
  \,
  ds
  < \infty
$
$ \P $-a.s.\ and
$
  X_t 
  = 
  X_{  \tzero  } 
  +
  \int_{  \tzero  }^t a_s \, ds
  +
  \int_{  \tzero  }^t b_s \, dW_s
$
$ \P $-a.s.\ for all
$ t \in [ \tzero ,T] $.
Then
\begin{equation}
\label{eq:first_perturbation}
\begin{split}
&
  \frac{
    V( X_t )
  }{
    \Psi[ 
      \chi , \zeta
    ]_t
  }
=
  V( X_{  \tzero  } )
  +
  \int_{  \tzero  }^t
  \tfrac{
    V'( X_s ) \,
      b_s
    - V( X_s ) \zeta_s
  }{
    \Psi[ 
      \chi , \zeta
    ]_s
  }
  \, dW_s
\\ & \quad
  +
  \int_{  \tzero  }^t
  \tfrac{
    V'( X_s ) \,
      a_s 
    +
    \frac{ 1 }{ 2 }
    \operatorname{trace}\left(
      b_s^* 
      \,
      ( \operatorname{Hess} V)( X_s )
      \,
      b_s
    \right)
    +
    \operatorname{trace}\left(
      \zeta_s^*
      \left[ 
        V( X_s ) \zeta_s
        -
        V'( X_s ) \, b_s
      \right]    
    \right)
    - 
    V( X_s ) \chi_s
  }{
    \Psi[ 
      \chi , \zeta
    ]_s
  }
  \, ds
\end{split}
\end{equation}
$ \P $-a.s.\ for all 
$ t \in [  \tzero , T ] $.
\end{lemma}

\begin{proof}[Proof
of Lemma~\ref{lem:perturbation_formula}]
Applying 
It\^{o}'s formula 
to the process
$
  \frac{ V( X_t ) }{ \Psi[ \chi, \zeta ]_t } 
$,
$ t \in [  \tzero , T ] $,
results in
\begin{equation}
\begin{split}
&
  \frac{ V( X_t ) }{
    \Psi[ \chi, \zeta ]_t
  }
\\ & =
  V( X_{  \tzero  } )
  +
  \int_{  \tzero  }^t
  \frac{
    V'( X_s )
    \, b_s
    -
    V( X_s ) \, \zeta_s
  }{
    \Psi[ \chi, \zeta ]_s
  }
  \,
  dW_s
\\ & \quad 
  +
  \int_{  \tzero  }^t
  \frac{
    V'( X_s )
    \, a_s
    +
    \tfrac{ 1 }{ 2 }
    \operatorname{trace}\!\left(
      ( b_s )^*
      ( \operatorname{Hess} V)( X_s )
      \,
      b_s
    \right)
    -
    V( X_s ) 
    \left[ 
      \chi_s - 
      \tfrac{ 1 }{ 2 } 
      \| \zeta_s \|^2_{ U^* }
    \right]
  }{
    \Psi[ \chi, \zeta ]_s
  }
  \,
  ds
\\ & \quad 
  +
  \int_{  \tzero  }^t
  \frac{
    \frac{ 1 }{ 2 }
    V( X_s ) \, \| \zeta_s \|^2_{ U^* }
    -
    \operatorname{trace}\!\left(
      \zeta_s^*
      \,
      V'( X_s )
      \,
      b_s
    \right)
  }{
    \Psi[ \chi, \zeta ]_s
  }
  \,
  ds
\end{split}
\label{eq:ito_formel}
\end{equation}
$ \P $-a.s.\ for all
$ t \in [ \tzero ,T] $.
Combining this with 
the elementary identity
\begin{equation}
\begin{split}
&
  V( X_s ) \, \| \zeta_s \|^2_{ U^* }  
    -
    \operatorname{trace}\!\left(
      \zeta_s^*
      \,
      V'( X_s )
      \,
      b_s
    \right)
  =
  V( X_s ) \, \| \zeta_s \|^2_{ HS( U, \R ) }  
    -
    \operatorname{trace}\!\left(
      \zeta_s^*
      \,
      V'( X_s )
      \,
      b_s
    \right)
\\ & =
  \operatorname{trace}\!\left(
    \zeta_s^* V( X_s ) \zeta_s
  \right)
    -
    \operatorname{trace}\!\left(
      \zeta_s^*
      \,
      V'( X_s )
      \,
      b_s
    \right)
  =
    \operatorname{trace}\!\left(
      \zeta_s^*
      \,
      \left[ 
        V( X_s )
        \zeta_s
        -
        V'( X_s )
        \,
        b_s
      \right]
    \right)
\end{split}
\end{equation}
for all $ s \in [ \tzero ,T] $
completes the proof of Lemma~\ref{lem:perturbation_formula}.
\end{proof}

In the next lemma, Lemma~\ref{lem:perturbation_formula2}, 
we present a slightly different formulation of Lemma~\ref{lem:perturbation_formula2},
that is, we add and substract in \eqref{eq:first_perturbation} 
the generator in \eqref{eq:generator} and the 
noise operator in \eqref{eq:noise_operator}.
Lemma~\ref{lem:perturbation_formula2} is thus an immediate
consequence of Lemma~\ref{lem:perturbation_formula} and 
its proof is therefore omitted.

\begin{lemma}
\label{lem:perturbation_formula2}
Assume the setting in Subsection~\ref{sec:setting},
let $ V \in C^2( O, \R ) $,
$ 
  \mu \in 
  \mathcal{L}^0( \mathcal{O}; H )
$,
$
  \sigma \in 
  \mathcal{L}^0( \mathcal{O}; HS( U, H ) )
$,
let 
$ X \colon [  \tzero  , T ] \times \Omega \to \mathcal{O} $,
$ a \colon [  \tzero , T ] \times \Omega \to H $,
$ b \colon [  \tzero , T ] \times \Omega \to HS( U, H ) $,
$ \chi \colon [  \tzero , T ] \times \Omega \to \R $,
%and
$ \zeta \colon [  \tzero , T ] \times \Omega \to U^* $
be predictable stochastic processes
with 
$
  \int_{  \tzero  }^T
  \| a_s \|_H
  +
  \| b_s \|_{ HS( U, H ) }^2
  +
  | \chi_s |
  +
  \| \zeta_s \|^2_{ U^* }
  +
  \| \mu( X_s ) \|_H
  +
  \| \sigma( X_s ) \|^2_{ HS( U, H ) }
  \,
  ds
  < \infty
$
$ \P $-a.s.\ and
$
  X_t 
  = 
  X_{  \tzero  } 
  +
  \int_{  \tzero  }^t a_s \, ds
  +
  \int_{  \tzero  }^t b_s \, dW_s
$
$ \P $-a.s.\ for all
$ t \in [ \tzero ,T] $.
Then
\begin{equation}
\begin{split}
  \frac{
    V( X_t )
  }{
    \Psi[ 
      \chi, \zeta
    ]_t
  }
& =
  V( X_{  \tzero  } )
  +
  \int_{  \tzero  }^t
  \tfrac{
    ( \mathcal{G}_{ \mu, \sigma } V)( 
      X_s
    )
    - 
    \chi_s V( X_s )
    +
    \operatorname{trace}\left(
      \zeta_s^*
      \left[ 
        V( X_s ) \zeta_s
        -
        V'( X_s ) \, b_s
      \right]    
    \right)
  }{
    \Psi[ 
      \chi, \zeta
    ]_s
  }
  \, ds
\\ & \quad
  +
  \int_{  \tzero  }^t
  \tfrac{
    V'( X_s ) 
    \left[
      a_s 
      -
      \mu( X_s )
    \right]
    +
    \frac{ 1 }{ 2 }
    \operatorname{trace}\left(
      \left[ 
        b_s + 
        \sigma( 
          X_s 
        ) 
      \right]^*
      ( \operatorname{Hess} V)( X_s )
      \left[ 
        b_s
        -
        \sigma( 
          X_s 
        )
      \right]
    \right)
  }{
    \Psi[ 
      \chi, \zeta
    ]_s
  }
  \, ds
\\ & \quad 
  +
  \int_{  \tzero  }^t
  \tfrac{
    V'( X_s ) 
    \left[ 
      b_s
      - 
      \sigma( 
        X_s 
      )
    \right]
    +
    ( G_{ \sigma } V)( 
      X_s 
    )
    - V( X_s ) \zeta_s
  }{
    \Psi[ 
      \chi, \zeta
    ]_s
  }
  \, dW_s
\end{split}
\end{equation}
$ \P $-a.s.\ for all 
$ t \in [  \tzero , T ] $.
\end{lemma}

\subsection{A perturbation formula}

In the next result, Proposition~\ref{prop:perturbation_formula},
we formulate the special case of Lemma~\ref{lem:perturbation_formula2}
where the stochastic process $ ( X_t )_{ \in [  \tzero , T] } $ 
in Lemma~\ref{lem:perturbation_formula2} is the pairing 
of two stochastic processes $ X = ( X^1, X^2 ) $.
% This is the subject of 
% Proposition~\ref{prop:perturbation_formula}.
% Proposition~\ref{prop:perturbation_formula} is a special case
% of Lemma~\ref{lem:perturbation_formula2} and thus omitted.

\begin{prop}
\label{prop:perturbation_formula}
Assume the setting in Subsection~\ref{sec:setting},
let $ 
  V = 
  ( V(x_1,x_2) )_{ (x_1, x_2) \in O^2 } 
  \in C^2( O^2, \R ) 
$,
$
  \mu \in \mathcal{L}^0( \mathcal{O}; H ) 
$,
$
  \sigma \in 
  \mathcal{L}^0( \mathcal{O}; HS( U, H ) )
$,
let
$ X^i \colon [  \tzero  , T ] \times \Omega \to \mathcal{O} $,
$ a^i \colon [  \tzero , T ] \times \Omega \to H $,
$ b^i \colon [  \tzero , T ] \times \Omega \to HS( U, H ) $,
$ i \in \{ 1, 2 \} $,
$ \chi \colon [  \tzero , T ] \times \Omega \to \R $,
$ \zeta \colon [  \tzero , T ] \times \Omega \to U^* $
be predictable stochastic processes
with 
$
  \int_{  \tzero  }^T
  \| a_s^i \|_H
  +
  \| b_s^i \|_{ HS( U, H ) }^2
  +
  | \chi_s |
  +
  \| \zeta_s \|^2_{ U^* }
  +
  \| \mu( X_s^i ) \|_H
  +
  \| \sigma( X_s^i ) \|_{ HS( U, H ) }^2
  \,
  ds
  < \infty
$
$ \P $-a.s.\ and
$
  X_t^i 
  = 
  X^i_{  \tzero  } 
  +
  \int_{  \tzero  }^t a^i_s \, ds
  +
  \int_{  \tzero  }^t b^i_s \, dW_s
$
$ \P $-a.s.\ for all
$ t \in [ \tzero ,T] $,
$ i \in \{ 1, 2 \} $.
Then
\begin{equation}
\begin{split}
&
  \frac{
    V( X^1_t, X^2_t )
  }{
    \Psi[ 
      \chi, \zeta
    ]_t
  }
=
  V( X^1_{  \tzero  }, X^2_{  \tzero  } )
  +
  \smallint\limits_{  \tzero  }^t
  \tfrac{
    \sum_{ i = 1 }^2
    (
      \partial_{ x_i }
      V
    )( X^1_s, X^2_s ) 
    \left[ 
      b^i_s
      - 
      \sigma( 
        X^i_s
      )
    \right]
    +
    ( 
      \overline{G}_{ \sigma } V
    )( 
      X^1_s ,
      X^2_s 
    )
    - V( X^1_s, X^2_s ) \zeta_s
  }{
    \Psi[ 
      \chi, \zeta
    ]_s
  }
  \, dW_s
\\ & \quad
  +
  \smallint\limits_{  \tzero  }^t
%   \int_{  \tzero  }^t
  \tfrac{
    ( 
      \overline{\mathcal{G}}_{ \mu, \sigma } 
      V
    )( 
      X^1_s ,
      X^2_s 
    )
    - 
    \chi_s V( X^1_s, X^2_s )
    +
    \operatorname{trace}\left(
      \zeta_s^*
      \left[ 
        V( X^1_s, X^2_s ) \zeta_s
        -
        \sum_{ i = 1 }^2
        ( \partial_{ x_i } V )( X^1_s, X^2_s ) 
        \, b_s^i
      \right]    
    \right)
  }{
    \Psi[ 
      \chi, \zeta
    ]_s
  }
  \, ds
\\ & \quad
  +
  \smallint\limits_{  \tzero  }^t
%  \int_{  \tzero  }^t
  \tfrac{
    \sum_{ i = 1 }^2
    ( \partial_{ x_i } V )( X^1_s, X^2_s ) 
    \left[
      a^i_s 
      -
      \mu( 
        X^i_s 
      )
    \right]
    +
    \frac{ 1 }{ 2 }
    \sum_{ i = 1 }^2
    \operatorname{trace}\left(
      \left[ 
        b^i_s + 
        \sigma( 
          X^i_s 
        )
      \right]^*
      ( \operatorname{Hess}_{ x_i } V)( X^1_s, X^2_s )
      \left[ 
        b^i_s
        -
        \sigma( 
          X^i_s 
        )
      \right]
    \right)
  }{
    \Psi[ 
      \chi, \zeta
    ]_s
  }
  \, ds
\\ & \quad
  +
  \smallsum\limits_{ k = 1 }^{ \infty }
  \smallint\limits_{  \tzero  }^t
  \tfrac{
    \sum_{ i = 1 }^2
    \left(
      \partial_{ x_i } \partial_{ x_{ 3 - i } }
      V
    \right)( X^1_s, X^2_s ) 
    \left(
      \left[ 
        b_s^i 
        + 
        \sigma( 
          X^i_s
        )
      \right] 
      e_k
      ,
      \left[ 
        b_s^{ 3 - i }
        -
        \sigma( 
          X^{ 3 - i }_s
        )
      \right] 
      e_k
    \right)
%     +
%     \left(
%       \partial_{ x_2 } \partial_{ x_1 }
%       V
%     \right)( X^1_s, X^2_s ) 
%     \left(
%       [ b_s^2 + \sigma( X_s^2 ) ] e_k
%       ,
%       [ b_s^1 - \sigma( X_s^1 ) ] e_k
%     \right)
  }{
    2 \,
    \Psi[ 
      \chi, \zeta
    ]_s
  }
  \, ds
\end{split}
\end{equation}
$ \P $-a.s.\ for all 
$ t \in [  \tzero , T ] $.
\end{prop}

Next we formulate the special case
of Proposition~\ref{prop:perturbation_formula}
where the stochastic process $ ( X^1_t )_{ t \in [ \tzero , T] } $ 
in Proposition~\ref{prop:perturbation_formula}
is a solution process of 
the SDE
with drift coefficient $ \mu $ and 
diffusion coefficient $ \sigma $.

\begin{corollary}
\label{cor:perturbation}
Assume the setting in Subsection~\ref{sec:setting},
let $ V = ( V(x, y) )_{ (x, y) \in O^2 } \in C^2( O^2, \R ) $,
$
  \mu \in \mathcal{L}^0( \mathcal{O}; H ) 
$,
$
  \sigma \in 
  \mathcal{L}^0( \mathcal{O}; HS( U, H ) )
$,
let
$ X, Y \colon [  \tzero  , T ] \times \Omega \to \mathcal{O} $,
$ a \colon [  \tzero , T ] \times \Omega \to H $,
$ b \colon [  \tzero , T ] \times \Omega \to HS( U, H ) $,
$ \chi \colon [  \tzero , T ] \times \Omega \to \R $,
%and
$ \zeta \colon [  \tzero , T ] \times \Omega \to U^* $
be predictable stochastic processes
with 
$
  \int_{  \tzero  }^T
  \| a_s \|_H
  +
  \| b_s \|_{ HS( U, H ) }^2
  +
  | \chi_s |
  +
  \| \zeta_s \|^2_{ U^* }
  +
  \| \mu( X_s ) \|_H
  +
  \| \sigma( X_s ) \|^2_{ HS( U, H ) }
  +
  \| \mu( Y_s ) \|_H
  +
  \| \sigma( Y_s ) \|^2_{ HS( U, H ) }
  \,
  ds
  < \infty
$
$ \P $-a.s.,
%\ and 
%let
%$ X, Y \colon [  \tzero  , T ] \times \Omega \to \mathcal{O} $
%be adapted stochastic processes with c.s.p.\ satisfying
$
  X_t 
  = 
  X_{  \tzero  } 
  +
  \int_{  \tzero  }^t \mu( X_s ) \, ds
  +
  \int_{  \tzero  }^t \sigma( X_s ) \, dW_s
$
$ \P $-a.s.\ and
$
  Y_t 
  = 
  Y_{  \tzero  } 
  +
  \int_{  \tzero  }^t a_s \, ds
  +
  \int_{  \tzero  }^t b_s \, dW_s
$
$ \P $-a.s.\ for all
$ t \in [ \tzero ,T] $.
Then
\begin{equation}
\label{eq:cor_perturbation}
\begin{split}
&
  \frac{
    V( X_t, Y_t )
  }{
    \Psi[ 
      \chi, \zeta
    ]_t
  }
=
  V( X_{  \tzero  }, Y_{  \tzero  } )
  +
  \smallint\limits_{  \tzero  }^t
  \tfrac{
    (
      \partial_y
      V
    )( X_s, Y_s ) 
    \left[ 
      b_s
      - \sigma( Y_s  )
    \right]
    +
    ( \overline{G}_{ \sigma } V)( X_s , Y_s  )
    - V( X_s, Y_s ) \zeta_s
  }{
    \Psi[ 
      \chi, \zeta
    ]_s
  }
  \, dW_s
\\ & 
  +
  \smallint\limits_{  \tzero  }^t
%   \int_{  \tzero  }^t
  \tfrac{
    ( \overline{\mathcal{G}}_{ \mu, \sigma } V)( X_s , Y_s )
    - 
    \chi_s V( X_s, Y_s )
    +
    \operatorname{trace}\left(
      \zeta_s^*
      \left[ 
        V( X_s, Y_s ) \zeta_s
        -
        ( \partial_x V )( X_s, Y_s ) 
        \, \sigma( X_s )
        -
        ( \partial_y V )( X_s, Y_s ) 
        \, b_s
      \right]    
    \right)
  }{
    \Psi[ 
      \chi, \zeta
    ]_s
  }
  \, ds
\\ & 
  +
  \smallint\limits_{  \tzero  }^t
%  \int_{  \tzero  }^t
  \tfrac{
    ( \partial_y V )( X_s, Y_s ) 
    \left[
      a_s 
      -
      \mu( Y_s )
    \right]
    +
    \frac{ 1 }{ 2 }
    \operatorname{trace}\left(
      \left[ 
        b_s + \sigma( Y_s ) 
      \right]^*
      ( \operatorname{Hess}_y V)( X_s, Y_s )
      \left[ 
        b_s
        -
        \sigma( Y_s )
      \right]
    \right)
  }{
    \Psi[ 
      \chi, \zeta
    ]_s
  }
  \, ds
\\ & 
  +
  \smallsum\limits_{ k = 1 }^{ \infty }
  \smallint\limits_{  \tzero  }^t
%  \int\limits_{  \tzero  }^t
  \tfrac{
    \left(
%       \frac{ \partial^2 }{
        \partial_x \partial_y 
%       }
      V
    \right)( X_s, Y_s ) 
    \left(
      \sigma( X_s ) e_k
      ,
      [ b_s - \sigma( Y_s ) ] e_k
    \right)
  }{
    \Psi[ 
      \chi, \zeta
    ]_s
  }
  \, ds
\end{split}
\end{equation}
$ \P $-a.s.\ for all 
$ t \in [  \tzero , T ] $.
\end{corollary}

Note in the setting of Corollary~\ref{cor:perturbation}
that if $ Y $ is also a solution of the SDE with drift
coefficient $ \mu $ and diffusion coefficient $ \sigma $
too
and if $ \chi $ and $ \zeta $ are appropriate
(see Proposition~2.12
in Cox et al.~\cite{CoxHutzenthalerJentzen2013}),
then Corollary~\ref{cor:perturbation} essentially reduces to
%is a special case of 
Proposition~2.12 in Cox et al.~\cite{CoxHutzenthalerJentzen2013}
and can be used to study the regularity of solutions of SDEs in
the initial value.
The next result, Proposition~\ref{prop:main_perturbation},
formulates the special case of Corollary~\ref{cor:perturbation}
where the process $ \zeta \equiv 0 $ vanishes.

\begin{prop}
\label{prop:main_perturbation}
Assume the setting in Subsection~\ref{sec:setting},
let $ V = ( V(x, y) )_{ ( x, y ) \in O^2 } \in C^2( O^2, \R ) $,
$
  \mu \in \mathcal{L}^0( \mathcal{O}; H ) 
$,
$
  \sigma \in 
  \mathcal{L}^0( \mathcal{O}; HS( U, H ) )
$,
let
$ 
  X, Y \colon [  \tzero  , T ] \times \Omega \to \mathcal{O} 
$,
$ a \colon [  \tzero , T ] \times \Omega \to H $,
$ b \colon [  \tzero , T ] \times \Omega \to HS( U, H ) $,
$ \chi \colon [  \tzero , T ] \times \Omega \to \R $
be predictable stochastic processes
with 
$
  \int_{  \tzero  }^T
  \| a_s \|_H
  +
  \| b_s \|_{ HS( U, H ) }^2
  +
  | \chi_s |
  +
  \| \mu( X_s ) \|_H
  +
  \| \sigma( X_s ) \|_{ HS( U, H ) }^2
  +
  \| \mu( Y_s ) \|_H
  +
  \| \sigma( Y_s ) \|_{ HS( U, H ) }^2
  \,
  ds
  < \infty
$
$ \P $-a.s.,
$
  X_t 
  = 
  X_{  \tzero  } 
  +
  \int_{  \tzero  }^t \mu( X_s ) \, ds
  +
  \int_{  \tzero  }^t \sigma( X_s ) \, dW_s
$
$ \P $-a.s.\ and
$
  Y_t 
  = 
  Y_{  \tzero  } 
  +
  \int_{  \tzero  }^t a_s \, ds
  +
  \int_{  \tzero  }^t b_s \, dW_s
$
$ \P $-a.s.\ for all
$ t \in [ \tzero ,T] $.
Then
\begin{equation}
\label{eq:main_perturbation}
\begin{split}
&
  \frac{
    V( X_t, Y_t )
  }{
    \exp\!\big(
      \int_{  \tzero  }^t \chi_s \, ds
    \big)
  }
=
  V( X_{  \tzero  }, Y_{  \tzero  } )
  +
  \smallint\limits_{  \tzero  }^t
  \tfrac{
    (
        \partial_{ y }
      V
    )( X_s, Y_s ) 
    \left[ 
      b_s
      - \sigma( Y_s )
    \right]
    +
    ( \overline{G}_{ \sigma } V)( X_s , Y_s )
  }{
    \exp(
      \int_{  \tzero  }^s \chi_u \, du
    )
  }
  \, dW_s
\\ & \quad
  +
  \smallint\limits_{  \tzero  }^t
%   \int_{  \tzero  }^t
  \tfrac{
    ( \overline{\mathcal{G}}_{ \mu, \sigma } V)( X_s , Y_s )
    - 
    \chi_s V( X_s, Y_s )
    +
    \sum_{ k = 1 }^{ \infty }
    \left(
      \partial_x \partial_y 
      V
    \right)( X_s, Y_s ) 
    \left(
      \sigma( X_s ) e_k
      ,
      [ b_s - \sigma( Y_s ) ] e_k
    \right)
  }{
    \exp(
      \int_{  \tzero  }^s \chi_u \, du
    )
  }
  \, ds
\\ & \quad
  +
  \smallint\limits_{  \tzero  }^t
%  \int_{  \tzero  }^t
  \tfrac{
    ( \partial_y V )( X_s, Y_s ) 
    \left[
      a_s 
      -
      \mu( Y_s )
    \right]
    +
    \frac{ 1 }{ 2 }
    \operatorname{trace}\left(
      \left[ 
        b_s + \sigma( Y_s ) 
      \right]^*
      ( \operatorname{Hess}_{ y } V)( X_s, Y_s )
      \left[ 
        b_s
        -
        \sigma( Y_s )
      \right]
    \right)
  }{
    \exp(
      \int_{  \tzero  }^s \chi_u \, du
    )
  }
  \, ds
\end{split}
\end{equation}
$ \P $-a.s.\ for all 
$ t \in [  \tzero , T ] $.
\end{prop}

\subsection{Perturbation estimates}
\label{ssec:perturbation.estimates}

Our central goal is to estimate 
the quantity 
$
  \sup_{ t \in [  \tzero , T ] }
  \| V( X_t, Y_t ) \|_{ L^r( \Omega; \R ) }
$
for some $ r \in (0,\infty) $
in \eqref{eq:main_perturbation} in 
Proposition~\ref{prop:main_perturbation}.
To do so, we apply in 
the next lemma
a localization argument 
together with H\"{o}lder's inequality to \eqref{eq:main_perturbation}.

\begin{lemma}
\label{lem:generalV_perturbation_estimate}
Assume the setting in Subsection~\ref{sec:setting},
let $ V = ( V(x, y) )_{ ( x, y ) \in O^2 } \in C^2( O^2, [0,\infty) ) $,
$
  \mu \in \mathcal{L}^0( \mathcal{O}; H ) 
$,
$
  \sigma \in 
  \mathcal{L}^0( \mathcal{O}; HS( U, H ) )
$,
let
$ \tau \colon \Omega \to [  \tzero , T ] $
be a stopping time,
let
$ 
  X, Y \colon [  \tzero  , T ] \times \Omega \to \mathcal{O} 
$
be adapted stochastic processes with continuous sample paths (c.s.p.),
let
$ a \colon [  \tzero , T ] \times \Omega \to H $,
$ b \colon [  \tzero , T ] \times \Omega \to HS( U, H ) $,
$ \chi \colon [  \tzero , T ] \times \Omega \to \R $
be predictable stochastic processes
with 
$
  \int_{  \tzero  }^T
  | \chi_s |
  +
  \| a_s \|_H
  +
  \| b_s \|_{ HS( U, H ) }^2
  +
  \| \mu( X_s ) \|_H
  +
  \| \sigma( X_s ) \|_{ HS( U, H ) }^2
  +
  \| \mu( Y_s ) \|_H
  +
  \| \sigma( Y_s ) \|_{ HS( U, H ) }^2
  \,
  ds
  < \infty
$
$ \P $-a.s.,
$
  X_t 
  = 
  X_{  \tzero  } 
  +
  \int_{  \tzero  }^t \mu( X_s ) \, ds
  +
  \int_{  \tzero  }^t \sigma( X_s ) \, dW_s
$
$ \P $-a.s.\ and
$
  Y_t 
  = 
  Y_{  \tzero  } 
  +
  \int_{  \tzero  }^t a_s \, ds
  +
  \int_{  \tzero  }^t b_s \, dW_s
$
$ \P $-a.s.\ for all
$ t \in [ \tzero ,T] $.
Then it holds for all $ p \in (0,1] $ that
{\small
\begin{align}
&
  \left\|
    V( X_{ \tau }, Y_{ \tau } )
  \right\|_{
    L^p( \Omega; \R )
  }
\leq
  \left\|
    \exp\!\left(
      \smallint_{  \tzero  }^{ \tau }
        \chi_s 
      \,
      ds
    \right)
  \right\|_{
    L^{
      p / ( 1 - p ) 
    }( \Omega; \R )
  }
  \sup\Bigg\{
  \E\bigg[
    V( X_{  \tzero  }, Y_{  \tzero  } )
    +
  \smallint\limits_{  \tzero  }^{ \nu \wedge \tau }
  \Big[
    ( \overline{\mathcal{G}}_{ \mu, \sigma } V)( X_s , Y_s )
\nonumber
\\ &
\nonumber
    - 
    \chi_s V( X_s, Y_s )
    +
    \smallsum_{ k = 1 }^{ \infty }
    \left(
      \partial_x \partial_y 
      V
    \right)\!( X_s, Y_s ) 
    \left(
      \sigma( X_s ) e_k
      ,
      [ b_s - \sigma( Y_s ) ] e_k
    \right)
  +
    ( \partial_y V )( X_s, Y_s ) 
    \left[
      a_s 
      -
      \mu( Y_s )
    \right]
\\ & 
    +
    \tfrac{ 1 }{ 2 }
    \operatorname{trace}\!\big(
      \left[ 
        b_s + \sigma( Y_s ) 
      \right]^*
      ( \operatorname{Hess}_{ y } V)( X_s, Y_s )
      \left[ 
        b_s
        -
        \sigma( Y_s )
      \right]
    \big)
  \Big]
  \exp\!\big(
    - \smallint\nolimits_{  \tzero  }^{ s } \chi_u \, du
  \big)
  \, ds
  \bigg]
  \colon
  \substack{ 
    \nu \text{ stopping}
    \\
    \text{time such that}
  }
\nonumber
\\ &
      \smallsum\limits_{ i = 0 }^2
      \,
      \smallint\limits_{  \tzero  }^{ \nu }
      \left\|
        V^{ (i) }( X_s, Y_s )
      \right\|_{
        L^{ (i) }( H^2, \R )
      }^2
      \Big[
        | \chi_s |
        +
        \| a_s \|_H
        +
        \| b_s \|^2_{ HS( U, H ) }
        +
        \| \mu( X_s ) \|_H
        +
        \| \mu( Y_s ) \|_H
\label{eq:perturbation_general}
\\ &
\nonumber
        +
        \|
          \sigma( X_s )
        \|^2_{ HS( U, H ) }
        +
        \|
          \sigma( Y_s )
        \|^2_{ HS( U, H ) }
      \Big]
      \, ds
    \in L^{ \infty }( \Omega; \R ) 
  \Bigg\}
\nonumber
\end{align}}\end{lemma}

\begin{proof}[Proof
of Lemma~\ref{lem:generalV_perturbation_estimate}]
Throughout this proof let 
$ \tau_n \colon \Omega \to [  \tzero , T ] $,
$ n \in \N $,
be stopping times given by
\begin{multline}
  \tau_n
:=
  \inf\!\bigg(
    \left\{ \tau \right\}
    \cup
    \Big\{ 
      t \in [  \tzero , T ]
      \colon
      \smallsum\limits_{ i = 0 }^2
      \smallint\limits_{  \tzero  }^{ t }
      \left\|
        V^{ (i) }( X_s, Y_s )
      \right\|_{
        L^{ (i) }( H^2, \R )
      }^2
      \Big[
        | \chi_s |
        +
        \| a_s \|_H
        +
        \| b_s \|^2_{ HS( U, H ) }
\\
        +
        \| \mu( X_s ) \|_H
        +
        \| \mu( Y_s ) \|_H
        +
        \|
          \sigma( X_s )
        \|^2_{ HS( U, H ) }
        +
        \|
          \sigma( Y_s )
        \|^2_{ HS( U, H ) }
      \Big]
      \,
      ds
      \geq n
    \Big\}
  \bigg)
\end{multline}
for all $ n \in \N $.
Proposition~\ref{prop:main_perturbation}
then implies
\begin{equation}
\label{eq:general_perturbation_proof}
\begin{split}
&
  \frac{
    V( X_{ \tau_n } , Y_{ \tau_n } )
  }{
    \exp\!\big(
      \int_{  \tzero  }^{ \tau_n } \chi_s \, ds
    \big)
  }
=
  V( X_{  \tzero  }, Y_{  \tzero  } )
  +
  \smallint\limits_{  \tzero  }^{ \tau_n }
  \tfrac{
    (
        \partial_{ y }
      V
    )( X_s, Y_s ) 
    \left[ 
      b_s
      - \sigma( Y_s )
    \right]
    +
    ( \overline{G}_{ \sigma } V)( X_s , Y_s )
  }{
    \exp(
      \int_{  \tzero  }^s \chi_u \, du
    )
  }
  \, dW_s
\\ & \quad
  +
  \smallint\limits_{  \tzero  }^{ \tau_n }
  \tfrac{
    ( \overline{\mathcal{G}}_{ \mu, \sigma } V)( X_s , Y_s )
    - 
    \chi_s V( X_s, Y_s )
    +
    \sum_{ k = 1 }^{ \infty }
    \left(
      \partial_x \partial_y 
      V
    \right)( X_s, Y_s ) 
    \left(
      \sigma( X_s ) e_k
      ,
      [ b_s - \sigma( Y_s ) ] e_k
    \right)
  }{
    \exp(
      \int_{  \tzero  }^s \chi_u \, du
    )
  }
  \, ds
\\ & \quad
  +
  \smallint\limits_{  \tzero  }^{ \tau_n }
  \tfrac{
    ( \partial_y V )( X_s, Y_s ) 
    \left[
      a_s 
      -
      \mu( Y_s )
    \right]
    +
    \frac{ 1 }{ 2 }
    \operatorname{trace}\left(
      \left[ 
        b_s + \sigma( Y_s ) 
      \right]^*
      ( \operatorname{Hess}_{ y } V)( X_s, Y_s )
      \left[ 
        b_s
        -
        \sigma( Y_s )
      \right]
    \right)
  }{
    \exp(
      \int_{  \tzero  }^s \chi_u \, du
    )
  }
  \, ds
\end{split}
\end{equation}
$ \P $-a.s.\ for all 
$ n \in \N $.
Taking expectations in \eqref{eq:general_perturbation_proof}
shows that for all $ n \in \N $ it holds that
\begin{equation}
\label{eq:general_perturbation_proof2}
\begin{split}
&
  \E\!\left[
  \frac{
    V( X_{ \tau_n } , Y_{ \tau_n } )
  }{
    \exp\!\big(
      \int_{  \tzero  }^{ \tau_n } \chi_s \, ds
    \big)
  }
  \right]
=
  \E\big[
    V( X_{  \tzero  }, Y_{  \tzero  } )
  \big]
\\ & \quad
  +
  \E\bigg[
  \smallint\limits_{  \tzero  }^{ \tau_n }
  \tfrac{
    ( \overline{\mathcal{G}}_{ \mu, \sigma } V)( X_s , Y_s )
    - 
    \chi_s V( X_s, Y_s )
    +
    \sum_{ k = 1 }^{ \infty }
    \left(
      \partial_x \partial_y 
      V
    \right)( X_s, Y_s ) 
    \left(
      \sigma( X_s ) e_k
      ,
      [ b_s - \sigma( Y_s ) ] e_k
    \right)
  }{
    \exp(
      \int_{  \tzero  }^s \chi_u \, du
    )
  }
  \, ds
\\ & \quad
  +
  \smallint\limits_{  \tzero  }^{ \tau_n }
  \tfrac{
    ( \partial_y V )( X_s, Y_s ) 
    \left[
      a_s 
      -
      \mu( Y_s )
    \right]
    +
    \frac{ 1 }{ 2 }
    \operatorname{trace}\left(
      \left[ 
        b_s + \sigma( Y_s ) 
      \right]^*
      ( \operatorname{Hess}_{ y } V)( X_s, Y_s )
      \left[ 
        b_s
        -
        \sigma( Y_s )
      \right]
    \right)
  }{
    \exp(
      \int_{  \tzero  }^s \chi_u \, du
    )
  }
  \, ds
  \bigg]
  .
\end{split}
\end{equation}
Next note that H\"{o}lder's inequality proves that
for all $ p \in (0,1] $ it holds that
\begin{equation}
\label{eq:general_perturbation_proof3}
\begin{split}
  \left\|
    V( X_{ \tau }, Y_{ \tau } )
  \right\|_{
    L^p( \Omega; \R )
  }
& =
  \left\|
    \frac{
      V( X_{ \tau }, Y_{ \tau } )
    }{
      \exp(
        \int_{  \tzero  }^{ \tau }
          \chi_s 
        \,
        ds
      )
    }
    \exp\!\left(
      \smallint_{  \tzero  }^{ \tau }
        \chi_s 
      \,
      ds
    \right)
  \right\|_{
    L^p( \Omega; \R )
  }
\\ & \leq
  \left\|
    \frac{
      V( X_{ \tau }, Y_{ \tau } )
    }{
      \exp(
        \int_{  \tzero  }^{ \tau }
          \chi_s 
        \,
        ds
      )
    }
  \right\|_{
    L^1( \Omega; \R )
  }
  \left\|
    \exp\!\left(
      \smallint_{  \tzero  }^{ \tau }
        \chi_s 
      \,
      ds
    \right)
  \right\|_{
    L^{ p / ( 1 - p ) }( \Omega; \R )
  }
\\ & =
  \E\!\left[
    \frac{
      V( X_{ \tau }, Y_{ \tau } )
    }{
      \exp(
        \int_{  \tzero  }^{ \tau }
          \chi_s 
        \,
        ds
      )
    }
  \right]
  \left\|
    \exp\!\left(
      \smallint_{  \tzero  }^{ \tau }
        \chi_s 
      \,
      ds
    \right)
  \right\|_{
    L^{ p / ( 1 - p ) }( \Omega; \R )
  }
  .
\end{split}
\end{equation}
Combining \eqref{eq:general_perturbation_proof2} and 
\eqref{eq:general_perturbation_proof3} with Fatou's lemma
completes the proof of 
Lemma~\ref{lem:generalV_perturbation_estimate}.
\end{proof}

If the right-hand side of \eqref{eq:perturbation_general}
is further estimated in an appropriate way, then a 
more compact statement can be obtained.
This is the subject of the next corollary.

\begin{corollary}
\label{cor:perturbation_estimate2}
Assume the setting in Subsection~\ref{sec:setting},
let $ V = ( V(x, y) )_{ ( x, y ) \in O^2 } \in C^2( O^2, [0,\infty) ) $,
$
  \mu \in \mathcal{L}^0( \mathcal{O}; H ) 
$,
$
  \sigma \in 
  \mathcal{L}^0( \mathcal{O}; HS( U, H ) )
$,
let
$ \tau \colon \Omega \to [  \tzero , T ] $
be a stopping time,
let
$ 
  X, Y \colon [  \tzero  , T ] \times \Omega \to \mathcal{O} 
$
be adapted stochastic processes with c.s.p.,
let
$ a \colon [  \tzero , T ] \times \Omega \to H $,
$ b \colon [  \tzero , T ] \times \Omega \to HS( U, H ) $,
$ \chi \colon [  \tzero , T ] \times \Omega \to [0,\infty) $
be predictable stochastic processes
with 
$
  \int_{  \tzero  }^T
  \| a_s \|_H
  +
  \| b_s \|_{ HS( U, H ) }^2
  +
  \| \mu( X_s ) \|_H
  +
  \| \sigma( X_s ) \|_{ HS( U, H ) }^2
  +
  \| \mu( Y_s ) \|_H
  +
  \| \sigma( Y_s ) \|_{ HS( U, H ) }^2
  +
  \chi_s
  \,
  ds
  < \infty
$
$ \P $-a.s.,
$
  X_t 
  = 
  X_{  \tzero  } 
  +
  \int_{  \tzero  }^t \mu( X_s ) \, ds
  +
  \int_{  \tzero  }^t \sigma( X_s ) \, dW_s
$
$ \P $-a.s.\ and
$
  Y_t 
  = 
  Y_{  \tzero  } 
  +
  \int_{  \tzero  }^t a_s \, ds
  +
  \int_{  \tzero  }^t b_s \, dW_s
$
$ \P $-a.s.\ for all
$ t \in [ \tzero ,T] $.
Then it holds for all $ p \in (0,1] $ that
{\small
\begin{align}
\label{cor:perturbation_general2}
&
  \left\|
    V( X_{ \tau }, Y_{ \tau } )
  \right\|_{
    L^p( \Omega; \R )
  }
\leq
  \left\|
    \exp\!\left(
      \smallint_{  \tzero  }^{ \tau }
        \chi_s 
      \,
      ds
    \right)
  \right\|_{
    L^{
      p / ( 1 - p ) 
    }( \Omega; \R )
  }
  \E\Big[
    V( X_{  \tzero  }, Y_{  \tzero  } )
    +
  \smallint\limits_{  \tzero  }^{ \tau }
  \Big[
    ( \overline{\mathcal{G}}_{ \mu, \sigma } V)( X_s , Y_s )
\nonumber
\\ &
\nonumber
    - 
    \chi_s V( X_s, Y_s )
    +
    \smallsum_{ k = 1 }^{ \infty }
    \left(
      \partial_x \partial_y 
      V
    \right)\!( X_s, Y_s ) 
    \left(
      \sigma( X_s ) e_k
      ,
      [ b_s - \sigma( Y_s ) ] e_k
    \right)
  +
    ( \partial_y V )( X_s, Y_s ) 
    \left[
      a_s 
      -
      \mu( Y_s )
    \right]
\\ & 
    +
    \tfrac{ 1 }{ 2 }
    \operatorname{trace}\!\big(
      \left[ 
        b_s + \sigma( Y_s ) 
      \right]^*
      ( \operatorname{Hess}_{ y } V)( X_s, Y_s )
      \left[ 
        b_s
        -
        \sigma( Y_s )
      \right]
    \big)
  \Big]^+
  ds
  \Big]
  .
\end{align}}\end{corollary}

Lemma~\ref{lem:generalV_perturbation_estimate}
can be used to study the regularity of solutions of SDEs
with respect to the initial values.
This is illustrates in the next result, Corollary~\ref{cor:perturbation_initial}.
Corollary~\ref{cor:perturbation_initial} follows immediately from
Lemma~\ref{lem:generalV_perturbation_estimate}
and its proof is thus omitted.

\begin{corollary}
\label{cor:perturbation_initial}
Assume the setting in Subsection~\ref{sec:setting},
let $ V \in C^2( O^2, [0,\infty) ) $,
$
  \mu \in \mathcal{L}^0( \mathcal{O}; H ) 
$,
$
  \sigma \in 
  \mathcal{L}^0( \mathcal{O}; HS( U, H ) )
$,
let
$ \tau \colon \Omega \to [  \tzero , T ] $
be a stopping time,
let
$ 
  X, Y \colon [  \tzero  , T ] \times \Omega \to \mathcal{O} 
$,
$
  \chi \colon [  \tzero , T ] \times \Omega \to \R
$
be predictable stochastic processes
with 
$
  \int_{  \tzero  }^T
  | \chi_s |
  +
  \| \mu( X_s ) \|_H
  +
  \| \sigma( X_s ) \|_{ HS( U, H ) }^2
  +
  \| \mu( Y_s ) \|_H
  +
  \| \sigma( Y_s ) \|_{ HS( U, H ) }^2
  \,
  ds
  < \infty
$
$ \P $-a.s.,
$
  \int\nolimits_{  \tzero  }^{ \tau }
    \big[
      ( \overline{\mathcal{G}}_{ \mu, \sigma } V)( X_s , Y_s )
      - 
      \chi_s V( X_s, Y_s )
    \big]^+
  \, ds
  \leq 0
$
$ \P $-a.s.,
$
  X_t 
  = 
  X_{  \tzero  } 
  +
  \int_{  \tzero  }^t \mu( X_s ) \, ds
  +
  \int_{  \tzero  }^t \sigma( X_s ) \, dW_s
$
$ \P $-a.s.\ and
$
  Y_t 
  = 
  Y_{  \tzero  } 
  +
  \int_{  \tzero  }^t \mu( Y_s ) \, ds
  +
  \int_{  \tzero  }^t \sigma( Y_s ) \, dW_s
$
$ \P $-a.s.\ for all
$ t \in [ \tzero ,T] $.
Then it holds for all $ p \in (0,1] $ that
\begin{equation}
\begin{split}
&
  \left\|
    V( X_{ \tau }, Y_{ \tau } )
  \right\|_{
    L^p( \Omega; \R )
  }
\leq
  \E\big[
    V( X_{  \tzero  }, Y_{  \tzero  } )
  \big]
  \left\|
    \exp\!\left(
      \smallint\nolimits_{  \tzero  }^{ \tau }
      \chi_s \,
      ds
    \right)
  \right\|_{
    L^{
      p / ( 1 - p ) 
    }( \Omega; \R )
  }
  .
\end{split}
\end{equation}
\end{corollary}

Corollary~\ref{cor:perturbation_initial} is a quite similar statement
to 
Proposition~2.17 in 
Cox et al.~\cite{CoxHutzenthalerJentzen2013}
in the case $ p = 1 $ in the setting of the proposition.
As Proposition~2.17 in 
Cox et al.~\cite{CoxHutzenthalerJentzen2013},
Corollary~\ref{cor:perturbation_initial} can now be used to study
the regularity with respect to the initial value for a number of nonlinear 
SDEs from the literature
(such as the Duffing-van der Pol oscillator, the Cox-Ingersoll-Ross process
and the Cahn-Hilliard-Cook equation); see Section~4 in Cox et al.~\cite{CoxHutzenthalerJentzen2013}
for a list of examples.

\subsection{Perturbation estimates
in the case of Hilbert space distances}
\label{sec:perturbation_estimate2}

This subsection investigates the special case
of Proposition~\ref{prop:main_perturbation}
where $ V \in C^2( O^2 , \R ) $ 
satisfies
$ 
  V(x,y) = \| x - y \|_H^p
$
for all $ x, y \in O $ and some $ p \in [2,\infty) $.

\begin{prop}
\label{prop:norm2}
Assume the setting in Subsection~\ref{sec:setting},
let
$ \mu \in \mathcal{L}^0( \mathcal{O} ; H ) $,
$ \sigma \in \mathcal{L}^0( \mathcal{O} ; HS( U, H ) ) $,
let
$ X, Y \colon [  \tzero , T ] \times \Omega \to \mathcal{O} $,
$ a \colon [  \tzero , T ] \times \Omega \to H $,
$ b \colon [  \tzero , T ] \times \Omega \to HS( U, H ) $,
$ \chi \colon [  \tzero , T ] \times \Omega \to \R $
be predictable stochastic processes
with 
$
  \int_{  \tzero  }^T
  \| a_s \|_H
  +
  \| b_s \|_{ HS( U, H ) }^2
  +
  | \chi_s |
  +
  \| \mu( X_s ) \|_H
  +
  \| \sigma( X_s ) \|^2_{ HS( U, H ) }
  +
  \| \mu( Y_s ) \|_H
  +
  \| \sigma( Y_s ) \|^2_{ HS( U, H ) }
  \,
  ds
  < \infty
$
$ \P $-a.s.,
$
  X_t 
  = 
  X_{  \tzero  } 
  +
  \int_{  \tzero  }^t \mu( X_s ) \, ds
  +
  \int_{  \tzero  }^t \sigma( X_s ) \, dW_s
$
$ \P $-a.s.\ and
$
  Y_t 
  = 
  Y_{  \tzero  } 
  +
  \int_{  \tzero  }^t a_s \, ds
  +
  \int_{  \tzero  }^t b_s \, dW_s
$
$ \P $-a.s.\ for all
$ t \in [ \tzero ,T] $.
Then
\begin{equation}
\label{eq:prop_norm2}
\begin{split}
&
  \frac{
    \left\| X_t - Y_t \right\|^p_H
  }{
    \exp(
      \int_{  \tzero  }^t \chi_s \, ds
    )
  }
\leq
  \left\| X_{  \tzero  } - Y_{  \tzero  } \right\|^p_H
  +
  \smallint\limits_{  \tzero  }^t
    \big\langle
    \tfrac{
      p
      \,
      \left\|
        X_s - Y_s 
      \right\|^{ ( p - 2 ) }_H
      \,
      \left[
        X_s - Y_s 
      \right]
    }{
      \exp(
        \int_{  \tzero  }^s \chi_u \, du
      )
    }
      ,
      \left[
      \sigma( X_s ) 
      - 
      b_s
      \right]
      dW_s
    \big\rangle_H
\\ & \quad
  +
  \smallint\limits_{  \tzero  }^t
  \tfrac{
    p
    \,
    \left\|
      X_s - Y_s 
    \right\|^{ ( p - 2 ) }_H
  \left[ 
    \left<
      X_s - Y_s
      ,
      \mu( Y_s )
      -
      a_s 
    \right>_H
    +
    \frac{ ( p - 1 ) \, ( 1 + 1 / \varepsilon ) 
    }{ 2 }
    \,
      \left\|
        b_s
        -
        \sigma( Y_s )
      \right\|^2_{ HS( U, H ) }
  \right]
    - 
    \chi_s \left\| X_s - Y_s \right\|^p_H
  }{
    \exp(
      \int_{  \tzero  }^s \chi_u \, du
    )
  }
  \, ds
\\ & \quad
  +
  \smallint\limits_{  \tzero  }^t
  \tfrac{
    p 
    \,
    \left\| X_s - Y_s \right\|^{ ( p - 2 ) }_H
    \left[
    \left< X_s - Y_s, \mu( X_s ) - \mu( Y_s ) \right>_H
    +
    \frac{ ( p - 1 ) \, ( 1 + \varepsilon ) 
    }{ 2 }
    \,
    \left\|
      \sigma( X_s ) - \sigma( Y_s )
    \right\|^2_{ HS( U, H ) }
    \right]
  }{
    \exp(
      \int_{  \tzero  }^s \chi_u \, du
    )
  }
  \, ds
\end{split}
\end{equation}
$ \P $-a.s.\ for all 
$ t \in [  \tzero , T ] $,
$ \varepsilon \in [0,\infty] $,
$ p \in [2,\infty) $.
\end{prop}

\begin{proof}[Proof of
Proposition~\ref{prop:norm2}]
Throughout this proof 
let $ p \in [2,\infty) $
be a real number and let $ V \in C^2( O^2 , \R ) $ 
be given by
$ 
  V(x,y) = \| x - y \|_H^p
$
for all $ x, y \in O $. 
Then observe that Remark~2.13
in Cox et al.~\cite{CoxHutzenthalerJentzen2013}
proves that for all $ s \in [ \tzero , T] $
it holds that
\begin{equation}
\begin{split}
&
    \smallsum_{ k = 1 }^{ \infty }
    \left(
      \partial_x \partial_y 
      V
    \right)\!( X_s, Y_s ) 
    \big(
      \sigma( X_s ) e_k
      ,
      \left[ 
        b_s - \sigma( Y_s ) 
      \right] e_k
    \big)
\\ & =
    - 
  \sum_{ k = 1 }^{ \infty }
    p
    \left\| X_s - Y_s \right\|^{ ( p - 2 ) }_H
    \left<
      \sigma( X_s ) e_k
      ,
      \left[ 
        b_s - \sigma( Y_s )
      \right]
      e_k
    \right>_H
\\ & 
  -
  \sum_{ k = 1 }^{ \infty }
    \mathbbm{1}_{
      \{ X_s \neq Y_s \}
    }
    \,
    p \left( p - 2 \right)
    \left\| X_s - Y_s \right\|^{ ( p - 4 ) }_H
    \left< X_s - Y_s , \sigma( X_s ) e_k \right>_H
    \left< X_s - Y_s , \left[ b_s - \sigma( Y_s ) \right] e_k \right>_H
\\ & =
    - 
    p
    \left\| X_s - Y_s \right\|^{ ( p - 2 ) }_H
    \operatorname{trace}\!\big(
      \sigma( X_s )^{ * } 
      \left[ 
        b_s - \sigma( Y_s )
      \right]
    \big)
\\ & 
  -
    \mathbbm{1}_{
      \{ X_s \neq Y_s \}
    }
    \,
    p \left( p - 2 \right)
    \left\| X_s - Y_s \right\|^{ ( p - 4 ) }_H
    \operatorname{trace}\!\big(
      \sigma( X_s )^*
      \left[ X_s - Y_s \right]
      \left[ X_s - Y_s \right]^*
      \left[
        b_s - \sigma( Y_s )
      \right]
    \big)
\end{split}
\label{eq:norm2_term1}
\end{equation}
and
\begin{equation}
\label{eq:norm2_term2}
\begin{split}
&
    \tfrac{ 1 }{ 2 }
    \operatorname{trace}\!\left(
      \left[ 
        b_s + \sigma( Y_s ) 
      \right]^*
      ( \operatorname{Hess}_{ y } V)( X_s, Y_s )
      \left[ 
        b_s
        -
        \sigma( Y_s )
      \right]
    \right)
\\ & =
  \tfrac{ p }{ 2 }
  \left\| X_s - Y_s \right\|^{ ( p - 2 ) }_H
    \operatorname{trace}\!\left(
      \left[ 
        b_s + \sigma( Y_s ) 
      \right]^*
      \left[ 
        b_s
        -
        \sigma( Y_s )
      \right]
    \right)
\\ &
  +
  \mathbbm{1}_{
    \{ X_s \neq Y_s \}
  }
  \,
  \tfrac{ p \left( p - 2 \right) }{ 2 }
  \left\| X_s - Y_s \right\|^{ ( p - 4 ) }_H
    \operatorname{trace}\!\left(
      \left[ 
        b_s + \sigma( Y_s ) 
      \right]^*
      \left[ X_s - Y_s \right]
      \left[ X_s - Y_s \right]^*
      \left[ 
        b_s
        -
        \sigma( Y_s )
      \right]
    \right)
  .
\end{split}
\end{equation}
Combining 
\eqref{eq:main_perturbation}
in Proposition~\ref{prop:main_perturbation}
together with 
\eqref{eq:norm2_term1}, \eqref{eq:norm2_term2}
and
%a minor generalization of Example~2.15
Remark~2.13
in Cox et al.~\cite{CoxHutzenthalerJentzen2013}
implies that
% 
% 
% 
% 
% Assume the setting in Subsection~\ref{sec:setting},
% let 
% $ X, Y \colon [  \tzero  , T ] \times \Omega \to \mathcal{O} $,
% $ a \colon [  \tzero , T ] \times \Omega \to H $,
% $ b \colon [  \tzero , T ] \times \Omega \to HS( U, H ) $
% and
% $ \chi \colon [  \tzero , T ] \times \Omega \to \R $
% be predictable stochastic processes
% with 
% $
%   \int_{  \tzero  }^T
%   \| a_s \|_H
%   +
%   \| b_s \|_{ HS( U, H ) }^2
%   +
%   | \chi_s |
%   +
%   \| \mu( X_s ) \|_H
%   +
%   \| \sigma( X_s ) \|^2_{ HS( U, H ) }
%   +
%   \| \mu( Y_s ) \|_H
%   +
%   \| \sigma( Y_s ) \|^2_{ HS( U, H ) }
%   \,
%   ds
%   < \infty
% $
% $ \P $-a.s.,
% $
%   X_t 
%   = 
%   X_{  \tzero  } 
%   +
%   \int_{  \tzero  }^t \mu( X_s ) \, ds
%   +
%   \int_{  \tzero  }^t \sigma( X_s ) \, dW_s
% $
% $ \P $-a.s.\ and
% $
%   Y_t 
%   = 
%   Y_{  \tzero  } 
%   +
%   \int_{  \tzero  }^t a_s \, ds
%   +
%   \int_{  \tzero  }^t b_s \, dW_s
% $
% $ \P $-a.s.\ for all
% $ t \in [ \tzero ,T] $
% and let 
% $ e_k \in U $,
% $ k \in \N $,
% be an orthonormal basis of $ U $.
% Then
\begin{align}
  \frac{
    \left\| X_t - Y_t \right\|^p_H
  }{
    \exp(
      \int_{  \tzero  }^t \chi_s \, ds
    )
  }
& =
  \left\| X_{  \tzero  } - Y_{  \tzero  } \right\|^p_H
  +
  \smallint\limits_{  \tzero  }^t
    p
    \left\|
      X_s - Y_s 
    \right\|^{ ( p - 2 ) }_H
    \big\langle
    \tfrac{
      X_s - Y_s 
    }{
      \exp(
        \int_{  \tzero  }^s \chi_u \, du
      )
    }
      ,
      \left[
      \sigma( X_s ) 
      - 
      b_s
      \right]
      dW_s
    \big\rangle_H
\nonumber
\\ & 
  +
  \smallint\limits_{  \tzero  }^t
  \tfrac{
    ( \overline{\mathcal{G}}_{ \mu, \sigma } V)( X_s, Y_s )
    - 
    \chi_s \left\| X_s - Y_s \right\|^p_H
    +
    p
    \,
    \left\|
      X_s - Y_s 
    \right\|^{ ( p - 2 ) }_H
    \left<
      X_s - Y_s
      ,
      \mu( Y_s )
      -
      a_s 
    \right>_H
  }{
    \exp(
      \int_{  \tzero  }^s \chi_u \, du
    )
  }
  \, ds
\nonumber
\\ & 
  +
  \smallint\limits_{  \tzero  }^t
  \tfrac{
    \frac{ p }{ 2 }
    \,
    \left\|
      X_s - Y_s 
    \right\|^{ ( p - 2 ) }_H
    \operatorname{trace}\left(
      \left[ 
        b_s + \sigma( Y_s ) - 2 \sigma( X_s ) 
      \right]^*
      \left[ 
        b_s
        -
        \sigma( Y_s )
      \right]
    \right)
  }{
    \exp(
      \int_{  \tzero  }^s \chi_u \, du
    )
  }
  \, ds
\\ & 
\nonumber
  +
  \smallint\limits_{  \tzero  }^t
  \tfrac{
    \mathbbm{1}_{
      \{ X_s \neq Y_s \}
    }
    \,
    \frac{ p \, \left( p - 2 \right) }{ 2 }
    \,
    \left\| X_s - Y_s \right\|^{ ( p - 4 ) }_H
    \operatorname{trace}\left(
      \left[ 
        b_s + \sigma( Y_s ) - 2 \sigma( X_s )
      \right]^*
      \left[ X_s - Y_s \right]
      \left[ X_s - Y_s \right]^*
      \left[ 
        b_s
        -
        \sigma( Y_s )
      \right]
    \right)
  }{
    \exp(
      \int_{  \tzero  }^s \chi_u \, du
    )
  }
  \, ds
\end{align}
$ \P $-a.s.\ for all 
$ t \in [  \tzero , T ] $.
This, in turn, shows that
\begin{equation}
\begin{split}
&
  \frac{
    \left\| X_t - Y_t \right\|^p_H
  }{
    \exp(
      \int_{  \tzero  }^t \chi_s \, ds
    )
  }
=
  \left\| X_{  \tzero  } - Y_{  \tzero  } \right\|^p_H
  +
  \smallint\limits_{  \tzero  }^t
    \big\langle
    \tfrac{
      p
      \,
      \left\|
        X_s - Y_s 
      \right\|^{ ( p - 2 ) }_H
      \,
      \left[
        X_s - Y_s 
      \right]
    }{
      \exp(
        \int_{  \tzero  }^s \chi_u \, du
      )
    }
      ,
      \left[
      \sigma( X_s ) 
      - 
      b_s
      \right]
      dW_s
    \big\rangle_H
\\ & 
  +
  \smallint\limits_{  \tzero  }^t
  \tfrac{
    ( \overline{\mathcal{G}}_{ \mu, \sigma } V)( X_s, Y_s )
    - 
    \chi_s \left\| X_s - Y_s \right\|^p_H
    +
    \mathbbm{1}_{
      \{ X_s \neq Y_s \}
    }
    \,
    \frac{ 
      p \, \left( p - 2 \right) 
    }{ 2 }
    \,
    \left\| X_s - Y_s \right\|^{ ( p - 4 ) }_H
    \,
    \left\|
      \left[ 
        b_s
        -
        \sigma( Y_s )
      \right]^*
      \left[ X_s - Y_s \right]
    \right\|^2_{
      U
    }
  }{
    \exp(
      \int_{  \tzero  }^s \chi_u \, du
    )
  }
  \, ds
\\ & 
  +
  \smallint\limits_{  \tzero  }^t
  \tfrac{
    p
    \,
    \left\|
      X_s - Y_s 
    \right\|^{ ( p - 2 ) }_H
  \left[ 
    \left<
      X_s - Y_s
      ,
      \mu( Y_s )
      -
      a_s 
    \right>_H
    +
    \frac{ 1 }{ 2 }
      \left\|
        b_s
        -
        \sigma( Y_s )
      \right\|^2_{ HS( U, H ) }
    +
    \operatorname{trace}\left(
      \left[ 
        \sigma( Y_s ) - \sigma( X_s ) 
      \right]^*
      \left[ 
        b_s
        -
        \sigma( Y_s )
      \right]
    \right)
  \right]
  }{
    \exp(
      \int_{  \tzero  }^s \chi_u \, du
    )
  }
  \, ds
\\ & 
  +
  \smallint\limits_{  \tzero  }^t
  \tfrac{
    \mathbbm{1}_{
      \{ X_s \neq Y_s \}
    }
    \,
    p \, \left( p - 2 \right) 
    \,
    \left\| X_s - Y_s \right\|^{ ( p - 4 ) }_H
    \operatorname{trace}\left(
      \left[ 
        \sigma( Y_s ) - \sigma( X_s )
      \right]^*
      \left[ X_s - Y_s \right]
      \left[ X_s - Y_s \right]^*
      \left[ 
        b_s
        -
        \sigma( Y_s )
      \right]
    \right)
  }{
    \exp(
      \int_{  \tzero  }^s \chi_u \, du
    )
  }
  \, ds
\end{split}
\end{equation}
$ \P $-a.s.\ for all 
$ t \in [  \tzero , T ] $.
A straightforward generalization 
of Example~2.15 in Cox et
al.~\cite{CoxHutzenthalerJentzen2013}
hence shows that
\begin{align}
&
  \frac{
    \left\| X_t - Y_t \right\|^p_H
  }{
    \exp(
      \int_{  \tzero  }^t \chi_s \, ds
    )
  }
=
  \left\| X_{  \tzero  } - Y_{  \tzero  } \right\|^p_H
  +
  \smallint\limits_{  \tzero  }^t
    \big\langle
    \tfrac{
      p
      \,
      \left\|
        X_s - Y_s 
      \right\|^{ ( p - 2 ) }_H
      \,
      \left[
        X_s - Y_s 
      \right]
    }{
      \exp(
        \int_{  \tzero  }^s \chi_u \, du
      )
    }
      ,
      \left[
      \sigma( X_s ) 
      - 
      b_s
      \right]
      dW_s
    \big\rangle_H
\nonumber
\\ & 
%\nonumber
  +
  \smallint\limits_{  \tzero  }^t
  \tfrac{
    \mathbbm{1}_{
      \{ X_s \neq Y_s \}
    }
    \,
    \frac{ 
      p \, \left( p - 2 \right) 
    }{ 2 }
    \,
    \left\| X_s - Y_s \right\|^{ ( p - 4 ) }_H
    \,
    \left[
    \left\|
      [ 
        \sigma( X_s ) - \sigma( Y_s ) 
      ]^*
      ( X_s - Y_s )
    \right\|^2_{ U }
    +
    \left\|
      \left[ 
        b_s
        -
        \sigma( Y_s )
      \right]^*
      \left[ X_s - Y_s \right]
    \right\|^2_{
      U
    }
    \right]
  }{
    \exp(
      \int_{  \tzero  }^s \chi_u \, du
    )
  }
  \, ds
\\ & 
\nonumber
  +
  \smallint\limits_{  \tzero  }^t
  \tfrac{
    p 
    \,
    \left\| X_s - Y_s \right\|^{ ( p - 2 ) }_H
    \left[
    \left< X_s - Y_s, \mu( X_s ) - \mu( Y_s ) \right>_H
    +
    \frac{ 1 }{ 2 }
    \left\|
      \sigma( X_s ) - \sigma( Y_s )
    \right\|^2_{ HS( U, H ) }
    \right]
    - 
    \chi_s \left\| X_s - Y_s \right\|^p_H
  }{
    \exp(
      \int_{  \tzero  }^s \chi_u \, du
    )
  }
  \, ds
\\ & 
\nonumber
  +
  \smallint\limits_{  \tzero  }^t
  \tfrac{
    p
    \,
    \left\|
      X_s - Y_s 
    \right\|^{ ( p - 2 ) }_H
  \left[ 
    \left<
      X_s - Y_s
      ,
      \mu( Y_s )
      -
      a_s 
    \right>_H
    +
    \frac{ 1 }{ 2 }
      \left\|
        b_s
        -
        \sigma( Y_s )
      \right\|^2_{ HS( U, H ) }
    +
    \operatorname{trace}\left(
      \left[ 
        \sigma( Y_s ) - \sigma( X_s ) 
      \right]^*
      \left[ 
        b_s
        -
        \sigma( Y_s )
      \right]
    \right)
  \right]
  }{
    \exp(
      \int_{  \tzero  }^s \chi_u \, du
    )
  }
  \, ds
\\ & 
\nonumber
  +
  \smallint\limits_{  \tzero  }^t
  \tfrac{
    \mathbbm{1}_{
      \{ X_s \neq Y_s \}
    }
    \,
    p \, \left( p - 2 \right) 
    \,
    \left\| X_s - Y_s \right\|^{ ( p - 4 ) }_H
    \operatorname{trace}\left(
      \left[ 
        \sigma( Y_s ) - \sigma( X_s )
      \right]^*
      \left[ X_s - Y_s \right]
      \left[ X_s - Y_s \right]^*
      \left[ 
        b_s
        -
        \sigma( Y_s )
      \right]
    \right)
  }{
    \exp(
      \int_{  \tzero  }^s \chi_u \, du
    )
  }
  \, ds
\end{align}
$ \P $-a.s.\ for all 
$ t \in [  \tzero , T ] $.
The Cauchy-Schwarz inequality
in the Hilbert space $ HS( U, H ) $
(see, e.g., Remark~B.0.4 and Proposition~B.0.8
in Pr\'{e}v\^{o}t \& R\"{o}ckner~\cite{PrevotRoeckner2007})
and the H\"{o}lder estimate for Schatten norms
(see, e.g., Remark~B.0.6
in Pr\'{e}v\^{o}t \& R\"{o}ckner~\cite{PrevotRoeckner2007})
therefore proves that
\begin{equation}
\begin{split}
&
  \frac{
    \left\| X_t - Y_t \right\|^p_H
  }{
    \exp(
      \int_{  \tzero  }^t \chi_s \, ds
    )
  }
\leq
  \left\| X_{  \tzero  } - Y_{  \tzero  } \right\|^p_H
  +
  \smallint\limits_{  \tzero  }^t
    \big\langle
    \tfrac{
      p
      \,
      \left\|
        X_s - Y_s 
      \right\|^{ ( p - 2 ) }_H
      \,
      \left[
        X_s - Y_s 
      \right]
    }{
      \exp(
        \int_{  \tzero  }^s \chi_u \, du
      )
    }
      ,
      \left[
      \sigma( X_s ) 
      - 
      b_s
      \right]
      dW_s
    \big\rangle_H
\\ & 
  +
  \smallint\limits_{  \tzero  }^t
  \tfrac{
    p 
    \,
    \left\| X_s - Y_s \right\|^{ ( p - 2 ) }_H
    \left[
    \left< X_s - Y_s, \mu( X_s ) - \mu( Y_s ) \right>_H
    +
    \frac{ ( p - 1 ) }{ 2 }
    \left\|
      \sigma( X_s ) - \sigma( Y_s )
    \right\|^2_{ HS( U, H ) }
    \right]
    - 
    \chi_s \left\| X_s - Y_s \right\|^p_H
  }{
    \exp(
      \int_{  \tzero  }^s \chi_u \, du
    )
  }
  \, ds
\\ &
  +
  \smallint\limits_{  \tzero  }^t
  \tfrac{
    p
    \,
    \left\|
      X_s - Y_s 
    \right\|^{ ( p - 2 ) }_H
  \left[ 
    \left<
      X_s - Y_s
      ,
      \mu( Y_s )
      -
      a_s 
    \right>_H
    +
    \frac{ ( p - 1 ) }{ 2 }
      \left\|
        b_s
        -
        \sigma( Y_s )
      \right\|^2_{ HS( U, H ) }
    +
    \operatorname{trace}\left(
      \left[ 
        \sigma( Y_s ) - \sigma( X_s ) 
      \right]^*
      \left[ 
        b_s
        -
        \sigma( Y_s )
      \right]
    \right)
  \right]
  }{
    \exp(
      \int_{  \tzero  }^s \chi_u \, du
    )
  }
  \, ds
\\ & 
  +
  \smallint\limits_{  \tzero  }^t
  \tfrac{
    \mathbbm{1}_{
      \{ X_s \neq Y_s \}
    }
    \,
    p \, \left( p - 2 \right) 
    \,
    \left\| X_s - Y_s \right\|^{ ( p - 4 ) }_H
      \left\| 
        \sigma( Y_s ) - \sigma( X_s )
      \right\|_{ HS( U, H ) }
      \left\|
        \left[ X_s - Y_s \right]
        \left[ X_s - Y_s \right]^*
      \right\|_{ L(H) }
      \left\|
        b_s
        -
        \sigma( Y_s )
      \right\|_{ HS( U, H ) }
  }{
    \exp(
      \int_{  \tzero  }^s \chi_u \, du
    )
  }
  \, ds
\end{split}
\end{equation}
$ \P $-a.s.\ for all 
$ t \in [  \tzero , T ] $.
Again the Cauchy-Schwarz inequality
in the Hilbert space $ HS( U, H ) $
(see, e.g., Remark~B.0.4 and Proposition~B.0.8
in Pr\'{e}v\^{o}t \& R\"{o}ckner~\cite{PrevotRoeckner2007})
hence shows that
\begin{equation}
\begin{split}
&
  \frac{
    \left\| X_t - Y_t \right\|^p_H
  }{
    \exp(
      \int_{  \tzero  }^t \chi_s \, ds
    )
  }
\leq
  \left\| X_{  \tzero  } - Y_{  \tzero  } \right\|^p_H
  +
  \smallint\limits_{  \tzero  }^t
    \big\langle
    \tfrac{
      p
      \,
      \left\|
        X_s - Y_s 
      \right\|^{ ( p - 2 ) }_H
      \,
      \left[
        X_s - Y_s 
      \right]
    }{
      \exp(
        \int_{  \tzero  }^s \chi_u \, du
      )
    }
      ,
      \left[
      \sigma( X_s ) 
      - 
      b_s
      \right]
      dW_s
    \big\rangle_H
\\ & 
  +
  \smallint\limits_{  \tzero  }^t
  \tfrac{
    p 
    \,
    \left\| X_s - Y_s \right\|^{ ( p - 2 ) }_H
    \left[
    \left< X_s - Y_s, \mu( X_s ) - \mu( Y_s ) \right>_H
    +
    \frac{ ( p - 1 ) }{ 2 }
    \left\|
      \sigma( X_s ) - \sigma( Y_s )
    \right\|^2_{ HS( U, H ) }
    +
    \left<
      X_s - Y_s
      ,
      \mu( Y_s )
      -
      a_s 
    \right>_H
    \right]
    - 
    \chi_s \left\| X_s - Y_s \right\|^p_H
  }{
    \exp(
      \int_{  \tzero  }^s \chi_u \, du
    )
  }
  \, ds
\\ &
  +
  \smallint\limits_{  \tzero  }^t
  \tfrac{
    p
    \,
    \left\|
      X_s - Y_s 
    \right\|^{ ( p - 2 ) }_H
  \left[ 
    \frac{ ( p - 1 ) }{ 2 }
      \left\|
        b_s
        -
        \sigma( Y_s )
      \right\|^2_{ HS( U, H ) }
    +
    \left( p - 1 \right) 
    \,
      \left\| 
        \sigma( Y_s ) - \sigma( X_s )
      \right\|_{ HS( U, H ) }
    \,
      \left\|
        b_s
        -
        \sigma( Y_s )
      \right\|_{ HS( U, H ) }
  \right]
  }{
    \exp(
      \int_{  \tzero  }^s \chi_u \, du
    )
  }
  \, ds
\end{split}
\end{equation}
$ \P $-a.s.\ for all 
$ t \in [  \tzero , T ] $.
The estimate 
$ 
  a b \leq 
  \frac{ \varepsilon }{ 2 } a^2 
  + 
  \frac{ 1 }{ 2 \varepsilon } b^2 
$
for all $ a, b \in \R $, $ \varepsilon \in [ 0, \infty ] $
hence implies that
\begin{equation}
\begin{split}
&
  \frac{
    \left\| X_t - Y_t \right\|^p_H
  }{
    \exp(
      \int_{  \tzero  }^t \chi_s \, ds
    )
  }
\leq
  \left\| X_{  \tzero  } - Y_{  \tzero  } \right\|^p_H
  +
  \smallint\limits_{  \tzero  }^t
    \big\langle
    \tfrac{
      p
      \,
      \left\|
        X_s - Y_s 
      \right\|^{ ( p - 2 ) }_H
      \,
      \left[
        X_s - Y_s 
      \right]
    }{
      \exp(
        \int_{  \tzero  }^s \chi_u \, du
      )
    }
      ,
      \left[
      \sigma( X_s ) 
      - 
      b_s
      \right]
      dW_s
    \big\rangle_H
\\ & \quad
  +
  \smallint\limits_{  \tzero  }^t
  \tfrac{
    p
    \,
    \left\|
      X_s - Y_s 
    \right\|^{ ( p - 2 ) }_H
  \left[ 
    \left<
      X_s - Y_s
      ,
      \mu( Y_s )
      -
      a_s 
    \right>_H
    +
    \frac{ ( p - 1 ) \, ( 1 + 1 / \varepsilon ) 
    }{ 2 }
    \,
      \left\|
        b_s
        -
        \sigma( Y_s )
      \right\|^2_{ HS( U, H ) }
  \right]
  }{
    \exp(
      \int_{  \tzero  }^s \chi_u \, du
    )
  }
  \, ds
\\ & \quad
  +
  \smallint\limits_{  \tzero  }^t
  \tfrac{
    p 
    \,
    \left\| X_s - Y_s \right\|^{ ( p - 2 ) }_H
    \left[
    \left< X_s - Y_s, \mu( X_s ) - \mu( Y_s ) \right>_H
    +
    \frac{ ( p - 1 ) \, ( 1 + \varepsilon ) }{ 2 }
    \,
    \left\|
      \sigma( X_s ) - \sigma( Y_s )
    \right\|^2_{ HS( U, H ) }
    \right]
    - 
    \chi_s \, \left\| X_s - Y_s \right\|^p_H
  }{
    \exp(
      \int_{  \tzero  }^s \chi_u \, du
    )
  }
  \, ds
\end{split}
\end{equation}
$ \P $-a.s.\ for all 
$ t \in [  \tzero , T ] $
and all
$ \varepsilon \in [0,\infty] $.
This completes the 
proof 
of Proposition~\ref{prop:norm2}.
\end{proof}
The next result, Theorem~\ref{thm:norm2_expectation_max0}, further develops 
our theory of perturbations for SDEs.
In particular, we apply a localization argument to the 
right-hand side 
of \eqref{eq:prop_norm2}, then take expectations on both sides and thereafter
apply H\"{o}lder's inequality.
\begin{theorem}
\label{thm:norm2_expectation_max0}
Assume the setting in Subsection~\ref{sec:setting},
let
$ \mu \in \mathcal{L}^0( \mathcal{O} ; H ) $,
$ \sigma \in \mathcal{L}^0( \mathcal{O} ; HS( U, H ) ) 
$,
$ \varepsilon \in [0,\infty] $,
$
  p \in [2,\infty)
$,
let
$
  \tau \colon \Omega \to [  \tzero , T ] 
$
be a stopping time,
let
$ X, Y \colon [  \tzero , T ] \times \Omega \to \mathcal{O} $
be adapted stochastic processes with c.s.p.,
let
$ a \colon [  \tzero , T ] \times \Omega \to H $,
$ b \colon [  \tzero , T ] \times \Omega \to HS( U, H ) $,
$ \chi \colon [  \tzero , T ] \times \Omega \to \R $
be predictable stochastic processes
with 
$
  \int_{  \tzero  }^T
  \| a_s \|_H
  +
  \| b_s \|_{ HS( U, H ) }^2
  +
  \| \mu( X_s ) \|_H
  +
  \| \sigma( X_s ) \|^2_{ HS( U, H ) }
  +
  \| \mu( Y_s ) \|_H
  +
  \| \sigma( Y_s ) \|^2_{ HS( U, H ) }
  \,
  ds
  < \infty
$
$ \P $-a.s.,
$
  X_t 
  = 
  X_{  \tzero  } 
  +
  \int_{  \tzero  }^t \mu( X_s ) \, ds
  +
  \int_{  \tzero  }^t \sigma( X_s ) \, dW_s
$
$ \P $-a.s.,
$
  Y_t 
  = 
  Y_{  \tzero  } 
  +
  \int_{  \tzero  }^t a_s \, ds
  +
  \int_{  \tzero  }^t b_s \, dW_s
$
$ \P $-a.s.\ for all
$ t \in [ \tzero ,T] $
and
\begin{equation}
\label{eq:perturbation_estimate_unique_assumption}
    \smallint_{  \tzero  }^{ \tau } 
      \Big[
        \tfrac{
          \left< 
            X_s - Y_s , \mu( X_s ) - \mu( Y_s ) 
          \right>_H
          +
          \frac{ ( p - 1 ) \, ( 1 + \varepsilon ) }{ 2 }
          \left\|
            \sigma( X_s ) - \sigma( Y_s )
          \right\|^2_{ HS( U, H ) }
        }{
          \left\| X_s - Y_s \right\|^2_H
        }
        +
        \chi_s 
    \Big]^+
    ds
  < \infty
\end{equation}
$ \P $-a.s.
Then it holds
for all 
$ r, q \in (0,\infty] $
with $ \frac{ 1 }{ p } + \frac{ 1 }{ q } = \frac{ 1 }{ r } $
that
{\small
\begin{align}
\label{eq:norm2_expectation}
&
  \left\|
    X_{ \tau } - Y_{ \tau } 
  \right\|_{
    L^r( \Omega; H )
  }
\leq
  \left\|
    \exp\!\left(
      \smallint_{  \tzero  }^{ \tau } 
      \Big[
        \tfrac{
          \left< 
            X_s - Y_s , \mu( X_s ) - \mu( Y_s ) 
          \right>_H
          +
          \frac{ ( p - 1 ) \, ( 1 + \varepsilon ) }{ 2 }
          \left\|
            \sigma( X_s ) - \sigma( Y_s )
          \right\|^2_{ HS( U, H ) }
        }{
          \left\| X_s - Y_s \right\|^2_H
        }
        +
        \chi_s 
      \Big]^+
      ds
    \right)
  \right\|_{
    L^q( \Omega; \R )
  }
\nonumber
\\ & 
  \Big[
    \big\|
      \,
      p
      \,
      \|
        X - Y 
      \|^{ ( p - 2 ) }_H
  \big[ 
    \langle
      X - Y
      ,
      \mu( Y )
      -
      a 
    \rangle_H
    +
    \tfrac{ ( p - 1 ) \, ( 1 + 1 / \varepsilon ) 
    }{ 2 }
      \|
        b
        -
        \sigma( Y )
      \|^2_{ HS( U, H ) }
    -
    \chi
    \|
      X - Y
    \|_H^2
  \big]^+
  \big\|_{
    L^1( \llbracket  \tzero , \tau \rrbracket ; \R )
  }^{ 1 / p }
\nonumber
\\ &
    +
    \| X_{  \tzero  } - Y_{  \tzero  } \|_{ L^p( \Omega; H ) }
  \Big]
  .
\end{align}
}
\end{theorem}

\begin{proof}[Proof
of Theorem~\ref{thm:norm2_expectation_max0}]
Throughout this proof let
$ 
  \hat{ \chi } \colon 
  [ \tzero , T] \times \Omega \to [0,\infty) 
$
be a predictable stochastic process
given by
\begin{equation}
  \hat{ \chi }_t
  =
  p 
  \,
%  \mathbbm{1}_{
%    \{ X_t \neq Y_t \}
%  }
  \mathbbm{1}_{
    \{ t \leq \tau \}
  }
  \left[
  \tfrac{
    \left< X_t - Y_t , \mu( X_t ) - \mu( Y_t ) \right>_H
    +
    \frac{ ( p - 1 ) \, ( 1 + \varepsilon ) }{ 2 }
    \left\|
      \sigma( X_t ) - \sigma( Y_t )
    \right\|^2_{ HS( U, H ) }
  }{
    \left\| X_t - Y_t \right\|^2_H
  }
  +
  \chi_t
  \right]^+
\end{equation}
for all $ t \in [  \tzero , T ] $
and let 
$ \tau_n \colon \Omega \to [  \tzero , T ] $,
$ n \in \N $,
be a sequence of stopping times given by
\begin{equation}
  \tau_n
  =
  \inf\!\Bigg(
    \{ 
      \tau
    \}
    \cup
    \Bigg\{ 
      t \in [  \tzero  , T ] 
      \colon
      \smallint_{  \tzero  }^t
        \left\| X_s - Y_s \right\|_H^{ 2 ( p - 1 ) }
        \left\|
          \sigma( X_s ) 
          - 
          b_s
        \right\|^2_{ HS( U, H ) }
      ds
      \geq 
      n
    \Bigg\}
  \Bigg)
\end{equation}
for all $ n \in \N $.
%and let $ \hat{ \chi }^n \colon [  \tzero , T ] \times \Omega \to \R $,
%$ n \in \N $,
%be predictable stochastic processes
%given by
%$
%  \hat{ \chi }^n_t 
%= 
%  \mathbbm{1}_{
%    \{ t \leq \tau_n \}
%  }
%  \hat{ \chi }_t
%$
%for all $ t \in [  \tzero , T ] $,
%$ n \in \N $.
%Note for all $ n \in \N $ that
%$
%  \int_{  \tzero  }^T | \hat{ \chi }_t^n | \, dt 
%  =
%  \int_{  \tzero  }^{ \tau_n }
%  | \hat{ \chi }_t | \, dt
%  \leq n
%  < \infty
%$.
Observe that
assumption~\eqref{eq:perturbation_estimate_unique_assumption}
ensures that
$
  \int_0^{ T }
  | \hat{ \chi }_s | \, ds
  < \infty
$
$ \P $-a.s.
Proposition~\ref{prop:norm2}
then implies that
\begin{equation}
\begin{split}
&
  \frac{
    \left\| X_t - Y_t \right\|^p_H
  }{
    \exp(
      \int_{  \tzero  }^t \hat{ \chi }_s \, ds
    )
  }
\leq
  \left\| X_{  \tzero  } - Y_{  \tzero  } \right\|^p_H
  +
  \smallint\limits_{  \tzero  }^t
    \big\langle
    \tfrac{
      p
      \,
      \left\|
        X_s - Y_s 
      \right\|^{ ( p - 2 ) }_H
      \,
      \left[
        X_s - Y_s 
      \right]
    }{
      \exp(
        \int_{  \tzero  }^s \hat{ \chi }_u \, du
      )
    }
      ,
      \left[
      \sigma( X_s ) 
      - 
      b_s
      \right]
      dW_s
    \big\rangle_H
\\ & \quad
  +
  \smallint\limits_{  \tzero  }^t
  \tfrac{
    p
    \,
    \left\|
      X_s - Y_s 
    \right\|^{ ( p - 2 ) }_H
  \left[ 
    \left<
      X_s - Y_s
      ,
      \mu( Y_s )
      -
      a_s 
    \right>_H
    +
    \frac{ ( p - 1 ) \, ( 1 + 1 / \varepsilon ) }{ 2 }
      \left\|
        b_s
        -
        \sigma( Y_s )
      \right\|^2_{ HS( U, H ) }
  \right]
    - 
    \hat{ \chi }_s \left\| X_s - Y_s \right\|^p_H
  }{
    \exp(
      \int_{  \tzero  }^s \hat{ \chi }_u \, du
    )
  }
  \, ds
\\ & \quad
  +
  \smallint\limits_{  \tzero  }^t
  \tfrac{
    p 
    \,
    \left\| X_s - Y_s \right\|^{ ( p - 2 ) }_H
    \left[
    \left< X_s - Y_s, \mu( X_s ) - \mu( Y_s ) \right>_H
    +
    \frac{ ( p - 1 ) \, ( 1 + \varepsilon ) }{ 2 }
    \left\|
      \sigma( X_s ) - \sigma( Y_s )
    \right\|^2_{ HS( U, H ) }
    \right]
  }{
    \exp(
      \int_{  \tzero  }^s \hat{ \chi }_u \, du
    )
  }
  \, ds
\end{split}
\end{equation}
$ \P $-a.s.\ for all 
$ t \in [  \tzero , T ] $.
This implies that
\begin{equation}
\begin{split}
&
  \frac{
    \left\| X_{ \tau_n } - Y_{ \tau_n } \right\|^p_H
  }{
    \exp(
      \int_{  \tzero  }^{ \tau_n } \hat{ \chi }_s \, ds
    )
  }
\leq
  \left\| X_{  \tzero  } - Y_{  \tzero  } \right\|^p_H
  +
  \smallint\limits_{  \tzero  }^{ \tau_n }
    \big\langle
    \tfrac{
      p
      \,
      \left\|
        X_s - Y_s 
      \right\|^{ ( p - 2 ) }_H
      \,
      \left[
        X_s - Y_s 
      \right]
    }{
      \exp(
        \int_{  \tzero  }^s \hat{ \chi }_u \, du
      )
    }
      ,
      \left[
      \sigma( X_s ) 
      - 
      b_s
      \right]
      dW_s
    \big\rangle_H
\\ & \quad
  +
  \smallint\limits_{  \tzero  }^{ \tau_n }
  \tfrac{
    p
    \,
    \left\|
      X_s - Y_s 
    \right\|^{ ( p - 2 ) }_H
  \left[ 
    \left<
      X_s - Y_s
      ,
      \mu( Y_s )
      -
      a_s 
    \right>_H
    +
    \frac{ ( p - 1 ) \, ( 1 + 1 / \varepsilon ) }{ 2 }
      \left\|
        b_s
        -
        \sigma( Y_s )
      \right\|^2_{ HS( U, H ) }
  \right]
    - 
    \hat{ \chi }_s \left\| X_s - Y_s \right\|^p_H
  }{
    \exp(
      \int_{  \tzero  }^s \hat{ \chi }_u \, du
    )
  }
  \, ds
\\ & \quad
  +
  \smallint\limits_{  \tzero  }^{ \tau_n }
  \tfrac{
    p 
    \,
    \left\| X_s - Y_s \right\|^{ ( p - 2 ) }_H
    \left[
    \left< X_s - Y_s, \mu( X_s ) - \mu( Y_s ) \right>_H
    +
    \frac{ ( p - 1 ) \, ( 1 + \varepsilon ) }{ 2 }
    \left\|
      \sigma( X_s ) - \sigma( Y_s )
    \right\|^2_{ HS( U, H ) }
    \right]
  }{
    \exp(
      \int_{  \tzero  }^s \hat{ \chi }_u \, du
    )
  }
  \, ds
\end{split}
\end{equation}
$ \P $-a.s.\ for all 
$ n \in \N $.
The definition of $ \hat{ \chi } $
hence gives that
\begin{equation}
\label{eq:norm2_perturbation_est_exp}
\begin{split}
&
  \frac{
    \left\| X_{ \tau_n } - Y_{ \tau_n } \right\|^p_H
  }{
    \exp(
      \int_{  \tzero  }^{ \tau_n } \hat{ \chi }_s \, ds
    )
  }
\leq
  \left\| X_{  \tzero  } - Y_{  \tzero  } \right\|^p_H
  +
  \smallint\limits_{  \tzero  }^{ \tau_n }
    \big\langle
    \tfrac{
      p
      \,
      \left\|
        X_s - Y_s 
      \right\|^{ ( p - 2 ) }_H
      \,
      \left[
        X_s - Y_s 
      \right]
    }{
      \exp(
        \int_{  \tzero  }^s \hat{ \chi }_u \, du
      )
    }
      ,
      \left[
      \sigma( X_s ) 
      - 
      b_s
      \right]
      dW_s
    \big\rangle_H
\\ & \quad
  +
  \smallint\limits_{  \tzero  }^{ \tau_n }
  \tfrac{
    p
    \,
    \left\|
      X_s - Y_s 
    \right\|^{ ( p - 2 ) }_H
  \left[ 
    \left<
      X_s - Y_s
      ,
      \mu( Y_s )
      -
      a_s 
    \right>_H
    +
    \frac{ ( p - 1 ) \, ( 1 + 1 / \varepsilon ) }{ 2 }
      \left\|
        b_s
        -
        \sigma( Y_s )
      \right\|^2_{ HS( U, H ) }
    -
    \chi_s
    \left\|
      X_s - Y_s
    \right\|_H^2
  \right]^+
%     +
%       \rho^{
%         ( p / 2 - 1 )
%       }
%       \,
%       [ 
%         p + \frac{ 1 }{ \varepsilon } - 1 
%       ]^{ \frac{ p }{ 2 } } 
%       \,
%       \left\|
%         b_s
%         -
%         \sigma( Y_s )
%       \right\|^p_{ HS( U, H ) }
%  \right]
  }{
    \exp(
      \int_{  \tzero  }^s \hat{ \chi }_u \, du
    )
  }
  \, ds
\end{split}
\end{equation}
$ \P $-a.s.\ for all 
$ n \in \N $.
Taking expectations in \eqref{eq:norm2_perturbation_est_exp}
proves that for all $ n \in \N $ it holds that
\begin{equation}
\begin{split}
&
  \E\!\left[
  \frac{
    \left\| X_{ \tau_n } - Y_{ \tau_n } \right\|^p_H
  }{
    \exp(
      \int_{  \tzero  }^{ \tau_n } \hat{ \chi }_s \, ds
    )
  }
  \right]
\leq
  \E\big[ 
    \| X_{  \tzero  } - Y_{  \tzero  } \|^p_H
  \big]
\\ & \quad
  +
  \E\!\left[
  \smallint\limits_{  \tzero  }^{ \tau_n }
  \tfrac{
    p
    \,
    \left\|
      X_s - Y_s 
    \right\|^{ ( p - 2 ) }_H
  \left[ 
    \left<
      X_s - Y_s
      ,
      \mu( Y_s )
      -
      a_s 
    \right>_H
    +
    \frac{ ( p - 1 ) \, ( 1 + 1 / \varepsilon ) }{ 2 }
      \left\|
        b_s
        -
        \sigma( Y_s )
      \right\|^2_{ HS( U, H ) }
    -
    \chi_s
    \left\|
      X_s - Y_s
    \right\|_H^2
  \right]^+
%     +
%       \rho^{
%         ( p / 2 - 1 )
%       }
%       \,
%       [ 
%         p + \frac{ 1 }{ \varepsilon } - 1 
%       ]^{ \frac{ p }{ 2 } } 
%       \,
%       \left\|
%         b_s
%         -
%         \sigma( Y_s )
%       \right\|^p_{ HS( U, H ) }
  }{
    \exp(
      \int_{  \tzero  }^s \hat{ \chi }_u \, du
    )
  }
  \, ds
  \right]
  .
\end{split}
\end{equation}
Fatou's lemma hence implies that
\begin{equation}
\begin{split}
&
  \E\!\left[
  \frac{
    \left\| X_{ \lim_{ n \to \infty } \tau_n } - Y_{ \lim_{ n \to \infty } \tau_n } \right\|^p_H
  }{
    \exp(
      \int_{  \tzero  }^{ \lim_{ n \to \infty } \tau_n } \hat{ \chi }_s \, ds
    )
  }
  \right]
\leq
  \E\big[ 
    \| X_{  \tzero  } - Y_{  \tzero  } \|^p_H
  \big]
\\ & \quad
  +
  \liminf_{ n \to \infty }
  \E\!\left[
  \smallint\limits_{  \tzero  }^{ \tau_n }
  \tfrac{
    p
    \,
    \left\|
      X_s - Y_s 
    \right\|^{ ( p - 2 ) }_H
  \left[ 
    \left<
      X_s - Y_s
      ,
      \mu( Y_s )
      -
      a_s 
    \right>_H
    +
    \frac{ ( p - 1 ) \, ( 1 + 1 / \varepsilon )}{ 2 }
      \left\|
        b_s
        -
        \sigma( Y_s )
      \right\|^2_{ HS( U, H ) }
    -
    \chi_s
    \left\|
      X_s - Y_s
    \right\|_H^2
  \right]^+
  }{
    \exp(
      \int_{  \tzero  }^s \hat{ \chi }_u \, du
    )
  }
  \, ds
  \right]
  .
\end{split}
\end{equation}
Monotonicity therefore shows that
\begin{equation}
\begin{split}
&
  \E\!\left[
  \frac{
    \left\| X_{ \tau } - Y_{ \tau } \right\|^p_H
  }{
    \exp(
      \int_{  \tzero  }^{ \tau } \hat{ \chi }_s \, ds
    )
  }
  \right]
\leq
  \E\big[ 
    \| X_{  \tzero  } - Y_{  \tzero  } \|^p_H
  \big]
\\ & \quad
  +
  \E\!\left[
  \smallint\limits_{  \tzero  }^{ \tau }
  \tfrac{
    p
    \,
    \left\|
      X_s - Y_s 
    \right\|^{ ( p - 2 ) }_H
  \left[ 
    \left<
      X_s - Y_s
      ,
      \mu( Y_s )
      -
      a_s 
    \right>_H
    +
    \frac{ ( p - 1 ) \, ( 1 + 1 / \varepsilon ) }{ 2 }
      \left\|
        b_s
        -
        \sigma( Y_s )
      \right\|^2_{ HS( U, H ) }
    -
    \chi_s
    \left\|
      X_s - Y_s
    \right\|_H^2
  \right]^+
  }{
    \exp(
      \int_{  \tzero  }^s \hat{ \chi }_u \, du
    )
  }
  \, ds
  \right]
  .
\end{split}
\end{equation}
This and the fact that $ \hat{ \chi } \geq 0 $ yield that
\begin{equation}
\begin{split}
&
  \E\!\left[
  \frac{
    \left\| X_{ \tau } - Y_{ \tau } \right\|^p_H
  }{
    \exp(
      \int_{  \tzero  }^{ \tau } \hat{ \chi }_s \, ds
    )
  }
  \right]
\leq
  \E\big[
    \| X_{  \tzero  } - Y_{  \tzero  } \|^p_H
  \big]
  +
  \E\bigg[
  \smallint\limits_{  \tzero  }^{ \tau }
    p
    \left\|
      X_s - Y_s 
    \right\|^{ ( p - 2 ) }_H
\\ & 
  \cdot
  \big[ 
    \left<
      X_s - Y_s
      ,
      \mu( Y_s )
      -
      a_s 
    \right>_H
    +
    \tfrac{ ( p - 1 ) \, ( 1 + 1 / \varepsilon ) }{ 2 }
      \left\|
        b_s
        -
        \sigma( Y_s )
      \right\|^2_{ HS( U, H ) }
    -
    \chi_s
    \left\|
      X_s - Y_s
    \right\|_H^2
  \big]^+
  \, ds
  \bigg]
  .
\end{split}
\end{equation}
H\"{o}lder's inequality hence proves that
for all $ q \in (0,\infty] $, $ r \in (0,p] $
with $ \frac{ 1 }{ p } + \frac{ 1 }{ q } = \frac{ 1 }{ r } $
it holds that
\begin{equation}
\begin{split}
&
  \left\|
    X_{ \tau } - Y_{ \tau } 
  \right\|_{
    L^r( \Omega; H )
  }^p
  =
  \left\|
    \frac{
      \left\|
        X_{ \tau } - Y_{ \tau }
      \right\|_H
    }{
      \exp(
        \frac{ 1 }{ p }
        \int_{  \tzero  }^{ \tau } \hat{ \chi }_s \, ds
      )
    }
    \,
      \exp\!\left(
        \tfrac{ 1 }{ p } \smallint_{  \tzero  }^{ \tau } \hat{ \chi }_s \, ds
      \right)
  \right\|_{
    L^r( \Omega; \R )
  }^p
\\ & \leq
  \left\|
    \frac{
      \left\|
        X_{ \tau } - Y_{ \tau }
      \right\|_H
    }{
      \exp(
        \frac{ 1 }{ p }
        \int_{  \tzero  }^{ \tau } \hat{ \chi }_s \, ds
      )
    }
  \right\|_{
    L^p( \Omega; \R )
  }^p
  \left\|
      \exp\!\left(
        \tfrac{ 1 }{ p } \smallint_{  \tzero  }^{ \tau } \hat{ \chi }_s \, ds
      \right)
  \right\|_{
    L^q( \Omega; \R )
  }^p
\\ & \leq
  \left\|
      \exp\!\left(
        \tfrac{ 1 }{ p } \smallint_{  \tzero  }^{ \tau } \hat{ \chi }_s \, ds
      \right)
  \right\|_{
    L^q( \Omega; \R )
  }^p
  \E\bigg[
    \| X_{  \tzero  } - Y_{  \tzero  } \|^p_H
    +
  \smallint\limits_{  \tzero  }^{ \tau }
    p
    \,
    \left\|
      X_s - Y_s 
    \right\|^{ ( p - 2 ) }_H
\\ & \cdot
  \Big[ 
    \left<
      X_s - Y_s
      ,
      \mu( Y_s )
      -
      a_s 
    \right>_H
    +
    \tfrac{ ( p - 1 ) \, ( 1 + 1 / \varepsilon ) }{ 2 }
      \left\|
        b_s
        -
        \sigma( Y_s )
      \right\|^2_{ HS( U, H ) }
    -
    \chi_s
    \left\|
      X_s - Y_s
    \right\|_H^2
  \Big]^+
  \, ds
  \bigg]
  .
\end{split}
\end{equation}
This implies \eqref{eq:norm2_expectation} 
and the proof of 
Theorem~\ref{thm:norm2_expectation_max0}
is thus completed.
\end{proof}

The next corollary, Corollary~\ref{cor:different_coefficients},
uses 
Theorem~\ref{thm:norm2_expectation_max0} 
to study 
the difference of solutions processes of two semilinear SPDEs with 
possibly different coefficient functions.

\begin{corollary}
\label{cor:different_coefficients}
Assume the setting in Subsection~\ref{sec:setting},
let $ A \colon D(A) \subseteq H \to H $
be a densely defined linear operator with $ \mathcal{O} \subseteq D(A) $,
let
$ F_1, F_2 \in \mathcal{L}^0( \mathcal{O} ; H ) $,
$ B_1, B_2 \in \mathcal{L}^0( \mathcal{O} ; HS( U, H ) ) $,
$ \varepsilon \in [0,\infty] $,
$ p \in [2,\infty) $,
let
$ X^1, X^2 \colon [  \tzero , T ] \times \Omega \to \mathcal{O} $,
$
  \hat{X} \colon [0,T] \times \Omega \to H
$,
$ \chi \colon [  \tzero , T ] \times \Omega \to \R $
be predictable stochastic processes
with 
$
  \int_{  \tzero  }^T
  \| A X_s^j \|_H
  +
  \| A \hat{X}_s \|_H
  +
  \| F_i( X_s^j ) \|_H
  +
  \| B_i( X_s^j ) \|^2_{ HS( U, H ) }
  +
  \| F_2( \hat{X}_s ) \|_H
  +
  \| B_2( \hat{X}_s ) \|^2_{ HS( U, H ) }
  \,
  ds
  < \infty
$
$ \P $-a.s.,
$
  X_t^i 
  = 
  X_{  \tzero  }^i
  +
  \int_{  \tzero  }^t A X_s^i + F_i( X_s^i ) \, ds
  +
  \int_{  \tzero  }^t B_i( X_s^i ) \, dW_s
$
$ \P $-a.s.,
$
  \hat{X}_t
  =
  X_{  \tzero  }^2
  +
  \int_{  \tzero  }^t A \hat{X}_s + F_2( X_s^1 ) \, ds
  +
  \int_{  \tzero  }^t B_2( X_s^1 ) \, dW_s
$
$ \P $-a.s.\ for all
$ t \in [ \tzero ,T] $,
$ (i, j) \in \{ 1, 2 \}^2 \backslash \{ ( 1, 2 ) \} $
and
\begin{equation}
      \smallint_{  \tzero  }^T 
      \Big[
        \tfrac{
          \langle 
            X^2_s - \hat{X}_s , 
            A [ X^2_s - \hat{X}_s ] 
            + F_2( X^2_s ) - F_2( \hat{X}_s ) 
          \rangle_H
          +
          \frac{ ( p - 1 ) \, ( 1 + \varepsilon ) }{ 2 }
          \,
          \|
            B_2( X^2_s ) - B_2( \hat{X}_s )
          \|^2_{ HS( U, H ) }
        }{
          \| X^2_s - \hat{X}_s \|^2_H
        }
        +
        \chi_s 
      \Big]^+
      ds
  < \infty
\end{equation}
$ \P $-a.s. Then it holds
for all 
$ t \in [0,T] $,
$ r, q \in (0,\infty] $
with 
$ 
  \frac{ 1 }{ p } + \frac{ 1 }{ q } 
  = \frac{ 1 }{ r } 
$
that
% {\small
\begin{align}
\label{eq:different_coefficients}
&
  \|
    X^1_t - X^2_t
  \|_{
    L^r( \Omega; H )
  }
\leq
  \|
    X^1_t - \hat{X}_t
  \|_{
    L^r( \Omega; H )
  }
  +
    \big\|
      \,
      p
      \,
      \|
        X^2 - \hat{X} 
      \|^{ ( p - 2 ) }_H
  \big[ 
    \langle
      X^2 - \hat{X}
      ,
      F_2( \hat{X} )
      -
      F_2( X^1 )
    \rangle_H
\nonumber
\\ & \quad +
% \nonumber
%   \Big[
    \tfrac{ ( p - 1 ) \, ( 1 + 1 / \varepsilon )
    }{ 2 }
    \,
      \|
        B_2( X^1 )
        -
        B_2( \hat{X} )
      \|^2_{ HS( U, H ) }
    -
    \chi
    \,
    \|
      X^2 - \hat{X}
    \|_H^2
  \big]^+
  \big\|_{
    L^1( [0,t] \times \Omega ; \R )
  }^{ 1 / p }
%   \Big]
\\ & \quad
  \cdot
  \left\|
    \exp\!\left(
      \smallint_{  \tzero  }^{ t } 
      \Big[
        \tfrac{
          \langle 
            X^2_s - \hat{X}_s , A [ X^2_s - \hat{X}_s ] + F_2( X^2_s ) - F_2( \hat{X}_s ) 
          \rangle_H
          +
          \frac{ ( p - 1 ) \, ( 1 + \varepsilon ) }{ 2 }
          \,
          \|
            B_2( X^2_s ) - B_2( \hat{X}_s )
          \|^2_{ HS( U, H ) }
        }{
          \| X^2_s - \hat{X}_s \|^2_H
        }
        +
        \chi_s 
      \Big]^+
      ds
    \right)
  \right\|_{
    L^q( \Omega; \R )
  }
\nonumber
  .
\end{align}
% }
\end{corollary}

Corollary~\ref{cor:different_coefficients}
follows immediately from
the triangle inequality and
from an application of
Theorem~\ref{thm:norm2_expectation_max0}
to the stochastic process 
$ X_t^2 $, $ t \in [0,T] $,
with the perturbation process
$ \hat{X}_t $, $ t \in [0,T] $,
and its proof is thus omitted.
In a number of cases it is convenient 
to further estimate 
the right-hand side of \eqref{eq:norm2_expectation}
in an appropriate way.
This is the subject of the next corollary
of Theorem~\ref{thm:norm2_expectation_max0}.

\begin{corollary}
\label{cor:norm2}
Assume the setting in Subsection~\ref{sec:setting},
let 
$ \mu \in \mathcal{L}^0( \mathcal{O}, H ) $, 
$ \sigma \in \mathcal{L}^0( \mathcal{O}, HS( U, H ) ) $,
$ \varepsilon \in [0,\infty] $,
$ p \in [2,\infty) $,
let $ \tau \colon \Omega \to [  \tzero  , T ] $ 
be a stopping time,
let
$ X, Y \colon [  \tzero , T ] \times \Omega \to \mathcal{O} $
be adapted stochastic processes with c.s.p.,
let
$ a \colon [  \tzero , T ] \times \Omega \to H $,
$ b \colon [  \tzero , T ] \times \Omega \to HS( U, H ) $
be predictable stochastic processes
with 
$
  \int_{  \tzero  }^T
  \| a_s \|_H
  +
  \| b_s \|_{ HS( U, H ) }^2
  +
  \| \mu( X_s ) \|_H
  +
  \| \sigma( X_s ) \|^2_{ HS( U, H ) }
  +
  \| \mu( Y_s ) \|_H
  +
  \| \sigma( Y_s ) \|^2_{ HS( U, H ) }
  \,
  ds
  < \infty
$
$ \P $-a.s.,
$
  X_t 
  = 
  X_{  \tzero  } 
  +
  \int_{  \tzero  }^t \mu( X_s ) \, ds
  +
  \int_{  \tzero  }^t \sigma( X_s ) \, dW_s
$
$ \P $-a.s.,
$
  Y_t 
  = 
  Y_{  \tzero  } 
  +
  \int_{  \tzero  }^t a_s \, ds
  +
  \int_{  \tzero  }^t b_s \, dW_s
$
$ \P $-a.s.\ for all
$ t \in [ \tzero ,T] $
and
\begin{equation}
    \smallint_{  \tzero  }^{ \tau }
  \Big[
  \tfrac{
    \left< 
      X_s - Y_s , \mu( X_s ) - \mu( Y_s ) 
    \right>_H
    +
    \frac{ ( p - 1 ) \, ( 1 + \varepsilon ) }{ 2 }
    \left\|
      \sigma( X_s ) - \sigma( Y_s )
    \right\|^2_{ HS( U, H ) }
  }{
    \left\| X_s - Y_s \right\|^2_H
  }
  \Big]^+
    ds
  < \infty
\end{equation}
$ \P $-a.s. Then it holds for all 
$ \delta, \rho, r \in (0,\infty) $,
$ q \in (0,\infty] $
with 
$ \frac{ 1 }{ p } + \frac{ 1 }{ q } = \frac{ 1 }{ r } $
that
{\footnotesize
\begin{align}
&
  \big\| X_{ \tau } - Y_{ \tau }
  \big\|_{
    L^r( \Omega; H )
  }
  \leq
  \left\|
  \exp\!\left(
    \smallint_{  \tzero  }^{ \tau }
  \bigg[
  \tfrac{
    \left< 
      X_s - Y_s , \mu( X_s ) - \mu( Y_s ) 
    \right>_H
    +
    \frac{ ( p - 1 ) \, ( 1 + \varepsilon ) }{ 2 }
    \left\|
      \sigma( X_s ) - \sigma( Y_s )
    \right\|^2_{ HS( U, H ) }
  }{
    \left\| X_s - Y_s \right\|^2_H
  }
  +
  \tfrac{
      ( 1 - \frac{ 1 }{ p } ) 
  }{
      \delta
  }
      +
  \tfrac{
      ( \frac{ 1 }{ 2 } - \frac{ 1 }{ p } )
  }{
      \rho
  }
  \bigg]^+
  \!
    ds
  \right)
  \right\|_{
    L^q( \Omega; \R )
  }
\nonumber
\\ & \cdot
\nonumber
  \left[
    \big\| X_{  \tzero  } - Y_{  \tzero  } 
    \big\|_{ L^p( \Omega; H ) }
  +
    \delta^{
      ( 1 - \frac{ 1 }{ p } )
    }
      \left\|
        a -
        \mu( Y )
      \right\|_{ L^p( \llbracket  \tzero , \tau \rrbracket ; H ) }
    +
      \rho^{ ( \frac{ 1 }{ 2 } - \frac{ 1 }{ p } ) }
      \sqrt{ ( p - 1 ) ( 1 + 1 / \varepsilon ) }
        \left\|
          b -
          \sigma( Y )
        \right\|_{ 
          L^p( \llbracket  \tzero  , \tau \rrbracket ; HS( U, H ) )
        }
  \right]
  .
\end{align}}
\end{corollary}

\begin{proof}[Proof
of Corollary~\ref{cor:norm2}]
Throughout this proof let 
$ \rho, \delta \in (0,\infty) $
be arbitrary and let
$ \chi \colon [  \tzero , T ] \times \Omega \to \R $
be given by
\begin{equation}
  \chi_t(\omega) 
  = 
  \frac{ 
    ( 1 - \frac{ 1 }{ p } ) 
  }{
    \delta
  }
  +
  \frac{ 
    ( \frac{ 1 }{ 2 } - \frac{ 1 }{ p } ) 
  }{
    \rho
  }
\end{equation}
for all $ t \in [  \tzero , T ] $,
$ \omega \in \Omega $.
Theorem~\ref{thm:norm2_expectation_max0}
then implies that for all 
$ r, q \in (0,\infty] $
with $ \frac{ 1 }{ p } + \frac{ 1 }{ q } = \frac{ 1 }{ r } $
it holds that
\begin{equation}
\begin{split}
&
  \left\|
    X_{ \tau } - Y_{ \tau }
  \right\|_{
    L^r( \Omega; H )
  }
\\ & \leq
  \left\|
    \exp\!\left(
      \smallint_{  \tzero  }^{ \tau } 
      \Big[
        \tfrac{
          \left< 
            X_s - Y_s , \mu( X_s ) - \mu( Y_s ) 
          \right>_H
          +
          \frac{ ( p - 1 ) \, ( 1 + \varepsilon ) }{ 2 }
          \left\|
            \sigma( X_s ) - \sigma( Y_s )
          \right\|^2_{ HS( U, H ) }
        }{
          \left\| X_s - Y_s \right\|^2_H
        }
        +
        \chi_s 
      \Big]^+
      ds
    \right)
  \right\|_{
    L^q( \Omega; \R )
  }
\\ & \quad \cdot
  \bigg[
    \left\| 
      X_{  \tzero  } - Y_{  \tzero  } 
    \right\|_{ L^p( \Omega; H ) }
    +
  \Big\|
    \Big[
    \tfrac{ p }{ 2 } \,
      \rho^{
        ( \frac{ 2 }{ p } - 1 )
      }
      \,
    \|
      X - Y 
    \|^{ ( p - 2 ) }_H
      \left( p - 1 \right)
      \left( 1 + 1 / \varepsilon \right)
      \rho^{
        ( 1 - \frac{ 2 }{ p } )
      }
      \left\|
        b
        -
        \sigma( Y )
      \right\|_{ HS( U, H ) }^2
\\ & \qquad
    +
    p
    \,
    \delta^{
      ( 1 - p ) / p
    }
    \,
    \|
      X - Y 
    \|^{ ( p - 1 ) }_H
    \,
    \delta^{
      ( p - 1 ) / p
    }
    \,
      \|
        \mu( Y )
        -
        a 
      \|_H
      - 
      p \, \chi \, \| X - Y \|^p_H
    \Big]^+
  \Big\|_{
    L^1( \llbracket  \tzero , \tau \rrbracket ; \R )
  }^{ 1 / p }
  \bigg]
  .
\end{split}
\end{equation}
Young's inequality 
hence proves that for all
$ r, q \in (0,\infty] $
with $ \frac{ 1 }{ p } + \frac{ 1 }{ q } = \frac{ 1 }{ r } $
it holds that
\begin{equation}
\begin{split}
&
  \left\|
    X_{ \tau } - Y_{ \tau }
  \right\|_{
    L^r( \Omega; H )
  }
\\ & \leq
  \left\|
    \exp\!\left(
      \smallint_{  \tzero  }^{ \tau } 
      \Big[
        \tfrac{
          \left< 
            X_s - Y_s , \mu( X_s ) - \mu( Y_s ) 
          \right>_H
          +
          \frac{ ( p - 1 ) \, ( 1 + \varepsilon ) }{ 2 }
          \left\|
            \sigma( X_s ) - \sigma( Y_s )
          \right\|^2_{ HS( U, H ) }
        }{
          \left\| X_s - Y_s \right\|^2_H
        }
        +
        \tfrac{
          ( 1 - \frac{ 1 }{ p } )
        }{ \delta }
        +
        \tfrac{
          ( \frac{ 1 }{ 2 } - \frac{ 1 }{ p } )
        }{ \rho }
      \Big]^+
      ds
    \right)
  \right\|_{
    L^q( \Omega; \R )
  }
\\ & \cdot
  \bigg[
    \left\| 
      X_{  \tzero  } - Y_{  \tzero  } 
    \right\|_{ L^p( \Omega; H ) }
    +
  \Big\|
    \Big[
      \rho^{
        ( \frac{ p }{ 2 } - 1 )
      }
      \left[ 
        \left( p - 1 \right)
        \left( 1 + 1 / \varepsilon \right)
      \right]^{ \frac{ p }{ 2 } } 
      \left\|
        b
        -
        \sigma( Y )
      \right\|_{ HS( U, H ) }^p
\\ & 
    +
    \delta^{ ( p - 1 ) }
    \,
      \|
        \mu( Y )
        -
        a 
      \|_H^p
      +
      \tfrac{ ( p - 2 ) }{ 2 \, \rho } \,
      \| X - Y \|^p_H
      +
      \tfrac{ ( p - 1 ) }{ \delta } \,
      \| X - Y \|^p_H
      - 
      p \, \chi \, \| X - Y \|^p_H
    \Big]^+
  \Big\|_{
    L^1( \llbracket  \tzero , \tau \rrbracket ; \R )
  }^{ 1 / p }
  \bigg]
  .
\end{split}
\end{equation}
This and the definition of $ \chi $
completes the proof
of Corollary~\ref{cor:norm2}.
\end{proof}

\section{Applications of the perturbation theory for SDEs}

\subsection{Numerical approximations of SODEs}
\label{sec:SODEs}

This subsection uses Corollary~\ref{cor:norm2}
to establish strong convergence rates for the
stopped-tamed Euler-Maruyama method in 
\cite{HutzenthalerJentzenWang2013}
(see (6) in \cite{HutzenthalerJentzenWang2013}).
To do so,
%In this analysis 
the following elementary lemma is used.

\begin{lemma}
\label{lem:psi_estimate}
Let $ d \in \N $ and let
$ \psi \colon \R^d \to \R^d $
be given by
$
  \psi( v ) = \frac{ v }{ 1 + \| v \|^2_{ \R^d } }
$
for all $ v \in \R^d $.
Then it holds
for all $ v \in \R^d $ 
that
$
  \| \psi'( v ) \|_{ L( \R^d ) }
  \leq 
  3
$,
that
$
  \left\| 
    \psi'( v ) - I_{ \R^d }
  \right\|_{
    L( \R^d )
  }
\leq
  3 
  \left[
    1 
    \wedge
    \left\| v \right\|_{ \R^d }
  \right]^2
$
and 
that
$
  \sup_{ 
    u \in \R^d 
    ,
    \,
    \| u \|_{ \R^d } \leq 1
  }
  \left\|
    \psi''( v )( u, u )
  \right\|_{
    \R^d
  }
\leq
  14
  \left[
    1 
    \wedge
    \left\| v \right\|_{ \R^d }
  \right]
$.
\end{lemma}

\begin{proof}[Proof of Lemma~\ref{lem:psi_estimate}]
Observe that, for example,
%(76)--(77) in \cite{HutzenthalerJentzenWang2013} 
%prove that
Section 2.3 in \cite{HutzenthalerJentzenWang2013} 
shows that
for all $ z \in \R^d $ 
it holds that
\begin{equation}
  \| \psi'( z ) \|_{ L( \R^d ) }
  \leq 
  3
  ,
\end{equation}
that
\begin{equation}
\label{eq:psi_est1}
\begin{split}
&
  \left\| 
    \psi'( z ) - I_{ \R^d }
  \right\|_{
    L( \R^d )
  }
\leq
  \left\| 
    \frac{ I_{ \R^d } }{
      1 + \| z \|^2_{ \R^d }
    }
    - 
    I_{ \R^d }
  \right\|_{
    L( \R^d )
  }
  +
  \left\|
    \frac{
      2 z z^*
    }{
      \left[ 
        1 + \| z \|_{ \R^d }^2 
      \right]^2
    }
  \right\|_{ L( \R^d ) }
\\ & =
  \left| 
    \frac{ 1 }{
      1 + \| z \|^2_{ \R^d }
    }
    - 
    1
  \right|
  +
    \frac{
      2 \left\| z \right\|^2_{ \R^d }
    }{
      \left[ 
        1 + \| z \|_{ \R^d }^2
      \right]^2
    }
=
    \frac{
      \left\| z \right\|^2_{ \R^d }
    }{
        1 + \| z \|_{ \R^d }^2
    }
  +
    \frac{
      2 \left\| z \right\|^2_{ \R^d }
    }{
      \left[ 
        1 + \| z \|_{ \R^d }^2
      \right]^2
    }
\\ & \leq
    \frac{
      3 \left\| z \right\|^2_{ \R^d }
    }{
        1 + \| z \|_{ \R^d }^2
    }
\leq
  3 
  \left[
    1 
    \wedge
    \left\| z \right\|_{ \R^d }
  \right]^2
\end{split}
\end{equation}
and that
\begin{equation}
\label{eq:psi_est2}
\begin{split}
  \sup_{ 
    u \in \R^d 
    ,
    \,
    \| u \|_{ \R^d } \leq 1
  }
  \left\|
    \psi''( z )( u, u )
  \right\|_{
    \R^d
  }
& \leq
  \frac{
    8 \, \| z \|^3_{ \R^d }
  }{
    \left[ 
      1 + \| z \|_{ \R^d }^2
    \right]^3
  }
  +
  \frac{
    6 \, \| z \|_{ \R^d }
  }{
    \left[
      1 + \| z \|^2_{ \R^d }
    \right]^2
  }
\\ & \leq
  \frac{
    14 \, \| z \|_{ \R^d }
  }{
    \left[
      1 + \| z \|^2_{ \R^d }
    \right]^2
  }
\leq
  14
  \left[
    1 
    \wedge
    \left\| z \right\|_{ \R^d }
  \right]
  .
\end{split}
\end{equation}
The proof of Lemma~\ref{lem:psi_estimate}
is thus completed.
\end{proof}

We now use Lemma~\ref{lem:psi_estimate}
together with Corollary~\ref{cor:norm2}
to prove a suitable strong convergence rate estimate
(see \eqref{eq:strong_convergence_rate_estimate} below)
for the stopped-tamed Euler-Maruyama
approximations in \cite{HutzenthalerJentzenWang2013}.

\begin{lemma}
\label{lem:ST}
Let $ d, m, n \in \N $, 
$ 0 = t_0 < t_1 < \ldots < t_n = T < \infty $,
$ \mathcal{O} \in \mathcal{B}( \R^d ) $,
$ \phi \in \mathcal{L}^0( \R^d ; \R ) $,
$ 
  \mu \in \mathcal{L}^0( \R^d ; \R^d )
$,
$
  \sigma \in \mathcal{L}^0( \R^d ; \R^{ d \times m } )
$
satisfy 
$
  \left\| 
    \mu( x )
    -
    \mu( y )
  \right\|_{ \R^d }
  \vee
  \left\| 
    \sigma( x )
    -
    \sigma( y )
  \right\|_{ HS( \R^m, \R^d ) }
  \leq
$
$
  (
    \phi( x ) + \phi( y )
  )
$
$
  \left\| x - y \right\|_{ \R^d }
$
for all $ x, y \in \R^d $,
let 
$ 
  ( 
    \Omega, \mathcal{F}, \P, 
    ( \mathcal{F}_t )_{ t \in [ 0 , T ] } 
  ) 
$
be a stochastic basis,
let
$
  W \colon [0,T] \times \Omega \to \R^m
$
be a standard 
$ ( \mathcal{F}_t )_{ t \in [0,T] } $-Brownian motion,
let
$ 
  X, Y \colon [ 0 , T ] \times \Omega \to \R^d 
$
be adapted stochastic processes with 
c.s.p.\ satisfying
$
  \int_{0}^T
  \| \mu( X_s ) \|_{ \R^d }
  +
  \| \sigma( X_s ) \|_{ \R^{ d \times m } }^2
  \, ds
$
$ \P $-a.s.,
$
  X_t = 
  X_0 + \int_{0}^t \mu( X_s ) \, ds +
  \int_{0}^t \sigma( X_s ) \, dW_s
$
$ \P $-a.s.\ for all $ t \in [0,T] $,
$ Y_0 = X_0 $
and
\begin{equation}
\label{eq:def_stopped_EM}
  Y_t
  = 
  Y_{ t_k }
  +
  \mathbbm{1}_{
    \{
      Y_{ t_k } \in \mathcal{O}
    \}
  }
  \left[
  \tfrac{
    \mu( Y_{ t_k } ) \left( t - t_k \right)
    +
    \sigma( Y_{ t_k } )
    ( W_t - W_{ t_k } )
  }{
    1 +
    \|
      \mu( Y_{ t_k } ) \left( t - t_k \right)
      +
      \sigma( Y_{ t_k } )
      ( W_t - W_{ t_k } )
    \|^2_{ \R^d }
  }
  \right]
\end{equation}
for all $ t \in [ t_k, t_{ k + 1 } ] $,
$ k \in \{ 0, 1, \dots, n - 1 \} $
and let $ \tau \colon \Omega \to [ 0, T ] $
be given by
$
  \tau 
= 
  \inf\!\big( 
    \{ T \} 
    \cup
    \{
      t \in \{ t_0 , t_1, \dots, t_n \}
      \colon
      Y_t \notin \mathcal{O}
    \}
  \big)
$.
Then it holds 
for all 
stopping times 
$
  \nu \colon \Omega \to [0,T]
$
and all
$ \varepsilon, r \in (0,\infty) $,
$ p \in [2,\infty) $,
$ q, u, v \in (0,\infty] $
with 
$
  \frac{ 1 }{ p } + \frac{ 1 }{ q } = \frac{ 1 }{ r }
$
and
$
  \frac{ 1 }{ u } + \frac{ 1 }{ v } = \frac{ 1 }{ p }
$
that
\begin{align}
\label{eq:strong_convergence_rate_estimate}
&
  \big\| 
    X_{ \nu \wedge \tau } - Y_{ \nu \wedge \tau }
  \big\|_{
    L^r( \Omega; \R^d )
  }
\leq
  \sup_{ s \in [0,T] }
  \max\!\left(
    1,
  \sqrt{T}
  \,
  \big\|
    \mu(
      Y_s
    )
  \big\|_{
    L^v( \Omega; \R^d )
  }
  +
  v \,
    \|
      \sigma(
        Y_s
      )
    \|_{
      L^v( \Omega; HS( \R^m, \R^d ) )
    }
  \right)
\nonumber
\\ & \cdot 
  30 \, p \left( 1 + \tfrac{ 1 }{ \varepsilon } \right) 
  e^T
  \left\|
  \exp\!\left(
    \smallint_{0}^{ \nu \wedge \tau }
  \Big[
  \tfrac{
    \left< 
      X_s - Y_s , \mu( X_s ) - \mu( Y_s ) 
    \right>_{ \R^d }
    +
    \frac{ ( p - 1 ) \, ( 1 + \varepsilon ) }{ 2 }
    \left\|
      \sigma( X_s ) - \sigma( Y_s )
    \right\|^2_{ HS( \R^m , \R^d ) }
  }{
    \left\| X_s - Y_s \right\|^2_{ \R^d }
  }
  \Big]^+
    ds
  \right)
  \right\|_{
    L^q( \Omega; \R )
  }
\\ & \cdot
\nonumber
  \left[ 
  \sup_{ s \in [0,T] }
    \big\|
      \|
        \mu(
          Y_s
        )
      \|_{ \R^d }
      +
      [
        1 \vee
        \|
          \sigma(
            Y_s
          )
        \|_{ HS( \R^m, \R^d ) }
      ]^2
      +
      |
        \phi(
          Y_s
        )
      |
    \big\|_{
      L^u( \Omega; \R )
    }
  \right]
  \left[ 
    \max_{ 0 \leq k \leq n - 1 }
    | t_{ k + 1 } - t_k |
  \right]^{ \frac{ 1 }{ 2 } }
  .
\end{align}
\end{lemma}

\begin{proof}[Proof 
of Lemma~\ref{lem:ST}]
Throughout this proof let 
$ e^{ (m) }_1 = ( 1, 0, \dots, 0 ) , \dots, e^{ (m) } = ( 0, \dots, 0, 1 ) \in \R^m $
be the Euclidean orthonormal basis of the $ \R^m $,
let 
$ 
  \lfloor \cdot 
  \rfloor_{ t_0 , t_1, \dots, t_n }
  \colon
  [0,T] \to \{ t_0 , t_1, \dots, t_n \}
$
be given by
$
  \lfloor t \rfloor_{ t_0 , t_1 , \dots, t_n }
$
$
  =
  \max\!\big(
    [0,t] \cap \{ t_0 , t_1, \dots, t_n \}
  \big)
$
for all $ t \in [ 0,T ] $,
let $ \nu \colon \Omega \to [0,T] $
be a stopping time,
let $ \psi \colon \R^d \to \R^d $
be given by
$
  \psi( v ) = \frac{ v }{ 1 + \| v \|^2_{ \R^d } }
$
for all $ v \in \R^d $
and let
$ Z \colon [0,T] \times \Omega \to \R^d $,
$ a \colon [0,T] \times \Omega \to \R^d $,
$ b \colon [0,T] \times \Omega \to HS( \R^m, \R^d ) $
be given by
\begin{equation}
  Z_t
  =
    \mu\big( 
      Y_{ 
        \lfloor t \rfloor_{ t_0 , t_1, \dots, t_n }
      }
    \big)
    \left( 
      t - \lfloor t \rfloor_{ t_0 , t_1, \dots, t_n }
    \right)
    +
    \sigma\big( 
      Y_{ 
        \lfloor t \rfloor_{ t_0 , t_1, \dots, t_n }
      }
    \big)
    \big( 
      W_t
      -
      W_{ 
        \lfloor t \rfloor_{ t_0 , t_1, \dots, t_n }
      }
    \big)
  ,
\end{equation}
\begin{equation}
\begin{split}
&
  a_t 
= 
  \psi'\!\left(
    Z_t
  \right)
    \mu( 
      Y_{ 
        \lfloor t \rfloor_{ t_0 , t_1, \dots, t_n }
      }
    )
  +
  \tfrac{ 1 }{ 2 }
  \smallsum\limits_{ j = 1 }^{ m }
  \psi''\!\left(
    Z_t
  \right)\!
  \left(
    \sigma( 
      Y_{ 
        \lfloor t \rfloor_{ t_0 , t_1, \dots, t_n }
      }
    )
    \,
    e_j^{ (m) }
    ,
    \sigma( 
      Y_{ 
        \lfloor t \rfloor_{ t_0 , t_1, \dots, t_n }
      }
    )
    \,
    e_j^{ (m) }
  \right)
\end{split}
\end{equation}
and
$
  b_t = 
  \psi'\!\left(
    Z_t
  \right)
    \sigma( 
      Y_{ 
        \lfloor t \rfloor_{ t_0 , t_1, \dots, t_n }
      }
    )
$
for all $ t \in [0,T] $.
W.l.o.g.\ we assume that
\begin{equation}
\label{eq:wlog_numeric}
    \smallint_{0}^{ \nu \wedge \tau }
  \Big[
  \tfrac{
    \left< 
      X_s - Y_s , \mu( X_s ) - \mu( Y_s ) 
    \right>_{ \R^d }
    +
    \frac{ ( p - 1 ) \, ( 1 + \varepsilon ) }{ 2 }
    \left\|
      \sigma( X_s ) - \sigma( Y_s )
    \right\|^2_{ HS( \R^m , \R^d ) }
  }{
    \left\| X_s - Y_s \right\|^2_{ \R^d }
  }
  \Big]^+
    ds
  < \infty
\end{equation}
$ \P $-a.s.\ (otherwise estimate 
\eqref{eq:strong_convergence_rate_estimate} 
is clear).
It\^{o}'s formula then proves that
$
  Y_t
  =
  Y_{ t \wedge \tau }
  =
  X_0 
  +
  \int_{0}^t
  \mathbbm{1}_{
    \{ s < \tau \}
  }
  \,
  a_s \, ds
  +
  \int_{0}^t
  \mathbbm{1}_{
    \{ s < \tau \}
  }
  \,
  b_s \, dW_s
$
$ \P $-a.s.\ for all $ t \in [0,T] $.
Next note that
Lemma~7.7 in Da Prato \& Zabczyk~\cite{dz92}
proves that
for all $ r \in [2,\infty) $ it holds that
\begin{equation}
\label{eq:Z_estimate}
\begin{split}
&
  \left\|
    Z_t
  \right\|_{
    L^r( \Omega; \R^d )
  }
\\ & \leq
  \sup_{ u \in [0,T] }
  \left[
  \sqrt{ T }
  \,
  \big\|
    \mu(
      Y_u
    )
  \big\|_{
    L^r( \Omega; \R^d )
  }
  +
  \tfrac{
    \sqrt{ r \, \left( r - 1 \right) }
  }{
    \sqrt{ 2 }
  }
    \|
      \sigma(
        Y_u
      )
    \|_{
      L^r( \Omega; HS( \R^m, \R^d ) )
    }
  \right]
  \left[ 
    \max_{ 0 \leq k \leq n - 1 }
    | t_{ k + 1 } - t_k |
  \right]^{ \frac{ 1 }{ 2 } }
  .
\end{split}
\end{equation}
Combining 
Lemma~\ref{lem:psi_estimate}
with the estimate 
$
  \frac{ 
    \| v \|_{ \R^d } 
  }{
    1 + \| v \|_{ \R^d }^2 
  }
  \leq 
  \| v \|_{ \R^d }
$
for all $ v \in \R^d $
shows that
for all $ t \in [0,T] $ it holds that
\begin{equation}
\label{eq:mu_est}
\begin{split}
&
  \big\|
    a_t 
    -
    \mu( 
      Y_t 
    )
  \big\|_{ \R^d }
\\ & \leq
  \big\|
    \psi'( Z_t ) - I_{ \R^d }
  \big\|_{ L( \R^d ) }
  \,
  \big\|
    \mu\big(
      Y_{ 
        \lfloor t \rfloor_{ t_0 , t_1, \dots, t_n }
      }
    \big)
  \big\|_{ \R^d }
  +
  \big\|
    \mu\big( 
      Y_{
        \lfloor t \rfloor_{ t_0 , t_1 , \dots, t_n }
      }
    \big)
    -
    \mu\big( Y_t \big)
  \big\|_{ \R^d }
\\ &
  +
  \tfrac{ 1 }{ 2 }
  \,
  \big\|
    \sigma\big(
      Y_{
        \lfloor t \rfloor_{ t_0 , t_1, \dots, t_n }
      }
    \big)
  \big\|_{ HS( \R^m, \R^d ) }^2
  \left[
  \sup\nolimits_{ 
    u \in \R^d 
    ,
    \,
    \| u \|_{ \R^d } \leq 1
  }
  \left\|
    \psi''( Z_t )( u, u )
  \right\|_{
    \R^d
  }
  \right]
\\ & \leq
  \| Z_t \|_{ \R^d }
  \left[ 
  \,
  3
  \,
  \big\|
    \mu\big(
      Y_{ 
        \lfloor t \rfloor_{ t_0 , t_1, \dots, t_n }
      }
    \big)
  \big\|_{ \R^d }
  +
  7 \,
  \big\|
    \sigma\big(
      Y_{
        \lfloor t \rfloor_{ t_0 , t_1, \dots, t_n }
      }
    \big)
  \big\|_{ HS( \R^m, \R^d ) }^2
  +
  \big|
    \phi\big(
      Y_t
    \big)
    +
    \phi\big(
      Y_{
        \lfloor t \rfloor_{ t_0 , t_1, \dots, t_n }
      }
    \big)
  \big|
  \,
  \right]
\end{split}
\end{equation}
and
\begin{equation}
\label{eq:sigma_est}
\begin{split}
&
  \big\|
    b_t
    -
    \sigma( Y_t )
  \big\|_{ HS( \R^m, \R^d ) }
\\ & \leq
  \big\|
    \psi'( Z_t ) - I_{ \R^d }
  \big\|_{ L( \R^d ) }
  \,
  \big\|
    \sigma\big(
      Y_{ 
        \lfloor t \rfloor_{ t_0 , t_1, \dots, t_n }
      }
    \big)
  \big\|_{ HS( \R^m, \R^d ) }
  +
  \big\|
    \sigma\big( 
      Y_{
        \lfloor t \rfloor_{ t_0 , t_1, \dots, t_n }
      }
    \big)
    -
    \sigma\big( 
      Y_t
    \big)
  \big\|_{ 
    HS( \R^m, \R^d )
  }
\\ & \leq
  \| Z_t \|_{ \R^d }
  \,
  \Big[
  3 \,
  \big\|
    \sigma\big(
      Y_{ 
        \lfloor t \rfloor_{ t_0 , t_1, \dots, t_n }
      }
    \big)
  \big\|_{ HS( \R^m, \R^d ) }
  +
  \big|
    \phi\big(
      Y_t
    \big)
    +
    \phi\big(
      Y_{
        \lfloor t \rfloor_{ t_0 , t_1, \dots, t_n }
      }
    \big)
  \big|
  \,
  \Big]
  .
\end{split}
\end{equation}
Combining \eqref{eq:mu_est} and \eqref{eq:sigma_est}
with \eqref{eq:Z_estimate}
proves that
for all $ p, u \in [2,\infty) $,
$ v \in (2,\infty] $
with $ \frac{ 1 }{ u } + \frac{ 1 }{ v } = \frac{ 1 }{ p } $
it holds that
\begin{equation}
\label{eq:residuum_mu}
\begin{split}
&
  \left\|
    a - \mu( Y )
  \right\|_{
    L^p( 
      \llbracket 0,\tau \rrbracket ; \R^d
    )
  }
\leq
  14 \, T^{ \frac{ 1 }{ p } }
  \left[ 
  \sup_{ s \in [0,T] }
    \big\|
      \|
        \mu(
          Y_s
        )
      \|_{ \R^d }
      +
      \|
        \sigma(
          Y_s
        )
      \|_{ HS( \R^m, \R^d ) }^2
      +
      |
        \phi(
          Y_s
        )
      |
    \big\|_{
      L^v( \Omega; \R )
    }
  \right]
\\ & \quad
  \cdot
  \sup_{ s \in [0,T] }
  \left[
  \sqrt{ T }
  \,
  \big\|
    \mu(
      Y_s
    )
  \big\|_{
    L^u( \Omega; \R^d )
  }
  +
  \tfrac{
    \sqrt{ u \, \left( u - 1 \right) } \,
    \|
      \sigma(
        Y_s
      )
    \|_{
      L^u( \Omega; HS( \R^m, \R^d ) )
    }
  }{
    \sqrt{ 2 }
  }
  \right]
  \left[ 
    \max_{ 0 \leq k \leq n - 1 }
    | t_{ k + 1 } - t_k |
  \right]^{ \frac{ 1 }{ 2 } }
\end{split}
\end{equation}
and
\begin{equation}
\label{eq:residuum_sigma}
\begin{split}
&
  \left\|
    b - \sigma( Y )
  \right\|_{
    L^p( 
      \llbracket 0 , \tau \rrbracket ; HS( \R^m, \R^d )
    )
  }
\leq
  6 \, T^{ \frac{ 1 }{ p } }
  \left[
    \sup_{ s \in [0,T] }
    \big\|
      \|
        \sigma(
          Y_s
        )
      \|_{
        HS( \R^m , \R^d )
      }
      +
      |
        \phi( Y_s )
      |
    \big\|_{ 
      L^v( \Omega; \R ) 
    }
  \right]
%   \left\|
%     Z_t
%   \right\|_{ L^r( \Omega; \R^d ) }
\\ & \quad
  \cdot
  \sup_{ s \in [0,T] }
  \left[
  \sqrt{ T }
  \,
  \big\|
    \mu(
      Y_s
    )
  \big\|_{
    L^u( \Omega; \R^d )
  }
  +
  \tfrac{
    \sqrt{ u \, \left( u - 1 \right) }
    \,
    \|
      \sigma(
        Y_s
      )
    \|_{
      L^u( \Omega; HS( \R^m, \R^d ) )
    }
  }{
    \sqrt{ 2 }
  }
  \right]
  \left[ 
    \max_{ 0 \leq k \leq n - 1 }
    | t_{ k + 1 } - t_k |
  \right]^{ \frac{ 1 }{ 2 } }
  .
\end{split}
\end{equation}
Corollary~\ref{cor:norm2}
together with 
\eqref{eq:wlog_numeric},
\eqref{eq:residuum_mu}
and \eqref{eq:residuum_sigma} implies that
for all 
$ \varepsilon, \delta, \rho, r \in (0,\infty) $,
$ p \in [2,\infty) $,
$ q, u, v \in (0,\infty] $
with 
$
  \frac{ 1 }{ p } + \frac{ 1 }{ q } = \frac{ 1 }{ r }
$
and
$
  \frac{ 1 }{ u } + \frac{ 1 }{ v } = \frac{ 1 }{ p }
$
it holds that
\begin{align}
\label{eq:stopped_tamed_proof1}
&
  \big\| X_{ \nu \wedge \tau } - Y_{ \nu \wedge \tau }
  \big\|_{
    L^r( \Omega; \R^d )
  }
\\ & \leq
  \left\|
  \exp\!\left(
    \smallint_{0}^{ \nu \wedge \tau }
  \bigg[
  \tfrac{
    \left< 
      X_s - Y_s , \mu( X_s ) - \mu( Y_s ) 
    \right>_{ \R^d }
    +
    \frac{ ( p - 1 ) \, ( 1 + \varepsilon ) }{ 2 }
    \left\|
      \sigma( X_s ) - \sigma( Y_s )
    \right\|^2_{ HS( \R^m, \R^d ) }
  }{
    \left\| X_s - Y_s \right\|^2_{ \R^d }
  }
  +
  \tfrac{
      ( 1 - \frac{ 1 }{ p } ) 
  }{
      \delta
  }
      +
  \tfrac{
      ( \frac{ 1 }{ 2 } - \frac{ 1 }{ p } )
  }{
      \rho
  }
  \bigg]^+
  \!
    ds
  \right)
  \right\|_{
    L^q( \Omega; \R )
  }
\nonumber
\\ & \cdot
\nonumber
    6 \, T^{ \frac{ 1 }{ p } }
  \Bigg[
    \frac{ 
      7 
      \delta^{
        ( 1 - \frac{ 1 }{ p } )
      }
    }{ 3 }
  \bigg[ 
    \sup_{ s \in [0,T] }
    \big\|
      \|
        \mu(
          Y_s
        )
      \|_{ \R^d }
      +
      \|
        \sigma(
          Y_s
        )
      \|_{ HS( \R^m, \R^d ) }^2
      +
      |
        \phi(
          Y_s
        )
      |
    \big\|_{
      L^v( \Omega; \R )
    }
  \bigg]
\\ &
    +
      \rho^{ ( \frac{ 1 }{ 2 } - \frac{ 1 }{ p } ) }
      \sqrt{ ( p - 1 ) \, ( 1 + 1 / \varepsilon ) }
    \bigg[
      \sup_{ s \in [0,T] }
      \big\|
      \|
        \sigma(
          Y_s
        )
      \|_{
        HS( \R^m , \R^d )
      }
      +
      |
        \phi( Y_s )
      |
      \big\|_{ 
        L^v( \Omega; \R ) 
      }
    \bigg]
  \Bigg]
\nonumber
\\ &
\nonumber
  \cdot
  \sup_{ s \in [0,T] }
  \left[
  \sqrt{ T }
  \,
  \big\|
    \mu(
      Y_s
    )
  \big\|_{
    L^u( \Omega; \R^d )
  }
  +
  \tfrac{
    \sqrt{ u \, \left( u - 1 \right) }
    \,
    \|
      \sigma(
        Y_s
      )
    \|_{
      L^u( \Omega; HS( \R^m, \R^d ) )
    }
  }{
    \sqrt{ 2 }
  }
  \right]
  \left[ 
    \max_{ 0 \leq k \leq n - 1 }
    | t_{ k + 1 } - t_k |
  \right]^{ \frac{ 1 }{ 2 } }
  .
\end{align}
The choice $ \rho = 1 $ and $ \delta = 2 $ in 
\eqref{eq:stopped_tamed_proof1}
shows that
for all 
$ \varepsilon, r \in (0,\infty) $,
$ p \in [2,\infty) $,
$ q, u, v \in (0,\infty] $
with 
$
  \frac{ 1 }{ p } + \frac{ 1 }{ q } = \frac{ 1 }{ r }
$
and
$
  \frac{ 1 }{ u } + \frac{ 1 }{ v } = \frac{ 1 }{ p }
$
it holds that
\begin{align}
&
  \big\| X_{ \nu \wedge \tau } - Y_{ \nu \wedge \tau }
  \big\|_{
    L^r( \Omega; \R^d )
  }
\\ & \leq
  \left\|
  \exp\!\left(
    \smallint_{0}^{ \nu \wedge \tau }
  \bigg[
  \tfrac{
    \left< 
      X_s - Y_s , \mu( X_s ) - \mu( Y_s ) 
    \right>_{ \R^d }
    +
    \frac{ ( p - 1 ) \, ( 1 + \varepsilon ) }{ 2 }
    \left\|
      \sigma( X_s ) - \sigma( Y_s )
    \right\|^2_{ HS( \R^m, \R^d ) }
  }{
    \left\| X_s - Y_s \right\|^2_{ \R^d }
  }
  +
  1
  -
  \tfrac{ 3 }{ 2 p }
  \bigg]^+
  \!
    ds
  \right)
  \right\|_{
    L^q( \Omega; \R )
  }
\nonumber
\\ & \cdot
\nonumber
    6 \, T^{ \frac{ 1 }{ p } }
  \Bigg[
    \frac{ 
      7 
      \cdot
      2^{
        ( 1 - \frac{ 1 }{ p } )
      }
    }{ 3 }
    +
      \sqrt{ ( p - 1 ) \, ( 1 + 1 / \varepsilon ) }
  \Bigg]
\\ &
  \cdot
  \bigg[ 
    \sup_{ s \in [0,T] }
    \big\|
      \|
        \mu(
          Y_s
        )
      \|_{ \R^d }
      +
      [
      1 \vee
      \|
        \sigma(
          Y_s
        )
      \|_{ HS( \R^m, \R^d ) }
      ]^2
      +
      |
        \phi(
          Y_s
        )
      |
    \big\|_{
      L^v( \Omega; \R )
    }
  \bigg]
\nonumber
\\ &
\nonumber
  \cdot
  \sup_{ s \in [0,T] }
  \left[
  \sqrt{ T }
  \,
  \big\|
    \mu(
      Y_s
    )
  \big\|_{
    L^u( \Omega; \R^d )
  }
  +
  \tfrac{
    \sqrt{ u \, \left( u - 1 \right) }
    \,
    \|
      \sigma(
        Y_s
      )
    \|_{
      L^u( \Omega; HS( \R^m, \R^d ) )
    }
  }{
    \sqrt{ 2 }
  }
  \right]
  \left[ 
    \max_{ 0 \leq k \leq n - 1 }
    | t_{ k + 1 } - t_k |
  \right]^{ \frac{ 1 }{ 2 } }
  .
\end{align}
This implies
that
for all 
$ \varepsilon, r \in (0,\infty) $,
$ p \in [2,\infty) $,
$ q, u, v \in (0,\infty] $
with 
$
  \frac{ 1 }{ p } + \frac{ 1 }{ q } = \frac{ 1 }{ r }
$
and
$
  \frac{ 1 }{ u } + \frac{ 1 }{ v } = \frac{ 1 }{ p }
$
it holds that
\begin{align}
&
  \big\| X_{ \nu \wedge \tau } - Y_{ \nu \wedge \tau }
  \big\|_{
    L^r( \Omega; \R^d )
  }
\\ & \leq
  \left\|
  \exp\!\left(
  T 
  \left[
    1
    -
    \tfrac{ 3 }{ 2 p }
  \right]
  +
    \smallint_{0}^{ \nu \wedge \tau }
  \bigg[
  \tfrac{
    \left< 
      X_s - Y_s , \mu( X_s ) - \mu( Y_s ) 
    \right>_{ \R^d }
    +
    \frac{ ( p - 1 ) \, ( 1 + \varepsilon ) }{ 2 }
    \left\|
      \sigma( X_s ) - \sigma( Y_s )
    \right\|^2_{ HS( \R^m, \R^d ) }
  }{
    \left\| X_s - Y_s \right\|^2_{ \R^d }
  }
  \bigg]^+
  \!
    ds
  \right)
  \right\|_{
    L^q( \Omega; \R )
  }
\nonumber
\\ &
  \cdot
  30 \, p \left( 1 + \tfrac{ 1 }{ \varepsilon } \right) T^{ \frac{ 1 }{ p } }
  \bigg[ 
    \sup_{ s \in [0,T] }
    \big\|
      \|
        \mu(
          Y_s
        )
      \|_{ \R^d }
      +
      [
      1 \vee
      \|
        \sigma(
          Y_s
        )
      \|_{ HS( \R^m, \R^d ) }
      ]^2
      +
      |
        \phi(
          Y_s
        )
      |
    \big\|_{
      L^v( \Omega; \R )
    }
  \bigg]
\nonumber
\\ &
\nonumber
  \cdot
  \sup_{ s \in [0,T] }
  \left[
  \sqrt{ T }
  \,
  \big\|
    \mu(
      Y_s
    )
  \big\|_{
    L^u( \Omega; \R^d )
  }
  +
  \tfrac{
    \sqrt{ u \, \left( u - 1 \right) }
    \,
    \|
      \sigma(
        Y_s
      )
    \|_{
      L^u( \Omega; HS( \R^m, \R^d ) )
    }
  }{
    \sqrt{ 2 }
  }
  \right]
  \left[ 
    \max_{ 0 \leq k \leq n - 1 }
    | t_{ k + 1 } - t_k |
  \right]^{ \frac{ 1 }{ 2 } }
  .
\end{align}
This completes the proof of Lemma~\ref{lem:ST}.
\end{proof}

Lemma~\ref{lem:ST} is only of use if the right-hand side
of \eqref{eq:strong_convergence_rate_estimate} is finite.
The next result (Proposition~\ref{prop:ST2}), in particular, provides
sufficient conditions to ensure that the right-hande side
of \eqref{eq:strong_convergence_rate_estimate} is finite 
and thereby establish strong convergence rates 
for the 
stopped-tamed Euler-Maruyama approximations
in \cite{HutzenthalerJentzenWang2013}.

\begin{prop}
\label{prop:ST2}
Let $ d, m \in \N $, 
$ r, \varepsilon, c, T \in (0,\infty) $, 
$ q_0, q_1 \in (0,\infty] $,
$ \alpha \in [0,\infty) $,
$ p \in [2,\infty) $,
$
  U_0 \in \mathcal{C}^3_{ \mathcal{D} }( \R^d, [0,\infty) )
$,
$ 
  U_1 \in \mathcal{C}^1_{ \mathcal{P} }( \R^d, [0,\infty) ) 
$,
$ 
  \mu \in \mathcal{C}^1_{ \mathcal{P} }( \R^d , \R^d )
$,
$
  \sigma \in \mathcal{C}^1_{ \mathcal{P} }( \R^d , \R^{ d \times m } )
$
satisfy 
$
  \frac{ 1 }{ p } + \frac{ 1 }{ q_0 } + \frac{ 1 }{ q_1 } = \frac{ 1 }{ r } 
$,
$ \| x \|^{ 1 / c } \leq c \left( 1 + U_0( x ) \right) $
and
{\small
\begin{equation*}
\begin{array}{c}
  ( \mathcal{G}_{ \mu, \sigma } U_0 )( x )
  +
  \tfrac{ 1 }{ 2 } 
  \, \| \sigma( x )^* ( \nabla U_0 )( x ) \|^2_{ \R^d } 
  +
  U_1( x )
\leq
  \alpha \, U_0( x )
  +
  c
  ,
\\[1ex]
    \left< 
      x - y , \mu( x ) - \mu( y ) 
    \right>_{ \R^d }
    +
    \tfrac{ ( p - 1 ) \, ( 1 + \varepsilon ) }{ 2 }
    \left\|
      \sigma( x ) - \sigma( y )
    \right\|^2_{ HS( \R^m , \R^d ) }
\leq
  \left[
    c
    +
    \tfrac{
      U_0( x ) + U_0( y )
    }{
      2 q_0 T e^{ \alpha T }
    }
    +
    \tfrac{
      U_1( x ) + U_1( y )
    }{
      2 q_1 e^{ \alpha T }
    }
  \right]
    \left\| x - y \right\|^2_{ \R^d }
\end{array}
\end{equation*}}for 
all $ x, y \in \R^d $,
let 
$ 
  ( 
    \Omega, \mathcal{F}, \P, 
    ( \mathcal{F}_t )_{ t \in [ 0 , T ] } 
  ) 
$
be a stochastic basis,
let
$
  W \colon [0,T] \times \Omega \to \R^m
$
be a standard 
$ ( \mathcal{F}_t )_{ t \in [0,T] } $-Brownian motion,
let
$ 
  X \colon [ 0, T ] \times \Omega \to \R^d 
$
and 
$ 
  Y^{ \theta } \colon [ 0 , T ] \times \Omega \to \R^d 
$,
$ \theta \in \mathcal{P}_T $,
be adapted stochastic processes with 
c.s.p.,
$
  \E\big[
    e^{ U_0( X_0 ) }
  \big]
  < \infty
$,
$
  X_t = 
  X_0 + \int_{0}^t \mu( X_s ) \, ds +
  \int_{0}^t \sigma( X_s ) \, dW_s
$
$ \P $-a.s.\ for all $ t \in [0,T] $
and
$ Y_0^{ \theta } = X_0 $
and
\begin{equation}
  Y_t^{ \theta }
  = 
  Y_{ t_k }^{ \theta }
  +
  \mathbbm{1}_{
    \left\{
      \| Y_{ t_k }^{ \theta } \|_{ \R^d }
      \,
      < 
      \,
      \exp(
        \,
        | 
          \!
          \ln( 
            \max_{ 0 \leq i \leq n - 1 }
            t_{ i + 1 } - t_i 
          ) 
        |^{ 1 / 2 }
        \,
      )
    \right\}
  }
  \left[
  \tfrac{
    \mu( Y_{ t_k }^{ \theta } ) \left( t - t_k \right)
    +
    \sigma( Y_{ t_k }^{ \theta } )
    ( W_t - W_{ t_k } )
  }{
    1 +
    \|
      \mu( Y_{ t_k }^{ \theta } ) \left( t - t_k \right)
      +
      \sigma( Y_{ t_k }^{ \theta } )
      ( W_t - W_{ t_k } )
    \|^2_{ \R^d }
  }
  \right]
\end{equation}
for all $ t \in [ t_k, t_{ k + 1 } ] $,
$ k \in \{ 0, 1, \dots, n - 1 \} $,
$ \theta = ( t_0 , \dots, t_n ) \in \mathcal{P}_T $,
$ n \in \N $.
Then there exists a real number $ C \in [0,\infty) $
such that 
for all $ n \in \N $
and all
$ \theta = (  t_0 , t_1, \dots, t_n ) \in \mathcal{P}_T $
it holds that
\begin{align}
&
  \sup_{ t \in [0,T] }
  \big\| 
    X_t - Y^{ \theta }_t
  \big\|_{
    L^r( \Omega; \R^d )
  }
\leq
  C
  \,
  \big[
    \max\nolimits_{ 
      k \in \{ 0, 1, \dots, n - 1 \} 
    }
    \left| t_{ k + 1 } - t_k \right|
  \big]^{
    \frac{ 1 }{ 2 }
  }
  \,
  .
\end{align}
\end{prop}

\begin{proof}[Proof
of Proposition~\ref{prop:ST2}]
Throughout this proof let
$ q \in (0,\infty] $ be given by
$
  \frac{ 1 }{ q } = \frac{ 1 }{ q_0 } + \frac{ 1 }{ q_1 }
$
and 
let
$
  \tau_{ \theta } \colon \Omega \to [ 0, T ]
$,
$ \theta \in \mathcal{P}_T $,
be mappings given by
\begin{equation}
  \tau_{ \theta }
= 
  \inf\!\left( 
    \{ T \}
    \cup
    \left\{
      t \in \{  t_0 , t_1, \dots, t_n \}
      \colon
      \| Y_t^{ \theta } \|_{ \R^d } 
      \geq
      \exp\!\left(
        \left| 
          \ln\!\left( 
            \max_{ i \in \{ 0, 1, \dots, n - 1 \} }
            \left[
              t_{ i + 1 } - t_i 
            \right]
          \right) 
        \right|^{ 1 / 2 }
      \right)
    \right\}
  \right)
\end{equation}
for all $ \theta = ( t_0, t_1, \dots, t_n ) \in \mathcal{P}_T $ and all $ n \in \N $.
Next note that the assumptions
$ \mu \in \mathcal{C}^1_{ \mathcal{P} }( \R^d, \R^d ) $
and
$ \sigma \in \mathcal{C}^1_{ \mathcal{P} }( \R^d, \R^{ d \times m } ) $
ensure that there exists a real number 
$ 
  \hat{c} \in 
  \big[
    1 + \| \mu( 0 ) \|_{ \R^d } + \| \sigma( 0 ) \|_{ HS( \R^m, \R^d ) }
    ,
    \infty
  \big) 
$
such that for all $ x, y \in \R^d $ it holds that
\begin{equation}
  \left\| 
    \mu( x )
    -
    \mu( y )
  \right\|_{ \R^d }
  \vee
  \left\| 
    \sigma( x )
    -
    \sigma( y )
  \right\|_{ HS( \R^m, \R^d ) }
  \leq
  \hat{c}
  \left(
    1 + \| x \|^{ \hat{c} } + 
    \| y \|^{ \hat{c} }
  \right)
  \left\| x - y \right\|_{ \R^d } 
  .
\end{equation}
In view of this, let $ \phi \in C( \R^d , \R ) $
be given by
$
  \phi( x ) 
  = 
  4 \,
  | \hat{c} |^2 
  \,
  \big[ 
    1 + \| x \| 
  \big]^{ ( 2 \hat{c} + 2 ) }
$
for all $ x \in \R^d $
and note that for all $ x , y \in \R^d $
it holds that
$
  \max\{
  1 ,
  \| \mu( x ) \|_{ \R^d } ,
  \|
    \sigma(x)
  \|^2_{ HS( \R^m, \R^d ) }
  \}
  \leq 
  \phi( x )
$
and
$
  \left\| 
    \mu( x )
    -
    \mu( y )
  \right\|_{ \R^d }
  \vee
  \left\| 
    \sigma( x )
    -
    \sigma( y )
  \right\|_{ HS( \R^m, \R^d ) }
  \leq
  \left(
    \phi( x ) + \phi( y )
  \right)
  \left\| x - y \right\|_{ \R^d } 
$.
We can thus apply 
Lemma~\ref{lem:ST}
to obtain that
for all 
$ n \in \N $,
$ \theta = ( t_0 , t_1, \dots, t_n ) \in \mathcal{P}_T $,
$ u, v \in (0,\infty] $
with 
$
  \frac{ 1 }{ u } + \frac{ 1 }{ v } = \frac{ 1 }{ p }
$
it holds that
\begin{align}
\label{eq:numeric_proof_1}
&
  \sup_{ t \in [0,T] }
  \big\| 
    X_{ t \wedge \tau_{ \theta } } - Y^{ \theta }_{ t \wedge \tau_{ \theta } }
  \big\|_{
    L^r( \Omega; \R^d )
  }
\leq
  \sup_{ s \in [0,T] }
  \max\!\left( 1,
  \sqrt{T}
  \,
  \big\|
    \phi(
      Y^{ \theta }_s
    )
  \big\|_{
    L^v( \Omega; \R )
  }
  +
  v \,
    \|
      \phi(
        Y^{ \theta }_s
      )
    \|_{
      L^v( \Omega; \R )
    }
  \right)
\nonumber
\\ & \quad \cdot 
  30 \,
  p
  \left[
    1 + \tfrac{ 1 }{ \varepsilon }
  \right]
  e^T
  \left\|
  \exp\!\left(
    \smallint_{0}^{ \tau_{ \theta } }
  \Big[
  \tfrac{
    \langle 
      X_s - Y^{ \theta }_s , \mu( X_s ) - \mu( Y^{ \theta }_s ) 
    \rangle_{ \R^d }
    +
    \frac{ ( p - 1 ) \, ( 1 + \varepsilon ) }{ 2 }
    \|
      \sigma( X_s ) - \sigma( Y^{ \theta }_s )
    \|^2_{ HS( \R^m , \R^d ) }
  }{
    \| 
      X_s - Y^{ \theta }_s 
    \|^2_{ \R^d }
  }
  \Big]^+
    ds
  \right)
  \right\|_{
    L^q( \Omega; \R )
  }
\nonumber
\\ & \quad \cdot
  \left[ 
  \sup_{ s \in [0,T] }
    \big\|
      3
      \phi(
        Y^{ \theta }_s
      )
    \big\|_{
      L^u( \Omega; \R )
    }
  \right]
  \left[ 
    \max_{ 0 \leq k \leq n - 1 }
    | t_{ k + 1 } - t_k |
  \right]^{ \frac{ 1 }{ 2 } }
  .
\end{align}
The choice $ u = v = 2 p $ in \eqref{eq:numeric_proof_1}
shows 
that
for all 
$ n \in \N $,
$ \theta = ( t_0 , t_1, \dots, t_n ) \in \mathcal{P}_T $
it holds that
\begin{align}
&
  \sup_{ t \in [0,T] }
  \big\| 
    X_{ t \wedge \tau_{ \theta } } - Y^{ \theta }_{ t \wedge \tau_{ \theta } }
  \big\|_{
    L^r( \Omega; \R^d )
  }
\leq
  360 
  \,
  p^2 
  \left[
    1 + \tfrac{ 1 }{ \varepsilon }
  \right]
  e^{ 2 T }
  \left[ 
  \sup_{ s \in [0,T] }
    \big\|
      \phi(
        Y^{ \theta }_s
      )
    \big\|_{
      L^{ 2 p }( \Omega; \R )
    }^2
  \right]
  \left[ 
    \max_{ 0 \leq k \leq n - 1 }
    | t_{ k + 1 } - t_k |
  \right]^{ \frac{ 1 }{ 2 } }
\nonumber
\\ & \quad \cdot 
  \left\|
  \exp\!\left(
    \smallint_{0}^{ \tau_{ \theta } }
  \Big[
  \tfrac{
    \left< 
      X_s - Y^{ \theta }_s , \mu( X_s ) - \mu( Y^{ \theta }_s ) 
    \right>_{ \R^d }
    +
    \frac{ ( p - 1 ) \, ( 1 + \varepsilon ) }{ 2 }
    \left\|
      \sigma( X_s ) - \sigma( Y^{ \theta }_s )
    \right\|^2_{ HS( \R^m , \R^d ) }
  }{
    \left\| X_s - Y^{ \theta }_s \right\|^2_{ \R^d }
  }
  \Big]^+
    ds
  \right)
  \right\|_{
    L^q( \Omega; \R )
  }
  .
\end{align}
The assumptions of Proposition~\ref{prop:ST2}
hence imply 
that
for all 
$ n \in \N $,
$ \theta = ( t_0 , t_1, \dots, t_n ) \in \mathcal{P}_T $
it holds that
\begin{equation}
\begin{split}
&
  \sup_{ t \in [0,T] }
  \big\| 
    X_{ t \wedge \tau_{ \theta } } - Y^{ \theta }_{ t \wedge \tau_{ \theta } }
  \big\|_{
    L^r( \Omega; \R^d )
  }
\leq
  360
  \, p^2
  \left[
    1 + \tfrac{ 1 }{ \varepsilon }
  \right]
  e^{ ( 2 + c ) T }
  \left[ 
    \max_{ 0 \leq k \leq n - 1 }
    | t_{ k + 1 } - t_k |
  \right]^{ \frac{ 1 }{ 2 } }
\\ & \quad \cdot 
  \left[  
  \sup_{ s \in [0,T] }
    \big\|
      \phi(
        Y^{ \theta }_s
      )
    \big\|_{
      L^{ 2 p }( \Omega; \R )
    }^2
  \right]
  \left\|
  \exp\!\left(
    \smallint_{0}^{ \tau_{ \theta } }
    \tfrac{
      U_0( X_s ) + U_0( Y^{ \theta }_s )
    }{
      2 q_0 T e^{ \alpha T }
    }
    +
    \tfrac{
      U_1( X_s ) + U_1( Y^{ \theta }_s )
    }{
      2 q_1 e^{ \alpha T }
    }
    \,
    ds
  \right)
  \right\|_{
    L^q( \Omega; \R )
  }
  .
\end{split}
\end{equation}
H\"{o}lder's inequality hence shows that
for all 
$ n \in \N $,
$ \theta = ( t_0 , t_1, \dots, t_n ) \in \mathcal{P}_T $
it holds that
\begin{equation}
\begin{split}
&
  \sup_{ t \in [0,T] }
  \big\| 
    X_{ t \wedge \tau_{ \theta } } - Y^{ \theta }_{ t \wedge \tau_{ \theta } }
  \big\|_{
    L^r( \Omega; \R^d )
  }
\leq
  360
  \, p^2
  \left[
    1 + \tfrac{ 1 }{ \varepsilon }
  \right]
  e^{ ( 2 + c ) T }
  \left[ 
    \max_{ 0 \leq k \leq n - 1 }
    | t_{ k + 1 } - t_k |
  \right]^{ \frac{ 1 }{ 2 } }
\\ & \cdot 
  \left[  
  \sup_{ s \in [0,T] }
    \big\|
      \phi(
        Y^{ \theta }_s
      )
    \big\|_{
      L^{ 2 p }( \Omega; \R )
    }
  \right]
  \left\|
  \exp\!\left(
    \smallint_{0}^{ \tau_{ \theta } }
    \tfrac{
      U_0( X_s ) 
    }{
      2 q_0 T e^{ \alpha T }
    }
    \,
    ds
  \right)
  \right\|_{
    L^{ 2 q_0 }( \Omega; \R )
  }
  \left\|
  \exp\!\left(
    \smallint_{0}^{ \tau_{ \theta } }
    \tfrac{
      U_0( Y^{ \theta }_s )
    }{
      2 q_0 T e^{ \alpha T }
    }
    \,
    ds
  \right)
  \right\|_{
    L^{ 2 q_0 }( \Omega; \R )
  }
\\ & \cdot
  \left\|
  \exp\!\left(
    \smallint_{0}^{ \tau_{ \theta } }
    \tfrac{
      U_1( X_s ) 
    }{
      2 q_1 e^{ \alpha T }
    }
    \,
    ds
  \right)
  \right\|_{
    L^{ 2 q_1 }( \Omega; \R )
  }
  \left\|
  \exp\!\left(
    \smallint_{0}^{ \tau_{ \theta } }
    \tfrac{
      U_1( Y^{ \theta }_s ) 
    }{
      2 q_1 e^{ \alpha T }
    }
    \,
    ds
  \right)
  \right\|_{
    L^{ 2 q_1 }( \Omega; \R )
  }
  .
\end{split}
\end{equation}
A simple consequence of Jensen's inequality 
(see, e.g., inequality~(19) in Li~\cite{Li1994} and 
Lemma~2.21 in Cox et al.~\cite{CoxHutzenthalerJentzen2013})
hence shows that
for all 
$ n \in \N $,
$ \theta = ( t_0 , t_1, \dots, t_n ) \in \mathcal{P}_T $
it holds that
\begin{equation}
\begin{split}
&
  \sup_{ t \in [0,T] }
  \big\| 
    X_{ t \wedge \tau_{ \theta } } - Y^{ \theta }_{ t \wedge \tau_{ \theta } }
  \big\|_{
    L^r( \Omega; \R^d )
  }
\leq
  360
  \, p^2
  \left[
    1 + \tfrac{ 1 }{ \varepsilon }
  \right]
  e^{ ( 2 + c ) T }
  \left[ 
    \max_{ 0 \leq k \leq n - 1 }
    | t_{ k + 1 } - t_k |
  \right]^{ \frac{ 1 }{ 2 } }
\\ & \cdot 
  \left[  
  \sup_{ s \in [0,T] }
    \big\|
      \phi(
        Y^{ \theta }_s
      )
    \big\|_{
      L^{ 2 p }( \Omega; \R )
    }
  \right]
  \sup_{ s \in [0,T] }
  \left|
  \E\!\left[
  \exp\!\left(
    \tfrac{
      U_0( X_s ) 
    }{
      e^{ \alpha s }
    }
  \right)
  \right]
  \right|^{
    1 / ( 2 q_0 )
  }
  \sup_{ s \in [0,T] }
  \left|
  \E\!\left[
  \exp\!\left(
    \tfrac{
      U_0( Y^{ \theta }_s )
    }{
      e^{ \alpha s }
    }
  \right)
  \right]
  \right|^{
    1 / ( 2 q_0 )
  }
\\ & \cdot
  \sup_{ t \in [0,T] }
  \left|
  \E\!\left[
  \exp\!\left(
    \tfrac{
      U_0( X_t ) 
    }{
      e^{ \alpha t }
    }
    +
    \smallint_{0}^{ t \wedge \tau_{ \theta } }
    \tfrac{
      U_1( X_s ) 
    }{
      e^{ \alpha s }
    }
    \,
    ds
  \right)
  \right]
  \right|^{
    \frac{ 1 }{ 2 q_1 }
  }
  \sup_{ t \in [0,T] }
  \left|
  \E\!\left[
  \exp\!\left(
    \tfrac{
      U_0( Y^{ \theta }_t ) 
    }{
      e^{ \alpha t }
    }
    +
    \smallint_{0}^{ t \wedge \tau_{ \theta } }
    \tfrac{
      U_1( Y^{ \theta }_s ) 
    }{
      e^{ \alpha s }
    }
    \,
    ds
  \right)
  \right]
  \right|^{
    \frac{ 1 }{ 2 q_1 }
  }
  .
\end{split}
\end{equation}
Therefore, we obtain that
for all 
$ n \in \N $,
$ \theta = ( t_0 , t_1, \dots, t_n ) \in \mathcal{P}_T $
it holds that
\begin{equation}
\begin{split}
&
  \sup_{ t \in [0,T] }
  \big\| 
    X_{ t \wedge \tau_{ \theta } } - Y^{ \theta }_{ t \wedge \tau_{ \theta } }
  \big\|_{
    L^r( \Omega; \R^d )
  }
\leq
  360
  \, p^2
  \left[
    1 + \tfrac{ 1 }{ \varepsilon }
  \right]
  e^{ ( 2 + c ) T }
  \left[ 
    \max_{ 0 \leq k \leq n - 1 }
    | t_{ k + 1 } - t_k |
  \right]^{ \frac{ 1 }{ 2 } }
\\ & \cdot 
  \left[  
  \sup_{ s \in [0,T] }
    \big\|
      \phi(
        Y^{ \theta }_s
      )
    \big\|_{
      L^{ 2 p }( \Omega; \R )
    }
  \right]
  \exp\!\left(
    \smallint_{0}^T
    \tfrac{
      c
    }{
      q_0 e^{ \alpha s }
    }
    \,
    ds
    +
    \smallint_{0}^T
    \tfrac{
      c
    }{
      q_1 e^{ \alpha s }
    }
    \,
    ds
  \right)
\\ & \cdot
  \sup_{ s \in [0,T] }
  \left|
  \E\!\left[
  \exp\!\left(
    \tfrac{
      U_0( X_s ) 
    }{
      e^{ \alpha s }
    }
    +
    \smallint_{0}^s
    \tfrac{
      U_1( X_s ) - c
    }{
      e^{ \alpha u }
    }
    \, du
  \right)
  \right]
  \right|^{
    \frac{ 1 }{ 2 q }
  }
  \sup_{ s \in [0,T] }
  \left|
  \E\!\left[
  \exp\!\left(
    \tfrac{
      U_0( Y^{ \theta }_s )
    }{
      e^{ \alpha s }
    }
    +
    \smallint_{0}^{ s \wedge \tau_{ \theta } }
    \tfrac{
      U_1( Y^{ \theta }_s ) - c
    }{
      e^{ \alpha u }
    }
    \, du
  \right)
  \right]
  \right|^{
    \frac{ 1 }{ 2 q }
  }
  .
\end{split}
\end{equation}
Combining this with 
Corollary~2.4 in
Cox et al.~\cite{CoxHutzenthalerJentzen2013}
and
Corollary~2.9 in
\cite{HutzenthalerJentzenWang2013}
implies that there exists a real number
$ \tilde{C} \in [0,\infty) $
such that 
for all $ n \in \N $,
$ \theta = ( t_0 , t_1, \dots, t_n ) \in \mathcal{P}_T $
it holds that
\begin{equation}
\begin{split}
&
  \sup_{ t \in [0,T] }
  \big\| 
    X_{ t \wedge \tau_{ \theta } } - Y^{ \theta }_{ t \wedge \tau_{ \theta } }
  \big\|_{
    L^r( \Omega; \R^d )
  }
\leq
  \tilde{C}
  \left[ 
    \max_{ 0 \leq k \leq n - 1 }
    | t_{ k + 1 } - t_k |
  \right]^{ \frac{ 1 }{ 2 } }
  .
\end{split}
\end{equation}
This and H\"{o}lder's inequality prove that
for all $ n \in \N $,
$ \theta = ( t_0 , t_1, \dots, t_n ) \in \mathcal{P}_T $
it holds that
\begin{equation}
\label{eq:proof_numeric_tau}
\begin{split}
&
  \sup_{ t \in [0,T] }
  \left\| 
    X_t - Y^{ \theta }_t
  \right\|_{
    L^r( \Omega; \R^d )
  }
\\ & \leq
  \sup_{ t \in [0,T] }
  \left\| 
    \mathbbm{1}_{
      \{ 
        \tau_{ \theta } < T
      \}
    }
    \left[ 
      X_t - Y^{ \theta }_t
    \right]
  \right\|_{
    L^r( \Omega; \R^d )
  }
  +
  \sup_{ t \in [0,T] }
  \left\| 
    \mathbbm{1}_{
      \{ 
        \tau_{ \theta } = T
      \}
    }
    \left[ 
      X_{ t \wedge \tau_{ \theta } } - Y^{ \theta }_{ t \wedge \tau_{ \theta } }
    \right]
  \right\|_{
    L^r( \Omega; \R^d )
  }
\\ & \leq
  \left\|
    \mathbbm{1}_{
      \{ 
        \tau_{ \theta } < T
      \}
    }
  \right\|_{
    L^{ 2 r }( \Omega; \R^d )
  }
  \left[
    \sup_{ t \in [0,T] }
    \left\| 
      X_t - Y^{ \theta }_t
    \right\|_{
      L^{ 2 r }( \Omega; \R^d ) 
    }
  \right]
  +
  \sup_{ t \in [0,T] }
  \left\| 
    X_{ t \wedge \tau_{ \theta } } - Y^{ \theta }_{ t \wedge \tau_{ \theta } }
  \right\|_{
    L^r( \Omega; \R^d )
  }
\\ & \leq
  \left|
    \P\!\left[
      \tau_{ \theta } < T
    \right]
  \right|^{
    \frac{ 1 }{ 2 r }
  }
  \sup_{ t \in [0,T] }
  \left[
    \left\| 
      X_t 
    \right\|_{
      L^{ 2 r }( \Omega; \R^d ) 
    }
    +
    \left\| 
      Y^{ \theta }_t
    \right\|_{
      L^{ 2 r }( \Omega; \R^d ) 
    }
  \right]
  +
  \tilde{C}
  \left[ 
    \max_{ 0 \leq k \leq n - 1 }
    | t_{ k + 1 } - t_k |
  \right]^{ \frac{ 1 }{ 2 } }
  .
\end{split}
\end{equation}
Next we observe that Markov's inequality together with 
the fact that for all $ x \in \R^d $ it holds that
$
  \tfrac{ 1 }{ c } \| x \|^{ 1 / c } \leq 1 + U_0( x )
$
shows that
for all $ n \in \N $,
$ \theta = ( t_0 , t_1, \dots, t_n ) \in \mathcal{P}_T $
it holds that
\begin{equation}
\begin{split}
  \P\!\left[ 
    \tau_{ \theta }
    < T
  \right]
& \leq
  \P\!\left[ 
    \| Y^{ \theta }_T \| 
    \geq
    \exp\!\left(
      \left| 
        \ln\!\left( 
          \max\nolimits_{ i \in \{ 0, 1, \dots, n - 1 \} }
          t_{ i + 1 } - t_i 
        \right) 
      \right|^{ 1 / 2 }
    \right)
  \right]
\\ & =
  \P\!\left[ 
    \tfrac{ 1 }{ c }
    \,
    \| Y^{ \theta }_T \|^{ 1 / c }
  \geq
    \tfrac{ 1 }{ c }
    \exp\!\left(
      \tfrac{
        \left| 
          \ln\left( 
            \max\nolimits_{ i \in \{ 0, 1, \dots, n - 1 \} }
            t_{ i + 1 } - t_i 
          \right) 
        \right|^{ 1 / 2 }
      }{
        c
      }
    \right)
  \right]
\\ & \leq
  \P\!\left[ 
    \frac{
      1 + U_0( Y^{ \theta }_T )
    }{
      e^{ \alpha T }
    }
    \geq
    \tfrac{ 1 }{ 
      c \, e^{ \alpha T }
    }
    \exp\!\left(
      \tfrac{
        \left| 
          \ln\left( 
            \max\nolimits_{ i \in \{ 0, 1, \dots, n - 1 \} }
            t_{ i + 1 } - t_i 
          \right) 
        \right|^{ 1 / 2 }
      }{
        c
      }
    \right)
  \right]
\\ & =
  \P\!\left[ 
    \exp\!\left(
    \frac{
      1 + U_0( Y^{ \theta }_T )
    }{
      e^{ \alpha T }
    }
    \right)
    \geq
    \exp\!\left(
      \tfrac{ 1 }{ 
        c \, e^{ \alpha T }
      }
      \exp\!\left(
        \tfrac{
          \left| 
            \ln\left( 
              \max\nolimits_{ i \in \{ 0, 1, \dots, n - 1 \} }
              t_{ i + 1 } - t_i 
            \right) 
          \right|^{ 1 / 2 }
        }{
          c
        }
      \right)
    \right)
  \right]
\\ & \leq
  \E\!\left[ 
    \exp\!\left(
    \frac{
      1 + U_0( Y^{ \theta }_T )
    }{
      e^{ \alpha T }
    }
    \right)
  \right]
    \exp\!\left(
    \tfrac{ - 1 }{ 
      c \, e^{ \alpha T }
    }
      \exp\!\left(
        \tfrac{
          \left| 
            \ln\left( 
              \max\nolimits_{ i \in \{ 0, 1, \dots, n - 1 \} }
              t_{ i + 1 } - t_i 
            \right) 
          \right|^{ 1 / 2 }
        }{
          c
        }
      \right)
    \right)
\end{split}
\end{equation}
The estimate $ \frac{ 1 }{ 4! } x^4 \leq e^x $ for all $ x \in [0,\infty) $ 
hence implies that for all $ n \in \N $,
$ \theta = ( t_0 , t_1, \dots, t_n ) \in \mathcal{P}_T $
it holds that
\begin{equation}
\begin{split}
&
  \P\!\left[ 
    \tau_{ \theta }
    < T
  \right]
\\ & \leq
  \E\!\left[ 
    \exp\!\left(
    \tfrac{
      U_0( Y^{ \theta }_T )
    }{
      e^{ \alpha T }
    }
    +
    \smallint_{0}^{ \tau_{ \theta } }
    \tfrac{
      U_1( Y^{ \theta }_s ) - c
    }{
      e^{ \alpha s }
    }
    \,
    ds
    \right)
  \right]
    \exp\!\left(
      c T
    +
      \tfrac{ 1 }{ e^{ \alpha T } }
    -
    \tfrac{ 1 }{ 
      c \, e^{ \alpha T }
    }
      \exp\!\left(
        \tfrac{
          \left| 
            \ln\left( 
              \max\nolimits_{ i \in \{ 0, 1, \dots, n - 1 \} }
              t_{ i + 1 } - t_i 
            \right) 
          \right|^{ 1 / 2 }
        }{
          c
        }
      \right)
    \right)
\\ & \leq
  \E\!\left[ 
    \exp\!\left(
    \tfrac{
      U_0( Y^{ \theta }_T )
    }{
      e^{ \alpha T }
    }
    +
    \smallint_{0}^{ \tau_{ \theta } }
    \tfrac{
      U_1( Y^{ \theta }_s ) - c
    }{
      e^{ \alpha s }
    }
    \,
    ds
    \right)
  \right]
  \exp\!\left(
      c T
    +
      \tfrac{ 1 }{ e^{ \alpha T } }
    -
    \tfrac{ 
          \left| 
            \ln\left( 
              \max\nolimits_{ i \in \{ 0, 1, \dots, n - 1 \} }
              t_{ i + 1 } - t_i 
            \right) 
          \right|^2
    }{ 
      24 \, c^5 \, e^{ \alpha T }
    }
  \right)
  .
\end{split}
\end{equation}
This together with 
Corollary~2.9 in
\cite{HutzenthalerJentzenWang2013}
shows that there exists a real number
$ \hat{C} \in [0,\infty) $
such that
for all $ n \in \N $,
$ \theta = ( t_0 , t_1, \dots, t_n ) \in \mathcal{P}_T $
it holds that
\begin{equation}
  \left|
    \P\!\left[
      \tau_{ \theta } < T
    \right]
  \right|^{
    \frac{ 1 }{ 2 r }
  }
\leq
  \hat{C}
  \left[ 
    \max_{ k \in \{ 0, 1, \dots, n - 1 \} }
    | t_{ k + 1 } - t_k |
  \right]^{ \frac{ 1 }{ 2 } }
  .
\end{equation}
Combining this with \eqref{eq:proof_numeric_tau} implies that
for all $ n \in \N $,
$ \theta = ( t_0 , t_1, \dots, t_n ) \in \mathcal{P}_T $
it holds that
\begin{equation}
\begin{split}
&
  \sup_{ t \in [0,T] }
  \left\| 
    X_t - Y^{ \theta }_t
  \right\|_{
    L^r( \Omega; \R^d )
  }
\\ & \leq
  \hat{C}
  \left[ 
    \max_{ 0 \leq k \leq n - 1 }
    | t_{ k + 1 } - t_k |
  \right]^{ \frac{ 1 }{ 2 } }
  \sup_{ t \in [0,T] }
  \left[
    \left\| 
      X_t 
    \right\|_{
      L^{ 2 r }( \Omega; \R^d ) 
    }
    +
    \left\| 
      Y^{ \theta }_t
    \right\|_{
      L^{ 2 r }( \Omega; \R^d ) 
    }
  \right]
  +
  \tilde{C}
  \left[ 
    \max_{ 0 \leq k \leq n - 1 }
    | t_{ k + 1 } - t_k |
  \right]^{ \frac{ 1 }{ 2 } }
  .
\end{split}
\end{equation}
Combining this again with 
Corollary~2.4 in
Cox et al.~\cite{CoxHutzenthalerJentzen2013}
and
Corollary~2.9 in
\cite{HutzenthalerJentzenWang2013}
completes the proof
of Proposition~\ref{prop:ST2}.
\end{proof}

Proposition~\ref{prop:ST2} establishes under suitable assumptions
strong convergence rates 
for the stopped-tamed Euler-Maruyama approximations 
in \cite{HutzenthalerJentzenWang2013}
in the case of SDEs with possibly non-globally Lipschitz continuous drift
and possibly non-globally Lipschitz continuous diffusion coefficient functions.
A number of SDEs from the literature have a globally Lipschitz continuous diffusion coefficient.
This special case of Proposition~\ref{prop:ST2} is the subject of the next result, Corollary~\ref{cor:ST3}.
Corollary~\ref{cor:ST3} follows immediately from Proposition~\ref{prop:ST2}.
% and its proof is therefore omitted.

\begin{corollary}
\label{cor:ST3}
Let $ d, m \in \N $, 
let $ \sigma \colon \R^d \to \R^{ d \times m } $
be globally Lipschitz continuous,
let
$ c, T \in (0,\infty) $, 
$
  U_0 \in \mathcal{C}^3_{ \mathcal{D} }( \R^d, [0,\infty) )
$,
$ 
  U_1 \in \mathcal{C}^1_{ \mathcal{P} }( \R^d, [0,\infty) ) 
$,
$ 
  \mu \in \mathcal{C}^1_{ \mathcal{P} }( \R^d , \R^d )
$
satisfy 
\begin{eqnarray}
%\begin{array}{c}
\label{eq:ST_ass_exp}
%   \exists \,
%   \alpha \in [0,\infty)
%   \colon
%   \;\;
  \lim_{
    \eta \to \infty
  }
  \sup_{ 
    x \in \R^d
  }
  \left[
  ( \mathcal{G}_{ \mu, \sigma } U_0 )( x )
  +
  \tfrac{ 1 }{ 2 } 
  \, \| \sigma( x )^* ( \nabla U_0 )( x ) \|^2_{ \R^d } 
  +
  U_1( x )
  -
  \eta \, U_0( x )
  \right] 
  < \infty
  ,
\\
\label{eq:ST_ass_rate}
  \forall \, \varepsilon \in (0,\infty)
  \colon
  \sup_{ 
    \substack{
      x, y \in \R^d , \,
%     \\
      x \neq y 
    }
  }
  \left[
  \tfrac{
    \left< 
      x - y , \mu( x ) - \mu( y ) 
    \right>_{ \R^d }
%     +
%     \frac{ ( p + \varepsilon - 1 ) }{ 2 }
%     \left\|
%       \sigma( x ) - \sigma( y )
%     \right\|^2_{ HS( \R^m , \R^d ) }
  }{
    \left\| x - y \right\|^2_{ \R^d }
  }
    -
    \varepsilon 
    \,
    \big( 
      U_0(x) + U_0(y) + U_1(x) + U_1(y) 
    \big)
  \right]
  < \infty 
\\[-3ex]
\nonumber
%\end{array}
\end{eqnarray}
and $ \sup_{ x \in \R^d } \left[ \| x \|^{ 1 / c } - c \, U_0( x ) \right] < \infty $,
let 
$ 
  ( 
    \Omega, \mathcal{F}, \P, 
    ( \mathcal{F}_t )_{ t \in [ 0 , T ] } 
  ) 
$
be a stochastic basis,
let
$
  W \colon [0,T] \times \Omega \to \R^m
$
be a standard 
$ ( \mathcal{F}_t )_{ t \in [0,T] } $-Brownian motion,
let
$ 
  X \colon [ 0, T ] \times \Omega \to \R^d 
$
and 
$ 
  Y^{ \theta } \colon [ 0 , T ] \times \Omega \to \R^d 
$,
$ \theta \in \mathcal{P}_T $,
be adapted stochastic processes with 
c.s.p.,
$
  \E\big[
    e^{ U_0( X_0 ) }
  \big]
  < \infty
$,
$
  X_t = 
  X_0 + \int_{0}^t \mu( X_s ) \, ds +
  \int_{0}^t \sigma( X_s ) \, dW_s
$
$ \P $-a.s.\ for all $ t \in [0,T] $
and
$ Y_0^{ \theta } = X_0 $
and
\begin{equation}
  Y_t^{ \theta }
  = 
  Y_{ t_k }^{ \theta }
  +
  \mathbbm{1}_{
    \left\{
      \| Y_{ t_k }^{ \theta } \|_{ \R^d }
      \,
      < 
      \,
      \exp(
        \,
        | 
          \!
          \ln( 
            \max_{ 0 \leq i \leq n - 1 }
            t_{ i + 1 } - t_i 
          ) 
        |^{ 1 / 2 }
        \,
      )
    \right\}
  }
  \left[
  \tfrac{
    \mu( Y_{ t_k }^{ \theta } ) \left( t - t_k \right)
    +
    \sigma( Y_{ t_k }^{ \theta } )
    ( W_t - W_{ t_k } )
  }{
    1 +
    \|
      \mu( Y_{ t_k }^{ \theta } ) \left( t - t_k \right)
      +
      \sigma( Y_{ t_k }^{ \theta } )
      ( W_t - W_{ t_k } )
    \|^2_{ \R^d }
  }
  \right]
\end{equation}
for all $ t \in [ t_k, t_{ k + 1 } ] $,
$ k \in \{ 0, 1, \dots, n - 1 \} $,
$ \theta = ( t_0 , \dots, t_n ) \in \mathcal{P}_T $,
$ n \in \N $.
Then there exist 
$ C_r \in [0,\infty) $, $ r \in (0,\infty) $,
such that 
for all 
$ r \in (0,\infty) $,
$ n \in \N $,
$ \theta = (  t_0 , t_1, \dots, t_n ) \in \mathcal{P}_T $
it holds that
% \begin{align}
% &
$
  \sup_{ t \in [0,T] }
  \| 
    X_t - Y^{ \theta }_t
  \|_{
    L^r( \Omega; \R^d )
  }
\leq
  C_r
  \,
  \big[
    \max\nolimits_{ 
      k \in \{ 0, 1, \dots, n - 1 \} 
    }
    \left| t_{ k + 1 } - t_k \right|
  \big]^{
    1 / 2
  }
$.
%   \,
%   .
% \end{align}
\end{corollary}

We now apply Corollary~\ref{cor:ST3} and Proposition~\ref{prop:ST2} respectively 
to a selection of example SODEs with non-globally monotone coefficients.
In each of these example SODEs,
the particular choice of the functions
of $ U_0 $ and $ U_1 $ 
in Corollary~\ref{cor:ST3}
and the estimates associated with them
are particularly inspired from 
the article Cox 
et al.~\cite{CoxHutzenthalerJentzen2013}
in which regularity with respect to the initial
value for these example SODEs has been analyzed.
The following common setting is used
in our investigations of the example SODEs.

\subsubsection{Setting}
\label{sec:ex_setting}

Throughout Subsection~\ref{sec:SODEs}
the following setting is frequently used.
Let $ d, m \in \N $, $ T \in (0,\infty) $, 
$ 
  \mu \in C( \R^d , \R^d )
$,
$
  \sigma \in C( \R^d , \R^{ d \times m } )
$,
$ x_0 \in \R^d $,
let 
$ 
  ( 
    \Omega, \mathcal{F}, \P, 
    ( \mathcal{F}_t )_{ t \in [ 0 , T ] } 
  ) 
$
be a stochastic basis,
let
$
  W \colon [0,T] \times \Omega \to \R^m
$
be a standard 
$ ( \mathcal{F}_t )_{ t \in [0,T] } $-Brownian motion,
let
$ 
  X \colon [ 0, T ] \times \Omega \to \R^d
$
and 
$ 
  Y^{ \theta } \colon [ 0 , T ] \times \Omega \to \R^d 
$,
$ \theta \in \mathcal{P}_T $,
be adapted stochastic processes with 
c.s.p.,
$
  X_t = 
  x_0 + \int_{0}^t \mu( X_s ) \, ds +
  \int_{0}^t \sigma( X_s ) \, dW_s
$
$ \P $-a.s.\ for all $ t \in [0,T] $
and
$ Y_0^{ \theta } = X_0 $
and
\begin{equation}
  Y_t^{ \theta }
  = 
  Y_{ t_k }^{ \theta }
  +
  \mathbbm{1}_{
    \left\{
      \| Y_{ t_k }^{ \theta } \|_{ \R^d }
      \,
      < 
      \,
      \exp(
        \,
        | 
          \!
          \ln( 
            \max_{ 0 \leq i \leq n - 1 }
            t_{ i + 1 } - t_i 
          ) 
        |^{ 1 / 2 }
        \,
      )
    \right\}
  }
  \left[
  \tfrac{
    \mu( Y_{ t_k }^{ \theta } ) \left( t - t_k \right)
    +
    \sigma( Y_{ t_k }^{ \theta } )
    ( W_t - W_{ t_k } )
  }{
    1 +
    \|
      \mu( Y_{ t_k }^{ \theta } ) \left( t - t_k \right)
      +
      \sigma( Y_{ t_k }^{ \theta } )
      ( W_t - W_{ t_k } )
    \|^2_{ \R^d }
  }
  \right]
\end{equation}
for all $ t \in [ t_k, t_{ k + 1 } ] $,
$ k \in \{ 0, 1, \dots, n - 1 \} $,
$ \theta = ( t_0 , \dots, t_n ) \in \mathcal{P}_T $,
$ n \in \N $.

\subsubsection{Stochastic Lorenz equation with bounded noise}
\label{ssec:Lorenz}

In this subsection assume the setting in
Subsection~\ref{sec:ex_setting},
let
$
  \alpha_1, \alpha_2, \alpha_3 \in [0,\infty)
$
and
assume that
$ d = m = 3 $,
that $ \sigma $ is globally bounded and globally Lipschitz continuous
(e.g., $ \sigma( x ) = \sqrt{ \beta } I_{ \R^3 } $ for all $ x \in \R^3 $
and some $ \beta \in (0,\infty) $; see, for example, Zhou \& E~\cite{ZhouE2010})
and that
$ 
  \mu( x_1, x_2, x_3 ) 
  = 
  \big(
     \alpha_1 \left( x_2 - x_1 \right)
     ,
      \alpha_2 x_1 - x_2 - x_1 x_3 
      ,
      x_1 x_2 - \alpha_3 x_3
  \big)
$ 
for all $ x = ( x_1, x_2, x_3 ) \in \R^3 $.
If $ U_0 \in C( \R^3, [0,\infty) ) $
is given by $ U_0(x) = \| x \|^2_{ \R^3 } $
for all $ x \in \R^3 $,
then 
\begin{equation}
\begin{split}
&
  \lim_{
    \eta \to \infty
  }
  \sup_{ 
    x \in \R^3
  }
  \left[
  ( \mathcal{G}_{ \mu, \sigma } U_0 )( x )
  +
  \tfrac{ 1 }{ 2 } 
  \, \| \sigma( x )^* ( \nabla U_0 )( x ) \|^2_{ \R^3 } 
  -
  \eta \, U_0( x )
  \right] 
\\ & \leq
  \lim_{
    \eta \to \infty
  }
  \sup_{ 
    x \in \R^3
  }
  \left[
  2
  \left< x, \mu(x) \right>_{ \R^3 }
  +
  \big[ 
    2 \,
    \| \sigma(x) \|^2_{ HS( \R^3 ) }
    -
    \eta 
  \big]
  \,
  U_0( x )
  +
  \left\| \sigma(x) \right\|^2_{ HS( \R^3 ) }
  \right] 
  < \infty
\end{split}
\end{equation}
and this proves that \eqref{eq:ST_ass_exp}
is fulfilled.
Moreover, note that for all $ \varepsilon \in (0,\infty) $ it holds that
\begin{equation}
\begin{split}
&
  \sup_{ 
    \substack{
      x, y \in \R^3 , \,
      x \neq y 
    }
  }
  \left[
  \tfrac{
    \left< 
      x - y , \mu( x ) - \mu( y ) 
    \right>_{ \R^3 }
  }{
    \left\| x - y \right\|^2_{ \R^3 }
  }
    -
    \varepsilon 
    \left( 
      \left\| x \right\|^2_{ \R^3 } + 
      \left\| y \right\|^2_{ \R^3 } 
    \right)
  \right]
\\ & 
\leq
  \sup_{ 
    \substack{
      x, y \in \R^3 , \,
      x \neq y 
    }
  }
  \left[
  \tfrac{
    \left\| \mu( x ) - \mu( y ) \right\|_{ \R^3 }
  }{
    \left\| x - y \right\|_{ \R^3 }
  }
    -
    \varepsilon 
    \left( 
      \left\| x \right\|^2_{ \R^3 } + 
      \left\| y \right\|^2_{ \R^3 } 
    \right)
  \right]
  < \infty 
\end{split}
\end{equation}
and this shows that 
\eqref{eq:ST_ass_rate}
is satisfied.
We can thus apply Corollary~\ref{cor:ST3}
to obtain that there exist 
$ C_r \in [0,\infty) $, $ r \in (0,\infty) $,
such that 
for all 
$ r \in (0,\infty) $,
$ n \in \N $,
$ \theta = (  t_0 , t_1, \dots, t_n ) \in \mathcal{P}_T $
it holds that
$
  \sup_{ t \in [0,T] }
  \| 
    X_t - Y^{ \theta }_t
  \|_{
    L^r( \Omega; \R^d )
  }
\leq
  C_r
  \,
  \big[
    \max\nolimits_{ 
      k \in \{ 0, 1, \dots, n - 1 \} 
    }
    \left| t_{ k + 1 } - t_k \right|
  \big]^{
    1 / 2
  }
$.

\subsubsection{Stochastic van der Pol oscillator}
\label{ssec:van.der.Pol}

In this subsection assume the setting in
Subsection~\ref{sec:ex_setting},
let
$
  c, \alpha \in ( 0, \infty )
$,
$
  \gamma, \delta 
  \in [0,\infty)
$,
let
$ g \colon \R \to \R^{ 1 \times m } $
be a globally Lipschtiz continuous function
with
$
  \| g(x)^* \|_{ \R^m }^2
  \leq 
  c 
  \left(
    1 
    +
    x^2
  \right)
$
for all $ x \in \R $
and assume that
$ d = 2 $,
$
  \mu( x )
=
  \left(
    x_2 ,
    \left( \gamma - \alpha ( x_1 )^2 \right)
    x_2
    - \delta x_1
  \right)
$
and
$
  \sigma( x ) u
=
  \left(
    0 ,
    g( x_1 ) u
  \right)
$
for all
$
  x = (x_1, x_2)
  \in \mathbb{R}^2
$,
$ u \in \R^m $.
If 
$ \vartheta \in (0, \frac{ \alpha }{ 2 c } ) $
and if
$ U_0, U_1 \in C( \R^2, [0,\infty) ) $
are given by $ U_0(x) = \frac{ \vartheta }{ 2 } \| x \|^2_{ \R^2 } $
and
$
  U_1( x )
  =
  \vartheta 
  \left[ \alpha - 2 c \vartheta \right] \left( x_1 x_2 \right)^2
$
for all $ x \in \R^2 $,
then it holds that
{\small
\begin{align}
&
  \lim_{
    \eta \to \infty
  }
  \sup_{ 
    x \in \R^2
  }
  \left[
  ( \mathcal{G}_{ \mu, \sigma } U_0 )( x )
  +
  \tfrac{ 1 }{ 2 } 
  \, \| \sigma( x )^* ( \nabla U_0 )( x ) \|^2_{ \R^2 } 
  +
  U_1( x )
  -
  \eta \, U_0( x )
  \right] 
\nonumber
\\ & =
\nonumber
  \vartheta
  \lim_{
    \eta \to \infty
  }
  \sup_{ 
    \substack{
      x = 
      \\
      ( x_1, x_2 ) 
      \\
      \in \R^2
    }
  }
  \left[
    \left( 1 - \delta \right) x_1 x_2
    + \gamma \left( x_2 \right)^2
    - \alpha \left( x_1 x_2 \right)^2
  +
  \tfrac{ 
    \left\| \sigma(x) \right\|^2_{ HS( \R^2 ) }
  }{ 2 }
  +
  2 \vartheta
    \left\| \sigma(x)^* x \right\|^2_{ \R^m }
  +
  \tfrac{
    U_1( x )
  }{ \vartheta }
  - 
  \tfrac{
    \eta
    \left\| x \right\|^2_{ \R^2 }
  }{
    2 
  }
  \right] 
\\ & \leq
  \vartheta
  \lim_{
    \eta \to \infty
  }
  \sup_{ 
    \substack{
      x = 
      \\
      ( x_1, x_2 ) 
      \\
      \in \R^2
    }
  }
  \left[
    \left\| g(x_1)^* \right\|^2_{ \R^m }
    +
    \left( 1 + \gamma + \delta - \tfrac{ \eta }{ 2 } 
    \right) \left\| x \right\|^2_{ \R^2 }
  +
  2 \vartheta \left| x_2 \right|^2
    \left\| g(x_1)^* \right\|^2_{ \R^m }
  +
  \tfrac{
    U_1( x )
  }{ \vartheta }
    - \alpha \left( x_1 x_2 \right)^2
  \right] 
\nonumber
\\ & \leq
%   \vartheta
%   \lim_{
%     \eta \to \infty
%   }
%   \sup_{ 
%     \substack{
%       x = ( x_1, x_2 ) \in \R^2
%     }
%   }
%   \left[
%     c
%     +
%     \left( 1 + \gamma + \delta + c + 2 \vartheta c - \eta \right) \left\| x \right\|^2_{ \R^2 }
%     +
%     \left[ 2 c \vartheta - \alpha \right] \left( x_1 x_2 \right)^2
%   +
%   U_1( x )
%   \right] 
% \\ & =
  \vartheta
  \lim_{
    \eta \to \infty
  }
  \sup_{ 
    \substack{
      x \in \R^2
    }
  }
  \left[
    c
    +
    \left( 1 + \gamma + \delta + c + 2 \vartheta c - \eta 
    \right) \left\| x \right\|^2_{ \R^2 }
  \right] 
  < \infty
\end{align}}and 
% this proves that \eqref{eq:ST_ass_exp}
% is fulfilled.
% Moreover, note that 
Subsection~4.1 in Cox et al.~\cite{CoxHutzenthalerJentzen2013}
ensures that
for all $ \varepsilon \in (0,\infty) $ it holds that
\begin{equation}
\begin{split}
&
  \sup_{ 
    \substack{
      x = ( x_1, x_2 ), 
    \\
      y = ( y_1, y_2 ) 
      \in \R^2 , 
    \\
      x \neq y 
    }
  }
  \left[
  \tfrac{
    \left< 
      x - y , \mu( x ) - \mu( y ) 
    \right>_{ \R^2 }
  }{
    \left\| x - y \right\|^2_{ \R^2 }
  }
    -
    \varepsilon 
    \left( 
      U_1(x)
      +
      U_1(y)
    \right)
  \right]
\\ & 
\leq
  \sup_{ 
    \substack{
      x = ( x_1, x_2 ), 
    \\
      y = ( y_1, y_2 ) 
      \in \R^2 , 
    \\
      x \neq y 
    }
  }
  \left[
  \tfrac{
    \left< 
      x - y , \mu( x ) - \mu( y ) 
    \right>_{ \R^2 }
  }{
    \left\| x - y \right\|^2_{ \R^2 }
  }
    -
    \varepsilon 
    \left[ \alpha - 2 c \vartheta \right]
    \left( 
      \left( x_1 x_2 \right)^2
      +
      \left( y_1 y_2 \right)^2
    \right)
  \right]
  < \infty 
  .
\end{split}
\end{equation}
This shows that 
\eqref{eq:ST_ass_exp}
and
\eqref{eq:ST_ass_rate}
are satisfied.
We can thus apply Corollary~\ref{cor:ST3}
to obtain that there exist 
$ C_r \in [0,\infty) $, $ r \in (0,\infty) $,
such that 
for all 
$ r \in (0,\infty) $,
$ n \in \N $,
$ \theta = (  t_0 , t_1, \dots, t_n ) \in \mathcal{P}_T $
it holds that
$
  \sup_{ t \in [0,T] }
  \| 
    X_t - Y^{ \theta }_t
  \|_{
    L^r( \Omega; \R^d )
  }
\leq
  C_r
  \,
  \big[
    \max\nolimits_{ 
      k \in \{ 0, 1, \dots, n - 1 \} 
    }
    \left| t_{ k + 1 } - t_k \right|
  \big]^{
    1 / 2
  }
$.

\subsubsection{Stochastic Duffing-van
der Pol oscillator}
\label{ssec:stochastic.Duffing.van.der.Pol.oscillator}

In this subsection assume the setting in
Subsection~\ref{sec:ex_setting},
let 
$ \alpha_1, \alpha_2 \in \R $,
$
  \alpha_3, c
  \in (0,\infty)
$,
let $ g \colon \R \to \R^{ 1 \times m } $
be a globally Lipschitz continuous function
with
$
  \| g(x)^* \|^2_{ \R^m }
  \leq
  c \left( 1 + x^2 \right)
$
for all $ x \in \R $
(a common choice for the function $ g $ in the 
stochastic Duffing-van der Pol oscillator is
$
  g(x) u = \beta_1 x u_1 + \beta_2 u_2
$
for all $ x \in \R $,
$ u = ( u_1, u_2 ) \in \R^2 $
and some $ \beta_1, \beta_2 \in \R $;
see, e.g., Schenk-Hopp\'{e}~\cite{SchenkHoppe1996Deterministic})
and assume that
$ d = 2 $,
$ D = \R^2 $,
$
  \mu( x )
=
  \left(
    x_2 ,
    \alpha_2 x_2 - \alpha_1 x_1
    - \alpha_3 ( x_1 )^2 x_2
    - ( x_1 )^3
  \right)
$
and
$
  \sigma( x ) u
=
  \left(
    0 ,
    g( x_1 ) u
  \right)
$
for all
$
  x = (x_1, x_2)
  \in \mathbb{R}^2
$,
$ u \in \R^m $.
If 
$ \vartheta \in ( 0, \frac{ \alpha_3 }{ c } ) 
$
and if
$
  U_0, U_1 \in C( \R^2, [0,\infty) )
$
are given by
$
  U_0( x )
=
  \frac{ \vartheta }{ 2 }
  \big[
    \tfrac{ \left( x_1 \right)^4 }{ 2 }
%     +
%     \alpha_1
%     \left( x_1 \right)^2
    +
    \left( x_2 \right)^2
  \big]
$
and
$
  U_1( x )
  =
  \vartheta 
  \left[ \alpha_3 - c \vartheta \right]
  \left( x_1 x_2 \right)^2
$
for all
$ x = (x_1, x_2) \in \R^2 $,
then it holds that
\begin{equation}
\begin{split}
&
  \lim_{
    \eta \to \infty
  }
  \sup_{ 
    x \in \R^2
  }
  \left[
  ( \mathcal{G}_{ \mu, \sigma } U_0 )( x )
  +
  \tfrac{ 1 }{ 2 } 
  \, \| \sigma( x )^* ( \nabla U_0 )( x ) \|^2_{ \R^2 } 
  +
  U_1( x )
  -
  \eta \, U_0( x )
  \right] 
\\ & =
  \vartheta
  \lim_{
    \eta \to \infty
  }
  \sup_{ 
    \substack{
      x = ( x_1, x_2 ) 
    \\
      \in \R^2
    }
  }
  \left[
  \alpha_2 \left( x_2 \right)^2
  -
  \alpha_1 x_1 x_2
  -
  \alpha_3 \left( x_1 x_2 \right)^2
  +
  \tfrac{ 
    [ 1 + \vartheta ( x_2 )^2 ] 
    \, 
    \| g( x_1 )^* \|^2_{ \R^m } 
  }{ 2 }
  +
  \tfrac{
    U_1( x )
  }{
    \vartheta
  }
  -
  \tfrac{
    \eta \, U_0( x )
  }{ \vartheta }
  \right] 
\\ & \leq
  \vartheta
  \lim_{
    \eta \to \infty
  }
  \sup_{ 
    \substack{
      x = ( x_1, x_2 ) 
    \\
      \in \R^2
    }
  }
  \left[
    \left[ 
      \left| \alpha_1 \right|
      +
      \left| \alpha_2 \right|
    \right] 
    \left\| x \right\|^2_{ \R^2 }
  -
  \alpha_3 \left( x_1 x_2 \right)^2
  +
  \tfrac{ 
    c \,
    [ 1 + \vartheta ( x_2 )^2 ] 
    \, 
    [ 1 + ( x_1 )^2 ]
  }{ 2 }
  +
  \tfrac{
    U_1( x )
  }{
    \vartheta
  }
  -
  \tfrac{
    \eta \, U_0( x )
  }{ \vartheta }
  \right] 
\\ & \leq
  \vartheta
  \lim_{
    \eta \to \infty
  }
  \sup_{ 
    \substack{
      x = ( x_1, x_2 ) 
    \\
      \in \R^2
    }
  }
  \left[
  \tfrac{ c }{ 2 }
  +
  \left[ 
      \left| \alpha_1 \right|
      +
      \left| \alpha_2 \right|
      + \tfrac{ c ( 1 + \vartheta ) }{ 2 } 
  \right] 
  \left\| x \right\|^2_{ \R^2 }
  +
  \left[ 
    c \vartheta
    -
    \alpha_3 
  \right] \left( x_1 x_2 \right)^2
  +
  \tfrac{
    U_1( x )
  }{
    \vartheta
  }
  -
  \tfrac{
    \eta \, U_0( x )
  }{ \vartheta }
  \right] 
\\ & =
  \vartheta
  \lim_{
    \eta \to \infty
  }
  \sup_{ 
    x \in \R^2
  }
  \left[
  \tfrac{ c }{ 2 }
  +
  \left[ 
      \left| \alpha_1 \right|
      +
      \left| \alpha_2 \right|
      + 
      \tfrac{ c ( 1 + \vartheta ) }{ 2 } 
  \right] 
  \left\| x \right\|^2_{ \R^2 }
  -
  \tfrac{
    \eta \, U_0( x )
  }{ \vartheta }
  \right] 
  < \infty
\end{split}
\end{equation}
and it holds 
for all $ \varepsilon \in ( 0, \infty ) $ that
\begin{equation}
\begin{split}
&
  \sup_{ 
    \substack{
      x, y \in \R^2 , \,
      x \neq y 
    }
  }
  \left[
  \tfrac{
    \left< 
      x - y , \mu( x ) - \mu( y ) 
    \right>_{ \R^2 }
  }{
    \left\| x - y \right\|^2_{ \R^2 }
  }
    -
    \varepsilon 
    \,
    \big( 
      U_0(x) + U_0(y) 
    \big)
  \right]
\\ & \leq
  \sup_{ 
    \substack{
      x, y \in \R^2 , \,
      x \neq y 
    }
  }
  \left[
  \tfrac{
    \left\|
      \mu( x ) - \mu( y ) 
    \right\|_{ \R^2 }
  }{
    \left\| x - y \right\|_{ \R^2 }
  }
    -
    \varepsilon 
    \,
    \big( 
      U_0(x) + U_0(y) 
    \big)
  \right]
  < \infty 
  .
\end{split}
\end{equation}
This proves that \eqref{eq:ST_ass_exp} and \eqref{eq:ST_ass_rate}
are fulfilled.
We can thus apply Corollary~\ref{cor:ST3}
to obtain that there exist 
$ C_r \in [0,\infty) $, $ r \in (0,\infty) $,
such that 
for all 
$ r \in (0,\infty) $,
$ n \in \N $,
$ \theta = (  t_0 , t_1, \dots, t_n ) \in \mathcal{P}_T $
it holds that
$
  \sup_{ t \in [0,T] }
  \| 
    X_t - Y^{ \theta }_t
  \|_{
    L^r( \Omega; \R^d )
  }
\leq
  C_r
  \,
  \big[
    \max\nolimits_{ 
      k \in \{ 0, 1, \dots, n - 1 \} 
    }
    \left| t_{ k + 1 } - t_k \right|
  \big]^{
    1 / 2
  }
$.

\subsubsection{Experimental psychology model}
\label{ssec:experimental.psychology}

In this subsection assume the setting in
Subsection~\ref{sec:ex_setting},
let
$ \alpha, \delta \in (0,\infty) $,
$ \beta \in \R $
and assume that
$ d = 2 $, $ m = 1 $,
$ D = \R^2 $,
$
  \mu( x_1, x_2 )
=
  \big(
      ( x_2 )^2
      ( \delta + 4 \alpha x_1 )
      - \frac{ 1 }{ 2 } \beta^2 x_1
    ,
      - x_1 x_2
      ( \delta + 4 \alpha x_1 )
      - \frac{ 1 }{ 2 } \beta^2 x_2
  \big)
$
and
$
  \sigma( x_1, x_2 )
  =
  ( - \beta x_2 , \beta x_1 )
$
for all
$ x = ( x_1 , x_2 ) \in \R^2 $.
If
$ q \in [3,\infty) $
and if
$ U_0 \in C( \R^2, \R ) $
is given by
$
  U_0( x )
  =
  \| x \|^q
$
for all
$
  x \in \R^2
$,
then it holds that 
\begin{equation}
\begin{split}
&
  \lim_{
    \eta \to \infty
  }
  \sup_{ 
    x \in \R^d
  }
  \left[
  ( \mathcal{G}_{ \mu, \sigma } U_0 )( x )
  +
  \tfrac{ 1 }{ 2 } 
  \, \| \sigma( x )^* ( \nabla U_0 )( x ) \|^2_{ \R^d } 
  -
  \eta \, U_0( x )
  \right] 
\\ & \leq
  \lim_{
    \eta \to \infty
  }
  \sup_{ 
    x \in \R^2
  }
  \left[
    q 
    \left\| x \right\|^{ ( q - 2 ) }_{ \R^2 }
    \left< 
      x, \mu( x )
    \right>_{ \R^2 }
    +
    \tfrac{ q \left( q - 1 \right) }{ 2 } 
    \left\| x \right\|^{ ( q - 2 ) }_{ \R^2 }
    \left\| \sigma( x ) \right\|^2_{ HS( \R^2 ) }
    -
    \eta \left\| x \right\|^q
  \right] 
  < \infty
\end{split}
\end{equation}
and for all $ \varepsilon \in ( 0, \infty ) $
it holds that
\begin{equation}
\begin{split}
&
  \sup_{ 
    \substack{
      x, y \in \R^2 , \,
      x \neq y 
    }
  }
  \left[
  \tfrac{
    \left< 
      x - y , \mu( x ) - \mu( y ) 
    \right>_{ \R^2 }
  }{
    \left\| x - y \right\|^2_{ \R^2 }
  }
    -
    \varepsilon 
    \,
    \big(
      \| x \|_{ \R^2 }^q
      +
      \| y \|_{ \R^2 }^q
    \big)
  \right]
\\ & \leq
  \sup_{ 
    \substack{
      x, y \in \R^2 , \,
      x \neq y 
    }
  }
  \left[
  \tfrac{
    \left\|
      \mu( x ) - \mu( y ) 
    \right\|_{ \R^2 }
  }{
    \left\| x - y \right\|_{ \R^2 }
  }
    -
    \varepsilon 
    \,
    \big(
      \| x \|_{ \R^2 }^q
      +
      \| y \|_{ \R^2 }^q
    \big)
  \right]
  < \infty 
  .
\end{split}
\end{equation}
This proves that \eqref{eq:ST_ass_exp} and \eqref{eq:ST_ass_rate}
are fulfilled.
We can thus apply Corollary~\ref{cor:ST3}
to obtain that there exist 
$ C_r \in [0,\infty) $, $ r \in (0,\infty) $,
such that 
for all 
$ r \in (0,\infty) $,
$ n \in \N $,
$ \theta = (  t_0 , t_1, \dots, t_n ) \in \mathcal{P}_T $
it holds that
$
  \sup_{ t \in [0,T] }
  \| 
    X_t - Y^{ \theta }_t
  \|_{
    L^r( \Omega; \R^d )
  }
\leq
  C_r
  \,
  \big[
    \max\nolimits_{ 
      k \in \{ 0, 1, \dots, n - 1 \} 
    }
    \left| t_{ k + 1 } - t_k \right|
  \big]^{
    1 / 2
  }
$.

\subsubsection{Brownian dynamics (Overdamped Langevin dynamics)}
\label{sec:overdamped_Langevin}
\label{ssec:overdamped.Langevin.dynamics}

In this subsection assume the setting in
Subsection~\ref{sec:ex_setting},
let
$ c, \beta \in (0,\infty) $,
$ \theta \in [ 0, \frac{ 2 }{ \beta } ) $,
$
  V \in \mathcal{C}^3_{ \mathcal{D} }( \R^d , [0,\infty) )
$
and assume that
$
  \limsup_{ r \searrow 0 }
  \sup_{ z \in \R^d }
  \frac{ \| z \|^r_{ \R^d } }{ 1 + V(z) }
  < \infty
$,
$ d = m $,
$
  \mu( x ) = - ( \nabla V )( x )
$,
$
  \sigma( x ) =
  \sqrt{ \beta } I_{ \R^d }
$,
$
  ( \triangle V)( x )
\leq
  c
  +
  c \,
  V(x)
  +
  \theta
  \left\|
    ( \nabla V )( x )
  \right\|^2_{ \R^d }
$
for all $ x \in \R^d $
and 
that for all $ \varepsilon \in (0,\infty) $
it holds that
\begin{equation}
  \sup_{ 
    \substack{
      x, y \in \R^d , 
    \\
      x \neq y 
    }
  }
  \left[
  \tfrac{
    \left< 
      x - y , ( \nabla V )( y ) - ( \nabla V )( x ) 
    \right>_{ \R^d }
  }{
    \left\| x - y \right\|^2_{ \R^d }
  }
    -
    \varepsilon 
    \left( 
      V(x) + V(y) 
      + 
      \left\| ( \nabla V )(x) \right\|^2_{ \R^d }  
      + 
      \left\| ( \nabla V )(y) \right\|^2_{ \R^d }  
    \right)
  \right]
  < \infty 
  .
\end{equation}
If $ \vartheta \in (0, \frac{ 2 }{ \beta } - \theta ) $
and if 
$ U_0, U_1 \in C( \R^d, \R ) $
are given by
$
  U_0(x)
  =
  \vartheta \, V(x)
$
and
$
  U_1(x)
  =
  \vartheta
  \,
  (
    1
    -
    \frac{ \beta }{ 2 }
    ( \theta + \vartheta )
  )
  \,
  \|
    ( \nabla V )( x )
  \|^2_{ \R^d }
$
for all $ x \in \R^d $,
then 
\begin{equation}
\begin{split}
&
  \lim_{
    \eta \to \infty
  }
  \sup_{ 
    x \in \R^d
  }
  \left[
  ( \mathcal{G}_{ \mu, \sigma } U_0 )( x )
  +
  \tfrac{ 1 }{ 2 } 
  \, \| \sigma( x )^* ( \nabla U_0 )( x ) \|^2_{ \R^d } 
  +
  U_1( x )
  -
  \eta \, U_0( x )
  \right] 
\\ & =
  \vartheta
  \lim_{
    \eta \to \infty
  }
  \sup_{ 
    x \in \R^d
  }
  \left[
    - \left\| ( \nabla V )( x ) \right\|^2_{ \R^d }
    +
    \tfrac{ \beta }{ 2 }
    \,
    ( \triangle V )( x )
    +
    \tfrac{ \vartheta \beta }{ 2 } 
    \, \| ( \nabla V )( x ) \|^2_{ \R^d } 
    +
    \tfrac{ 
      U_1( x )
    }{
      \vartheta
    }
    -
    \eta \, V( x )
  \right] 
\\ & \leq
  \vartheta
  \lim_{
    \eta \to \infty
  }
  \sup_{ 
    x \in \R^d
  }
  \left[
    \tfrac{ c \beta }{ 2 }
    +
    \left[
      \tfrac{ ( \theta + \vartheta ) \beta }{ 2 }
      - 
      1
    \right] 
    \left\| ( \nabla V )( x ) \right\|^2_{ \R^d }
    +
    \tfrac{ 
      U_1( x )
    }{
      \vartheta
    }
    +
    \left[
      \tfrac{ c \beta }{ 2 }
      -
      \eta 
    \right]  
    V( x )
  \right] 
\\ & =
  \vartheta
  \lim_{
    \eta \to \infty
  }
  \sup_{ 
    x \in \R^d
  }
  \left[
    \tfrac{ c \beta }{ 2 }
    +
    \left[
      \tfrac{ c \beta }{ 2 }
      -
      \eta 
    \right]  
    V( x )
  \right] 
  < \infty
  .
\end{split}
\end{equation}
We can thus apply Corollary~\ref{cor:ST3}
to obtain that there exist 
$ C_r \in [0,\infty) $, $ r \in (0,\infty) $,
such that 
for all 
$ r \in (0,\infty) $,
$ n \in \N $,
$ \theta = (  t_0 , t_1, \dots, t_n ) \in \mathcal{P}_T $
it holds that
$
  \sup_{ t \in [0,T] }
  \| 
    X_t - Y^{ \theta }_t
  \|_{
    L^r( \Omega; \R^d )
  }
\leq
  C_r
  \big[
    \max\nolimits_{ 
      k \in \{ 0, 1, \dots, n - 1 \} 
    }
    \left| t_{ k + 1 } - t_k \right|
  \big]^{
    1 / 2
  }
$.
\begin{remark}[Higher order strong convergence rates for SDEs with possibly non-globally monotone coefficients]
\label{remark:higher_order}
Corollary~\ref{cor:ST3} applies both to SDEs with additive and non-additive noise
%and possibly non-globally monotone coefficients 
and establishes 
the strong convergence rate $ \frac{ 1 }{ 2 } $.
We expect that, in the case of SDEs with additive noise 
(see, e.g., Subsections~\ref{sec:overdamped_Langevin} and \ref{ssec:Langevin.dynamics})
and possibly non-globally monotone coefficients, 
an application of the perturbation theory in Section~\ref{sec:perturbation_theory}
(to be more specific, an application of Proposition~\ref{prop:norm2})
yields the strong convergence rate $ 1 $.
Similarly, 
we expect that 
Proposition~\ref{prop:norm2} 
can be used to establish higher order
strong convergence rates for suitable higher order schemes 
in the case of SDEs with possibly non-globally monotone coefficients.
\end{remark}

\subsubsection{Langevin dynamics and stochastic Duffing oscillator}
\label{ssec:Langevin.dynamics}

In this subsection assume the setting in
Subsection~\ref{sec:ex_setting},
let
$ \gamma \in [0,\infty) $,
$ \beta \in (0,\infty) $,
$ V \in \mathcal{C}_{ \mathcal{D} }^3( \R^m, [0,\infty) ) $
and assume that
$
  \limsup_{ r \searrow 0 }
  \sup_{ z \in \R^m }
  \frac{
    \| z \|^r
  }{
    1 + V( z )
  }
  < \infty
$,
$ d = 2 m $,
$
  \mu( x )
  =
  ( x_2,
  - ( \nabla V )( x_1 ) 
$
$  
  - \gamma x_2 )
$,
$
  \sigma( x ) u
  = ( 0, \sqrt{ \beta } u )
$
for all $ x = ( x_1, x_2 ) \in \R^{ 2 m } $, $ u \in \R^m $
and
that for all $ \varepsilon \in (0,\infty) $
it holds that
\begin{equation}
\label{eq:Langevin_ST_ass_rate}
  \sup_{ 
    \substack{
      x, y \in \R^m , \,
%     \\
      x \neq y 
    }
  }
  \left[
  \tfrac{
    \left\|
      ( \nabla V )( x ) - ( \nabla V )( y )
    \right\|_{ \R^m }
  }{
    \left\| x - y \right\|_{ \R^m }
  }
    -
    \varepsilon 
    \left(
      \left\| x \right\|^2_{ \R^m }
      +
      \left\| y \right\|^2_{ \R^m }
      +
      V(x)
      +
      V(y)
    \right)
  \right]
  < \infty 
  .
\end{equation}
These assumptions are, for example, satisfied in the case of 
\emph{stochastic
Duffing oscillator with additive noise}
(see, e.g., equation~(9) in Datta \& Bhattacharjee~\cite{DattaBhattacharjee2001})
in which $ m = 1 $ and in which $ V $ fulfills $ V( x ) = \frac{ 1 }{ 2 } x^2 + \frac{ \lambda }{ 4 } x^4 $
for all $ x \in \R $
and some $ \lambda \in (0,\infty) $.
If
$
  \vartheta \in
  ( 0, \infty )
$
and if
$
  U_0 \in C( \R^{ 2 m }, \R )
$
is given by
$
  U_0( x ) = \frac{ \vartheta }{ 2 } \, \| x_1 \|^2_{ \R^m } + \vartheta \, V( x_1 ) + \frac{ \vartheta }{ 2 } \, \| x_2 \|^2_{ \R^m }
$
for all $ x = ( x_1, x_2 ) \in \R^{ 2 m } $,
then 
\begin{equation}
\label{eq:Langevin_ST_ass_exp}
\begin{split}
&
  \lim_{
    \eta \to \infty
  }
  \sup_{ 
    x \in \R^d
  }
  \left[
  ( \mathcal{G}_{ \mu, \sigma } U_0 )( x )
  +
  \tfrac{ 1 }{ 2 } 
  \, \| \sigma( x )^* ( \nabla U_0 )( x ) \|^2_{ \R^d } 
  -
  \eta \, U_0( x )
  \right] 
\\ & =
  \lim_{
    \eta \to \infty
  }
  \sup_{ 
    x = ( x_1, x_2 ) \in \R^{ 2 m }
  }
  \left[
  \vartheta
  \left< x_1, x_2 \right>_{ \R^m }
  - \vartheta \gamma \left\| x_2 \right\|^2_{ \R^m }
  + \tfrac{ \vartheta \beta m }{ 2 }
  +
  \tfrac{ \beta \vartheta^2 }{ 2 } 
  \, \| x_2 \|^2_{ \R^m } 
  -
  \eta \, U_0( x )
  \right] 
\\ & \leq 
  \vartheta
  \lim_{
    \eta \to \infty
  }
  \sup_{ 
    x_1, x_2 \in \R^m
  }
  \Big[
    \left[ 
      \tfrac{ 1 }{ 2 }
      -
      \tfrac{ \eta }{ 2 }
    \right] 
    \left\| x_1 \right\|^2_{ \R^m }
    +
    \left[ 
      \tfrac{ 1 }{ 2 }
      +
      \tfrac{ \beta \vartheta }{ 2 } 
      - \gamma 
      - \tfrac{ \eta }{ 2 }
    \right] 
    \left\| x_2 \right\|^2_{ \R^m }
    + 
    \tfrac{ \beta m }{ 2 }
  \Big] 
  < \infty
\end{split}
\end{equation}
(cf.\ Subsection~4.4 in
Cox et al.~\cite{CoxHutzenthalerJentzen2013}).
Inequalities~\eqref{eq:Langevin_ST_ass_rate}
and \eqref{eq:Langevin_ST_ass_exp} show that
\eqref{eq:ST_ass_rate} and \eqref{eq:ST_ass_exp}
are fulfilled.
We can thus apply Corollary~\ref{cor:ST3}
to obtain that there exist 
$ C_r \in [0,\infty) $, $ r \in (0,\infty) $,
such that 
for all 
$ r \in (0,\infty) $,
$ n \in \N $,
$ \theta = (  t_0 , t_1, \dots, t_n ) \in \mathcal{P}_T $
it holds that
$
  \sup_{ t \in [0,T] }
  \| 
    X_t - Y^{ \theta }_t
  \|_{
    L^r( \Omega; \R^d )
  }
\leq
  C_r
  \big[
    \max\nolimits_{ 
      k \in \{ 0, 1, \dots, n - 1 \} 
    }
    \left| t_{ k + 1 } - t_k \right|
  \big]^{
    1 / 2
  }
$
(see also 
Remark~\ref{remark:higher_order} above).

\subsection{Galerkin approximations of stochastic partial differential equations (SPDEs)}
\label{ssec:Galerkin}

The next result, Corollary~\ref{cor:Galerkin},
is useful for the estimation of approximation errors
of Galerkin approximations of solutions
of SPDEs.

\begin{corollary}
\label{cor:Galerkin}
Assume the setting in Subsection~\ref{sec:setting},
let
$ \varepsilon \in [0,\infty] $,
$ p \in [2,\infty) $,
$ \mu \in \mathcal{L}^0( \mathcal{O} ; H ) $,
$ \sigma \in \mathcal{L}^0( \mathcal{O} ; HS( U, H ) ) $,
$ P \in L( H ) $
satisfy 
% $ P^2 = P = P^* $
% $ \| P \|_{ L(H) } \leq 1 $
% and 
$ P( \mathcal{O} ) \subseteq \mathcal{O} $,
let
$ X, Y \colon [  \tzero , T ] \times \Omega \to \mathcal{O} $,
$ \chi \colon [  \tzero , T ] \times \Omega \to \R $
be predictable stochastic processes
with 
$
  \int_{  \tzero  }^T
  \| \mu( X_s ) \|_H
  +
  \| \sigma( X_s ) \|^2_{ HS( U, H ) }
  +
  \| \mu( P X_s ) \|_H
  +
  \| \sigma( P X_s ) \|^2_{ HS( U, H ) }
  +
  \| \mu( Y_s ) \|_H
  +
  \| \sigma( Y_s ) \|^2_{ HS( U, H ) }
  \,
  ds
  < \infty
$
$ \P $-a.s.,
$
  X_t 
  = 
  X_{  \tzero  } 
  +
  \int_{  \tzero  }^t \mu( X_s ) \, ds
  +
  \int_{  \tzero  }^t \sigma( X_s ) \, dW_s
$
$ \P $-a.s.,
$
  Y_t 
  = 
  P X_{  \tzero  } 
  +
  \int_{  \tzero  }^t P \mu( Y_s ) \, ds
  +
  \int_{  \tzero  }^t P \sigma( Y_s ) \, dW_s
$
$ \P $-a.s.\ for all
$ t \in [ \tzero ,T] $
and
\begin{equation}
\label{eq:cor_Galerkin_assumption}
      \smallint_{  \tzero  }^T 
      \Big[
        \tfrac{
          \langle 
            Y_s - P X_s , 
            P \mu( Y_s ) - P \mu( P X_s ) 
          \rangle_H
          +
          \frac{ ( p - 1 ) \, ( 1 + \varepsilon ) }{ 2 }
          \,
          \|
            P \sigma( Y_s ) - P \sigma( P X_s )
          \|^2_{ HS( U, H ) }
        }{
          \| Y_s - P X_s \|^2_H
        }
        +
        \chi_s 
      \Big]^+
      ds
  < \infty
\end{equation}
$ \P $-a.s. 
Then it holds
for all 
$ r, q \in (0,\infty] $
with $ \frac{ 1 }{ p } + \frac{ 1 }{ q } = \frac{ 1 }{ r } $
that
%{\small
\begin{equation}
\label{eq:cor_Galerkin}
\begin{split}
&
  \sup_{ t \in [  \tzero , T] }
  \left\|
    X_t - Y_t
  \right\|_{
    L^r( \Omega; H )
  }
% \\ & 
\leq
  \sup_{ t \in [  \tzero , T] }
  \left\|
    ( I - P ) X_t
  \right\|_{
    L^r( \Omega; H ) 
  }
\\ & \quad 
  + 
  \left\|
    \exp\!\left(
      \smallint_{  \tzero  }^T 
      \Big[
        \tfrac{
          \left< 
            Y_s - P X_s , P \mu( Y_s ) - P \mu( P X_s ) 
          \right>_H
          +
          \frac{ ( p - 1 ) \, ( 1 + \varepsilon ) }{ 2 }
          \left\|
            P \sigma( Y_s ) - P \sigma( P X_s )
          \right\|^2_{ HS( U, H ) }
        }{
          \left\| Y_s - P X_s \right\|^2_H
        }
        +
        \chi_s 
      \Big]^+
      ds
    \right)
  \right\|_{
    L^q( \Omega; \R )
  }
\\ & \quad \cdot
    \Big\|
    \,
    p	
    \,
    \|
      Y - P X 
    \|^{ ( p - 2 ) }_H
    \big[
    \langle
      Y - P X
      ,
      P \mu( P X )
      -
      P \mu( X )
    \rangle_H
\\ & \quad
    +
    \tfrac{ ( p - 1 ) \, ( 1 + 1 / \varepsilon ) }{ 2 }
      \left\|
        P \sigma( X )
        -
        P \sigma( P X )
      \right\|_{ HS( U, H ) }^2
    - \chi \, \| Y - P X \|^2_H
    \big]^+
  \Big\|_{
    L^1( [  \tzero , T ] \times \Omega ; \R )
  }^{ 1 / p }
  .
\end{split}
\end{equation}
%}
\end{corollary}

Corollary~\ref{cor:Galerkin}
is a special case of Corollary~\ref{cor:different_coefficients}
(choose $ D(A) = H $, $ A = 0 $, $ F_1 = \mu $, $ B_1 = \sigma $, 
$ F_2 = P \mu $, $ B_2 = P \sigma $,
$ X^1 = X $, $ X^2 = Y $ and $ \hat{X} = P( X ) $
in the setting of Corollary~\ref{cor:different_coefficients} and
Corollary~\ref{cor:Galerkin} respectively).
If the processes $ X $ and $ Y $ in Corollary~\ref{cor:Galerkin}
satisfy suitable exponential integrability properties
(see Corollary~2.4 in Cox et al.~\cite{CoxHutzenthalerJentzen2013}),
then the right-hand side of \eqref{eq:cor_Galerkin}
can be further estimated in an appropriate way.
This is the subject of the next result.

\begin{prop}
\label{prop:Galerkin2}
Assume the setting in Subsection~\ref{sec:setting},
let 
$ \varepsilon \in [0,\infty] $,
$ r, q_0, q_1, \hat{q}_0, \hat{q}_1 \in (0,\infty] $,
$ 
  c, 
  \alpha, \beta, \hat{\alpha}, \hat{\beta} \in [0,\infty) 
$,
$ p \in [2,\infty) $,
$
  U_0, \hat{U}_0 \in C^2( O, [0,\infty) )
$,
$ 
  U_1, \hat{U}_1 \in C( \mathcal{O}, [0,\infty) ) 
$,
$
  \varphi \in \mathcal{L}^0( \mathcal{O}; \R )
$,
$ \mu \in \mathcal{L}^0( \mathcal{O} ; H ) $,
$ \sigma \in \mathcal{L}^0( \mathcal{O} ; HS( U, H ) ) $,
$ P \in L( H ) $
satisfy 
% $ P^2 = P = P^* $,
% $ \| P \|_{ L(H) } \leq 1 $,
$ P( \mathcal{O} ) \subseteq \mathcal{O} $,
$
  \frac{ 1 }{ p } + \frac{ 1 }{ q_0 } + \frac{ 1 }{ q_1 } + \frac{ 1 }{ \hat{q}_0 } + \frac{ 1 }{ \hat{q}_1 } = \frac{ 1 }{ r } 
$
and
{\small
\begin{eqnarray}
&
  ( \mathcal{G}_{ \mu, \sigma } U_0 )( x )
  +
  \tfrac{ 1 }{ 2 } 
  \, \| \sigma( x )^* ( \nabla U_0 )( x ) \|^2_H 
  +
  U_1( x )
\leq
  \alpha \, U_0( x )
  +
  \beta
  ,
\nonumber
\\[0.5ex]
&
\label{eq:Galerkin_assumption}
  ( \mathcal{G}_{ P \mu, P \sigma } \hat{U}_0 )( y )
  +
  \tfrac{ 1 }{ 2 } 
  \, \| \sigma( y )^* P^* ( \nabla \hat{U}_0 )( y ) \|^2_H 
  +
  \hat{U}_1( y )
\leq
  \hat{\alpha} \, \hat{U}_0( y )
  +
  \hat{\beta}
  ,
\\[0.5ex]
&
\nonumber
    \left< 
      P x - y , P \mu( P x ) - P \mu( y ) 
    \right>_H
    +
    \frac{ ( p - 1 ) \, ( 1 + \varepsilon ) }{ 2 }
    \left\|
      P \sigma( P x ) - P \sigma( y ) 
    \right\|^2_{ HS( U , H ) }
  +
    \left< 
      y - P x , P \mu( P x ) - P \mu( x ) 
    \right>_H
\\ & 
\nonumber
    +
    \frac{ ( p - 1 ) \, ( 1 + 1 / \varepsilon ) }{ 2 }
    \left\|
      P \sigma( P x ) - P \sigma( x )
    \right\|^2_{ HS( U , H ) }
\leq
  \tfrac{
    \left|
      \varphi( x )
    \right|^2
  }{ 2 }
  +
  \left[
    c
    +
    \tfrac{
      U_0( x ) 
    }{
      q_0 T e^{ \alpha T }
    }
    +
    \tfrac{
      \hat{U}_0( y )
    }{
      \hat{q}_0 T e^{ \hat{\alpha} T }
    }
    +
    \tfrac{
      U_1( x ) 
    }{
      q_1 e^{ \alpha T }
    }
    +
    \tfrac{
      \hat{U}_1( y )
    }{
      \hat{q}_1 e^{ \hat{\alpha} T }
    }
  \right] 
  \left\| P x - y \right\|^2_H
\end{eqnarray}}for 
all $ x \in \mathcal{O} $,
$ y \in P( H ) \cap \mathcal{O} $,
let
$ X, Y \colon [  \tzero , T ] \times \Omega \to \mathcal{O} $
be predictable stochastic processes
with 
$
  \int_{  \tzero  }^T
  \| \mu( X_s ) \|_H
  +
  \| \sigma( X_s ) \|^2_{ HS( U, H ) }
  +
  \| \mu( P X_s ) \|_H
  +
  \| \sigma( P X_s ) \|^2_{ HS( U, H ) }
  +
  \| \mu( Y_s ) \|_H
  +
  \| \sigma( Y_s ) \|^2_{ HS( U, H ) }
  \,
  ds
  < \infty
$
$ \P $-a.s.,
$
  X_t 
  = 
  X_{  \tzero  } 
  +
  \int_{  \tzero  }^t \mu( X_s ) \, ds
  +
  \int_{  \tzero  }^t \sigma( X_s ) \, dW_s
$
$ \P $-a.s.,
$
  Y_t 
  = 
  P X_{  \tzero  } 
  +
  \int_{  \tzero  }^t P \mu( Y_s ) \, ds
  +
  \int_{  \tzero  }^t P \sigma( Y_s ) \, dW_s
$
$ \P $-a.s.\ for all
$ t \in [ \tzero ,T] $
and
$
  \E\big[
    e^{
      U_0( X_0 ) 
    }
    +
    e^{
      \hat{U}_0( Y_0 )
    }
  \big] 
  < \infty
$.
Then 
\begin{align}
&
  \sup_{ t \in [  \tzero , T] }
  \left\|
    X_t - Y_t
  \right\|_{
    L^r( \Omega; H )
  }
\leq
  T^{
    ( \frac{ 1 }{ 2 } - \frac{ 1 }{ p } )
  }
  \exp\!\left(
    \tfrac{ 1 }{ 2 } - \tfrac{ 1 }{ p }
    +
    \smallint_0^T
    c
    +
    \smallsum\limits_{ i = 0 }^1
    \left[
    \tfrac{
      \beta
    }{
      q_i e^{ \alpha s }
    }
    +
    \tfrac{
      \hat{\beta}
    }{
      \hat{q}_i e^{ \hat{ \alpha } s }
    }
    \right]
    ds
  \right)
\\ & \quad 
\nonumber
  \cdot
  \left\|
    \varphi( X )
  \right\|_{
    L^p( [  \tzero , T ] \times \Omega ; \R )
  }
  \left|
  \E\Big[
    e^{
      U_0( X_0 ) 
    }
  \Big]
  \right|^{
    \left[ 
      \frac{ 1 }{ q_0 }
      +
      \frac{ 1 }{ q_1 }
    \right]
  }
  \left|
  \E\!\left[
    e^{
      \hat{U}_0( Y_0 )
    }
  \right] 
  \right|^{
    \left[
      \frac{ 1 }{ \hat{q}_0 }
      +
      \frac{ 1 }{ \hat{q}_1 }
    \right]
  }
  +
  \sup_{ t \in [  \tzero , T] }
  \left\|
    ( I - P ) X_t
  \right\|_{
    L^r( \Omega; H ) 
  }
  .
\end{align}
\end{prop}

\begin{proof}[Proof
of Proposition~\ref{prop:Galerkin2}]
Throughout this proof 
let 
$ q \in (0,\infty] $
be given by $ \frac{ 1 }{ q_0 } + \frac{ 1 }{ q_1 } + \frac{ 1 }{ \hat{q}_0 } + \frac{ 1 }{ \hat{q}_1 } = \frac{ 1 }{ q } $
and let 
$ \chi \colon [  \tzero , T ] \times \Omega \to \R $ be given by
\begin{equation}
\label{eq:def_chi}
\begin{split}
  \chi_t 
& = 
    c
    +
    \tfrac{
      U_0( X_t ) 
    }{
      q_0 T e^{ \alpha T }
    }
    +
    \tfrac{
      \hat{U}_0( Y_t )
    }{
      \hat{q}_0 T e^{ \hat{\alpha} T }
    }
    +
    \tfrac{
      U_1( X_t ) 
    }{
      q_1 e^{ \alpha T }
    }
    +
    \tfrac{
      \hat{U}_1( Y_t )
    }{
      \hat{q}_1 e^{ \hat{\alpha} T }
    }
    +
    \tfrac{
      ( 1 / 2 - 1 / p )
    }{
      T
    }
\\ & \quad
    -
        \tfrac{
          \left< 
            Y_t - P X_t , P \mu( Y_t ) - P \mu( P X_t ) 
          \right>_H
          +
          \frac{ ( p - 1 ) \, ( 1 + \varepsilon ) }{ 2 }
          \left\|
            P \sigma( Y_t ) - P \sigma( P X_t )
          \right\|^2_{ HS( U, H ) }
        }{
          \left\| Y_t - P X_t \right\|^2_H
        }
\end{split}
\end{equation}
for all $ t \in [ 0, T ] $.
We intend to apply Corollary~\ref{cor:Galerkin}.
To do so, we need to verify 
assumption~\eqref{eq:cor_Galerkin_assumption}
in Corollary~\ref{cor:Galerkin}.
For this note that
the definition of $ \chi $ (see \eqref{eq:def_chi})
% assumption~\eqref{eq:Galerkin_assumption}
together with H\"{o}lder's inequality
ensures that
\begin{equation}
\begin{split}
&
  \left\|
    \exp\!\left(
      \smallint_{  \tzero  }^T 
      \Big[
        \tfrac{
          \left< 
            Y_s - P X_s , P \mu( Y_s ) - P \mu( P X_s ) 
          \right>_H
          +
          \frac{ ( p - 1 ) \, ( 1 + \varepsilon ) }{ 2 }
          \left\|
            P \sigma( Y_s ) - P \sigma( P X_s )
          \right\|^2_{ HS( U, H ) }
        }{
          \left\| Y_s - P X_s \right\|^2_H
        }
        +
        \chi_s 
      \Big]^+
      ds
    \right)
  \right\|_{
    L^q( \Omega; \R )
  }
\\ & =
  \left\|
    \exp\!\left(
      \tfrac{ 1 }{ 2 } - \tfrac{ 1 }{ p }
    +
      \smallint_{  \tzero  }^T 
    c
    +
    \tfrac{
      U_0( X_s ) 
    }{
      q_0 T e^{ \alpha T }
    }
    +
    \tfrac{
      \hat{U}_0( Y_s )
    }{
      \hat{q}_0 T e^{ \hat{\alpha} T }
    }
    +
    \tfrac{
      U_1( X_s ) 
    }{
      q_1 e^{ \alpha T }
    }
    +
    \tfrac{
      \hat{U}_1( Y_s )
    }{
      \hat{q}_1 e^{ \hat{\alpha} T }
    }
    \,
      ds
    \right)
  \right\|_{
    L^q( \Omega; \R )
  }
\\ & \leq
  \exp\!\left(
    \tfrac{ 1 }{ 2 } - \tfrac{ 1 }{ p }
    +
    c T 
  \right)
  \left\|
    \exp\!\left(
      \smallint_{  \tzero  }^T 
    \tfrac{
      U_0( X_s ) 
    }{
      q_0 T e^{ \alpha T }
    }
    \,
      ds
    \right)
  \right\|_{
    L^{ q_0 }( \Omega; \R )
  }
  \left\|
    \exp\!\left(
      \smallint_{  \tzero  }^T 
    \tfrac{
      \hat{U}_0( Y_s )
    }{
      \hat{q}_0 T e^{ \hat{\alpha} T }
    }
    \,
      ds
    \right)
  \right\|_{
    L^{ \hat{q}_0 }( \Omega; \R )
  }
\\ & \quad \cdot
  \left\|
    \exp\!\left(
      \smallint_{  \tzero  }^T 
    \tfrac{
      U_1( X_s ) 
    }{
      q_1 e^{ \alpha T }
    }
    \,
      ds
    \right)
  \right\|_{
    L^{ q_1 }( \Omega; \R )
  }
  \left\|
    \exp\!\left(
      \smallint_{  \tzero  }^T 
    \tfrac{
      \hat{U}_1( Y_s )
    }{
      \hat{q}_1 e^{ \hat{\alpha} T }
    }
    \,
      ds
    \right)
  \right\|_{
    L^{ \hat{q}_1 }( \Omega; \R )
  }
  .
\end{split}
\end{equation}
A simple consequence of Jensen's inequality 
(see, e.g., inequality~(19) in Li~\cite{Li1994} and 
Lemma~2.21 in Cox et al.~\cite{CoxHutzenthalerJentzen2013})
together with nonnegativity of $ U_0 $ and $ \hat{U}_0 $
hence proves that
\begin{equation}
\begin{split}
&
  \left\|
    \exp\!\left(
      \smallint_{  \tzero  }^T 
      \Big[
        \tfrac{
          \left< 
            Y_s - P X_s , P \mu( Y_s ) - P \mu( P X_s ) 
          \right>_H
          +
          \frac{ ( p - 1 ) \, ( 1 + \varepsilon ) }{ 2 }
          \left\|
            P \sigma( Y_s ) - P \sigma( P X_s )
          \right\|^2_{ HS( U, H ) }
        }{
          \left\| Y_s - P X_s \right\|^2_H
        }
        +
        \chi_s 
      \Big]^+
      ds
    \right)
  \right\|_{
    L^q( \Omega; \R )
  }
\\ & \leq
  \exp\!\left(
    \tfrac{ 1 }{ 2 } - \tfrac{ 1 }{ p }
    +
    c T 
  \right)
  \sup_{ s \in [0,T] }
  \left\|
    \exp\!\left(
    \tfrac{
      U_0( X_s ) 
    }{
      q_0 e^{ \alpha s }
    }
    \right)
  \right\|_{
    L^{ q_0 }( \Omega; \R )
  }
  \sup_{ s \in [0,T] }
  \left\|
    \exp\!\left(
    \tfrac{
      \hat{U}_0( Y_s )
    }{
      \hat{q}_0 e^{ \hat{\alpha} s }
    }
    \right)
  \right\|_{
    L^{ \hat{q}_0 }( \Omega; \R )
  }
\\ & \quad \cdot
  \left\|
    \exp\!\left(
    \tfrac{
      U_0( X_T ) 
    }{
      q_1 e^{ \alpha T }
    }
    +
      \smallint_{  \tzero  }^T 
    \tfrac{
      U_1( X_s ) 
    }{
      q_1 e^{ \alpha T }
    }
    \,
      ds
    \right)
  \right\|_{
    L^{ q_1 }( \Omega; \R )
  }
  \left\|
    \exp\!\left(
    \tfrac{
      \hat{U}_0( Y_T )
    }{
      \hat{q}_1 e^{ \hat{\alpha} T }
    }
    +
      \smallint_{  \tzero  }^T 
    \tfrac{
      \hat{U}_1( Y_s )
    }{
      \hat{q}_1 e^{ \hat{\alpha} T }
    }
    \,
      ds
    \right)
  \right\|_{
    L^{ \hat{q}_1 }( \Omega; \R )
  }
  .
\end{split}
\end{equation}
The nonnegativity of $ U_1 $ and $ \hat{U}_1 $ hence implies that
\begin{equation}
\begin{split}
&
  \left\|
    \exp\!\left(
      \smallint_{  \tzero  }^T 
      \Big[
        \tfrac{
          \left< 
            Y_s - P X_s , P \mu( Y_s ) - P \mu( P X_s ) 
          \right>_H
          +
          \frac{ ( p - 1 ) \, ( 1 + \varepsilon ) }{ 2 }
          \left\|
            P \sigma( Y_s ) - P \sigma( P X_s )
          \right\|^2_{ HS( U, H ) }
        }{
          \left\| Y_s - P X_s \right\|^2_H
        }
        +
        \chi_s 
      \Big]^+
      ds
    \right)
  \right\|_{
    L^q( \Omega; \R )
  }
\\ & \leq
  \exp\!\left(
    \tfrac{ 1 }{ 2 } - \tfrac{ 1 }{ p }
    +
    \smallint_0^T
    c
    +
    \tfrac{
      \beta
    }{
      q_0 e^{ \alpha s }
    }
    +
    \tfrac{
      \hat{\beta}
    }{
      \hat{q}_0 e^{ \hat{ \alpha } s }
    }
    +
    \tfrac{
      \beta
    }{
      q_1 e^{ \alpha s }
    }
    +
    \tfrac{
      \hat{\beta}
    }{
      \hat{q}_1 e^{ \hat{ \alpha } s }
    }
    \, ds
  \right)
\\ & \cdot
  \sup_{ s \in [0,T] }
  \left\|
    \exp\!\left(
    \tfrac{
      U_0( X_s ) 
    }{
      q_0 e^{ \alpha s }
    }
    +
      \smallint_{  \tzero  }^s 
    \tfrac{
      U_1( X_u ) 
      -
      \beta
    }{
      q_0 e^{ \alpha u }
    }
    \,
      du
    \right)
  \right\|_{
    L^{ q_0 }( \Omega; \R )
  }
  \sup_{ s \in [0,T] }
  \left\|
    \exp\!\left(
    \tfrac{
      \hat{U}_0( Y_s )
    }{
      \hat{q}_0 e^{ \hat{\alpha} s }
    }
    +
      \smallint_{  \tzero  }^s 
    \tfrac{
      \hat{U}_1( Y_u )
      -
      \hat{\beta}
    }{
      \hat{q}_0 e^{ \hat{\alpha} u }
    }
    \,
      du
    \right)
  \right\|_{
    L^{ \hat{q}_0 }( \Omega; \R )
  }
\\ & \cdot
  \left\|
    \exp\!\left(
    \tfrac{
      U_0( X_T ) 
    }{
      q_1 e^{ \alpha T }
    }
    +
      \smallint_{  \tzero  }^T 
    \tfrac{
      U_1( X_s ) 
      -
      \beta
    }{
      q_1 e^{ \alpha s }
    }
    \,
      ds
    \right)
  \right\|_{
    L^{ q_1 }( \Omega; \R )
  }
  \left\|
    \exp\!\left(
    \tfrac{
      \hat{U}_0( Y_T )
    }{
      \hat{q}_1 e^{ \hat{\alpha} T }
    }
    +
      \smallint_{  \tzero  }^T 
    \tfrac{
      \hat{U}_1( Y_s )
      -
      \hat{\beta}
    }{
      \hat{q}_1 e^{ \hat{\alpha} s }
    }
    \,
      ds
    \right)
  \right\|_{
    L^{ \hat{q}_1 }( \Omega; \R )
  }
\end{split}
\end{equation}
and this proves that
\begin{equation}
\begin{split}
&
  \left\|
    \exp\!\left(
      \smallint_{  \tzero  }^T 
      \Big[
        \tfrac{
          \left< 
            Y_s - P X_s , P \mu( Y_s ) - P \mu( P X_s ) 
          \right>_H
          +
          \frac{ ( p - 1 ) \, ( 1 + \varepsilon ) }{ 2 }
          \left\|
            P \sigma( Y_s ) - P \sigma( P X_s )
          \right\|^2_{ HS( U, H ) }
        }{
          \left\| Y_s - P X_s \right\|^2_H
        }
        +
        \chi_s 
      \Big]^+
      ds
    \right)
  \right\|_{
    L^q( \Omega; \R )
  }
\\ & \leq
  \exp\!\left(
    \tfrac{ 1 }{ 2 } - \tfrac{ 1 }{ p }
    +
    \smallint_0^T
    c
    +
    \smallsum\limits_{ i = 0 }^1
    \left[
    \tfrac{
      \beta
    }{
      q_i e^{ \alpha s }
    }
    +
    \tfrac{
      \hat{\beta}
    }{
      \hat{q}_i e^{ \hat{ \alpha } s }
    }
    \right]
    ds
  \right)
  \sup_{ s \in [0,T] }
  \left|
  \E\!\left[
    \exp\!\left(
    \tfrac{
      U_0( X_s ) 
    }{
      e^{ \alpha s }
    }
    +
      \smallint_{  \tzero  }^s 
    \tfrac{
      U_1( X_u ) 
      -
      \beta
    }{
      e^{ \alpha u }
    }
    \,
      du
    \right)
  \right]
  \right|^{
    \left[ 
      \frac{ 1 }{ q_0 }
      +
      \frac{ 1 }{ q_1 }
    \right]
  }
\\ & \quad \cdot
  \sup_{ s \in [0,T] }
  \left|
  \E\!\left[
    \exp\!\left(
    \tfrac{
      \hat{U}_0( Y_s )
    }{
      e^{ \hat{\alpha} s }
    }
    +
      \smallint_{  \tzero  }^s 
    \tfrac{
      \hat{U}_1( Y_u )
      -
      \hat{\beta}
    }{
      e^{ \hat{\alpha} u }
    }
    \,
      du
    \right)
  \right] 
  \right|^{
    \left[
      \frac{ 1 }{ \hat{q}_0 }
      +
      \frac{ 1 }{ \hat{q}_1 }
    \right]
  }
  .
\end{split}
\end{equation}
Corollary~2.4 in Cox et al.~\cite{CoxHutzenthalerJentzen2013}
together with assumption \eqref{eq:Galerkin_assumption}
and the assumption that
$
  \E\big[
    e^{
      U_0( X_0 ) 
    }
    +
    e^{
      \hat{U}_0( Y_0 )
    }
  \big] 
  < \infty
$
hence implies that
\begin{equation}
\label{eq:expU_est}
\begin{split}
&
  \left\|
    \exp\!\left(
      \smallint_{  \tzero  }^T 
      \Big[
        \tfrac{
          \left< 
            Y_s - P X_s , P \mu( Y_s ) - P \mu( P X_s ) 
          \right>_H
          +
          \frac{ ( p - 1 ) \, ( 1 + \varepsilon ) }{ 2 }
          \left\|
            P \sigma( Y_s ) - P \sigma( P X_s )
          \right\|^2_{ HS( U, H ) }
        }{
          \left\| Y_s - P X_s \right\|^2_H
        }
        +
        \chi_s 
      \Big]^+
      ds
    \right)
  \right\|_{
    L^q( \Omega; \R )
  }
\\ & \leq
  \exp\!\left(
    \tfrac{ 1 }{ 2 } - \tfrac{ 1 }{ p }
    +
    \smallint_0^T
    c
    +
    \smallsum\limits_{ i = 0 }^1
    \left[
    \tfrac{
      \beta
    }{
      q_i e^{ \alpha s }
    }
    +
    \tfrac{
      \hat{\beta}
    }{
      \hat{q}_i e^{ \hat{ \alpha } s }
    }
    \right]
    ds
  \right)
  \left|
  \E\Big[
    e^{
      U_0( X_0 ) 
    }
  \Big]
  \right|^{
    \left[ 
      \frac{ 1 }{ q_0 }
      +
      \frac{ 1 }{ q_1 }
    \right]
  }
  \left|
  \E\!\left[
    e^{
      \hat{U}_0( Y_0 )
    }
  \right] 
  \right|^{
    \left[
      \frac{ 1 }{ \hat{q}_0 }
      +
      \frac{ 1 }{ \hat{q}_1 }
    \right]
  }
  < \infty
\end{split}
\end{equation}
and this ensures that
\begin{equation}
      \smallint_{  \tzero  }^T 
      \Big[
        \tfrac{
          \left< 
            Y_s - P X_s , P \mu( Y_s ) - P \mu( P X_s ) 
          \right>_H
          +
          \frac{ ( p - 1 ) \, ( 1 + \varepsilon ) }{ 2 }
          \left\|
            P \sigma( Y_s ) - P \sigma( P X_s )
          \right\|^2_{ HS( U, H ) }
        }{
          \left\| Y_s - P X_s \right\|^2_H
        }
        +
        \chi_s 
      \Big]^+
      ds
  < \infty
\end{equation}
$ \P $-a.s.\ We 
can thus apply 
Corollary~\ref{cor:Galerkin}
to obtain that
% together with the estimate
% $
%   \| P \|_{ L( H ) } \leq 1
% $
%{\small
\begin{equation}
\begin{split}
&
  \sup_{ t \in [  \tzero , T] }
  \left\|
    X_t - Y_t
  \right\|_{
    L^r( \Omega; H )
  }
% \\ & 
\leq
  \sup_{ t \in [  \tzero , T] }
  \left\|
    ( I - P ) X_t
  \right\|_{
    L^r( \Omega; H ) 
  }
\\ & \quad 
  + 
  \left\|
    \exp\!\left(
      \smallint_{  \tzero  }^T 
      \Big[
        \tfrac{
          \left< 
            Y_s - P X_s , P \mu( Y_s ) - P \mu( P X_s ) 
          \right>_H
          +
          \frac{ ( p - 1 ) \, ( 1 + \varepsilon ) }{ 2 }
          \left\|
            P \sigma( Y_s ) - P \sigma( P X_s )
          \right\|^2_{ HS( U, H ) }
        }{
          \left\| Y_s - P X_s \right\|^2_H
        }
        +
        \chi_s 
      \Big]^+
      ds
    \right)
  \right\|_{
    L^q( \Omega; \R )
  }
\\ & \quad \cdot
    \Big\|
    \,
    p	
    \,
    \|
      Y - P X 
    \|^{ ( p - 2 ) }_H
    \big[
    \langle
      Y - P X
      ,
      P \mu( P X )
      -
      P \mu( X )
    \rangle_H
\\ & \quad
    +
    \tfrac{ ( p - 1 ) \, ( 1 + 1 / \varepsilon ) }{ 2 }
      \left\|
        P \sigma( X )
        -
        P \sigma( P X )
      \right\|_{ HS( U, H ) }^2
    - \chi \, \| Y - P X \|^2_H
    \big]^+
  \Big\|_{
    L^1( [  \tzero , T ] \times \Omega ; \R )
  }^{ 1 / p }
  .
\end{split}
\end{equation}
Assumption~\eqref{eq:Galerkin_assumption} 
together with the fact that
$
  Y_t \in P(H) \cap \mathcal{O}
$
$ \P $-a.s.\ for all $ t \in [0,T] $
and the definition 
of $ \chi $ (see \eqref{eq:def_chi})
hence shows that
\begin{equation}
\label{eq:proof_Galerkin2_start}
\begin{split}
&
  \sup_{ t \in [  \tzero , T] }
  \left\|
    X_t - Y_t
  \right\|_{
    L^r( \Omega; H )
  }
% \\ & 
\leq
  \sup_{ t \in [  \tzero , T] }
  \left\|
    ( I - P ) X_t
  \right\|_{
    L^r( \Omega; H ) 
  }
\\ & \quad 
  + 
  \left\|
    \exp\!\left(
      \smallint_{  \tzero  }^T 
      \Big[
        \tfrac{
          \left< 
            Y_s - P X_s , P \mu( Y_s ) - P \mu( P X_s ) 
          \right>_H
          +
          \frac{ ( p - 1 ) \, ( 1 + \varepsilon ) }{ 2 }
          \left\|
            P \sigma( Y_s ) - P \sigma( P X_s )
          \right\|^2_{ HS( U, H ) }
        }{
          \left\| Y_s - P X_s \right\|^2_H
        }
        +
        \chi_s 
      \Big]^+
      ds
    \right)
  \right\|_{
    L^q( \Omega; \R )
  }
\\ & \quad \cdot
    \Big\|
    \,
    p	
    \,
    \|
      Y - P X 
    \|^{ ( p - 2 ) }_H
    \big[
      \tfrac{
        1
      }{ 2 }
        \left| \varphi( X ) \right|^2
      - 
      \tfrac{
        ( 1 / 2 - 1 / p )
      }{
        T
      }
      \, 
      \| Y - P X \|^2_H
    \big]^+
  \Big\|_{
    L^1( [  \tzero , T ] \times \Omega ; \R )
  }^{ 1 / p }
  .
\end{split}
\end{equation}
Next observe that Young's inequality proves that
for all $ u \in ( 0, \infty ) $ it holds that
\begin{equation}
\label{eq:Young_estimate}
\begin{split}
&
    \Big\|
    \,
    p	
    \,
    \|
      Y - P X 
    \|^{ ( p - 2 ) }_H
    \big[
      \tfrac{
        1
      }{ 2 }
        \left| \varphi( X ) \right|^2
      - 
      \tfrac{
        ( 1 / 2 - 1 / p )
      }{
        T
      }
      \, 
      \| Y - P X \|^2_H
    \big]^+
  \Big\|_{
    L^1( [  \tzero , T ] \times \Omega ; \R )
  }^{ 1 / p }
\\ & = 
  \left[
  p 
  \int_0^T
  \E\!\left[
  \max\!\left(
    \tfrac{ 
      u^{ ( 1 - 2 / p ) }
    }{ 2 }
      \left| \varphi( X ) \right|^2
    \tfrac{
      \left\|
        Y_s - P X_s
      \right\|^{ ( p - 2 ) }_H
    }{ 
      u^{ ( 1 - 2 / p ) } 
    }
    -
      \tfrac{
        ( 1 / 2 - 1 / p )
      }{
        T
      }
    \left\|
      Y_s - P X_s 
    \right\|^p_H
    , 
    0
    \right)
  \right]
    ds
  \right]^{ 1 / p }
\\ & \leq
  \left[
  p 
  \int_0^T
  \E\!\left[
    \max\!\left(
    \tfrac{ 2 }{ p }
    \left[
      \tfrac{
        u^{
          \left( 1 - 2 / p \right)
        }
      }{
        2
      }
    \right]^{ \frac{ p }{ 2 } }
    \left|
      \varphi( X_s )
    \right|^p
    +
    \tfrac{ 
      ( p - 2 ) \,
    \left\|
      Y_s - P X_s
    \right\|^p_H
    }{ 
      p 
      \,
      u
    }
      - 
      \tfrac{
        ( 1 / 2 - 1 / p ) \,
    \left\|
      Y_s - P X_s 
    \right\|^p_H
      }{
        T
      }
    , 
    0
    \right)
  \right]
    ds
  \right]^{ 
    1 / p
  }
  .
\end{split}
\end{equation}
The choice $ u = 2 T $ 
in \eqref{eq:Young_estimate}
results in 
\begin{equation}
\label{eq:Young_estimate2_Galerkin}
\begin{split}
&
    \Big\|
    \,
    p	
    \,
    \|
      Y - P X 
    \|^{ ( p - 2 ) }_H
    \big[
      \tfrac{
        1
      }{ 2 }
        \left| \varphi( X ) \right|^2
      - 
      \tfrac{
        ( 1 / 2 - 1 / p )
      }{
        T
      }
      \, \| Y - P X \|^2_H
    \big]^+
  \Big\|_{
    L^1( [  \tzero , T ] \times \Omega ; \R )
  }^{ 1 / p }
\\ & 
\leq
  T^{
    \left( 
      1 / 2 - 1 / p
    \right)
  }
  \left\|
    \varphi( X )
  \right\|_{
    L^p( 
      [ 0, T ] \times \Omega ; \R
    )
  }
  .
\end{split}
\end{equation}
Putting this and \eqref{eq:expU_est} 
into \eqref{eq:proof_Galerkin2_start}
completes the proof of Proposition~\ref{prop:Galerkin2}.
\end{proof}

In a number of cases the functions 
$ U_0 $ and $ \hat{U}_0 $ 
in Proposition~\ref{prop:Galerkin2}
satisfy 
$
  U_0( x ) = \hat{U}_0( x ) = \frac{ \rho }{ 2 } \left\| x \right\|^2_H
$
for all $ x \in O $
and some $ \rho \in (0,\infty) $.
This special case of Proposition~\ref{prop:Galerkin2} is 
the subject of the next
result, Corollary~\ref{cor:Galerkin4}.
Corollary~\ref{cor:Galerkin4} follows immediately
from Proposition~\ref{prop:Galerkin2}.

\begin{corollary}
\label{cor:Galerkin4}
Assume the setting in Subsection~\ref{sec:setting},
let 
$ \varepsilon \in [0,\infty] $,
$ r, \rho \in (0,\infty) $, 
$ q \in (0,\infty] $,
$ c, \beta \in [0,\infty) $,
$ p \in [2,\infty) $,
$ 
  U \in C( \mathcal{O}, [0,\infty) ) 
$,
$ \mu \in \mathcal{L}^0( \mathcal{O} ; H ) $,
$ \sigma \in \mathcal{L}^0( \mathcal{O} ; HS( U, H ) ) $,
$
  \varphi \in \mathcal{L}^0( \mathcal{O}; \R )
$,
$ P \in L( H ) $
satisfy 
$ P^2 = P = P^* $,
$ \| P \|_{ L(H) } \leq 1 $,
$ P( \mathcal{O} ) \subseteq \mathcal{O} $,
$
  \frac{ 1 }{ p } + \frac{ 1 }{ q } 
  = \frac{ 1 }{ r } 
$
and
% {\small
\begin{eqnarray}
&
  \left< x, \mu(x) \right>_H
  +
  \tfrac{ 1 }{ 2 }
  \left\| \sigma( x ) \right\|^2_{ HS( U, H ) }
  +
  \tfrac{ \rho }{ 2 } 
  \, \| \sigma( x )^* x \|^2_H 
  +
  U( x )
\leq
  \beta
  ,
\label{eq:Galerkin_assumption_cor4}
\\[0.5ex]
&
\nonumber
    \left< 
      P x - y , \mu( P x ) - \mu( y ) 
    \right>_H
    +
    \frac{ ( p - 1 ) \, ( 1 + \varepsilon ) }{ 2 }
    \left\|
      \sigma( P x ) - \sigma( y ) 
    \right\|^2_{ HS( U , H ) }
  +
    \left< 
      y - P x , P \mu( P x ) - P \mu( x ) 
    \right>_H
\\ & 
\nonumber
    +
    \frac{ ( p - 1 ) \, ( 1 + 1 / \varepsilon ) }{ 2 }
    \left\|
      \sigma( P x ) - \sigma( x )
    \right\|^2_{ HS( U , H ) }
\leq
  \tfrac{
    \left|
      \varphi( x )
    \right|^2
  }{ 2 }
  +
  \big[
    c
    +
    \tfrac{
      \rho 
    }{
      2 q
    }
    \,
    U( x )
    +
    \tfrac{
      \rho 
    }{
      2 q
    }
    \,
    U( y )
  \big] 
  \left\| P x - y \right\|^2_H
\end{eqnarray}
% }
for 
all $ x \in \mathcal{O} $,
$ y \in P( H ) \cap \mathcal{O} $,
let
$ X, Y \colon [  \tzero , T ] \times \Omega \to \mathcal{O} $
be predictable stochastic processes
with 
$
  \E\big[
    e^{
      \frac{ \rho }{ 2 } \left\| X_0 \right\|^2_H
    }
  \big]
  < \infty
$,
$
  \int_{  \tzero  }^T
  \| \mu( X_s ) \|_H
  +
  \| \sigma( X_s ) \|^2_{ HS( U, H ) }
  +
  \| \mu( P X_s ) \|_H
  +
  \| \sigma( P X_s ) \|^2_{ HS( U, H ) }
  +
  \| \mu( Y_s ) \|_H
  +
  \| \sigma( Y_s ) \|^2_{ HS( U, H ) }
  \,
  ds
  < \infty
$
$ \P $-a.s.,
$
  X_t 
  = 
  X_{  \tzero  } 
  +
  \int_{  \tzero  }^t \mu( X_s ) \, ds
  +
  \int_{  \tzero  }^t \sigma( X_s ) \, dW_s
$
$ \P $-a.s.\ and
$
  Y_t 
  = 
  P X_{  \tzero  } 
  +
  \int_{  \tzero  }^t P \mu( Y_s ) \, ds
  +
  \int_{  \tzero  }^t P \sigma( Y_s ) \, dW_s
$
$ \P $-a.s.\ for all
$ t \in [ \tzero ,T] $.
Then 
\begin{equation}
\begin{split}
&
  \sup_{ t \in [  \tzero , T] }
  \left\|
    X_t - Y_t
  \right\|_{
    L^r( \Omega; H )
  }
\\ & \leq
  \left\|
    \varphi( X )
  \right\|_{
    L^p( [  \tzero , T ] \times \Omega ; \R )
  }
%   \sqrt{ T }
  T^{
    ( \frac{ 1 }{ 2 } - \frac{ 1 }{ p } )
  }
  \,
  e^{
    \left[
    \frac{ 1 }{ 2 } 
    - 
    \frac{ 1 }{ p }
    +
    c T
    +
    \frac{
      \beta \rho T
    }{
      q
    }
    \right]
  }
  \big|
  \E\big[
    e^{
      \frac{ \rho }{ 2 } \left\| X_0 \right\|^2_H
    }
  \big]
  \big|^{
    \frac{ 1 }{ q }
  }
  +
  \sup_{ t \in [  \tzero , T] }
  \left\|
    ( I - P ) X_t
  \right\|_{
    L^r( \Omega; H ) 
  }
  .
\end{split} 
\end{equation}
\end{corollary}

We now apply Corollary~\ref{cor:Galerkin4} %Corollary~\ref{cor:Galerkin3} 
and Proposition~\ref{prop:Galerkin2} 
respectively 
to two semilinear example SPDEs with non-globally monotone nonlinearities. 
In both example SPDEs,
the particular choice of the functions
of $ U_0 $ and $ U_1 $ 
in Proposition~\ref{prop:Galerkin2}
and the estimates associated with them
are particularly inspired by 
a revised version of the article Cox 
et al.~\cite{CoxHutzenthalerJentzen2013}.
In both example SPDEs, 
the following common setting is used.

\subsubsection{Setting}
\label{ssec:SPDE_setting}

In the remainder of Subsection~\ref{ssec:Galerkin} the following setting is used.
Let $ k, l \in \N $, $ T \in (0,\infty) $, 
$ \varrho, \theta \in [0,\infty) $,
$ \vartheta \in ( - 1, 0] $ satisfy
$
  \theta - \vartheta < 1
$,
let 
$ ( \Omega, \mathcal{F}, \P, ( \mathcal{F}_t )_{ t \in [0,T] } ) $
be a stochastic basis,
let $ D = (0,1) $,
%for simplicity
%let $ D = ( 0, 1 )^d $,
%let $ D \subseteq \R^d $ be a bounded 
%non-empty domain with a sufficiently 
%smooth boundary,
let
$ 
  ( H, \left< \cdot, \cdot \right>_H , \left\| \cdot \right\|_H ) 
  = 
  ( 
    L^2( D ; \R^k ) , 
    \left< \cdot , \cdot \right>_{
      L^2( D ; \R^k )
    } 
    ,
    \left\| \cdot \right\|_{
      L^2( D ; \R^k )
    } 	
  )
$
and
$ 
  ( U, \left< \cdot, \cdot \right>_U , \left\| \cdot \right\|_U ) 
  = 
  ( 
    L^2( D ; \R^l ) , 
    \left< \cdot , \cdot \right>_{
      L^2( D ; \R^l )
    } 
    ,
    \left\| \cdot \right\|_{
      L^2( D ; \R^l )
    } 	
  )
$
be the $ \R $-Hilbert spaces of equivalence classes
of Lebesgue square integrable functions
from $ D $ to $ \R^k $
and $ D $ to $ \R^l $ respectively
and let $ ( W_t )_{ t \in [0,T] } $ be a cylindrical
$ \operatorname{Id}_U $-Wiener process on $ U $
with respect to $ ( \mathcal{F}_t )_{ t \in [0,T] } $.
Moreover, 
let
$ A \colon D(A) \subset H \to H $
be a generator of a strongly continuous
analytic semigroup with the property that
$
  \rho - A
$
is strictly positive
and let
$
  ( 
    H_r, \left< \cdot, \cdot \right>_{ H_r }, \left\| \cdot \right\|_{ H_r }
  )
  =
  ( 
    D( ( \varrho - A )^r ), 
    \left< ( \varrho  - A )^r ( \cdot ), ( \varrho  - A )^r ( \cdot ) \right>_{ H }, \left\| ( \varrho - A )^r ( \cdot ) \right\|_{ H }
  )
$,
$ r \in \R $,
be the 
$ \R $-Hilbert spaces of domains of fractional powers of $ \varrho - A $.
Furthermore, let 
$
  x_0 \in H_{ \theta }
$,
$
  F \in C( H_{ \theta } , H_{ \vartheta } )
$,
$ B \in C( H , HS( U, H ) ) $,
% $ B \in C( H_{ \theta }, HS( U, H ) ) $,
let 
$ P_N \in L( H_{ \vartheta }, D(A) ) $,
$ N \in \N $,
be bounded linear operators 
with $ \dim( P_N( H ) ) < \infty $
for all $ N \in \N $
and let 
$ X \colon [0,T] \times \Omega \to H_{ \theta } $
be an adapted stochastic process with 
continuous sample paths satisfying
\begin{equation}
\label{eq:SPDE}
  X_t
  =
  e^{ A t } x_0
  +
  \smallint_0^t
  e^{ A ( t - s ) } F( X_s ) \, ds
  +
  \smallint_0^t
  e^{ A ( t - s ) } B( X_s ) \, dW_s
\end{equation}
$ \P $-a.s.\ for all
$ t \in [0,T] $.
Finally, let 
$
  \mu_N \colon P_N( H ) \to P_N( H )
$
and
$
  \sigma_N \colon P_N( H ) \to \HS( U, P_N( H ) )
$
be given by
$
  \mu_N( v ) 
  = P_N( A v + F(v) )
$
and
$
  \sigma_N(v) u 
  = 
  P_N\big( 
    B(v) u 
  \big)
$
for all 
$ v \in P_N( H ) $,
$ u \in U $,
$ N \in \N $
and let
$ X^N \colon [0,T] \times \Omega \to P_N( H ) $,
$ N \in \N $,
be adapted stochastic processes with 
continuous sample paths satisfying
\begin{equation}
\label{eq:Galerkin}
  X_t^N
  =
  P_N( X_0 )
  +
  \smallint_0^t
  \mu_N( X^N_s ) \, ds
  +
  \smallint_0^t
  \sigma_N( X^N_s ) \, dW_s
\end{equation}
$ \P $-a.s.\ for all
$ t \in [0,T] $,
$ N \in \N $.

\subsubsection{Cahn-Hilliard-Cook type equations}
\label{ssec:Cahn_Hilliard}

In this subsection
assume the setting 
in Subsection~\ref{ssec:SPDE_setting}, 
assume 
$ \theta \in ( \frac{ 1 }{ 12 }, \frac{ 1 }{ 2 } ) 
$,
$ k = 1 $,
$ \vartheta = - \frac{ 1 }{ 2 } $,
let $ c \in (0,\infty) $,
let
$ L \colon D(L) \subset H \to H $
be the Laplacian with 
Neumann boundary conditions on $ (0,1) $,
that is,
$ 
  D( L ) = \{ v \in H^2( D, \R ) \colon v'|_{ \partial D } \equiv 0 \}
$
and
$
  L v = v''
$
for all $ v \in D( L ) $,
assume that 
$ D( A ) = D( L^2 ) $ 
and that $ A v = - L^2 v $
for all $ v \in D(A) $,
let
$ e_n \in H $, $ n \in \N $,
be given by
$
  e_1( x ) = 1
$
and
$
  e_{ n + 1 }( x ) = \sqrt{ 2 } \cos( n \pi x )
$
for all 
$ x \in (0,1) $,
$ n \in \N $,
assume that
$
  P_N( v ) = \sum_{ n = 1 }^N
  \left< e_n, v \right>_H e_n
$
for all $ v \in H $,
$ N \in \N $,
assume that
$ F( v ) = c \, \Delta \left( v^3 - v \right) $
for all $ v \in H_{ \theta } $,
assume that 
$
  \eta :=
  \sup_{ 
    v, w \in H , v \neq w
  }
  \frac{
    \| B(v) - B(v) \|_{ HS( U, H ) }
  }{
    \| v - w \|_H
  }
  < \infty
$
and assume that
for all $ \varepsilon \in (0,\infty) $
it holds that
\begin{equation}
\label{eq:varsigma_assumption}
  \varsigma_{ \varepsilon } 
:=
  \sup_{ v \in H_{ \theta } }
  \left[
    \|
      ( I - P_1 ) B( v ) 
    \|^2_{ HS( U, H ) }
    -
    \varepsilon \,
    \| 
      ( ( I - P_1 ) v )^2 
    \|_H^2
    -
    \varepsilon
    \,
    \|
      ( I - P_1 ) v
    \|_H^2
    \left\| v \right\|^2_H
  \right] 
  < \infty
  .
\end{equation}
(For example, 
if $ Q \in L(U) $
is a trace class operator (see, e.g., Appendix~B
in Pr\'{e}v\^{o}t \& R\"{o}ckner~\cite{PrevotRoeckner2007}),
if $ l = 1 $
% $ a, b \colon D \to \R^{ 1 \times l } $
% are continuously differentiable functions
% with 
% $
%   \sup_{ x \in D } | A(x) | 
% %   +
% %   \int_{ D } | b( x ) |^2 \, dx
% %   < \infty
% $
and if $ B $ satisfies
$ ( B( v ) u )( x ) = ( \sqrt{ Q } u )( x ) $
for all $ x \in D $, $ v \in H $, $ u \in U $,
then 
$ B $ fulfills
$
  \sup_{ 
    v, w \in H , v \neq w
  }
  \frac{
    \| B(v) - B(w) \|_{ HS( U, H ) }
  }{
    \| v - w \|_H
  }
  = 0
$
and \eqref{eq:varsigma_assumption} and in that case 
\eqref{eq:SPDE} is 
the Cahn-Hilliard-Cook type SPDE
\begin{equation}
  d X_t( x ) 
  =
  \left[ 
    - 
    \tfrac{ \partial^4 }{ \partial x^4 } X_t( x )
    +
    c
    \,
    \tfrac{ \partial^2 }{ \partial x^2 } 
    \left[
      \left( X_t( x ) \right)^3
      -
      X_t( x )
    \right]
  \right] dt
  +
  \sqrt{ Q }
  \,
  dW_t(x)
\end{equation}
for $ x \in (0,1) $, $ t \in [0,T] $
equipped with the Neumann and the non-flux boundary conditions
$
  X_t'(0) = X_t'(1) =
  X_t'''(0) = X_t'''( 1 )
  = 0
$
for $ t \in [0,T] $;
cf., e.g., Da Prato \& 
Debussche~\cite{DaPratoDebussche1996}.)
%
%
% Below we present a natural example where
% these conditions are fulfilled.
We will now apply Proposition~\ref{prop:Galerkin2}
to the spectral Galerkin approximation processes $ X^N $
and $ X^M $ for $ N, M $ with $ N < M $.
For this let $ \tilde{P}, \tilde{L} \in L(H) $
be linear operators given by
$
  \tilde{P} v 
= 
  ( I - P_1 ) v 
= 
  v - P_1( v )
=
  v - e_1 \left< e_1, v \right>_H
$
and
$
  \tilde{L} v = 
  -
  \sum_{ n = 2 }^{ \infty }
%  \left| 
  n^{ - 2 } \, \pi^{ - 2 } 
%  \right|^{ - 1 / 2 }
  \left< e_n , v \right>_H 
  e_n
$
for all $ v \in H $
and note that
for all $ v \in D(L) $ 
it holds that
\begin{equation}
  \tilde{L} L v = 
  L \tilde{L} v = 
  \tilde{P} v
  .
\end{equation}
Then observe that 
Young's inequality proves that
for all
$ \delta \in [ \frac{ 3 }{ 4 } , \infty ) $,
$ M \in \N $,
$ x \in P_M( H_{ \vartheta } ) $ it holds that
\begin{equation}
\label{eq:key_estimate_Neumann}
\begin{split}
&
  -
  c \,
  \langle 
    \tilde{P} x, x^3 
  \rangle_H
=
  -
  c \,
  \langle 
    \tilde{P} x, ( \tilde{P} x + P_1 x )^3 
  \rangle_H
\\ & =
  -
  c \,
  \langle 
    \tilde{P} x, ( \tilde{P} x )^3 
  \rangle_H
  -
  3 \, c \,
  \langle 
    \tilde{P} x, ( \tilde{P} x )^2 ( P_1 x ) 
  \rangle_H
  -
  3 \, c \,
  \langle 
    \tilde{P} x, ( \tilde{P} x ) ( P_1 x )^2 
  \rangle_H
  -
  c \,
  \langle 
    \tilde{P} x, ( P_1 x )^3 
  \rangle_H
\\ & =
  -
  c \,
  \| 
    ( \tilde{P} x )^2 
  \|_H^2
  -
  3 \, c \,
  \langle 
    \tilde{P} x, ( \tilde{P} x )^2
  \rangle_H
  \left< e_1, x \right>_H
  -
  3 \, c \,
  \|
    \tilde{P} x
  \|_H^2
  \left| \left< e_1, x \right>_H \right|^2 
  -
  c \,
  \langle 
    \tilde{P} x, e_1
  \rangle_H
  ( \left< e_1, x \right>_H )^3 
\\ & \leq
  -
  c \,
  \| 
    ( \tilde{P} x )^2 
  \|_H^2
  +
  \left[ 
    \sqrt{ 2 c } \,
    \delta^{ 1 / 2 } 
    \|
      ( \tilde{P} x )^2 
    \|_H
  \right]
  \left[
  \tfrac{ 3 \sqrt{ c } }{ \sqrt{ 2 } } 
  \,
  \delta^{ - 1 / 2 } 
  \|
    \tilde{P} x
  \|_H
  \left| \left< e_1, x \right>_H \right|
  \right]
  -
  3 \, c \,
  \|
    \tilde{P} x
  \|_H^2
  \left| \left< e_1, x \right>_H \right|^2 
\\ & \leq
  -
  c \left( 1 - \delta \right)
  \| 
    ( \tilde{P} x )^2 
  \|_H^2
  -
  3 \, c \,
  ( 1 - \tfrac{ 3 }{ 4 \delta } )
  \,
  \|
    \tilde{P} x
  \|_H^2
  \left| \left< e_1, x \right>_H \right|^2 
\\ & =
  -
  c \left( 1 - \delta \right)
  \| 
    ( \tilde{P} x )^2 
  \|_H^2
  -
  3 \, c \,
  ( 1 - \tfrac{ 3 }{ 4 \delta } )
  \,
  \|
    \tilde{P} x
  \|_H^2
  \left[ 
    \| x \|^2_H 
    -
    \| \tilde{P} x \|^2_H
  \right]
\\ & \leq
  c 
  \left[
    \delta 
    +
    3 
%    \,
%    \lambda_{ \R }( D )
    \left(
      1 - \tfrac{ 3 }{ 4 \delta }
    \right)
    - 1
  \right]
  \| 
    ( \tilde{P} x )^2 
  \|_H^2
  -
  3 \, c \,
  ( 1 - \tfrac{ 3 }{ 4 \delta } )
  \,
  \|
    \tilde{P} x
  \|_H^2
  \,
  \| x \|^2_H 
  \,
  .
\end{split}
\end{equation}
In the next step observe that
for all 
$ M \in \N $
and all 
$ x \in P_M( H ) $
it holds that
\begin{equation}
\begin{split}
  \langle 
    \tilde{L} x, \mu_M( x ) 
  \rangle_H
& = 
  \langle
    \tilde{L} x, P_M( A x + F(x) )
  \rangle_H
=
  \langle
    \tilde{L} P_M x, A x + F(x)
  \rangle_H
=
  \langle
    \tilde{L} x, A x + F(x)
  \rangle_H
\\ & =
  -
  \langle
    \tilde{L} x, L^2 x 
  \rangle_H
  +
  \langle
    \tilde{L} x, 
    F(x)
  \rangle_H
=
  -
  \langle
    \tilde{P} x, L x 
  \rangle_H
  +
  c
  \langle
    \tilde{P} x, 
    x^3 - x
  \rangle_H
\\ & =
  \langle
    ( - L )^{ 1 / 2 } \tilde{P} x, ( - L )^{ 1 / 2 } 
    \tilde{P} x 
  \rangle_H
  +
  c
  \langle
    \tilde{P} x, 
    x^3 
  \rangle_H
  -
  c
  \langle
    \tilde{P} x, 
    x 
  \rangle_H
\\ & =
  \|
    ( - L )^{ 1 / 2 } \tilde{P} x
  \|^2_H
  +
  c
  \langle
    \tilde{P} x, 
    x^3 
  \rangle_H
  -
  c
  \|
    \tilde{P} x
  \|^2_H
  \,
  .
\end{split}
\end{equation}
This implies that
if $ M \in \N $, 
$ \rho, \hat{ \rho } \in (0,\infty) $
and if $ U_0 \in C^2( P_M( H ) , [0,\infty) ) $
satisfies 
$
  U_0( x ) 
  = 
  \frac{ \rho }{ 2 } 
  \, \| ( - \tilde{L} )^{ 1 / 2 } x \|^2_H
  +
  \frac{ \hat{ \rho } }{ 2 } \,
  \| \tilde{P} x \|^2_H
$
for all $ x \in P_M( H ) $,
then it holds for all
$ x \in P_M( H ) $ that
\begin{equation}
\begin{split}
&
  ( \mathcal{G}_{ \mu_M , \sigma_M } U_0 )( x )
  +
  \tfrac{ 1 }{ 2 } 
  \, \| \sigma_M( x )^* ( \nabla U_0 )( x ) \|^2_H 
\\ & =
\left[
  -
  \rho
  \,
  \langle 
    \tilde{L} x, \mu_M( x ) 
  \rangle_H
  +
  \tfrac{ \rho }{ 2 }
  \,
  \|
    ( - \tilde{L} )^{ 1 / 2 } \sigma_M( x ) 
  \|^2_{ HS( U, P_M( H ) ) }
\right]
\\ &
  +
\left[
  \hat{ \rho }
  \,
  \langle 
    \tilde{P} x , 
    \mu_M( x ) 
  \rangle_H
  +
  \tfrac{ \hat{ \rho } }{ 2 }
  \,
  \|
    \tilde{P} \sigma_M( x ) 
  \|^2_{ HS( U, P_M( H ) ) }
\right]
% \\ &
  +
  \tfrac{ 1 }{ 2 } 
  \,
  \big\| 
    \sigma_M( x )^* 
    \big[ 
      \rho 
      \,
      ( - \tilde{L} ) \, x
      +
      \hat{ \rho } 
      \,
      \tilde{P} 
      \, x
    \big] 
  \big\|^2_H 
\\ & \leq
\rho
\left[
  c \,
  \|
    \tilde{P} x
  \|^2_H
  -
  \left\|
    x'
  \right\|^2_H
  -
  c \,
  \langle 
    \tilde{P} x, x^3 
  \rangle_H
  +
  \tfrac{ 1 }{ 2 }
  \,
  \|
    ( - \tilde{L} )^{ 1 / 2 } B( x ) 
  \|^2_{ HS( U, H ) }
\right]
\\ &
  +
  \hat{ \rho }
\left[
  c \,
  \| x' \|^2_H
  -
  \| x'' \|^2_H
  -
  c
  \,
  \langle 
    x' , 
    ( x^3 )'
  \rangle_H
  +
  \tfrac{ 1 }{ 2 }
  \,
  \|
    \tilde{P} B( x ) 
  \|^2_{ HS( U, H ) }
\right]
  +
  \tfrac{ 1 }{ 2 }
  \,
  \big\| 
    B( x )^* 
    \big[
      \hat{ \rho } \tilde{P} - \rho \tilde{L}
    \big]
    x
  \big\|^2_H 
  .
\end{split}
\end{equation} 
Combining this with 
\eqref{eq:key_estimate_Neumann}
and the estimate
$
  \| ( - \tilde{L} )^{ 1 / 2 } v 
  \|^2_H
\leq 
  \| 
    \tilde{P} v 
  \|_H^2
$
for all $ v \in H $
proves that
if $ M \in \N $, 
$ \rho, \hat{ \rho } \in (0,\infty) $
and if $ U_0 \in C^2( P_M( H ) , [0,\infty) ) $
satisfies 
$
  U_0( x ) 
  = 
  \frac{ \rho }{ 2 } 
  \, \| ( - \tilde{L} )^{ 1 / 2 } x \|^2_H
  +
  \frac{ \hat{ \rho } }{ 2 } \,
  \| \tilde{P} x \|^2_H
$
for all $ x \in P_M( H ) $,
then it holds for all
$ \delta \in [ \frac{ 3 }{ 4 } , \infty ) $, $ x \in P_M( H ) $ 
that
\begin{equation}
\label{eq:CHC_U_est}
\begin{split}
&
  ( \mathcal{G}_{ \mu_M , \sigma_M } U_0 )( x )
  +
  \tfrac{ 1 }{ 2 } 
  \, \| \sigma_M( x )^* ( \nabla U_0 )( x ) \|^2_H 
\\ & \leq
\rho
\left[
  c \,
  \|
    \tilde{P} x
  \|^2_H
  -
  \left\|
    x'
  \right\|^2_H
  +
  c 
  \left[
    \delta 
    +
    3 \,
    \lambda_{ \R }( D )
    \left(
      1 - \tfrac{ 3 }{ 4 \delta }
    \right)
    - 1
  \right]
  \| 
    ( \tilde{P} x )^2 
  \|_H^2
  -
  3 \, c \,
  ( 1 - \tfrac{ 3 }{ 4 \delta } )
  \,
  \|
    \tilde{P} x
  \|_H^2
  \,
  \| x \|^2_H 
\right]
\\ &
  +
  \hat{ \rho }
\left[
  c \,
  \| x' \|^2_H
  -
  \| x'' \|^2_H
  -
  3 \, c
  \,
  \|
    x' x
  \|_H^2
\right]
  +
  \tfrac{ 
    ( \rho + \hat{ \rho } ) 
  \,
  \|
    \tilde{P} B( x ) 
  \|^2_{ HS( U, H ) }
  }{ 2 }
  +
  \tfrac{ 
    \| B(x) \|^2_{ HS( U, H ) }
    \| 
      \hat{ \rho } \tilde{P} - \rho \tilde{L} 
    \|_{ L(H) }^2
    \| 
      \tilde{P} x
    \|^2_H 
  }{ 2 }
  .
\end{split}
\end{equation} 
This implies that
if $ M \in \N $, 
$ \rho, \hat{ \rho } \in (0,\infty) $
and if $ U_0 \in C^2( P_M( H ) , [0,\infty) ) $
satisfies 
$
  U_0( x ) 
  = 
  \frac{ \rho }{ 2 } 
  \, \| ( - \tilde{L} )^{ 1 / 2 } x \|^2_H
  +
  \frac{ \hat{ \rho } }{ 2 } \,
  \| \tilde{P} x \|^2_H
$
for all $ x \in P_M( H ) $,
then it holds for all
$ \varepsilon \in (0,\infty) $, 
$ 
  \delta \in [ \tfrac{ 3 }{ 4 } , \infty ) 
$,
$ x \in P_M( H ) $ 
that
\begin{equation}
\label{eq:CHC_U_est2}
\begin{split}
&
  ( \mathcal{G}_{ \mu_M , \sigma_M } U_0 )( x )
  +
  \tfrac{ 1 }{ 2 } 
  \, \| \sigma_M( x )^* ( \nabla U_0 )( x ) \|^2_H 
\\ & \leq
\rho
\left[
  -
  \left\|
    x'
  \right\|^2_H
  +
  c 
  \left[
    \delta 
    +
    3 
    \left(
      1 - \tfrac{ 3 }{ 4 \delta }
    \right)
    - 1
  \right]
  \| 
    ( \tilde{P} x )^2 
  \|_H^2
  -
  3 \, c \,
  ( 1 - \tfrac{ 3 }{ 4 \delta } )
  \,
  \|
    \tilde{P} x
  \|_H^2
  \,
  \| x \|^2_H 
\right]
\\ & \quad
  +
  \hat{ \rho }
\left[
  c \,
  \| x' \|^2_H
  -
  \| x'' \|^2_H
  -
  3 \, c
  \,
  \|
    x' x
  \|_H^2
\right]
  +
  \tfrac{ 
    ( \rho + \hat{ \rho } ) 
    \,
    \left[ 
      \varsigma_{ \varepsilon }
      +
      \varepsilon \,
      \| 
        ( \tilde{P} x )^2 
      \|_H^2
      +
      \varepsilon
      \,
      \|
        \tilde{P} x
      \|_H^2
      \left\| x \right\|^2_H
    \right]
  }{ 2 }
\\ & \quad
  +
  \tfrac{ 
    \| B(x) - B(0) + B(0) \|^2_{ HS( U, H ) }
    \| 
      \hat{ \rho } \tilde{P} - \rho \tilde{L} 
    \|_{ L(H) }^2
    \| 
      \tilde{P} x
    \|^2_H 
  }{ 2 }
  +
  \rho \, c \,
  \|
    \tilde{P} x
  \|^2_H
\\ & \leq
  \left[ 
    \tfrac{
      ( \rho + \hat{ \rho } ) \, \varepsilon
    }{
      2
    }
    +
    \rho
    \,
    c 
    \left[
      \delta 
      +
      \left(
        3 - \tfrac{ 9 }{ 4 \delta }
      \right)
      - 1
    \right]
  \right]
  \| 
    ( \tilde{P} x )^2 
  \|_H^2
\\ & \quad
  +
  \left[
    \hat{ \rho }
    \,
    c 
    -
    \rho
  \right]
  \| x' \|^2_H
  -
  \hat{ \rho }
\left[
  \| x'' \|^2_H
  +
  3 \, c
  \,
  \|
    x' x
  \|_H^2
\right]
  +
  \left[
    \tfrac{ 
      ( \rho + \hat{ \rho } ) \, \varepsilon
    }{ 2 }
    -
    \rho \, c \,
    ( 3 - \tfrac{ 9 }{ 4 \delta } )
  \right]
  \|
    \tilde{P} x
  \|_H^2
  \,
  \| x \|^2_H 
\\ & \quad
  +
  \left[
    \eta^2 \, \| x \|^2_H
    +
    \| B(0) \|^2_{ HS( U, H ) }
  \right]
    \| 
      \hat{ \rho } \tilde{P} - \rho \tilde{L} 
    \|_{ L(H) }^2
    \,
    \| 
      \tilde{P} x
    \|^2_H 
  +
  \rho \, c \,
  \|
    \tilde{P} x
  \|^2_H
  +
  \tfrac{ 
    \varsigma_{ \varepsilon }
    \,
    ( \rho + \hat{ \rho } ) 
  }{ 2 }
  .
\end{split}
\end{equation} 
Hence, we obtain that
if $ M \in \N $, 
$ \rho, \hat{ \rho } \in (0,\infty) $
and if $ U_0 \in C^2( P_M( H ) , [0,\infty) ) $
satisfies 
$
  U_0( x ) 
  = 
  \frac{ \rho }{ 2 } 
  \, \| ( - \tilde{L} )^{ 1 / 2 } x \|^2_H
  +
  \frac{ \hat{ \rho } }{ 2 } \,
  \| \tilde{P} x \|^2_H
$
for all $ x \in P_M( H ) $,
then it holds for all
$ \varepsilon \in (0,\infty) $, 
$ \delta \in [ \frac{ 3 }{ 4 }, \infty ) $,
$ x \in P_M( H ) $ 
that
\begin{equation}
\label{eq:CHC_U_est3}
\begin{split}
&
  ( \mathcal{G}_{ \mu_M , \sigma_M } U_0 )( x )
  +
  \tfrac{ 1 }{ 2 } 
  \, \| \sigma_M( x )^* ( \nabla U_0 )( x ) \|^2_H 
\\ & \leq
  \left[ 
    \tfrac{
      ( \rho + \hat{ \rho } ) \, \varepsilon
    }{
      2
    }
    +
    \rho
    \,
    c 
    \left[
      \delta 
      +
        2 - \tfrac{ 9 }{ 4 \delta }
    \right]
  \right]
  \| 
    ( \tilde{P} x )^2 
  \|_H^2
\\ & 
  +
  \left[ 
    \rho \, c 
    +
    \| B(0) \|^2_{ HS( U, H ) }
    \,
    \| 
      \hat{ \rho } \tilde{P} - \rho \tilde{L} 
    \|_{ L(H) }^2
  \right]
  \|
    \tilde{P} x
  \|^2_H
  +
  \left[
    \hat{ \rho }
    \,
    c 
    -
    \rho
  \right]
  \| x' \|^2_H
  -
  \hat{ \rho }
\left[
  \| x'' \|^2_H
  +
  3 \, c
  \,
  \|
    x' x
  \|_H^2
\right]
\\ & 
  +
  \left[
    \tfrac{ 
      ( \rho + \hat{ \rho } ) \, \varepsilon
    }{ 2 }
    +
    \eta^2 \, 
    \| 
      \hat{ \rho } \tilde{P} - \rho \tilde{L} 
    \|_{ L(H) }^2
    -
    \rho \, c \,
    ( 3 - \tfrac{ 9 }{ 4 \delta } )
  \right]
  \|
    \tilde{P} x
  \|_H^2
  \,
  \| x \|^2_H 
  +
  \tfrac{ 
    \varsigma_{ \varepsilon }
    \,
    ( \rho + \hat{ \rho } ) 
  }{ 2 }
  .
\end{split}
\end{equation} 
This implies that there exist
real numbers
$ \rho, \hat{ \rho }, \tilde{\rho} \in (0,\infty) $
such that
$ U_0, U_1 \in C^2( D(A) , [0,\infty) ) $
given by
$
  U_0( x ) 
  = 
  \frac{ \rho }{ 2 } 
  \, \| ( - \tilde{L} )^{ 1 / 2 } x \|^2_H
  +
  \frac{ \hat{ \rho } }{ 2 } \,
  \| \tilde{P} x \|^2_H
$
and
$
  U_1( x )
= 
  \hat{ \rho }
  \left\|
    x''
  \right\|^2_H
  +
  \tilde{ \rho }
  \,
  \| x \|^2_H 
  \,
  \| \tilde{P} x \|^2_H
$
for all $ x \in D(A) $
fulfill that
\begin{equation}
\label{eq:CHC_U_est4}
  \beta
  :=
  \sup_{ M \in \N }
  \sup_{ x \in P_M( H ) }
  \left[
  ( \mathcal{G}_{ \mu_M , \sigma_M } 
    U_0|_{ P_M( H ) } 
  )( x )
  +
  \tfrac{ 1 }{ 2 } 
  \, \| \sigma_M( x )^* ( \nabla U_0 )( x ) \|^2_H
  +
  U_1( x )
  \right]
  < \infty
  .
\end{equation}
Next note that
for all 
$ \varepsilon \in [0,\infty) $,
$ p \in [2,\infty) $,
$ M, N \in \N $,
$ x \in P_M( H ) $,
$ y \in P_N( H ) $
with $ M > N $
it holds that
\begin{equation}
\begin{split}
&
    \left< 
      P_N x - y , P_N \mu_M( P_N x ) - P_N \mu_M( y ) 
    \right>_H
    +
    \tfrac{ ( p - 1 ) \, ( 1 + \varepsilon ) }{ 2 }
    \left\|
      P_N \sigma_M( P_N x ) - P_N \sigma_M( y ) 
    \right\|^2_{ HS( U , P_N( H ) ) }
\\ & 
  +
    \left< 
      y - P_N x , P_N \mu_M( P_N x ) - P_N \mu_M( x ) 
    \right>_H
    +
    \tfrac{ ( p - 1 ) \, ( 1 + 1 / \varepsilon ) }{ 2 }
    \left\|
      P_N \sigma_M( P_N x ) - P_N \sigma_M( x )
    \right\|^2_{ HS( U , P_N( H ) ) }
\\ & \leq
    \left< 
      P_N x - y , F( P_N x ) - F( y ) 
    \right>_H
    -
    \left\| 
      L ( P_N x - y )
    \right\|^2_H
    +
    \tfrac{ ( p - 1 ) \, ( 1 + \varepsilon ) }{ 2 }
    \left\|
      B( P_N x ) - B( y ) 
    \right\|^2_{ HS( U , H ) }
\\ & 
  +
    \left< 
      y - P_N x , F( P_N x ) - F( x ) 
    \right>_H
    +
    \tfrac{ ( p - 1 ) \, ( 1 + 1 / \varepsilon ) }{ 2 }
    \left\|
      B( P_N x ) - B( x )
    \right\|^2_{ HS( U , H ) }
\\ & \leq
    c \,
    \|
      ( - L )^{ 1 / 2 }
      ( P_N x - y )
    \|_H^2
    +
    c
    \left< 
      P_N x - y , L \! \left[ ( P_N x )^3 - y^3 \right] 
    \right>_H
    -
    \left\| 
      L ( P_N x - y )
    \right\|^2_H
\\ & 
    +
    \tfrac{ ( p - 1 ) \, ( 1 + \varepsilon ) \, \eta^2 }{ 2 }
    \left\|
      P_N x - y 
    \right\|^2_H
  +
  c
    \left< 
      L ( y - P_N x ) , ( P_N x )^3 - x^3 - ( P_N x - x )
    \right>_H
\\ &
    +
    \tfrac{ ( p - 1 ) \, ( 1 + 1 / \varepsilon ) \, \eta^2 }{ 2 }
    \left\|
      ( I - P_N ) x 
    \right\|^2_H
    .
% \\ & \leq
%   \tfrac{
%     \left|
%       \varphi( x )
%     \right|^2
%   }{ 2 }
%   +
%   \left[
%     c
%     +
%     \tfrac{
%       U_1( x ) + U_1( y )
%     }{
%       2 q
%     }
%   \right] 
%   \left\| P_N( x ) - y \right\|^2_H
\end{split}
\end{equation}
This implies that
for all 
$ \varepsilon \in [0,\infty) $,
$ p \in [2,\infty) $,
$ M, N \in \N $,
$ x \in P_M( H ) $,
$ y \in P_N( H ) $
with $ M > N $
it holds that
\begin{equation}
\begin{split}
&
    \left< 
      P_N x - y , P_N \mu_M( P_N x ) - P_N \mu_M( y ) 
    \right>_H
    +
    \tfrac{ ( p - 1 ) \, ( 1 + \varepsilon ) }{ 2 }
    \left\|
      P_N \sigma_M( P_N x ) - P_N \sigma_M( y ) 
    \right\|^2_{ HS( U , P_N( H ) ) }
\\ & 
  +
    \left< 
      y - P_N x , P_N \mu_M( P_N x ) - P_N \mu_M( x ) 
    \right>_H
    +
    \tfrac{ ( p - 1 ) \, ( 1 + 1 / \varepsilon ) }{ 2 }
    \left\|
      P_N \sigma_M( P_N x ) - P_N \sigma_M( x )
    \right\|^2_{ HS( U , P_N( H ) ) }
\\ & \leq
    c \,
    \|
      ( P_N x - y )'
    \|_H^2
    -
    c
    \left< 
      ( P_N x - y )' , [ ( P_N x - y ) ( ( P_N x )^2 + ( P_N x ) y + y^2 ) ]' 
    \right>_H
  -
  \tfrac{ 
    \left\| L ( y - P_N x ) \right\|^2_H
  }{ 2 } 
\\ & 
    +
    \tfrac{ ( p - 1 ) \, ( 1 + \varepsilon ) \, \eta^2 }{ 2 }
    \left\|
      P_N x - y 
    \right\|^2_H
  +
    c^2 
    \left\|
      ( P_N x )^3 - x^3 
    \right\|^2_H
  +
  \big[
    c^2
    +
    \tfrac{ ( p - 1 ) \, ( 1 + 1 / \varepsilon ) \, \eta^2 }{ 2 }
  \big]
    \left\|
      ( I - P_N ) x 
    \right\|^2_H
\\ & \leq
    c \,
    \|
      ( P_N x - y )'
    \|_H^2
    -
    c
    \left< 
      [ ( P_N x - y )' ]^2, ( P_N x )^2 + ( P_N x ) y + y^2 
    \right>_H
  -
  \tfrac{ 
    1
  }{ 2 } 
    \left\| L ( y - P_N x ) \right\|^2_H
\\ &
    -
    c
    \left< 
      ( P_N x - y )' , ( P_N x - y ) [ ( P_N x )^2 + ( P_N x ) y + y^2 ]' 
    \right>_H
    +
    \tfrac{ ( p - 1 ) \, ( 1 + \varepsilon ) \, \eta^2 }{ 2 }
    \left\|
      P_N x - y 
    \right\|^2_H
\\ & 
  +
    c^2 
    \left\|
      [ x - P_N x ]
      [ x^2 + ( P_N x )^2 + ( P_N x ) x ]
    \right\|^2_H
  +
  \big[
    c^2
    +
    \tfrac{ ( p - 1 ) \, ( 1 + 1 / \varepsilon ) \, \eta^2 }{ 2 }
  \big]
    \left\|
      ( I - P_N ) x 
    \right\|^2_H
  .
\end{split}
\end{equation}
Hence, we obtain that
for all 
$ \varepsilon \in [0,\infty) $,
$ p \in [2,\infty) $,
$ M, N \in \N $,
$ x \in P_M( H ) $,
$ y \in P_N( H ) $
with $ M > N $
it holds that
\begin{equation}
\begin{split}
&
    \left< 
      P_N x - y , P_N \mu_M( P_N x ) - P_N \mu_M( y ) 
    \right>_H
    +
    \tfrac{ ( p - 1 ) ( 1 + \varepsilon ) }{ 2 }
    \left\|
      P_N \sigma_M( P_N x ) - P_N \sigma_M( y ) 
    \right\|^2_{ HS( U , P_N( H ) ) }
\\ & 
  +
    \left< 
      y - P_N x , P_N \mu_M( P_N x ) - P_N \mu_M( x ) 
    \right>_H
    +
    \tfrac{ ( p - 1 ) \, ( 1 + 1 / \varepsilon ) }{ 2 }
    \left\|
      P_N \sigma_M( P_N x ) - P_N \sigma_M( x )
    \right\|^2_{ HS( U , P_N( H ) ) }
\\ & \leq
    c \,
    \|
      ( P_N x - y )'
    \|_H^2
    -
    \tfrac{ c }{ 2 }
    \left< 
      [ ( P_N x - y )' ]^2, ( P_N x )^2 + y^2 
    \right>_H
  -
  \tfrac{ 
    1
  }{ 2 } 
    \left\| L ( y - P_N x ) \right\|^2_H
\\ &
    -
    c
    \left< 
      ( P_N x - y )' , 
      ( P_N x - y ) 
      \big[ 
        2 ( P_N x )' ( P_N x ) + ( P_N x )' y + ( P_N x ) y' + 2 y' y 
      \big] 
    \right>_H
    +
    \tfrac{ 
      ( p - 1 ) \, ( 1 + \varepsilon ) \, \eta^2
      \,
      \left\|
        P_N x - y 
      \right\|^2_H
    }{ 2 }
\\ & 
  +
    c^2 
    \left\|
      [ x - P_N x ]
      [ x^2 + ( P_N x )^2 + ( P_N x ) x ]
    \right\|^2_H
  +
  \big[
    c^2
    +
    \tfrac{ ( p - 1 ) \, ( 1 + 1 / \varepsilon ) \, \eta^2 }{ 2 }
  \big]
    \left\|
      ( I - P_N ) x 
    \right\|^2_H
\\ & \leq
    c \,
    \|
      ( P_N x - y )'
    \|_H^2
    -
    \tfrac{ c }{ 2 }
    \left< 
      [ ( P_N x - y )' ]^2, ( P_N x )^2 + y^2 
    \right>_H
  -
  \tfrac{ 
    1
  }{ 2 } 
    \left\| L ( y - P_N x ) \right\|^2_H
\\ &
    +
    2 \, c
    \left\|
      ( P_N x - y )' 
      \left( 
        | P_N x | + | y | 
      \right)
    \right\|_H
    \left\|
      \left|
        P_N x - y 
      \right|
      \left(
        | ( P_N x )' | + | y' | 
      \right)
    \right\|_H
    +
    \tfrac{ 
      ( p - 1 ) \, ( 1 + \varepsilon ) \, \eta^2
    }{ 2 }
      \left\|
        P_N x - y 
      \right\|^2_H
\\ & 
  +
    c^2 
    \left\|
      [ x - P_N x ]
      [ x^2 + ( P_N x )^2 + ( P_N x ) x ]
    \right\|^2_H
  +
  \big[
    c^2
    +
    \tfrac{ ( p - 1 ) \, ( 1 + 1 / \varepsilon ) \, \eta^2 }{ 2 }
  \big]
    \left\|
      ( I - P_N ) x 
    \right\|^2_H
  .
\end{split}
\end{equation}
Young's inequality therefore shows that
for all 
$ \varepsilon \in [0,\infty) $,
$ p \in [2,\infty) $,
$ M, N \in \N $,
$ x \in P_M( H ) $,
$ y \in P_N( H ) $
with $ M > N $
it holds that
\begin{equation}
\label{eq:CHC_est1}
\begin{split}
&
    \left< 
      P_N x - y , P_N \mu_M( P_N x ) - P_N \mu_M( y ) 
    \right>_H
    +
    \tfrac{ ( p - 1 ) \, ( 1 + \varepsilon ) }{ 2 }
    \left\|
      P_N \sigma_M( P_N x ) - P_N \sigma_M( y ) 
    \right\|^2_{ HS( U , P_N( H ) ) }
\\ & 
  +
    \left< 
      y - P_N x , P_N \mu_M( P_N x ) - P_N \mu_M( x ) 
    \right>_H
    +
    \tfrac{ ( p - 1 ) \, ( 1 + 1 / \varepsilon ) }{ 2 }
    \left\|
      P_N \sigma_M( P_N x ) - P_N \sigma_M( x )
    \right\|^2_{ HS( U , P_N( H ) ) }
\\ & \leq
    c \,
    \|
      ( P_N x - y )'
    \|_H^2
  -
  \tfrac{ 
    \left\| L ( y - P_N x ) \right\|^2_H
  }{ 2 } 
    +
    4
    c
    \left\|
      \left|
        P_N x - y 
      \right|
      \left(
        | ( P_N x )' | + | y' | 
      \right)
    \right\|_H^2
    +
    \tfrac{ 
      ( p - 1 ) \, ( 1 + \varepsilon ) \, \eta^2
      \,
      \left\|
        P_N x - y 
      \right\|^2_H
    }{ 2 }
\\ &
  +
    c^2 
    \left\|
      x - P_N x
    \right\|^2_H
    \|
      x^2 + ( P_N x )^2 + ( P_N x ) x 
    \|_{ 
      L^{ \infty }( D; \R )
    }
  +
  \big[
    c^2
    +
    \tfrac{ ( p - 1 ) \, ( 1 + 1 / \varepsilon ) \, \eta^2 }{ 2 }
  \big]
    \left\|
      ( I - P_N ) x 
    \right\|^2_H
\\ & \leq
    c \,
    \|
      ( P_N x - y )'
    \|_H^2
  -
  \tfrac{ 
    \left\| L ( y - P_N x ) \right\|^2_H
  }{ 2 } 
    +
    8 \,
    c
    \left\|
      P_N x - y 
    \right\|_H^2
    \left[
      \left\|
        ( P_N x )' 
      \right\|^2_{ 
        L^{ \infty }( D; \R ) 
      }
      +
      \left\|
        y' 
      \right\|^2_{ 
        L^{ \infty }( D; \R ) 
      }
    \right]
\\ &
    +
    \tfrac{ 
      ( p - 1 ) \, ( 1 + \varepsilon ) \, \eta^2
      \,
      \left\|
        P_N x - y 
      \right\|^2_H
    }{ 2 }
  +
  \tfrac{ 3 c^2 }{ 2 }
    \left\|
      x - P_N x
    \right\|^2_H
    \|
      x^2 + ( P_N x )^2 
    \|_{ 
      L^{ \infty }( D; \R )
    }
\\ &
  +
  \big[
    c^2
    +
    \tfrac{ ( p - 1 ) \, ( 1 + 1 / \varepsilon ) \, \eta^2 }{ 2 }
  \big]
    \left\|
      ( I - P_N ) x 
    \right\|^2_H
    .
\end{split}
\end{equation}
In the next step observe that
the Sobolev embedding theorem together
with interpolation shows that
there exist real numbers
$ \hat{ \kappa } \in [0,\infty) $
and
$ ( \kappa_q )_{ q \in (0,\infty) } \subset [0,\infty) $
such that
for all $ x \in D(A) $,
$ q \in ( 0, \infty ) $ it holds that
\begin{equation}
  c
  \left\| x' \right\|_H^2
\leq 
  \hat{ \kappa }
  \left\| x \right\|^2_H
  +
  \tfrac{ 1 }{ 2 }
  \left\| x'' \right\|^2_H
\qquad 
  \text{and}
\qquad 
  8 \, c
  \left\| x' \right\|^2_{ 
    L^{ \infty }( D; \R )
  }
\leq
  \tfrac{ \kappa_q }{ 2 }
  +
  \tfrac{ 1 }{ 2 q }
  \,
  U_1( x )
\end{equation}
(cf., e.g., Theorem~37.5
in Sell \& You~\cite{sy02}).
Putting this into \eqref{eq:CHC_est1}
proves that
for all 
$ \varepsilon \in [0,\infty) $,
$ p \in [2,\infty) $,
$ q \in (0,\infty) $,
$ M, N \in \N $,
$ x \in P_M( H ) $,
$ y \in P_N( H ) $
with $ M > N $
it holds that
\begin{equation}
\begin{split}
&
    \left< 
      P_N x - y , P_N \mu_M( P_N x ) - P_N \mu_M( y ) 
    \right>_H
    +
    \tfrac{ ( p - 1 ) \, ( 1 + \varepsilon ) }{ 2 }
    \left\|
      P_N \sigma_M( P_N x ) - P_N \sigma_M( y ) 
    \right\|^2_{ HS( U , P_N( H ) ) }
\\ & 
  +
    \left< 
      y - P_N x , P_N \mu_M( P_N x ) - P_N \mu_M( x ) 
    \right>_H
    +
    \tfrac{ ( p - 1 ) \, ( 1 + 1 / \varepsilon ) }{ 2 }
    \left\|
      P_N \sigma_M( P_N x ) - P_N \sigma_M( x )
    \right\|^2_{ HS( U , P_N( H ) ) }
\\ & \leq
  \left[
    \hat{ \kappa }
    +
    \kappa_q
    +
    \tfrac{ 
      ( p - 1 ) \, ( 1 + \varepsilon ) \, \eta^2
    }{ 2 }
  \right]
    \|
      P_N x - y 
    \|_H^2
    +
    \tfrac{ 1 }{ 2 q }
    \left\|
      P_N x - y 
    \right\|_H^2
    \left[
      U_1( x )
      +
      U_1( y )
    \right]
\\ &
  +
  \left[
    c^2
    +
    \tfrac{ ( p - 1 ) \, ( 1 + 1 / \varepsilon ) \, \eta^2 }{ 2 }
    +
    \tfrac{ 3 c^2 }{ 2 }
    \|
      x
    \|_{ 
      L^{ \infty }( D; \R )
    }^2
    +
    \tfrac{ 3 c^2 }{ 2 }
    \|
      P_N( x )  
    \|_{ 
      L^{ \infty }( D; \R )
    }^2
  \right]
    \left\|
      ( I - P_N ) x 
    \right\|^2_H
    .
\end{split}
\end{equation}
Combining this and \eqref{eq:CHC_U_est4}
with Proposition~\ref{prop:Galerkin2}
then shows that
for all 
$ \varepsilon \in [0,\infty) $,
$ p \in [2,\infty) $,
$ q, r \in (0,\infty) $,
$ M, N \in \N $
with $ M > N $
and
$
  \frac{ 1 }{ p } + \frac{ 1 }{ q } = \frac{ 1 }{ r }
$
it holds that
\begin{align}
\nonumber
&
  \sup_{ t \in [  \tzero , T] }
  \left\|
    X_t^M - X_t^N
  \right\|_{
    L^r( \Omega; H )
  }
\leq
  \sqrt{ 2 }
  \,
  T^{
    ( \frac{ 1 }{ 2 } - \frac{ 1 }{ p } )
  }
  \exp\!\left(
    \tfrac{ 1 }{ 2 } - \tfrac{ 1 }{ p }
    +
    \left[ 
      \hat{ \kappa }
      +
      \kappa_q
      +
      \tfrac{ 
        ( p - 1 ) \, ( 1 + \varepsilon ) \, \eta^2
      }{ 2 }
    \right]
    T
    +
    \tfrac{
      \beta T
    }{
      q
    }
  \right)
\\ &  
\nonumber
  \cdot
  \left\|
    \sqrt{
      c^2
      +
      \tfrac{ ( p - 1 ) \, ( 1 + 1 / \varepsilon ) \, \eta^2 }{ 2 }
      +
      \tfrac{ 3 c^2 }{ 2 }
      \|
        X^M
      \|_{ 
        L^{ \infty }( D; \R )
      }^2
      +
      \tfrac{ 3 c^2 }{ 2 }
      \|
        P_N( X^M )  
      \|_{ 
        L^{ \infty }( D; \R )
      }^2
    }
    \left\|
      ( I - P_N ) X^M 
    \right\|_H
  \right\|_{
    L^p( [  \tzero , T ] \times \Omega ; \R )
  }
\\ & \cdot
  \left|
  \E\Big[
    e^{
      U_0( X_0^M ) 
    }
  \Big]
  \E\!\left[
    e^{
      U_0( X^N_0 )
    }
  \right] 
  \right|^{
      \frac{ 1 }{ 2 q }
  }
  +
  \sup_{ t \in [  \tzero , T] }
  \left\|
    ( I - P_N ) X_t^M
  \right\|_{
    L^r( \Omega; H ) 
  }
  .
\end{align}
The estimates
\begin{equation}
\begin{split}
&
   \| P_N v \|_{ 
     L^{ \infty }( D ; \R ) 
   }
\leq 
   \sum_{ n = 0 }^{ \infty }
   \left|
     \left< e_{ n + 1 }, v \right>_H
   \right|
   \left\|
     e_{ n + 1 }
   \right\|_{ L^{ \infty }( D; \R ) }
\\ & \leq
   \sqrt{2}
   \left[
     \sum_{ n = 0 }^{ \infty }
     \left( \varrho + \pi^4 n^4 \right)^{ - \frac{ \alpha }{ 4 } }
     \left(
       \left( \varrho + \pi^4 n^4 \right)^{ \frac{ \alpha }{ 4 } }
       \left|
         \left< e_{ n + 1 }, v \right>_H
       \right|
     \right)
   \right]
\\ & 
\leq
   \sqrt{2}
   \left[
     \sum_{ n = 0 }^{ \infty }
     \left( \varrho + \pi^4 n^4 \right)^{ - \frac{ \alpha }{ 2 } }
   \right]^{ \frac{ 1 }{ 2 } }
  \left[
    \sum_{ n = 0 }^{ \infty }
    \left|
      \left( \varrho + \pi^4 n^4 \right)^{ \frac{ \alpha }{ 4 } }
      \left< e_{ n + 1 }, v \right>_H
    \right|^2
  \right]^{
    \frac{ 1 }{ 2 }
  }
\\ & =
   \sqrt{2}
   \left[
     \sum_{ n = 0 }^{ \infty }
     \left( \varrho + \pi^4 n^4 \right)^{ - \frac{ \alpha }{ 2 } }
   \right]^{ \frac{ 1 }{ 2 } }
  \left[
    \sum_{ n = 1 }^{ \infty }
    \left|
      \left< 
        e_n , 
        ( \varrho - A )^{ \frac{ \alpha }{ 4 } } v 
      \right>_H
    \right|^2
  \right]^{
    \frac{ 1 }{ 2 }
  }
  =
   \sqrt{2}
   \left[
     \sum_{ n = 0 }^{ \infty }
     \left( \varrho + \pi^4 n^4 
     \right)^{ - \frac{ \alpha }{ 2 } }
   \right]^{ \frac{ 1 }{ 2 } }
   \| v \|_{ 
     H_{ \alpha / 4 } 
   }
\end{split}
\end{equation}
and
\begin{equation}
\begin{split}
&
   \left\|
     ( I - P_N ) 
     v
   \right\|_H
 \leq
   \left\|
     ( I - P_N ) \,
     ( \varrho - A )^{ - \alpha / 4 }
   \right\|_{ L( H ) }
   \left\| v \right\|_{ H_{ \alpha / 4 } }
 \leq 
   \left( 
     \varrho + \pi^4 N^4
   \right)^{
     - \alpha / 4
   }
   \left\|
     v
   \right\|_{ H_{ \alpha / 4 } }
\\ & \leq 
   \left[
     N^4
     \pi^4
   \right]^{
     - \alpha / 4
   }
   \| v \|_{ 
     H_{ \alpha / 4 } 
   }
 =
   N^{ - \alpha }
   \,
   \pi^{ - \alpha }
   \,
   \| v \|_{ 
     H_{ \alpha / 4 } 
   }
\end{split}
\end{equation}
for all $ N \in \N $,
$ v \in H_{ \alpha / 4 } $,
$ \alpha \in ( 0, \infty ) $
hence prove that
for all 
$ \varepsilon \in [0,\infty) $,
$ p \in [2,\infty) $,
$ \alpha, q, r \in (0,\infty) $,
$ M, N \in \N $
with $ M > N $
and
$
  \frac{ 1 }{ p } + \frac{ 1 }{ q } = \frac{ 1 }{ r }
$
it holds that
\begin{align}
\nonumber
&
  \sup_{ t \in [  \tzero , T] }
  \left\|
    X_t^M - X_t^N
  \right\|_{
    L^r( \Omega; H )
  }
\leq
  N^{ - \alpha } \, 
  \pi^{ - \alpha } 
  \sqrt{ 2 } \,
  T^{
    ( \frac{ 1 }{ 2 } - \frac{ 1 }{ p } )
  }
  \exp\!\left(
    \tfrac{ 1 }{ 2 } - \tfrac{ 1 }{ p }
    +
    \left[ 
      \hat{ \kappa }
      +
      \kappa_q
      +
      \tfrac{ 
        ( p - 1 ) \, ( 1 + \varepsilon ) \, \eta^2
      }{ 2 }
    \right]
    T
    +
    \tfrac{
      \beta T
    }{
      q
    }
  \right)
\\ &  
\nonumber
  \cdot
  \left\|
    \|
      X^M 
    \|_{ H_{ \alpha / 4 } }
    \sqrt{
      c^2
      +
      \tfrac{ ( p - 1 ) \, ( 1 + 1 / \varepsilon ) \, \eta^2 }{ 2 }
      +
      6 c^2 
      \left[
        \smallsum_{ n = 0 }^{ \infty }
        \left( \varrho + \pi^4 n^4 
        \right)^{ - \frac{ \alpha }{ 2 } }
      \right]
      \|
        X^M
      \|_{ 
        H_{ \alpha / 4 }
      }^2
    }
  \right\|_{
    L^p( [  \tzero , T ] \times \Omega ; \R )
  }
\\ & \cdot
\label{eq:CHC_final1}
  \left|
  \E\Big[
    e^{
      U_0( X_0 ) 
    }
  \Big]
  \right|^{
    1 / q
  }
  +
  N^{ - \alpha } \, 
  \pi^{ - \alpha } 
  \left[
  \sup\nolimits_{ t \in [  \tzero , T] }
  \left\|
    X_t^M
  \right\|_{
    L^r( \Omega; H_{ \alpha / 4 } ) 
  }
  \right]
  .
\end{align}
The choice $ \varepsilon = 1 $
in \eqref{eq:CHC_final1}
then shows that
for all 
$ p \in [2,\infty) $,
$ \alpha, q, r \in (0,\infty) $,
$ M, N \in \N $
with $ M > N $
and
$
  \frac{ 1 }{ p } + \frac{ 1 }{ q } = \frac{ 1 }{ r }
$
it holds that
\begin{align}
\nonumber
&
  \sup_{ t \in [  \tzero , T] }
  \left\|
    X_t^M - X_t^N
  \right\|_{
    L^r( \Omega; H )
  }
\leq
  N^{ - \alpha } \, 
  \pi^{ - \alpha } 
  \sqrt{ 2 T } 
  \exp\!\left(
    \tfrac{ 1 }{ 2 } - \tfrac{ 1 }{ p }
    +
    \left[ 
      \hat{ \kappa }
      +
      \kappa_q
      +
      ( p - 1 ) \, \eta^2
      +
      \tfrac{ \beta }{ q }
    \right]
    T
  \right)
\\ &  
  \cdot
    \left[
      \eta \sqrt{ p - 1 } 
      +
      c
      +
      \sqrt{6} \,
      c
      \sqrt{
        \smallsum_{ n = 0 }^{ \infty }
        \left( \varrho + \pi^4 n^4 
        \right)^{ - \alpha / 2 }
      }
    \right]
  \max\!\left( 
    1 ,
    \sup\nolimits_{ t \in [0,T] }
    \|
      X^M_t
    \|^2_{ L^{ 2 p }( \Omega; H_{ \alpha / 4 } ) }
  \right)
\\ & \cdot
\nonumber
  \left|
  \E\Big[
    \exp\!\left(
      \tfrac{ \rho }{ 2 } 
      \, \| ( - \tilde{L} )^{ 1 / 2 } X_0 \|^2_H
      +
      \tfrac{ \hat{ \rho } }{ 2 } \,
      \| \tilde{P} X_0 \|^2_H
    \right)
  \Big]
  \right|^{
    1 / q
  }
  +
  N^{ - \alpha } \, 
  \pi^{ - \alpha } 
  \left[
  \sup\nolimits_{ t \in [  \tzero , T] }
  \left\|
    X_t^M
  \right\|_{
    L^r( \Omega; H_{ \alpha / 4 } ) 
  }
  \right]
  .
\end{align}
This together with the
estimate
$ 
  \pi^{ - \alpha } \sqrt{ 2 T } \exp( \frac{ 1 }{ 2 } ) 
\leq 
  \exp( \frac{ T }{ 2 } )
$
for all $ \alpha \in ( \frac{1}{2}, \infty) $
implies that
for all 
$ p \in [2,\infty) $,
$ \alpha, q, r \in (0,\infty) $,
$ M, N \in \N $
with $ M > N $
and
$
  \frac{ 1 }{ p } + \frac{ 1 }{ q } = \frac{ 1 }{ r }
$
it holds that
\begin{align}
\nonumber
&
  \sup_{ t \in [  \tzero , T] }
  \left\|
    X_t^M - X_t^N
  \right\|_{
    L^r( \Omega; H )
  }
\leq
  N^{ - \alpha }  
  \exp\!\left(
    \big[ 
      \tfrac{ 1 }{ 2 } 
      +
      \hat{ \kappa }
      +
      \kappa_q
      +
      ( p - 1 ) \, \eta^2
      +
      \tfrac{ \beta }{ q }
    \big]
    \,
    T
  \right)
  \left|
  \E\Big[
    \exp\!\left(
      \tfrac{ ( \rho + \hat{ \rho } ) }{ 2 } 
      \, \| X_0 \|^2_H
    \right)
  \Big]
  \right|^{
    1 / q
  }
\\ &  
  \cdot
    \left[
      1 +
      \eta \sqrt{ p - 1 }
      +
      c
      +
      \sqrt{ 6 } \,
      c \,
      \big[
        \smallsum_{ n = 0 }^{ \infty }
        \left( \varrho + \pi^4 n^4 
        \right)^{ - \alpha / 2 }
      \big]^{ 1 / 2 }
    \right]
  \max\!\left( 
    1 ,
    \sup\nolimits_{ t \in [0,T] }
    \|
      X^M_t
    \|^2_{ L^{ 2 p }( \Omega; H_{ \alpha / 4 } ) }
  \right)
  .
\end{align}
Fatou's lemma hence shows that
for all 
$ p \in [2,\infty) $,
$ \alpha, q, r \in (0,\infty) $,
$ N \in \N $
with
$
  \frac{ 1 }{ p } + \frac{ 1 }{ q } = \frac{ 1 }{ r }
$
it holds that
\begin{align}
\nonumber
&
  \sup_{ t \in [  \tzero , T] }
  \left\|
    X_t - X_t^N
  \right\|_{
    L^r( \Omega; H )
  }
\leq
  N^{ - \alpha }  
  \exp\!\left(
    \big[ 
      \tfrac{ 1 }{ 2 } + p \, \eta^2
      +
      \hat{ \kappa }
      +
      \kappa_q
      +
      \tfrac{ \beta }{ q }
    \big]
    \,
    T
  \right)
  \left|
  \E\Big[
    \exp\!\left(
      \tfrac{ ( \rho + \hat{ \rho } ) }{ 2 } 
      \, \| X_0 \|^2_H
    \right)
  \Big]
  \right|^{
    1 / q
  }
\\ &  
  \cdot
    \left[
      1 +
      \eta \sqrt{ p }
      +
      c
      +
      \sqrt{ 6 } \,
      c
      \,
      \big[
        \smallsum_{ n = 0 }^{ \infty }
        \left( \varrho + \pi^4 n^4 
        \right)^{ - \alpha / 2 }
      \big]^{ 1 / 2 }
    \right]
  \max\!\left( 
    1 ,
    \liminf_{ M \to \infty }
    \sup\nolimits_{ t \in [0,T] }
    \|
      X^M_t
    \|^2_{ L^{ 2 p }( \Omega; H_{ \alpha / 4 } ) }
  \right)
  .
\label{eq:CHC.Galerkin.estimate}
\end{align}

\subsubsection{Stochastic Burgers equation}
\label{ssec:stochastic.Burgers.equation}

In this subsection 
assume the setting in Subsection~\ref{ssec:SPDE_setting},
assume that $ D = (0,1) $,
that $ k = 1 $,
that $ A $ is the Laplacian 
with Dirichlet boundary conditions on $ D $,
that is,
$ D( A ) = H^2( D, \R ) \cap H^1_0( D, \R ) $
and
$
  A v = v''
$
for all $ v \in D(A) $
and that
$ \varrho = 0 $,
$ 
  \theta = \frac{ 1 }{ 4 } 
$,
$
  \vartheta = - \frac{ 1 }{ 2 }
$,
let $ c \in \R \backslash \{ 0 \} $,
let
$ e_n \in H $, $ n \in \N $,
be given by
$
  e_n(x) = \sqrt{ 2 } \sin( n \pi x )
$
for all 
$ x \in (0,1) $,
$ n \in \N $
and 
assume that
$
  P_N( v ) = \sum_{ n = 1 }^N
  \left< e_n, v \right>_H e_n
$
for all $ v \in H $,
$ N \in \N $,
that
$
  F( v ) = \tfrac{ c }{ 2 } \, ( v^2 )'
$
for all $ v \in H_{ 1 / 4 } \subset L^4( D; \R ) $
and that
$
  B \colon H \to \HS( U, H )
$
is globally Lipschitz continuous 
with
$
  \eta 
  :=
  \sup_{ x \in H }
  \| B( x ) \|^2_{ \HS( U, H ) }
  \in (0,\infty)
$.
(For example, 
if $ b \colon (0,1) \times \R \to \R $
is a globally bounded function
with a globally bounded continuous derivative,
if $ Q \in L(U) $
is a trace class operator (see, e.g., Appendix~B
in Pr\'{e}v\^{o}t \& R\"{o}ckner~\cite{PrevotRoeckner2007}),
if $ l = 1 $
and if $ B $ satisfies
$ ( B( v ) u )( x ) = b(x,v(x)) \cdot ( \sqrt{ Q } u )( x ) $
for all $ x \in D $, $ u, v \in H = U $,
then 
$ B $ fulfills
$
    \| B \|_{
      \operatorname{Lip}( H, HS( U, H ) )
    }
  :=
  \sup_{ 
    v, w \in H , v \neq w
  }
  \frac{
    \| B(v) - B(w) \|_{ HS( U, H ) }
  }{
    \| v - w \|_H
  }
  < \infty
$
and 
$
  \sup_{ v \in H }
  \| B(v) \|_{ HS( U, H ) }^2
  < \infty
$
and in that case 
\eqref{eq:SPDE} 
is the stochastic Burgers equation
\begin{equation}
  d X_t( x ) 
  =
  \left[ 
    \tfrac{ \partial^2 }{ \partial x^2 } X_t( x )
    +
    c
    \,
    X_t(x) \,
    \tfrac{ \partial }{ \partial x } 
    X_t(x)
  \right] dt
  +
  b(x, X_t(x))
  \sqrt{ Q }
  \,
  dW_t(x)
\end{equation}
for $ x \in (0,1) $, $ t \in [0,T] $
equipped with the Dirichlet boundary conditions
$
  X_t(0) = X_t(1) 
  = 0
$
for $ t \in [0,T] $.)
%
%
%
% The SPDE~\eqref{eq:SPDE} thus reduces to the stochastic
% Burgers equation
% \begin{equation}
%   d X_t( x ) = 
%   \left[ 
%     \tfrac{ \partial^2 }{ \partial x^2 } X_t( x )
%     +
%     c \,
%     X_t(x)
%     \,
%     \tfrac{ \partial }{ \partial x } X_t( x )
%   \right]
%   dt
%   +
%   dW_t( x )
%   ,
% \qquad 
%   X_t( 0 ) = X_t( 1 ) = 0
% \end{equation}
% for $ x \in (0,1) $, $ t \in [0,T] $.
% 
% 
% 
% 
We will now apply Corollary~\ref{cor:Galerkin4}
to the spectral Galerkin approximation processes $ X^N $ and $ X^M $
for $ N, M \in \N $
with $ N < M $
(see \eqref{eq:Galerkin}).
For this note that
for all 
$ M \in \N $, 
$ 
  x \in P_M( H )
$,
$ \rho \in (0,\infty) $
it holds that
\begin{equation}
\label{eq:Burgers_section_estimate0}
\begin{split}
&
  \left< x, \mu_M(x) \right>_H
  +
  \tfrac{ 1 }{ 2 }
  \left\| \sigma_M( x ) \right\|^2_{ HS( U, P_M( H ) ) }
  +
  \tfrac{ \rho }{ 2 } 
  \, \| \sigma_M( x )^* x \|^2_H 
\\ & \leq
  \left< x, A x \right>_H
  +
  \tfrac{ 1 }{ 2 }
  \left\| B( x ) \right\|^2_{ HS( U, H ) }
  +
  \tfrac{ \rho }{ 2 } 
  \, \| B( x )^* x \|^2_H 
\leq
  \tfrac{ \eta }{ 2 }
  +
    \tfrac{ \rho \eta }{ 2 } 
    \, \| x \|^2_H 
  -
    \left\|
      x'
    \right\|^2_H
\\ & =
  \tfrac{ \eta }{ 2 }
  +
    \tfrac{ \rho \eta }{ 2 } 
    \, \| x \|^2_H 
  -
    \tfrac{
      \rho \eta 
    }{
      2 \pi^2
    }
    \left\|
      x'
    \right\|^2_H
  -
  \left[ 
    1 - \tfrac{ \rho \eta }{ 2 \pi^2 }
  \right]
    \left\|
      x'
    \right\|^2_H
\leq
  \tfrac{ \eta }{ 2 }
  -
  \left[ 
    1 - \tfrac{ \rho \eta }{ 2 \pi^2 }
  \right]
    \left\|
      x'
    \right\|^2_H
  .
\end{split}
\end{equation}
In addition, note that
for all $ N, M \in \N $ with $ N < M $,
$ x \in P_M( H ) $,
$ y \in P_N( H ) $,
$ p \in [2,\infty) $,
$ \varepsilon \in (0,\infty) $
it holds that
\begin{equation}
\begin{split}
& \quad
    \left< 
      P_N x - y , \mu_M( P_N x ) - \mu_M( y ) 
    \right>_H
    +
    \tfrac{ ( p - 1 ) \, ( 1 + \varepsilon ) }{ 2 }
    \left\|
      \sigma_M( P_N x ) - \sigma_M( y ) 
    \right\|^2_{ HS( U , P_M( H ) ) }
\\ & 
%\quad
    +
    \left< 
      y - P_N x , P_N \mu_M( P_N x ) - P_N \mu_M( x ) 
    \right>_H
    +
    \tfrac{ ( p - 1 ) \, ( 1 + 1 / \varepsilon ) }{ 2 }
    \left\|
      \sigma_M( P_N x ) - \sigma_M( x )
    \right\|^2_{ HS( U, P_M( H ) ) }
\\ & \leq
   - \| ( - A )^{ 1 / 2 } ( P_N x - y ) \|^2_H
   +
    \left< 
      P_N x - y , F( P_N x ) - F( y ) 
    \right>_H
    +
    \tfrac{ 
      ( p - 1 ) \, ( 1 + \varepsilon ) 
      \,
      \left\|
        B( P_N x ) - B( y ) 
      \right\|^2_{ HS( U, H ) }
    }{ 2 }
\\ & 
%\quad
    +
    \left< 
      y - P_N x , F( P_N x ) - F( x ) 
    \right>_H
    +
    \tfrac{ ( p - 1 ) \, ( 1 + 1 / \varepsilon ) }{ 2 }
    \left\|
      B( P_N x ) - B( x )
    \right\|^2_{ HS( U, H ) }
\\ & \leq
   - \| ( P_N x - y )' \|^2_H
   + \tfrac{ c }{ 4 }
    \left< 
      ( P_N x - y )^2 , ( P_N x + y )'
    \right>_H
    +
    \tfrac{ 
      ( p - 1 ) \, ( 1 + \varepsilon ) 
      \,
      \| B \|^2_{
        \operatorname{Lip}( H, HS( U, H ) )
      }
      \,
      \left\|
        P_N x - y 
      \right\|^2_H
    }{ 2 }
\\ & 
%\quad
    -
    \tfrac{ c }{ 2 }
    \left< 
      ( y - P_N x )', \left( ( P_N - I ) x \right) \left( P_N x + x \right) 
    \right>_H
    +
    \tfrac{ ( p - 1 ) \, ( 1 + 1 / \varepsilon ) }{ 2 }
    \,
    \| B \|^2_{
      \operatorname{Lip}( H, HS( U, H ) )
    }
    \left\|
      ( I - P_N ) x
    \right\|^2_H
  .
\end{split} 
\end{equation}
Young's inequality hence shows that
for all $ N, M \in \N $ with $ N < M $,
$ x \in P_M( H ) $,
$ y \in P_N( H ) $,
$ p \in [2,\infty) $,
$ \varepsilon, \delta \in (0,\infty) $
it holds that
\begin{equation}
\label{eq:Burgers_section_estimate}
\begin{split}
& \quad
    \left< 
      P_N x - y , \mu_M( P_N x ) - \mu_M( y ) 
    \right>_H
    +
    \tfrac{ ( p - 1 ) \, ( 1 + \varepsilon ) }{ 2 }
    \left\|
      \sigma_M( P_N x ) - \sigma_M( y ) 
    \right\|^2_{ HS( U , P_M( H ) ) }
\\ & \quad
    +
    \left< 
      y - P_N x , P_N \mu_M( P_N x ) - P_N \mu_M( x ) 
    \right>_H
    +
    \tfrac{ ( p - 1 ) \, ( 1 + 1 / \varepsilon ) }{ 2 }
    \left\|
      \sigma_M( P_N x ) - \sigma_M( x )
    \right\|^2_{ HS( U , P_M( H ) ) }
\\ & \leq
     \tfrac{ | c | }{ 4 }
     \left\| P_N x - y \right\|_H
     \left\| P_N x - y \right\|_{ L^{ \infty }( (0,1) ; \R ) }
     \left\| ( P_N x + y )' \right\|_H
    +
    \tfrac{ 
      ( p - 1 ) \, ( 1 + \varepsilon ) 
    }{ 2 }
      \,
      \| B \|^2_{
        \operatorname{Lip}( H, HS( U, H ) )
      }
      \left\|
        P_N x - y 
      \right\|^2_H
\\ & \quad
     - 
     \tfrac{ 1 }{ 2 }
     \,
     \| ( P_N x - y )' \|^2_H
    +
    \tfrac{ c^2 }{ 8 }
    \left\|
      \left( ( P_N - I ) x \right) \left( P_N x + x \right) 
    \right\|_H^2
    +
    \tfrac{ ( p - 1 ) \, ( 1 + 1 / \varepsilon ) }{ 2 }
    \,
    \| B \|^2_{
      \operatorname{Lip}( H, HS( U, H ) )
    }
    \left\|
      ( I - P_N ) x
    \right\|^2_H
\\ & \leq
  \left[
     \tfrac{ c^2 \delta }{ 2 }
     \left\| ( P_N x + y )' \right\|_H^2
    +
    \tfrac{ 
      ( p - 1 ) \, ( 1 + \varepsilon ) 
    }{ 2 }
      \,
      \| B \|^2_{
        \operatorname{Lip}( H, HS( U, H ) )
      }
  \right]
      \left\|
        P_N x - y 
      \right\|^2_H
     +
     \tfrac{ 1 }{ 32 \delta }
     \left\| P_N x - y \right\|_{ L^{ \infty }( (0,1) ; \R ) }^2
\\ & \quad
     - 
     \tfrac{ 1 }{ 2 }
     \,
     \| ( P_N x - y )' \|^2_H
  +
  \left[
    \tfrac{ c^2 }{ 8 }
    \left\|
      P_N x + x  
    \right\|_{ L^{ \infty }( (0,1) ; \R ) }^2
    +
    \tfrac{ ( p - 1 ) \, ( 1 + 1 / \varepsilon ) }{ 2 }
    \,
    \| B \|^2_{
      \operatorname{Lip}( H, HS( U, H ) )
    }
  \right]
    \left\|
      ( I - P_N ) x
    \right\|^2_H
    .
\end{split}
\end{equation}
Next let $ \kappa \colon ( 0, \infty ) \to (0,\infty) $
be a strictly decreasing function satisfying
$
  \frac{ 1 }{ 32 r }
  \,
  \| v \|_{ L^{ \infty }( (0,1) ; \R ) }^2
\leq
  \kappa(r) \left\| v \right\|^2_H
  +
  \tfrac{ 1 }{ 2 } 
  \left\| v' \right\|^2_H
$
for all $ v \in D( A ) $,
$ r \in ( 0, \infty ) $
(cf., e.g., Theorem~37.5
in Sell \& You~\cite{sy02}).
Then we get from \eqref{eq:Burgers_section_estimate} that
for all $ N, M \in \N $ with $ N < M $,
$ x \in P_M( H ) $,
$ y \in P_N( H ) $,
$ p \in [2,\infty) $,
$ \varepsilon, \delta \in (0,\infty) $
it holds that
\begin{equation}
\label{eq:Burgers_section_estimate2}
\begin{split}
& \quad
    \left< 
      P_N x - y , \mu_M( P_N x ) - \mu_M( y ) 
    \right>_H
    +
    \tfrac{ 
      ( p - 1 ) \, ( 1 + \varepsilon ) 
    }{ 2 }
    \left\|
      \sigma_M( P_N x ) - \sigma_M( y ) 
    \right\|^2_{ HS( U , P_M( H ) ) }
\\ & \quad
    +
    \left< 
      y - P_N x , P_N \mu_M( P_N x ) - P_N \mu_M( x ) 
    \right>_H
    +
    \tfrac{ 
      ( p - 1 ) \, ( 1 + 1 / \varepsilon ) 
    }{ 2 }
    \left\|
      \sigma_M( P_N x ) - \sigma_M( x )
    \right\|^2_{ HS( U , P_M( H ) ) }
\\ & \leq
  \left[
    \kappa( \delta )
    +
     c^2 \delta 
     \left\| x' \right\|_H^2
    +
     c^2 \delta 
     \left\| y' \right\|_H^2
    +
    \tfrac{ 
      ( p - 1 ) \, ( 1 + \varepsilon ) 
    }{ 2 }
      \,
      \| B \|^2_{
        \operatorname{Lip}( H, HS( U, H ) )
      }
  \right]
      \left\|
        P_N x - y 
      \right\|^2_H
\\ & \quad
  +
  \tfrac{ 1 }{ 2 }
  \left[
    \tfrac{ c^2 }{ 4 }
    \left\|
      P_N x + x  
    \right\|_{ L^{ \infty }( (0,1) ; \R ) }^2
    +
    \left( 
      p - 1 
    \right)
    \left( 
      1 + 1 / \varepsilon
    \right)
    \| B \|^2_{
      \operatorname{Lip}( H, HS( U, H ) )
    }
  \right]
    \left\|
      ( I - P_N ) x
    \right\|^2_H
    .
\end{split}
\end{equation}
Combining 
\eqref{eq:Burgers_section_estimate0}
and
\eqref{eq:Burgers_section_estimate2}
allows us to apply
Corollary~\ref{cor:Galerkin4}
% (with 
% $ \beta = \frac{ \eta }{ 2 } $
% in the notation of Corollary~\ref{cor:Galerkin4})
to obtain that
for all 
$ N, M \in \N $,
$ r, q, \varepsilon, \delta, \rho \in (0,\infty) $,
$ p \in [2,\infty) $
with
$ 
  N < M 
$,
$
  \frac{ 1 }{ p } + \frac{ 1 }{ q } = \frac{ 1 }{ r }
$
and
$
  c^2 \delta \leq \frac{ \rho }{ 2 q } \left[ 1 - \frac{ \rho \eta }{ 2 \pi^2 } \right]
$
it holds that
\begin{align}
\label{eq:Burgers_final_1}
&
  \sup_{ t \in [  \tzero , T] }
  \left\|
    X_t^M - X^N_t
  \right\|_{
    L^r( \Omega; H )
  }
\leq
  \sup_{ t \in [  \tzero , T] }
  \left\|
    ( I - P_N ) X_t^M
  \right\|_{
    L^r( \Omega; H ) 
  }
\nonumber
\\ &
  +
  T^{
    ( \frac{ 1 }{ 2 } - \frac{ 1 }{ p } )
  }
  \,
  \exp\!\left(
    \tfrac{ 1 }{ 2 } 
    - 
    \tfrac{ 1 }{ p }
    +
    \left[ 
      \kappa( \delta ) + 
      \tfrac{
        ( p - 1 ) \, ( 1 + \varepsilon )
      }{
        2
      }
      \,
      \| B \|^2_{ \operatorname{Lip}( H, HS( U, H ) ) }
    \right] 
    T
    +
    \tfrac{
      \eta \rho T
    }{
      2 q
    }
  \right)
  \Big|
  \E\Big[
    e^{
      \frac{ \rho }{ 2 } \left\| X_0^M \right\|^2_H
    }
  \Big]
  \Big|^{
    1 / q
  }
\\ & \cdot
  \left\|
    \Big[
      \tfrac{ | c | }{ 2 }
      \left\|
        P_N X^M + X^M  
      \right\|_{ L^{ \infty }( D ; \R ) }
      +
      \sqrt{
        ( p - 1 ) \, ( 1 + \tfrac{ 1 }{ \varepsilon } )
      }
      \,
      \| B \|_{
        \operatorname{Lip}( H, HS( U, H ) )
      }
    \Big]
    \left\|
      ( I - P_N ) X^M
    \right\|_H
  \right\|_{
    L^p( \llbracket \tzero , T \rrbracket ; \R )
  }
  .
\nonumber
\end{align}
The estimate
\begin{equation}
\begin{split}
&
  \| P_N v \|_{ L^{ \infty }( D ; \R ) }
\leq
  \sum_{ n = 1 }^N
  \left|
    \left< e_n, v \right>_H
  \right|
  \left\|
    e_n
  \right\|_{ L^{ \infty }( D; \R ) }
\leq
  \frac{
    \sqrt{2}
  }{
    \pi^{ \alpha }
  }
  \left[ 
    \sum_{ n = 1 }^N
    n^{ - \alpha }
    \left(
      \pi^{ \alpha }
      \,
      n^{ \alpha }
      \left|
        \left< e_n, v \right>_H
      \right|
    \right)
  \right]
\\ & \leq
  \frac{
    \sqrt{2}
  }{
    \pi^{ \alpha }
  }
  \left[ 
    \sum_{ n = 1 }^{ \infty }
    n^{ - 2 \alpha }
  \right]^{ 1 / 2 }
  \left[
    \sum_{ n = 1 }^{ \infty }
      \pi^{ 2 \alpha }
      \,
      n^{ 2 \alpha }
      \left|
        \left< e_n, v \right>_H
      \right|^2
  \right]^{ 1 / 2 }
=
  \frac{
    \sqrt{2}
    \,
    \| v \|_{ 
      H_{ \alpha / 2 } 
    }
  }{
    \pi^{ \alpha }
  }
  \left[
    \sum_{ n = 1 }^{ \infty }
    n^{ - 2 \alpha }
  \right]^{ 1 / 2 }
\\ & \leq
  \left[
    \smallsum_{ n = 1 }^{ \infty }
    n^{ - 2 \alpha }
  \right]^{ 1 / 2 }
    \| v \|_{ 
      H_{ \alpha / 2 } 
    }
\end{split}
\end{equation}
for all 
$ N \in \N $,
$ v \in H_{ \alpha / 2 } $,
$ \alpha \in ( 0 , \infty ) $
hence proves that
for all 
$ N, M \in \N $,
$ \alpha, r, q, \varepsilon, \delta, \rho \in (0,\infty) $,
$ p \in [2,\infty) $
with
$ 
  N < M 
$,
$
  \frac{ 1 }{ p } + \frac{ 1 }{ q } = \frac{ 1 }{ r }
$
and
$
  c^2 \delta \leq \frac{ \rho }{ 2 q } \left[ 1 - \frac{ \rho \eta }{ 2 \pi^2 } \right]
$
it holds that
\begin{align}
\label{eq:Burgers_final_3}
&
  \sup_{ t \in [  \tzero , T] }
  \left\|
    X_t^M - X^N_t
  \right\|_{
    L^r( \Omega; H )
  }
\leq
  \sup_{ t \in [  \tzero , T] }
  \left\|
    ( I - P_N ) X_t^M
  \right\|_{
    L^r( \Omega; H ) 
  }
\nonumber
\\ &
  +
  T^{
    ( \frac{ 1 }{ 2 } - \frac{ 1 }{ p } )
  }
  \,
  \exp\!\left(
    \tfrac{ 1 }{ 2 } 
    - 
    \tfrac{ 1 }{ p }
    +
    \left[ 
      \kappa( \delta ) + 
      \tfrac{
        ( p - 1 ) \, ( 1 + \varepsilon )
      }{
        2
      }
      \,
      \| B \|^2_{ \operatorname{Lip}( H, HS( U, H ) ) }
    \right] 
    T
    +
    \tfrac{
      \eta \rho T
    }{
      2 q
    }
  \right)
  \Big|
  \E\Big[
    e^{
      \frac{ \rho }{ 2 } \left\| X_0^M \right\|^2_H
    }
  \Big]
  \Big|^{
    1 / q
  }
\\ & \cdot
  \left\|
    \Big[
      | c | 
  \left[
    \smallsum_{ n = 1 }^{ \infty }
    n^{ - 2 \alpha }
  \right]^{ 1 / 2 }
      \| X^M \|_{ H_{ \alpha / 2 } }
      +
      \sqrt{
        ( p - 1 ) \, ( 1 + \tfrac{ 1 }{ \varepsilon } )
      }
      \,
      \| B \|_{
        \operatorname{Lip}( H, HS( U, H ) )
      }
    \Big]
    \left\|
      ( I - P_N ) X^M
    \right\|_H
  \right\|_{
    L^p( \llbracket \tzero , T \rrbracket ; \R )
  }
  .
\nonumber
\end{align}
The estimate
\begin{equation}
\begin{split}
&
   \left\|
     ( I - P_N ) v
   \right\|_H
=
    \left\|
      ( - A )^{ - \alpha / 2 } ( I - P_N )
      ( - A )^{ \alpha / 2 } v
    \right\|_H
\\ & 
 \leq
    \left\|
      ( - A )^{ - \alpha / 2 } ( I - P_N )
    \right\|_{ L( H ) }
    \left\|
      ( - A )^{ \alpha / 2 } v
    \right\|_H
  =
    \left\|
      ( - A )^{ - \alpha / 2 } ( I - P_N )
    \right\|_{ L( H ) }
    \left\|
      v
    \right\|_{ H_{ \alpha / 2 } }
\\ &
 = 
   \left[
     ( N + 1 )^2 \pi^2
   \right]^{
     - \alpha / 2
   }
   \left\|
     v
   \right\|_{ H_{ \alpha / 2 } }
 =
     ( N + 1 )^{ - \alpha } 
     \,
     \pi^{ - \alpha }
   \left\|
     v
   \right\|_{ H_{ \alpha / 2 } }
\leq 
   N^{ - \alpha }
   \,
   \pi^{ - \alpha }
   \left\|
     v
   \right\|_{ H_{ \alpha / 2 } }
\end{split}
\end{equation}
for all $ N \in \N $,
$ v \in H_{ \alpha / 2 } $,
$ \alpha \in ( 0, \infty ) $
therefore shows that 
for all 
$ N, M \in \N $,
$ \alpha, r, q, \varepsilon, \delta, \rho \in (0,\infty) $,
$ p \in [2,\infty) $
with
$ 
  N < M 
$,
$
  \frac{ 1 }{ p } + \frac{ 1 }{ q } = \frac{ 1 }{ r }
$
and
$
  c^2 \delta \leq \frac{ \rho }{ 2 q } \left[ 1 - \frac{ \rho \eta }{ 2 \pi^2 } \right]
$
it holds that
\begin{align}
\label{eq:Burgers_final_4}
&
  \sup_{ t \in [  \tzero , T] }
  \left\|
    X_t^M - X^N_t
  \right\|_{
    L^r( \Omega; H )
  }
\leq
  N^{ - \alpha } \, \pi^{ - \alpha } 
  \left[
  \sup\nolimits_{ t \in [  \tzero , T] }
  \left\|
    X_t^M
  \right\|_{
    L^r( \Omega; H_{ \alpha / 2 } ) 
  }
  \right]
\nonumber
\\ &
  +
  N^{ - \alpha } \, \pi^{ - \alpha } \,
  T^{
    \frac{ 1 }{ 2 } 
  }
  \exp\!\left(
    \tfrac{ 1 }{ 2 } 
    - 
    \tfrac{ 1 }{ p }
    +
    \left[ 
      \kappa( \delta ) + 
      \tfrac{
        ( p - 1 ) \, ( 1 + \varepsilon )
      }{
        2
      }
      \,
      \| B \|^2_{ \operatorname{Lip}( H, HS( U, H ) ) }
    \right] 
    T
    +
    \tfrac{
      \eta \rho T
    }{
      2 q
    }
  \right)
  \Big|
  \E\Big[
    e^{
      \frac{ \rho }{ 2 } \left\| X_0 \right\|^2_H
    }
  \Big]
  \Big|^{
    1 / q
  }
\\ & \cdot
    \Big[
      | c | 
  \left[
    \smallsum_{ n = 1 }^{ \infty }
    n^{ - 2 \alpha }
  \right]^{ 1 / 2 }
      +
      \sqrt{
        ( p - 1 ) \, ( 1 + 1 / \varepsilon )
      }
      \,
      \| B \|_{
        \operatorname{Lip}( H, HS( U, H ) )
      }
    \Big]
  \max\!\left(
    1
    ,
    \sup\nolimits_{ t \in [0,T] }
    \|
      X^M_t
    \|_{
      L^{ 2 p }( \Omega ; H_{ \alpha / 2 } )
    }^2
  \right)
  .
\nonumber
\end{align}
The estimate
$
  T^{ 1 / 2 }
  \leq 
  \exp( \frac{ ( T - 1 ) }{ 2 } )
$
and the choice
$ \varepsilon = 1 $
in \eqref{eq:Burgers_final_4}
imply 
that
for all 
$ N, M \in \N $,
$ \alpha, r, q, \delta, \rho \in (0,\infty) $,
$ p \in [2,\infty) $
with
$ 
  N < M 
$,
$
  \frac{ 1 }{ p } + \frac{ 1 }{ q } = \frac{ 1 }{ r }
$
and
$
  c^2 \delta \leq \frac{ \rho }{ 2 q } \left[ 1 - \frac{ \rho \eta }{ 2 \pi^2 } \right]
$
it holds that
\begin{align}
\label{eq:Burgers_final_5}
&
  \sup_{ t \in [  \tzero , T] }
  \left\|
    X_t^M - X^N_t
  \right\|_{
    L^r( \Omega; H )
  }
\leq
  N^{ - \alpha } 
  \exp\!\left(
    \left[ 
      \tfrac{ q + \eta \rho }{ 2 q }
      +
      \kappa( \delta ) + 
        p
      \,
      \| B \|^2_{ \operatorname{Lip}( H, HS( U, H ) ) }
    \right] 
    T
  \right)
  \Big|
  \E\Big[
    e^{
      \frac{ \rho }{ 2 } \left\| X_0 \right\|^2_H
    }
  \Big]
  \Big|^{
    1 / q
  }
\nonumber
\\ & \cdot
    \Big[
      1 +
      | c | 
  \left[
    \smallsum_{ n = 1 }^{ \infty }
    n^{ - 2 \alpha }
  \right]^{ 1 / 2 }
      +
      \sqrt{
        p
      }
      \,
      \| B \|_{
        \operatorname{Lip}( H, HS( U, H ) )
      }
    \Big]
  \max\!\left(
    1
    ,
    \sup\nolimits_{ t \in [0,T] }
    \|
      X^M_t
    \|_{
      L^{ 2 p }( \Omega ; H_{ \alpha / 2 } )
    }^2
  \right)
  .
\end{align}
Fatou's lemma hence shows
that
for all 
$ N \in \N $,
$ \alpha, r, q \in (0,\infty) $,
$ \rho \in ( 0, \frac{ 2 \pi^2 }{ \eta } ) $,
$ p \in [2,\infty) $
with
$
  \frac{ 1 }{ p } + \frac{ 1 }{ q } = \frac{ 1 }{ r }
$
it holds that
\begin{align}
\label{eq:Burgers.Galerkin.estimate}
&
  \sup_{ t \in [  \tzero , T] }
  \left\|
    X_t - X^N_t
  \right\|_{
    L^r( \Omega; H )
  }
% \nonumber
% \\ & 
\leq
  \exp\!\left(
      \tfrac{ ( q + \eta \rho ) T }{ 2 q }
      +
      \kappa\big( 
        \tfrac{ \rho \, [ 2 \pi^2 - \rho \eta ] }{ 4 q c^2 \pi^2 } 
      \big) T
      + 
        p T
      \| B \|^2_{ \operatorname{Lip}( H, HS( U, H ) ) }
  \right)
  \Big|
  \E\Big[
    e^{
      \frac{ \rho }{ 2 } \left\| X_0 \right\|^2_H
    }
  \Big]
  \Big|^{
    1 / q
  }
\nonumber
\\ & \cdot
% \nonumber
    \left[
      1 +
      | c | 
      \sqrt{
    \smallsum_{ n = 1 }^{ \infty }
    \tfrac{ 1 }{ n^{ 2 \alpha } }
    }
      +
      \sqrt{
        p
      }
      \,
      \| B \|_{
        \operatorname{Lip}( H, HS( U, H ) )
      }
    \right]
  \max\!\left(
    1
    ,
    \liminf_{ M \to \infty }
    \sup_{ t \in [0,T] }
    \|
      X^M_t
    \|_{
      L^{ 2 p }( \Omega ; H_{ \alpha / 2 } )
    }^2
  \right)
  N^{ - \alpha }
  .
\end{align}

\subsection{Analysis of SDEs with small noise}
\label{ssec:small.noise}

In this subsection we use 
Corollary~\ref{cor:norm2}
to study the perturbations 
of deterministic ordinary differential equations
and deterministic 
partial differential equations
by a small noise term.

\begin{corollary}
\label{cor:small_noise}
Assume the setting in Subsection~\ref{sec:setting},
let $ \varepsilon \in [0,\infty) $,
$ \mu \in \mathcal{L}^0( \mathcal{O}; H ) $, 
$ \sigma \in \mathcal{L}^0( \mathcal{O}; HS( U, H ) ) $, 
let $ \tau \colon \Omega \to [  \tzero  , T ] $ 
be a stopping time,
let
$ X, Y \colon [  \tzero , T ] \times \Omega \to \mathcal{O} $
be adapted stochastic processes
with c.s.p.\ and
{\small
$
  \int_{  \tzero  }^{ \tau }
  \frac{
    1
  }{
    \left\| X_s - Y_s \right\|^2_H
  }
    [
      \left< 
        X_s - Y_s , \mu( X_s ) - \mu( Y_s ) 
      \right>_H
    ]^+
  \,
  ds
  +
  \int_{  \tzero  }^{ T }
  \| \sigma( Y_s ) \|^2_{ HS( U, H ) }
    +
  \| \mu( X_s ) \|_H
  +
  \| \mu( Y_s ) \|_H
  \,
  ds
  < \infty
$}
$ \P $-a.s.,
$
  X_t 
  = 
  X_{  \tzero  } 
  +
  \int_{  \tzero  }^t \mu( X_s ) \, ds
$
$ \P $-a.s.\ and
$
  Y_t 
  = 
  X_{  \tzero  } 
  +
  \int_{  \tzero  }^t \mu( Y_s ) \, ds
  +
  \int_{  \tzero  }^t \varepsilon \, \sigma( Y_s ) \, dW_s
$
$ \P $-a.s.\ for all
$ t \in [ \tzero ,T] $.
Then it holds for all 
$ \rho, r \in (0,\infty) $,
$ p \in [2,\infty) $, $ q \in (0,\infty] $
with 
$ \frac{ 1 }{ p } + \frac{ 1 }{ q } = \frac{ 1 }{ r } $
that
% {\footnotesize
\begin{align}
\label{eq:cor_small_noise}
&
  \big\| 
    X_{ \tau } - Y_{ \tau }
  \big\|_{
    L^r( \Omega; H )
  }
\\ & \leq
\nonumber
      \varepsilon \,
      \rho^{ ( \frac{ 1 }{ 2 } - \frac{ 1 }{ p } ) }
      \sqrt{ p - 1 }
        \left\|
          \sigma( Y )
        \right\|_{ 
          L^p( \llbracket  \tzero  , \tau \rrbracket ; HS( U, H ) )
        }
  \left\|
  \exp\!\left(
    \smallint_{  \tzero  }^{ \tau }
  \Big[
  \tfrac{
    \left< 
      X_s - Y_s , \mu( X_s ) - \mu( Y_s ) 
    \right>_H
  }{
    \left\| X_s - Y_s \right\|^2_H
  }
      +
  \tfrac{
      ( \frac{ 1 }{ 2 } - \frac{ 1 }{ p } )
  }{
      \rho
  }
  \Big]^+
  \!
    ds
  \right)
  \right\|_{
    L^q( \Omega; \R )
  }
  .
\end{align}%}
\end{corollary}

Corollary~\ref{cor:small_noise}
follows immediately from
Corollary~\ref{cor:norm2}.
If the suitable exponential integrability properties 
of the processes $ X $ and $ Y $ in Corollary~\ref{cor:small_noise}
are known (see, e.g., Corollary~2.4 in Cox et al.~\cite{CoxHutzenthalerJentzen2013}), 
then the right-hand side of \eqref{eq:cor_small_noise}
can be further estimated in an appropriate way.
This is the subject of the next result, Corollary~\ref{cor:small_noise2}.
Corollary~\ref{cor:small_noise2} can be applied
to a number of nonlinear ordinary
and partial differential equation
perturbed by a small 
noise term such as the 
examples in Subsections~\ref{ssec:Lorenz}--\ref{ssec:Langevin.dynamics}
as well as the examples in Subsections~\ref{ssec:Cahn_Hilliard}--\ref{ssec:stochastic.Burgers.equation}.

\begin{corollary}
\label{cor:small_noise2}
Assume the setting in Subsection~\ref{sec:setting},
let 
$ 
  \varepsilon, \alpha, \hat{ \alpha }, \beta, \hat{ \beta } 
  \in [0,\infty) 
$,
$ p \in [2,\infty) $,
$ r, q_0, \hat{q}_0, q_1, \hat{q}_1 \in (0,\infty] $,
$ U_1, \hat{U}_1 \in C( \mathcal{O}, [0,\infty) ) $,
$ U_0, \hat{U}_0 \in C^2( O, [0,\infty) ) $,
$ \mu \in \mathcal{L}^0( \mathcal{O}, H ) $, 
$ \sigma \in \mathcal{L}^0( \mathcal{O}, HS( U, H ) ) $
satisfy
$
  \frac{ 1 }{ q_0 } + \frac{ 1 }{ \hat{q}_0 }
  +
  \frac{ 1 }{ q_1 } + \frac{ 1 }{ \hat{q}_1 }
  +
  \frac{ 1 }{ p }
  =
  \frac{ 1 }{ r }
$
and
\begin{equation*}
%\label{eq:small_noise_assumption}
\begin{array}{c}
  U_0'(x) \, \mu( x )
  +
  U_1( x )
\leq
  \alpha \, U_0( x )
  +
  \beta
  ,
\\[1ex]
  \hat{U}_0'(x) \, \mu( x )
  +
  \tfrac{ \varepsilon^2 }{ 2 }
  \operatorname{tr}\!\big(
    \sigma( x ) \sigma( x )^*
    ( \operatorname{Hess} \hat{U}_0)( x )
  \big)
%   +
%   ( \mathcal{G}_{ \mu, \sigma } \hat{U}_0 )( x )
  +
  \tfrac{ \varepsilon^2 }{ 2 } 
  \, \| \sigma( x )^* ( \nabla \hat{U}_0 )( x ) \|^2_H 
  +
  \hat{U}_1( x )
\leq
  \hat{ \alpha } \, \hat{U}_0( x )
  +
  \hat{ \beta }
  ,
\\[1ex]
  \left< 
    x - y
    , 
    \mu( x ) - \mu( y )
  \right>_H
\leq
  \left[ 
    c
    +
    \tfrac{
      U_0( x ) 
    }{
      q_0 T e^{ \alpha_0 T }
    }
    +
    \tfrac{
      U_1( x )
    }{
      q_1 e^{ \alpha_1 T }
    }
    +
    \tfrac{
      \hat{U}_0( y )
    }{
      \hat{ q }_0 T e^{ \hat{ \alpha }_0 T }
    }
    +
    \tfrac{
      \hat{U}_1( y )
    }{
      \hat{q}_1 e^{ \hat{\alpha}_1 T }
    }
  \right]
  \left\| x - y \right\|^2_H  
\end{array}
\end{equation*}
for all $ x, y \in \mathcal{O} $,
let
$ X, Y \colon [ 0, T ] \times \Omega \to \mathcal{O} $
be predictable stochastic processes
with 
$
  \E\big[
    e^{
      U_0( X_0 )
    }
    +
    e^{
      \hat{U}_0( X_0 )
    }
  \big]
  < \infty
$,
$
  \int_{ 0 }^T
  \| \mu( X_s ) \|_H
  +
  \| \mu( Y_s ) \|_H
  +
  \| \sigma( Y_s ) \|^2_{ HS( U, H ) }
  \,
  ds
  < \infty
$
$ \P $-a.s.,
$
  X_t 
  = 
  X_{ 0 } 
  +
  \int_{ 0 }^t \mu( X_s ) \, ds
$
$ \P $-a.s.\ and
$
  Y_t 
  = 
  X_{ 0 } 
  +
  \int_{ 0 }^t \mu( Y_s ) \, ds
  +
  \int_{ 0 }^t \varepsilon \, \sigma( Y_s ) \, dW_s
$
$ \P $-a.s.\ for all
$ t \in [0,T] $.
Then
\begin{equation}
\begin{split}
&
  \sup_{ t \in [0,T] }
  \big\| X_t - Y_t
  \big\|_{
    L^r( \Omega; H )
  }
\leq
  \varepsilon \,
  T^{ ( \frac{ 1 }{ 2 } - \frac{ 1 }{ p } ) }
  \exp\!\left(
      \tfrac{ 1 }{ 2 } - \tfrac{ 1 }{ p } 
    +
    \smallint_{0}^T
    c
    +
    \tfrac{ 
      \beta 
    }{ 
      e^{ \alpha s } 
    }
    \!
    \left[ 
      \tfrac{ 1 }{ q_0 }
      +
      \tfrac{ 1 }{ q_1 }
    \right]
    +
    \tfrac{ 
      \hat{ \beta }
    }{ 
      e^{ \hat{ \alpha } s } 
    }
    \!
    \left[ 
      \tfrac{ 1 }{ \hat{q}_0 }
      +
      \tfrac{ 1 }{ \hat{q}_1 }
    \right]
    ds
  \right)
\\ & \quad \cdot
      \sqrt{ p - 1 }
        \left\|
          \sigma( Y )
        \right\|_{ 
          L^p( [ 0 , T ] \times \Omega ; HS( U, H ) )
        }
  \left|
  \E\Big[
    e^{ U_0( X_0 ) }
  \Big]
  \right|^{
    \left[ 
      \frac{ 1 }{ q_0 } + \frac{ 1 }{ q_1 } 
    \right]
  }
  \left|
  \E\!\left[
    e^{
      \hat{U}_0( X_0 ) 
    }
  \right]
  \right|^{
    \left[
      \frac{ 1 }{ \hat{q}_0 }
      +
      \frac{ 1 }{ \hat{q}_1 }
    \right]
  }
  .
\end{split}
\end{equation}
\end{corollary}

\begin{proof}[Proof 
of Corollary~\ref{cor:small_noise2}]
Throughout this proof let $ q \in (0,\infty] $
be a real number given by
$
  \frac{ 1 }{ q_0 } + \frac{ 1 }{ \hat{q}_0 } +
  \frac{ 1 }{ q_1 } + \frac{ 1 }{ \hat{q}_1 }
  =
  \frac{ 1 }{ q }
$.
We intend to prove Corollary~\ref{cor:small_noise2}
through an application of 
Corollary~\ref{cor:small_noise}.
To do so, we need to verify the assumptions
of Corollary~\ref{cor:small_noise}.
For this observe that
the assumptions of Corollary~\ref{cor:small_noise2}
and H\"{o}lder's inequality 
ensure that
\begin{equation}
\begin{split}
&
  \left\|
  \exp\!\left(
    \smallint_{0}^T
  \Big[
  \tfrac{
    \left< 
      X_s - Y_s , \mu( X_s ) - \mu( Y_s ) 
    \right>_H
  }{
    \left\| X_s - Y_s \right\|^2_H
  }
  \Big]^+
    ds
  \right)
  \right\|_{
    L^q( \Omega; \R )
  }
\\ & \leq 
  e^{ c T }
  \left\|
  \exp\!\left(
    \smallint_{0}^T
    \tfrac{ 
      U_0( X_s ) 
    }{
      q_0 T e^{ \alpha T }
    }
    +
    \tfrac{ 
      U_1( X_s ) 
    }{
      q_1 e^{ \alpha T }
    }
    +
    \tfrac{ 
      \hat{U}_0( Y_s ) 
    }{
      \hat{q}_0 T e^{ \hat{\alpha} T }
    }
    +
    \tfrac{ 
      \hat{U}_1( Y_s ) 
    }{
      \hat{q}_1 e^{ \hat{\alpha} T }
    }
    \,
    ds
  \right)
  \right\|_{
    L^q( \Omega; \R )
  }
\\ & \leq 
  e^{ c T }
  \left\|
  \exp\!\left(
    \smallint_{0}^T
    \tfrac{ 
      U_0( X_s ) 
    }{
      q_0 T e^{ \alpha T }
    }
    \,
    ds
  \right)
  \right\|_{
    L^{ q_0 }( \Omega; \R )
  }
  \left\|
  \exp\!\left(
    \smallint_{0}^T
    \tfrac{ 
      U_1( X_s ) 
    }{
      q_1 e^{ \alpha T }
    }
    \,
    ds
  \right)
  \right\|_{
    L^{ q_1 }( \Omega; \R )
  }
\\ & \quad \cdot
  \left\|
  \exp\!\left(
    \smallint_{0}^T
    \tfrac{ 
      \hat{U}_0( Y_s ) 
    }{
      \hat{q}_0 T e^{ \hat{\alpha} T }
    }
    \,
    ds
  \right)
  \right\|_{
    L^{ \hat{q}_0 }( \Omega; \R )
  }
  \left\|
  \exp\!\left(
    \smallint_{0}^T
    \tfrac{ 
      \hat{U}_1( Y_s ) 
    }{
      \hat{q}_1 e^{ \hat{\alpha} T }
    }
    \,
    ds
  \right)
  \right\|_{
    L^{ \hat{q}_1 }( \Omega; \R )
  }
  .
\end{split}
\end{equation}
A simple consequence of Jensen's inequality 
(see, e.g., inequality~(19) in Li~\cite{Li1994} and 
Lemma~2.21 in Cox et al.~\cite{CoxHutzenthalerJentzen2013})
together with nonnegativity of $ U_0 $ and $ \hat{U}_0 $
hence proves that
\begin{equation}
\begin{split}
&
  \left\|
  \exp\!\left(
    \smallint_{0}^T
  \Big[
  \tfrac{
    \left< 
      X_s - Y_s , \mu( X_s ) - \mu( Y_s ) 
    \right>_H
  }{
    \left\| X_s - Y_s \right\|^2_H
  }
  \Big]^+
    ds
  \right)
  \right\|_{
    L^q( \Omega; \R )
  }
\\ & \leq 
  e^{ c T }
  \sup_{ t \in [0,T] }
  \left\|
  \exp\!\left(
    \tfrac{ 
      U_0( X_t ) 
    }{
      q_0 e^{ \alpha t }
    }
  \right)
  \right\|_{
    L^{ q_0 }( \Omega; \R )
  }
  \left\|
  \exp\!\left(
    \tfrac{ 
      U_0( X_T ) 
    }{
      q_1 e^{ \alpha T }
    }
    +
    \smallint_{0}^T
    \tfrac{ 
      U_1( X_s ) 
    }{
      q_1 e^{ \alpha s }
    }
    \,
    ds
  \right)
  \right\|_{
    L^{ q_1 }( \Omega; \R )
  }
\\ & \quad \cdot
  \sup_{ t \in [0,T] }
  \left\|
  \exp\!\left(
    \tfrac{ 
      \hat{U}_0( Y_t ) 
    }{
      \hat{q}_0 e^{ \hat{\alpha} t }
    }
  \right)
  \right\|_{
    L^{ \hat{q}_0 }( \Omega; \R )
  }
  \left\|
  \exp\!\left(
    \tfrac{ 
      \hat{U}_0( Y_T ) 
    }{
      \hat{q}_1 e^{ \hat{\alpha} T }
    }
    +
    \smallint_{0}^T
    \tfrac{ 
      \hat{U}_1( Y_s ) 
    }{
      \hat{q}_1 e^{ \hat{\alpha} s }
    }
    \,
    ds
  \right)
  \right\|_{
    L^{ \hat{q}_1 }( \Omega; \R )
  }
  .
\end{split}
\end{equation}
The nonnegativity of $ U_1 $ and $ \hat{U}_1 $ therefore shows that
\begin{equation}
\begin{split}
&
  \left\|
  \exp\!\left(
    \smallint_{0}^T
  \Big[
  \tfrac{
    \left< 
      X_s - Y_s , \mu( X_s ) - \mu( Y_s ) 
    \right>_H
  }{
    \left\| X_s - Y_s \right\|^2_H
  }
  \Big]^+
    ds
  \right)
  \right\|_{
    L^q( \Omega; \R )
  }
\leq
  \exp\!\left(
    c T
    +
    \smallint_{0}^T
    \tfrac{ 
      \beta 
    }{ 
      q_0 e^{ \alpha s } 
    }
    +
    \tfrac{ 
      \beta 
    }{ 
      q_1 e^{ \alpha s } 
    }
    +
    \tfrac{ 
      \hat{ \beta }
    }{ 
      \hat{ q }_0 e^{ \hat{ \alpha } s } 
    }
    +
    \tfrac{ 
      \hat{ \beta }
    }{ 
      \hat{ q }_1 e^{ \hat{ \alpha } s } 
    }
    \, ds
  \right)
\\ & \cdot
  \sup_{ t \in [0,T] }
  \left\|
  \exp\!\left(
    \tfrac{ 
      U_0( X_t ) 
    }{
      q_0 e^{ \alpha t }
    }
    +
    \smallint_{0}^t
    \tfrac{ 
      U_1( X_s ) - \beta 
    }{
      q_0 e^{ \alpha s }
    }
    \,
    ds
  \right)
  \right\|_{
    L^{ q_0 }( \Omega; \R )
  }
  \left\|
  \exp\!\left(
    \tfrac{ 
      U_0( X_T ) 
    }{
      q_1 e^{ \alpha T }
    }
    +
    \smallint_{0}^T
    \tfrac{ 
      U_1( X_s ) - \beta
    }{
      q_1 e^{ \alpha s }
    }
    \,
    ds
  \right)
  \right\|_{
    L^{ q_1 }( \Omega; \R )
  }
\\ & \cdot
  \sup_{ t \in [0,T] }
  \left\|
  \exp\!\left(
    \tfrac{ 
      \hat{U}_0( Y_t ) 
    }{
      \hat{q}_0 e^{ \hat{\alpha} t }
    }
    +
    \smallint_{0}^t
    \tfrac{ 
      \hat{U}_1( Y_s ) - \hat{ \beta }
    }{
      \hat{q}_0 e^{ \hat{\alpha} s }
    }
    \,
    ds
  \right)
  \right\|_{
    L^{ \hat{q}_0 }( \Omega; \R )
  }
  \left\|
  \exp\!\left(
    \tfrac{ 
      \hat{U}_0( Y_T ) 
    }{
      \hat{q}_1 e^{ \hat{\alpha} T }
    }
    +
    \smallint_{0}^T
    \tfrac{ 
      \hat{U}_1( Y_s ) - \hat{ \beta }
    }{
      \hat{q}_1 e^{ \hat{\alpha} s }
    }
    \,
    ds
  \right)
  \right\|_{
    L^{ \hat{q}_1 }( \Omega; \R )
  }
  .
\end{split}
\end{equation}
This, in turn, implies that
\begin{equation}
\begin{split}
&
  \left\|
  \exp\!\left(
    \smallint_{0}^T
  \Big[
  \tfrac{
    \left< 
      X_s - Y_s , \mu( X_s ) - \mu( Y_s ) 
    \right>_H
  }{
    \left\| X_s - Y_s \right\|^2_H
  }
  \Big]^+
    ds
  \right)
  \right\|_{
    L^q( \Omega; \R )
  }
\leq
  \exp\!\left(
    c T
    +
    \smallint_{0}^T
    \tfrac{ 
      \beta 
    }{ 
      e^{ \alpha s } 
    }
    \!
    \left[ 
      \tfrac{ 1 }{ q_0 }
      +
      \tfrac{ 1 }{ q_1 }
    \right]
    +
    \tfrac{ 
      \hat{ \beta }
    }{ 
      e^{ \hat{ \alpha } s } 
    }
    \!
    \left[ 
      \tfrac{ 1 }{ \hat{q}_0 }
      +
      \tfrac{ 1 }{ \hat{q}_1 }
    \right]
    ds
  \right)
\\ & \cdot
  \sup_{ t \in [0,T] }
  \left|
  \E\!\left[
  \exp\!\left(
    \tfrac{ 
      U_0( X_t ) 
    }{
      e^{ \alpha t }
    }
    +
    \smallint_{0}^t
    \tfrac{ 
      U_1( X_s ) - \beta 
    }{
      e^{ \alpha s }
    }
    \,
    ds
  \right)
  \right]
  \right|^{
    \left[ 
      \frac{ 1 }{ q_0 } + \frac{ 1 }{ q_1 } 
    \right]
  }
\\ & \cdot
  \sup_{ t \in [0,T] }
  \left|
  \E\!\left[
  \exp\!\left(
    \tfrac{ 
      \hat{U}_0( Y_t ) 
    }{
      e^{ \hat{\alpha} t }
    }
    +
    \smallint_{0}^t
    \tfrac{ 
      \hat{U}_1( Y_s ) - \hat{ \beta }
    }{
      e^{ \hat{\alpha} s }
    }
    \,
    ds
  \right)
  \right]
  \right|^{
    \left[
      \frac{ 1 }{ \hat{q}_0 }
      +
      \frac{ 1 }{ \hat{q}_1 }
    \right]
  }
  .
\end{split}
\end{equation}
Corollary~2.4 in 
Cox et al.~\cite{CoxHutzenthalerJentzen2013}
therefore shows that
\begin{equation}
\label{eq:exp_moment_small_noise}
\begin{split}
&
  \left\|
  \exp\!\left(
    \smallint_{0}^T
  \Big[
  \tfrac{
    \left< 
      X_s - Y_s , \mu( X_s ) - \mu( Y_s ) 
    \right>_H
  }{
    \left\| X_s - Y_s \right\|^2_H
  }
  \Big]^+
    ds
  \right)
  \right\|_{
    L^q( \Omega; \R )
  }
\\ & \leq
  \exp\!\left(
    c T
    +
    \smallint_{0}^T
    \tfrac{ 
      \beta 
    }{ 
      e^{ \alpha s } 
    }
    \!
    \left[ 
      \tfrac{ 1 }{ q_0 }
      +
      \tfrac{ 1 }{ q_1 }
    \right]
    +
    \tfrac{ 
      \hat{ \beta }
    }{ 
      e^{ \hat{ \alpha } s } 
    }
    \!
    \left[ 
      \tfrac{ 1 }{ \hat{q}_0 }
      +
      \tfrac{ 1 }{ \hat{q}_1 }
    \right]
    ds
  \right)
  \left|
  \E\!\left[
    e^{ U_0( X_0 ) }
  \right]
  \right|^{
    \left[ 
      \frac{ 1 }{ q_0 } + \frac{ 1 }{ q_1 } 
    \right]
  }
  \left|
  \E\!\left[
    e^{
      \hat{U}_0( Y_0 ) 
    }
  \right]
  \right|^{
    \left[
      \frac{ 1 }{ \hat{q}_0 }
      +
      \frac{ 1 }{ \hat{q}_1 }
    \right]
  }
  .
\end{split}
\end{equation}
The assumption 
$
  \E\big[
    e^{
      U_0( X_0 )
    }
    +
    e^{
      \hat{U}_0( Y_0 )
    }
  \big] 
  < \infty
$
hence shows that
\begin{equation}
    \smallint_{0}^T
  \Big[
  \tfrac{
    \left< 
      X_s - Y_s , \mu( X_s ) - \mu( Y_s ) 
    \right>_H
  }{
    \left\| X_s - Y_s \right\|^2_H
  }
  \Big]^+
    ds
  < \infty
\end{equation}
$ \P $-a.s.\ Corollary~\ref{cor:small_noise}
can thus be applied to obtain that
for all $ \rho \in (0,\infty) $ it holds that
\begin{align}
\label{eq:proof_small_noise2}
&
  \sup_{ t \in [0,T] }
  \big\| X_t - Y_t
  \big\|_{
    L^r( \Omega; H )
  }
\\ & \leq
\nonumber
      \varepsilon \,
      \rho^{ ( \frac{ 1 }{ 2 } - \frac{ 1 }{ p } ) }
      \sqrt{ p - 1 }
        \left\|
          \sigma( Y )
        \right\|_{ 
          L^p( [ 0 , T ] \times \Omega ; HS( U, H ) )
        }
  \left\|
  \exp\!\left(
    \smallint_{0}^T
  \Big[
  \tfrac{
    \left< 
      X_s - Y_s , \mu( X_s ) - \mu( Y_s ) 
    \right>_H
  }{
    \left\| X_s - Y_s \right\|^2_H
  }
  \Big]^+
  \!
      +
  \tfrac{
      ( \frac{ 1 }{ 2 } - \frac{ 1 }{ p } )
  }{
      \rho
  }
  \,
    ds
  \right)
  \right\|_{
    L^q( \Omega; \R )
  }
  .
\end{align}
Combining this with 
\eqref{eq:exp_moment_small_noise}
proves that
for all $ \rho \in (0,\infty) $
it holds that
\begin{align}
&
  \sup_{ t \in [0,T] }
  \big\| X_t - Y_t
  \big\|_{
    L^r( \Omega; H )
  }
\leq
      \varepsilon \,
      \rho^{ ( \frac{ 1 }{ 2 } - \frac{ 1 }{ p } ) }
      \sqrt{ p - 1 }
        \left\|
          \sigma( Y )
        \right\|_{ 
          L^p( [ 0 , T ] \times \Omega ; HS( U, H ) )
        }
\\ &
  \cdot
  \exp\!\left(
    \smallint_{0}^T
    c
    +
    \tfrac{
      ( \frac{ 1 }{ 2 } - \frac{ 1 }{ p } ) 
    }{
      \rho
    }
    +
    \tfrac{ 
      \beta 
    }{ 
      e^{ \alpha s } 
    }
    \!
    \left[ 
      \tfrac{ 1 }{ q_0 }
      +
      \tfrac{ 1 }{ q_1 }
    \right]
    +
    \tfrac{ 
      \hat{ \beta }
    }{ 
      e^{ \hat{ \alpha } s } 
    }
    \!
    \left[ 
      \tfrac{ 1 }{ \hat{q}_0 }
      +
      \tfrac{ 1 }{ \hat{q}_1 }
    \right]
    ds
  \right)
  \left|
  \E\Big[
    e^{ U_0( X_0 ) }
  \Big]
  \right|^{
    \left[ 
      \frac{ 1 }{ q_0 } + \frac{ 1 }{ q_1 } 
    \right]
  }
  \left|
  \E\!\left[
    e^{
      \hat{U}_0( X_0 ) 
    }
  \right]
  \right|^{
    \left[
      \frac{ 1 }{ \hat{q}_0 }
      +
      \frac{ 1 }{ \hat{q}_1 }
    \right]
  }
  .
\nonumber
\end{align}
This and the identity
\begin{equation}
\begin{split}
&
  \inf_{ 
    \rho \in (0,\infty)
  }
  \left[
      \rho^{ ( \frac{ 1 }{ 2 } - \frac{ 1 }{ p } ) }
  \exp\!\left(
    \tfrac{
      T \, ( \frac{ 1 }{ 2 } - \frac{ 1 }{ p } ) 
    }{
      \rho
    }
  \right) 
  \right]
=
  \exp\!\left(
  \inf_{ 
    \rho \in (0,\infty)
  }
  \left[
    \tfrac{
      T \, ( \frac{ 1 }{ 2 } - \frac{ 1 }{ p } ) 
    }{
      \rho
    }
    +
    \ln( \rho )
    \,
      ( \tfrac{ 1 }{ 2 } - \tfrac{ 1 }{ p } ) 
  \right]
  \right) 
\\ & =
  \exp\!\left(
  \left(
  \min_{ 
    \rho \in (0,\infty)
  }
    \big[
      T \rho
      -
      \ln( \rho )
    \big]
  \right)
      ( \tfrac{ 1 }{ 2 } - \tfrac{ 1 }{ p } ) 
  \right) 
=
  \exp\!\left(
    \big[
      1
      -
      \ln( \tfrac{ 1 }{ T } )
    \big]
    \,
    \big[ 
      \tfrac{ 1 }{ 2 } - \tfrac{ 1 }{ p } 
    \big] 
  \right) 
\\ & =
  \exp\!\left(
    \big[
      1
      +
      \ln( T )
    \big]
    \,
    \big[ 
      \tfrac{ 1 }{ 2 } - \tfrac{ 1 }{ p } 
    \big] 
  \right) 
  =
  T^{ ( \frac{ 1 }{ 2 } - \frac{ 1 }{ p } ) }
  \exp\!\left(
    \tfrac{ 1 }{ 2 } - \tfrac{ 1 }{ p } 
  \right)
\end{split}
\end{equation}
complete
the proof of Corollary~\ref{cor:small_noise2}.
\end{proof}

\subsubsection*{Acknowledgements}

Ryan Kurniawan, Marco Noll and 
Christoph Schwab are gratefully acknowledged
for pointing out several misprints to us
and for useful comments regarding 
notation and the overall presentation 
of this work.
Special thanks are due to Andrew Stuart
for his suggestion to view a time-discrete
numerical approximation as a perturbation
of the solution process of the SDE.

This project has been partially supported
by the research project 
``Numerical approximations of stochastic
differential equations with non-globally
Lipschitz continuous coefficients''
funded by the German Research Foundation.

%\bibliographystyle{acm}
%\bibliography{../../Bib/bibfile}
\def\cprime{$'$} \def\cprime{$'$}

\end{document}